\documentclass[10pt, leqno]{amsart}

\usepackage{
  amsmath,
  amssymb,
  tikz,
  fancyhdr,
  enumitem,
  setspace,
  mathdots,
  multirow,
  mathabx,
  upgreek,
  ytableau
}

\usetikzlibrary{arrows,calc,matrix,shapes}

\usepackage[colorlinks=true, pdfstartview=FitV,linkcolor=black,citecolor=black,urlcolor=black]{hyperref}

\usepackage[vcentermath,enableskew]{youngtab}
\usepackage{mathrsfs}
\usepackage[center,small,sc]{caption}

\usepackage{multicol}
\usepackage{soul}

\usepackage[all]{xy}

\theoremstyle{plain}
\newtheorem{theorem}{Theorem}[section]
\newtheorem{lemma}[theorem]{Lemma}
\newtheorem{proposition}[theorem]{Proposition}
\newtheorem{corollary}[theorem]{Corollary}

\theoremstyle{definition}
\newtheorem{definition}[theorem]{Definition}
\newtheorem{example}[theorem]{Example}
\newtheorem{algorithm}[theorem]{Algorithm}
\newtheorem{conjecture}[theorem]{Conjecture}
\newtheorem{convention}[theorem]{Convention}
\newtheorem{remark}[theorem]{Remark}
\numberwithin{equation}{section} \numberwithin{figure}{section}
\numberwithin{table}{section}

\newcommand{\matr}[2]{\left( \begin{matrix} #1 \\ #2 \end{matrix} \right)}

\newcommand{\seteq}{\mathbin{:=}}

\newcommand{\C}{\mathfrak{C}}

\newcommand\skewpar{/}

\newcommand\Sym{\mathfrak{S}}
\newcommand\Inv{\mathcal{I}}
\newcommand\NI{\mathcal{NI}}
\newcommand\SYT{\mathcal{S}}
\newcommand\RYT{\mathcal{R}}
\newcommand\PT{\mathcal{P}}
\newcommand\RPT{\mathfrak{P}}
\newcommand{\Ae}{\sS}
\newcommand{\AAE}{\mathcal{AE}}
\newcommand{\cmA}{\mathsf{A}}
\newcommand{\Ss}{\mathfrak{B}}
\newcommand{\sS}{\mathfrak{D}}
\newcommand{\sfS}{\mathcal{S}}

\newcommand{\sfM}{\mathsf{M}}
\newcommand{\sfR}{\mathsf{R}}
\newcommand{\trace}{\mathrm{Tr}}
\newcommand\NR{\overline{\mathsf{R}}}
\newcommand\lm{{\lambda\skewpar\mu}}

\newcommand{\OG}{\mathrm{O}}

\newcommand{\Z}{\mathbb{Z}}

\newcommand{\dbs}[2]{\overset{#2}{\underset{#1}{*}} }
\newcommand{\pls}[1]{\underset{#1}{*}}
\newcommand{\upp}[1]{\underset{#1}{\uparrow}}
\newcommand{\dw}[1]{\overset{#1}{\downarrow}}

\newcommand{\nc}{\newcommand}

\newcommand\cb[1]{\binom{#1}{\lfloor \frac{#1}2\rfloor}}

\nc{\Hom}{\mathrm{Hom}}
\nc{\cl}{\mathrm{cl}}
\nc{\tens}{\mathop\otimes}
\nc{\wt}{\mathrm{wt}}
\nc{\het}{\mathrm{ht}}
\nc{\mx}{\mathrm{max}}
\nc{\Supp}{\mathrm{Supp}}
\nc{\cont}{\mathrm{cont}}
\nc{\rl}{\mathsf{Q}}
\nc{\cm}{\mathsf{A}}
\nc{\Sh}{\mathsf{Sh}}
\nc{\s}{\mathsf{s}}
\nc{\row}{\mathsf{row}}
\nc{\col}{\mathsf{col}}
\nc{\Rnorm}{R^{\rm{norm}}}
\nc{\Runiv}{R^{\rm{univ}}}
\nc{\aff}{\mathrm{aff}}
\nc{\La}{\Lambda}
\nc{\la}{\lambda}
\nc{\oal}{\alpha} 
\nc{\al}{\alpha}
\nc{\be}{\beta}
\nc{\vp}{\varphi}
\nc{\Oint}{\mathcal{O}_{{\rm int}}}
\nc{\Oqint}{\mathcal{O}^q_{{\rm int}}}

\nc{\eit}{\tilde{e}_i}
\nc{\fit}{\tilde{f}_i}
\nc{\ejt}{\tilde{e}_j}
\nc{\fjt}{\tilde{f}_j}
\nc{\YW}{\mathsf{Y}}

\newcommand{\yy}{ \mathbb{Y}}
\newcommand{\ofw}{ \mathbf{\Lambda}}
\newcommand{\ula}{ {\underline{\lambda}}}
\def\g{\mathfrak g}
\def\h{\mathfrak h}
\def\Z{\mathbb Z}
\def\C{\mathbb C}
\def\Q{\mathbb Q}
\def\A{\mathbb A}

\def\ZZ{\mathcal Z}
\def\YY{\mathcal Y}

\def\ve{\varepsilon}
\def\seteq{\mathbin{:=}}

\newcommand{\ot}{12}
\newcommand{\oo}{11}
\newcommand{\te}{10}
\nc\uqpg{U_q'(\g)}
\newcommand{\six}{{\color{red}6}}
\newcommand{\fv}{{\bf \color{blue}5}}
\newcommand{\bfo}{{\color{red}4}}
\newcommand{\tit}{{\color{yellow}13}}
\newcommand{\smax}{ {\rm smax}^+}
\newcommand{\hmaxp}{ {\rm smax}^+_{\Ss}}
\newcommand{\tmaxp}{ {\rm smax}^+_{\sS}}

\newcommand{\hwsix}
{\resizebox{2.4 cm}{2.4 cm}{\xy
(0,0)*{}="B1";(6,0)*{}="B2";
(0,6)*{}="T1";(6,6)*{}="T2";
"T1"; "B1" **\dir{-};"T2"; "B2" **\dir{-};"T2"; "T1" **\dir{-};"B2"; "B1" **\dir{-};
"T1"+(-30,0); "T2"+(-30,0) **\dir{-};"T1"+(-30,0); "B1"+(-30,0) **\dir{-};"B2"+(-30,0); "B1"+(-30,0) **\dir{-};
"T1"+(-24,0); "T2"+(-24,0) **\dir{-};"T1"+(-24,0); "B1"+(-24,0) **\dir{-};"B2"+(-24,0); "B1"+(-24,0) **\dir{-};
"T1"+(-24,6); "T2"+(-24,6) **\dir{-};"T1"+(-24,6); "B1"+(-24,6) **\dir{-};
"T1"+(-18,0); "T2"+(-18,0) **\dir{-};"T1"+(-18,0); "B1"+(-18,0) **\dir{-};"B2"+(-18,0); "B1"+(-18,0) **\dir{-};
"T1"+(-18,6); "T2"+(-18,6) **\dir{-};"T1"+(-18,6); "B1"+(-18,6) **\dir{-};
"T1"+(-18,12); "T2"+(-18,12) **\dir{-};"T1"+(-18,12); "B1"+(-18,12) **\dir{-};
"T1"+(-12,0); "T2"+(-12,0) **\dir{-};"T1"+(-12,0); "B1"+(-12,0) **\dir{-};"B2"+(-12,0); "B1"+(-12,0) **\dir{-};
"T1"+(-12,6); "T2"+(-12,6) **\dir{-};"T1"+(-12,6); "B1"+(-12,6) **\dir{-};
"T1"+(-12,12); "T2"+(-12,12) **\dir{-};"T1"+(-12,12); "B1"+(-12,12) **\dir{-};
"T1"+(-12,18); "T2"+(-12,18) **\dir{-};"T1"+(-12,18); "B1"+(-12,18) **\dir{-};
"T1"+(-6,0); "T2"+(-6,0) **\dir{-};"T1"+(-6,0); "B1"+(-6,0) **\dir{-};"B2"+(-6,0); "B1"+(-6,0) **\dir{-};
"T1"+(-6,6); "T2"+(-6,6) **\dir{-};"T1"+(-6,6); "B1"+(-6,6) **\dir{-};
"T1"+(-6,12); "B1"+(-6,12) **\dir{-};"T2"+(-6,12); "B2"+(-6,12) **\dir{-};"T2"+(-6,12); "T1"+(-6,12) **\dir{-};
"T1"+(-6,18); "B1"+(-6,18) **\dir{-};"T2"+(-6,18); "B2"+(-6,18) **\dir{-};"T2"+(-6,18); "T1"+(-6,18) **\dir{-};
"T1"+(-6,24); "B1"+(-6,24) **\dir{-};"T2"+(-6,24); "B2"+(-6,24) **\dir{-};"T2"+(-6,24); "T1"+(-6,24) **\dir{-};
"T1"+(0,6); "B1"+(0,6) **\dir{-};"T2"+(0,6); "B2"+(0,6) **\dir{-};"T2"+(0,6); "T1"+(0,6) **\dir{-};
"T1"+(0,12); "B1"+(0,12) **\dir{-};"T2"+(0,12); "B2"+(0,12) **\dir{-};"T2"+(0,12); "T1"+(0,12) **\dir{-};
"T1"+(0,18); "B1"+(0,18) **\dir{-};"T2"+(0,18); "B2"+(0,18) **\dir{-};"T2"+(0,18); "T1"+(0,18) **\dir{-};
"T1"+(0,24); "B1"+(0,24) **\dir{-};"T2"+(0,24); "B2"+(0,24) **\dir{-};"T2"+(0,24); "T1"+(0,24) **\dir{-};
"T1"+(0,30); "B1"+(0,30) **\dir{-};"T2"+(0,30); "B2"+(0,30) **\dir{-};"T2"+(0,30); "T1"+(0,30) **\dir{-};
"T2"; "B1" **\dir{-};"T2"+(-6,0); "B1"+(-6,0) **\dir{-};"T2"+(-12,0); "B1"+(-12,0) **\dir{-};
"T2"+(-18,0); "B1"+(-18,0) **\dir{-};"T2"+(-24,0); "B1"+(-24,0) **\dir{-};
"T2"+(-30,0); "B1"+(-30,0) **\dir{-};
"T2"+(0,-0.5); "B1"+(0.5,0) **\dir{.};"T2"+(0,-1); "B1"+(1,0) **\dir{.};"T2"+(0,-1.5); "B1"+(1.5,0) **\dir{.};"T2"+(0,-2); "B1"+(2,0) **\dir{.};
"T2"+(0,-2.5); "B1"+(2.5,0) **\dir{.};"T2"+(0,-3); "B1"+(3,0) **\dir{.};"T2"+(0,-3.5); "B1"+(3.5,0) **\dir{.};"T2"+(0,-4); "B1"+(4,0) **\dir{.};
"T2"+(0,-4.5); "B1"+(4.5,0) **\dir{.};"T2"+(0,-5); "B1"+(5,0) **\dir{.};
"T2"+(-6,-0.5); "B1"+(-5.5,0) **\dir{.};"T2"+(-6,-1); "B1"+(-5,0) **\dir{.};"T2"+(-6,-1.5); "B1"+(-4.5,0) **\dir{.};"T2"+(-6,-2); "B1"+(-4,0) **\dir{.};
"T2"+(-6,-2.5); "B1"+(-3.5,0) **\dir{.};"T2"+(-6,-3); "B1"+(-3,0) **\dir{.};"T2"+(-6,-3.5); "B1"+(-2.5,0) **\dir{.};"T2"+(-6,-4); "B1"+(-2,0) **\dir{.};
"T2"+(-6,-4.5); "B1"+(-1.5,0) **\dir{.};"T2"+(-6,-5); "B1"+(-1,0) **\dir{.};
"T2"+(-12,-0.5); "B1"+(-11.5,0) **\dir{.};"T2"+(-12,-1); "B1"+(-11,0) **\dir{.};"T2"+(-12,-1.5); "B1"+(-10.5,0) **\dir{.};"T2"+(-12,-2); "B1"+(-10,0) **\dir{.};
"T2"+(-12,-2.5); "B1"+(-9.5,0) **\dir{.};"T2"+(-12,-3); "B1"+(-9,0) **\dir{.};"T2"+(-12,-3.5); "B1"+(-8.5,0) **\dir{.};"T2"+(-12,-4); "B1"+(-8,0) **\dir{.};
"T2"+(-12,-4.5); "B1"+(-7.5,0) **\dir{.};"T2"+(-12,-5); "B1"+(-7,0) **\dir{.};
"T2"+(-18,-0.5); "B1"+(-17.5,0) **\dir{.};"T2"+(-18,-1); "B1"+(-17,0) **\dir{.};"T2"+(-18,-1.5); "B1"+(-16.5,0) **\dir{.};"T2"+(-18,-2); "B1"+(-16,0) **\dir{.};
"T2"+(-18,-2.5); "B1"+(-15.5,0) **\dir{.};"T2"+(-18,-3); "B1"+(-15,0) **\dir{.};"T2"+(-18,-3.5); "B1"+(-14.5,0) **\dir{.};"T2"+(-18,-4); "B1"+(-14,0) **\dir{.};
"T2"+(-18,-4.5); "B1"+(-13.5,0) **\dir{.};"T2"+(-18,-5); "B1"+(-13,0) **\dir{.};
"T2"+(-24,-0.5); "B1"+(-23.5,0) **\dir{.};"T2"+(-24,-1); "B1"+(-23,0) **\dir{.};"T2"+(-24,-1.5); "B1"+(-22.5,0) **\dir{.};"T2"+(-24,-2); "B1"+(-22,0) **\dir{.};
"T2"+(-24,-2.5); "B1"+(-21.5,0) **\dir{.};"T2"+(-24,-3); "B1"+(-22,0) **\dir{.};"T2"+(-24,-3.5); "B1"+(-20.5,0) **\dir{.};"T2"+(-24,-4); "B1"+(-20,0) **\dir{.};
"T2"+(-24,-4.5); "B1"+(-19.5,0) **\dir{.};"T2"+(-24,-5); "B1"+(-19,0) **\dir{.};
"T2"+(-30,-0.5); "B1"+(-29.5,0) **\dir{.};"T2"+(-30,-1); "B1"+(-29,0) **\dir{.};"T2"+(-30,-1.5); "B1"+(-28.5,0) **\dir{.};"T2"+(-30,-2); "B1"+(-28,0) **\dir{.};
"T2"+(-30,-2.5); "B1"+(-27.5,0) **\dir{.};"T2"+(-30,-3); "B1"+(-27,0) **\dir{.};"T2"+(-30,-3.5); "B1"+(-26.5,0) **\dir{.};"T2"+(-30,-4); "B1"+(-26,0) **\dir{.};
"T2"+(-30,-4.5); "B1"+(-25.5,0) **\dir{.};"T2"+(-30,-5); "B1"+(-25,0) **\dir{.};
(4.5,2)*{1};(1.5,4)*{0};  (3,9)*{2};(3,15)*{3};(3,21)*{4};(3,27)*{5};(3,33)*{6};
(-1.5,2)*{0};(-4.5,4)*{1};(-3,9)*{2};(-3,15)*{3};(-3,21)*{4};(-3,27)*{5};
(-7.5,2)*{1};(-10.5,4)*{0};(-9,9)*{2};(-9,15)*{3};(-9,21)*{4};
(-13.5,2)*{0};(-16.5,4)*{1};(-15,9)*{2};(-15,15)*{3};
(-19.5,2)*{1};(-22.5,4)*{0};(-21,9)*{2};
(-25.5,2)*{0};(-28.5,4)*{1};
\endxy}}

\newcommand{\hwtwo}
{\resizebox{0.8 cm}{0.8 cm}{\xy
(0,0)*{}="B1";(6,0)*{}="B2";
(0,6)*{}="T1";(6,6)*{}="T2";
"T1"; "B1" **\dir{-};"T2"; "B2" **\dir{-};"T2"; "T1" **\dir{-};"B2"; "B1" **\dir{-};
"T1"+(-6,0); "T2"+(-6,0) **\dir{-};"T1"+(-6,0); "B1"+(-6,0) **\dir{-};"B2"+(-6,0); "B1"+(-6,0) **\dir{-};
"T1"+(0,6); "B1"+(0,6) **\dir{-};"T2"+(0,6); "B2"+(0,6) **\dir{-};"T2"+(0,6); "T1"+(0,6) **\dir{-};
"T2"; "B1" **\dir{-};"T2"+(-6,0); "B1"+(-6,0) **\dir{-};
"T2"+(0,-0.5); "B1"+(0.5,0) **\dir{.};"T2"+(0,-1); "B1"+(1,0) **\dir{.};"T2"+(0,-1.5); "B1"+(1.5,0) **\dir{.};"T2"+(0,-2); "B1"+(2,0) **\dir{.};
"T2"+(0,-2.5); "B1"+(2.5,0) **\dir{.};"T2"+(0,-3); "B1"+(3,0) **\dir{.};"T2"+(0,-3.5); "B1"+(3.5,0) **\dir{.};"T2"+(0,-4); "B1"+(4,0) **\dir{.};
"T2"+(0,-4.5); "B1"+(4.5,0) **\dir{.};"T2"+(0,-5); "B1"+(5,0) **\dir{.};
"T2"+(-6,-0.5); "B1"+(-5.5,0) **\dir{.};"T2"+(-6,-1); "B1"+(-5,0) **\dir{.};"T2"+(-6,-1.5); "B1"+(-4.5,0) **\dir{.};"T2"+(-6,-2); "B1"+(-4,0) **\dir{.};
"T2"+(-6,-2.5); "B1"+(-3.5,0) **\dir{.};"T2"+(-6,-3); "B1"+(-3,0) **\dir{.};"T2"+(-6,-3.5); "B1"+(-2.5,0) **\dir{.};"T2"+(-6,-4); "B1"+(-2,0) **\dir{.};
"T2"+(-6,-4.5); "B1"+(-1.5,0) **\dir{.};"T2"+(-6,-5); "B1"+(-1,0) **\dir{.};
(4.5,2)*{0};(1.5,4)*{1};(3,9)*{2};
(-1.5,2)*{1};(-4.5,4)*{0};
\endxy}}

\newcommand{\hwLfivefourB}
{\xy
(0,0)*{}="B1";(6,0)*{}="B2";
(0,6)*{}="T1";(6,6)*{}="T2";
"T1"; "B1" **\dir{-};"T2"; "B2" **\dir{-};"T2"; "T1" **\dir{-};"B2"; "B1" **\dir{-};
"T1"+(-12,0); "T2"+(-12,0) **\dir{-};"T1"+(-12,0); "B1"+(-12,0) **\dir{-};"B2"+(-12,0); "B1"+(-12,0) **\dir{-};
"T1"+(-12,6); "T2"+(-12,6) **\dir{-};"T1"+(-12,6); "B1"+(-12,6) **\dir{-};
"T1"+(-6,0); "T2"+(-6,0) **\dir{-};"T1"+(-6,0); "B1"+(-6,0) **\dir{-};"B2"+(-6,0); "B1"+(-6,0) **\dir{-};
"T1"+(-6,6); "T2"+(-6,6) **\dir{-};"T1"+(-6,6); "B1"+(-6,6) **\dir{-};
"T1"+(-6,12); "B1"+(-6,12) **\dir{-};"T2"+(-6,12); "B2"+(-6,12) **\dir{-};"T2"+(-6,12); "T1"+(-6,12) **\dir{-};
"T1"+(0,6); "B1"+(0,6) **\dir{-};"T2"+(0,6); "B2"+(0,6) **\dir{-};"T2"+(0,6); "T1"+(0,6) **\dir{-};
"T1"+(0,12); "B1"+(0,12) **\dir{-};"T2"+(0,12); "B2"+(0,12) **\dir{-};"T2"+(0,12); "T1"+(0,12) **\dir{-};
"T1"+(0,18); "B1"+(0,18) **\dir{-};"T2"+(0,18); "B2"+(0,18) **\dir{-};"T2"+(0,18); "T1"+(0,18) **\dir{-};
"T2"; "B1" **\dir{-};"T2"+(-6,0); "B1"+(-6,0) **\dir{-};"T2"+(-12,0); "B1"+(-12,0) **\dir{-};
"T2"+(0,-0.5); "B1"+(0.5,0) **\dir{.};"T2"+(0,-1); "B1"+(1,0) **\dir{.};"T2"+(0,-1.5); "B1"+(1.5,0) **\dir{.};"T2"+(0,-2); "B1"+(2,0) **\dir{.};
"T2"+(0,-2.5); "B1"+(2.5,0) **\dir{.};"T2"+(0,-3); "B1"+(3,0) **\dir{.};"T2"+(0,-3.5); "B1"+(3.5,0) **\dir{.};"T2"+(0,-4); "B1"+(4,0) **\dir{.};
"T2"+(0,-4.5); "B1"+(4.5,0) **\dir{.};"T2"+(0,-5); "B1"+(5,0) **\dir{.};
"T2"+(-6,-0.5); "B1"+(-5.5,0) **\dir{.};"T2"+(-6,-1); "B1"+(-5,0) **\dir{.};"T2"+(-6,-1.5); "B1"+(-4.5,0) **\dir{.};"T2"+(-6,-2); "B1"+(-4,0) **\dir{.};
"T2"+(-6,-2.5); "B1"+(-3.5,0) **\dir{.};"T2"+(-6,-3); "B1"+(-3,0) **\dir{.};"T2"+(-6,-3.5); "B1"+(-2.5,0) **\dir{.};"T2"+(-6,-4); "B1"+(-2,0) **\dir{.};
"T2"+(-6,-4.5); "B1"+(-1.5,0) **\dir{.};"T2"+(-6,-5); "B1"+(-1,0) **\dir{.};
"T2"+(-12,-0.5); "B1"+(-11.5,0) **\dir{.};"T2"+(-12,-1); "B1"+(-11,0) **\dir{.};"T2"+(-12,-1.5); "B1"+(-10.5,0) **\dir{.};"T2"+(-12,-2); "B1"+(-10,0) **\dir{.};
"T2"+(-12,-2.5); "B1"+(-9.5,0) **\dir{.};"T2"+(-12,-3); "B1"+(-9,0) **\dir{.};"T2"+(-12,-3.5); "B1"+(-8.5,0) **\dir{.};"T2"+(-12,-4); "B1"+(-8,0) **\dir{.};
"T2"+(-12,-4.5); "B1"+(-7.5,0) **\dir{.};"T2"+(-12,-5); "B1"+(-7,0) **\dir{.};
(4.5,2)*{1};(1.5,4)*{0};  (3,9)*{2};(3,15)*{3};(3,21)*{4};
(-1.5,2)*{0};(-4.5,4)*{1};(-3,9)*{2};(-3,15)*{3};
(-7.5,2)*{1};(-10.5,4)*{0};(-9,9)*{2};
\endxy}
\newcommand{\hwLtwooneB}
{\xy
(0,0)*{}="B1";(6,0)*{}="B2";
(0,6)*{}="T1";(6,6)*{}="T2";
"T1"; "B1" **\dir{-};"T2"; "B2" **\dir{-};"T2"; "T1" **\dir{-};"B2"; "B1" **\dir{-};
"T1"+(-6,0); "T2"+(-6,0) **\dir{-};"T1"+(-6,0); "B1"+(-6,0) **\dir{-};"B2"+(-6,0); "B1"+(-6,0) **\dir{-};
"T1"+(0,6); "B1"+(0,6) **\dir{-};"T2"+(0,6); "B2"+(0,6) **\dir{-};"T2"+(0,6); "T1"+(0,6) **\dir{-};
"T1"+(0,12); "B1"+(0,12) **\dir{-};"T2"+(0,12); "B2"+(0,12) **\dir{-};"T2"+(0,12); "T1"+(0,12) **\dir{-};
"T1"+(0,18); "B1"+(0,18) **\dir{-};"T2"+(0,18); "B2"+(0,18) **\dir{-};"T2"+(0,18); "T1"+(0,18) **\dir{-};
"T1"+(0,24); "B1"+(0,24) **\dir{-};"T2"+(0,24); "B2"+(0,24) **\dir{-};"T2"+(0,24); "T1"+(0,24) **\dir{-};
"T2"; "B1" **\dir{-};"T2"+(-6,0); "B1"+(-6,0) **\dir{-};
"T2"+(0,-0.5); "B1"+(0.5,0) **\dir{.};"T2"+(0,-1); "B1"+(1,0) **\dir{.};"T2"+(0,-1.5); "B1"+(1.5,0) **\dir{.};"T2"+(0,-2); "B1"+(2,0) **\dir{.};
"T2"+(0,-2.5); "B1"+(2.5,0) **\dir{.};"T2"+(0,-3); "B1"+(3,0) **\dir{.};"T2"+(0,-3.5); "B1"+(3.5,0) **\dir{.};"T2"+(0,-4); "B1"+(4,0) **\dir{.};
"T2"+(0,-4.5); "B1"+(4.5,0) **\dir{.};"T2"+(0,-5); "B1"+(5,0) **\dir{.};
"T2"+(-6,-0.5); "B1"+(-5.5,0) **\dir{.};"T2"+(-6,-1); "B1"+(-5,0) **\dir{.};"T2"+(-6,-1.5); "B1"+(-4.5,0) **\dir{.};"T2"+(-6,-2); "B1"+(-4,0) **\dir{.};
"T2"+(-6,-2.5); "B1"+(-3.5,0) **\dir{.};"T2"+(-6,-3); "B1"+(-3,0) **\dir{.};"T2"+(-6,-3.5); "B1"+(-2.5,0) **\dir{.};"T2"+(-6,-4); "B1"+(-2,0) **\dir{.};
"T2"+(-6,-4.5); "B1"+(-1.5,0) **\dir{.};"T2"+(-6,-5); "B1"+(-1,0) **\dir{.};
(4.5,2)*{0};(1.5,4)*{1};(3,9)*{2};(3,15)*{3};(3,21)*{4};(3,27)*{5};
(-1.5,2)*{1};(-4.5,4)*{0};
\endxy}

\newcommand{\hwLtwooneBt}
{\xy
(0,0)*{}="B1";(6,0)*{}="B2";
(0,6)*{}="T1";(6,6)*{}="T2";
"T1"; "B1" **\dir{-};"T2"; "B2" **\dir{-};"T2"; "T1" **\dir{-};"B2"; "B1" **\dir{-};
"T1"+(-6,0); "T2"+(-6,0) **\dir{-};"T1"+(-6,0); "B1"+(-6,0) **\dir{-};"B2"+(-6,0); "B1"+(-6,0) **\dir{-};
"T1"+(-6,6); "T2"+(-6,6) **\dir{-};"T1"+(-6,6); "B1"+(-6,6) **\dir{-};
"T1"+(0,6); "B1"+(0,6) **\dir{-};"T2"+(0,6); "B2"+(0,6) **\dir{-};"T2"+(0,6); "T1"+(0,6) **\dir{-};
"T1"+(0,12); "B1"+(0,12) **\dir{-};"T2"+(0,12); "B2"+(0,12) **\dir{-};"T2"+(0,12); "T1"+(0,12) **\dir{-};
"T1"+(0,18); "B1"+(0,18) **\dir{-};"T2"+(0,18); "B2"+(0,18) **\dir{-};"T2"+(0,18); "T1"+(0,18) **\dir{-};
"T1"+(0,24); "B1"+(0,24) **\dir{-};"T2"+(0,24); "B2"+(0,24) **\dir{-};"T2"+(0,24); "T1"+(0,24) **\dir{-};
"T2"; "B1" **\dir{-};"T2"+(-6,0); "B1"+(-6,0) **\dir{-};
"T2"+(0,-0.5); "B1"+(0.5,0) **\dir{.};"T2"+(0,-1); "B1"+(1,0) **\dir{.};"T2"+(0,-1.5); "B1"+(1.5,0) **\dir{.};"T2"+(0,-2); "B1"+(2,0) **\dir{.};
"T2"+(0,-2.5); "B1"+(2.5,0) **\dir{.};"T2"+(0,-3); "B1"+(3,0) **\dir{.};"T2"+(0,-3.5); "B1"+(3.5,0) **\dir{.};"T2"+(0,-4); "B1"+(4,0) **\dir{.};
"T2"+(0,-4.5); "B1"+(4.5,0) **\dir{.};"T2"+(0,-5); "B1"+(5,0) **\dir{.};
"T2"+(-6,-0.5); "B1"+(-5.5,0) **\dir{.};"T2"+(-6,-1); "B1"+(-5,0) **\dir{.};"T2"+(-6,-1.5); "B1"+(-4.5,0) **\dir{.};"T2"+(-6,-2); "B1"+(-4,0) **\dir{.};
"T2"+(-6,-2.5); "B1"+(-3.5,0) **\dir{.};"T2"+(-6,-3); "B1"+(-3,0) **\dir{.};"T2"+(-6,-3.5); "B1"+(-2.5,0) **\dir{.};"T2"+(-6,-4); "B1"+(-2,0) **\dir{.};
"T2"+(-6,-4.5); "B1"+(-1.5,0) **\dir{.};"T2"+(-6,-5); "B1"+(-1,0) **\dir{.};
(4.5,2)*{1};(1.5,4)*{0};(3,9)*{2};(3,15)*{3};(3,21)*{4};(3,27)*{5};
(-1.5,2)*{0};(-4.5,4)*{1};(-3,9)*{2};
\endxy}

\newcommand{\hst}
{\xy
(0,0)*{}="B1";(6,0)*{}="B2";
(0,6)*{}="T1";(6,6)*{}="T2";
"T1"; "B1" **\dir{-};"T2"; "B2" **\dir{-};"T2"; "T1" **\dir{-};"B2"; "B1" **\dir{-};
"T2"; "B1" **\dir{-};
"T2"+(0,-0.5); "B1"+(0.5,0) **\dir{.};"T2"+(0,-1); "B1"+(1,0) **\dir{.};"T2"+(0,-1.5); "B1"+(1.5,0) **\dir{.};"T2"+(0,-2); "B1"+(2,0) **\dir{.};
"T2"+(0,-2.5); "B1"+(2.5,0) **\dir{.};"T2"+(0,-3); "B1"+(3,0) **\dir{.};"T2"+(0,-3.5); "B1"+(3.5,0) **\dir{.};"T2"+(0,-4); "B1"+(4,0) **\dir{.};
"T2"+(0,-4.5); "B1"+(4.5,0) **\dir{.};"T2"+(0,-5); "B1"+(5,0) **\dir{.};
(4.5,2)*{1};(1.5,4)*{0};
\endxy}
\newcommand{\hwLfivefourBs}
{\xy
(0,0)*{}="B1";(6,0)*{}="B2";
(0,6)*{}="T1";(6,6)*{}="T2";
"T1"; "B1" **\dir{-};"T2"; "B2" **\dir{-};"T2"; "T1" **\dir{-};"B2"; "B1" **\dir{-};
"T1"+(-6,0); "T2"+(-6,0) **\dir{-};"T1"+(-6,0); "B1"+(-6,0) **\dir{-};"B2"+(-6,0); "B1"+(-6,0) **\dir{-};
"T1"+(-6,6); "T2"+(-6,6) **\dir{-};"T1"+(-6,6); "B1"+(-6,6) **\dir{-};
"T1"+(-6,12); "B1"+(-6,12) **\dir{-};"T2"+(-6,12); "B2"+(-6,12) **\dir{-};"T2"+(-6,12); "T1"+(-6,12) **\dir{-};
"T1"+(0,6); "B1"+(0,6) **\dir{-};"T2"+(0,6); "B2"+(0,6) **\dir{-};"T2"+(0,6); "T1"+(0,6) **\dir{-};
"T1"+(0,12); "B1"+(0,12) **\dir{-};"T2"+(0,12); "B2"+(0,12) **\dir{-};"T2"+(0,12); "T1"+(0,12) **\dir{-};
"T1"+(0,18); "B1"+(0,18) **\dir{-};"T2"+(0,18); "B2"+(0,18) **\dir{-};"T2"+(0,18); "T1"+(0,18) **\dir{-};
"T2"; "B1" **\dir{-};"T2"+(-6,0); "B1"+(-6,0) **\dir{-};
"T2"+(0,-0.5); "B1"+(0.5,0) **\dir{.};"T2"+(0,-1); "B1"+(1,0) **\dir{.};"T2"+(0,-1.5); "B1"+(1.5,0) **\dir{.};"T2"+(0,-2); "B1"+(2,0) **\dir{.};
"T2"+(0,-2.5); "B1"+(2.5,0) **\dir{.};"T2"+(0,-3); "B1"+(3,0) **\dir{.};"T2"+(0,-3.5); "B1"+(3.5,0) **\dir{.};"T2"+(0,-4); "B1"+(4,0) **\dir{.};
"T2"+(0,-4.5); "B1"+(4.5,0) **\dir{.};"T2"+(0,-5); "B1"+(5,0) **\dir{.};
"T2"+(-6,-0.5); "B1"+(-5.5,0) **\dir{.};"T2"+(-6,-1); "B1"+(-5,0) **\dir{.};"T2"+(-6,-1.5); "B1"+(-4.5,0) **\dir{.};"T2"+(-6,-2); "B1"+(-4,0) **\dir{.};
"T2"+(-6,-2.5); "B1"+(-3.5,0) **\dir{.};"T2"+(-6,-3); "B1"+(-3,0) **\dir{.};"T2"+(-6,-3.5); "B1"+(-2.5,0) **\dir{.};"T2"+(-6,-4); "B1"+(-2,0) **\dir{.};
"T2"+(-6,-4.5); "B1"+(-1.5,0) **\dir{.};"T2"+(-6,-5); "B1"+(-1,0) **\dir{.};
(4.5,2)*{0};(1.5,4)*{1};  (3,9)*{2};(3,15)*{3};(3,21)*{4};
(-1.5,2)*{1};(-4.5,4)*{0};(-3,9)*{2};(-3,15)*{3};
\endxy}

\newcommand{\hwLtwooneBp}
{\xy
(0,0)*{}="B1";(6,0)*{}="B2";
(0,6)*{}="T1";(6,6)*{}="T2";
"T1"; "B1" **\dir{-};"T2"; "B2" **\dir{-};"T2"; "T1" **\dir{-};"B2"; "B1" **\dir{-};
"T1"+(-6,0); "T2"+(-6,0) **\dir{-};"T1"+(-6,0); "B1"+(-6,0) **\dir{-};"B2"+(-6,0); "B1"+(-6,0) **\dir{-};
"T1"+(0,6); "B1"+(0,6) **\dir{-};"T2"+(0,6); "B2"+(0,6) **\dir{-};"T2"+(0,6); "T1"+(0,6) **\dir{-};
"T1"+(0,12); "B1"+(0,12) **\dir{-};"T2"+(0,12); "B2"+(0,12) **\dir{-};"T2"+(0,12); "T1"+(0,12) **\dir{-};
"T1"+(0,18); "B1"+(0,18) **\dir{-};"T2"+(0,18); "B2"+(0,18) **\dir{-};"T2"+(0,18); "T1"+(0,18) **\dir{-};
"T1"+(0,24); "B1"+(0,24) **\dir{-};"T2"+(0,24); "B2"+(0,24) **\dir{-};"T2"+(0,24); "T1"+(0,24) **\dir{-};
"T2"; "B1" **\dir{-};"T2"+(-6,0); "B1"+(-6,0) **\dir{-};
"T2"+(0,-0.5); "B1"+(0.5,0) **\dir{.};"T2"+(0,-1); "B1"+(1,0) **\dir{.};"T2"+(0,-1.5); "B1"+(1.5,0) **\dir{.};"T2"+(0,-2); "B1"+(2,0) **\dir{.};
"T2"+(0,-2.5); "B1"+(2.5,0) **\dir{.};"T2"+(0,-3); "B1"+(3,0) **\dir{.};"T2"+(0,-3.5); "B1"+(3.5,0) **\dir{.};"T2"+(0,-4); "B1"+(4,0) **\dir{.};
"T2"+(0,-4.5); "B1"+(4.5,0) **\dir{.};"T2"+(0,-5); "B1"+(5,0) **\dir{.};
"T2"+(-6,-0.5); "B1"+(-5.5,0) **\dir{.};"T2"+(-6,-1); "B1"+(-5,0) **\dir{.};"T2"+(-6,-1.5); "B1"+(-4.5,0) **\dir{.};"T2"+(-6,-2); "B1"+(-4,0) **\dir{.};
"T2"+(-6,-2.5); "B1"+(-3.5,0) **\dir{.};"T2"+(-6,-3); "B1"+(-3,0) **\dir{.};"T2"+(-6,-3.5); "B1"+(-2.5,0) **\dir{.};"T2"+(-6,-4); "B1"+(-2,0) **\dir{.};
"T2"+(-6,-4.5); "B1"+(-1.5,0) **\dir{.};"T2"+(-6,-5); "B1"+(-1,0) **\dir{.};
(4.5,2)*{1};(1.5,4)*{0};(3,9)*{2};(3,15)*{3};(3,21)*{4};(3,27)*{5};
(-1.5,2)*{0};(-4.5,4)*{1};
\endxy}

\newcommand{\hwLfiveBfn}
{\xy
(0,0)*{}="B1";(6,0)*{}="B2";
(0,6)*{}="T1";(6,6)*{}="T2";
"T1"; "B1" **\dir{-};"T2"; "B2" **\dir{-};"T2"; "T1" **\dir{-};"B2"; "B1" **\dir{-};
"T1"+(-6,0); "T2"+(-6,0) **\dir{-};"T1"+(-6,0); "B1"+(-6,0) **\dir{-};"B2"+(-6,0); "B1"+(-6,0) **\dir{-};
"T1"+(0,6); "B1"+(0,6) **\dir{-};"T2"+(0,6); "B2"+(0,6) **\dir{-};"T2"+(0,6); "T1"+(0,6) **\dir{-};
"T2"+(0,-3); "T2"+(-12,-3) **\dir{-};
"T2"+(0,-3.5); "T2"+(-12,-3.5) **\dir{.};"T2"+(0,-4); "T2"+(-12,-4) **\dir{.};"T2"+(0,-4.5); "T2"+(-12,-4.5) **\dir{.};"T2"+(0,-5); "T2"+(-12,-5) **\dir{.};"T2"+(0,-5.5); "T2"+(-12,-5.5) **\dir{.};
(3,1.5)*{7};(3,4.5)*{7};  (3,9)*{6};
(-3,1.5)*{7};(-3,4.5)*{7};
\endxy}

\newcommand{\hwLfiveBB}
{\resizebox{2.0 cm}{2.0 cm}{\xy
(0,0)*{}="B1";(6,0)*{}="B2";
(0,6)*{}="T1";(6,6)*{}="T2";
"T1"; "B1" **\dir{-};"T2"; "B2" **\dir{-};"T2"; "T1" **\dir{-};"B2"; "B1" **\dir{-};
"T1"+(-6,0); "T2"+(-6,0) **\dir{-};"T1"+(-6,0); "B1"+(-6,0) **\dir{-};"B2"+(-6,0); "B1"+(-6,0) **\dir{-};
"T1"+(-12,0); "T2"+(-12,0) **\dir{-};"T1"+(-12,0); "B1"+(-12,0) **\dir{-};"B2"+(-12,0); "B1"+(-12,0) **\dir{-};
"T1"+(-18,0); "T2"+(-18,0) **\dir{-};"T1"+(-18,0); "B1"+(-18,0) **\dir{-};"B2"+(-18,0); "B1"+(-18,0) **\dir{-};
"T1"+(0,6); "B1"+(0,6) **\dir{-};"T2"+(0,6); "B2"+(0,6) **\dir{-};"T2"+(0,6); "T1"+(0,6) **\dir{-};
"T1"+(-6,6); "T2"+(-6,6) **\dir{-};"T1"+(-6,6); "B1"+(-6,6) **\dir{-};
"T1"+(0,12); "B1"+(0,12) **\dir{-};"T2"+(0,12); "B2"+(0,12) **\dir{-};"T2"+(0,12); "T1"+(0,12) **\dir{-};
"T1"+(0,18); "B1"+(0,18) **\dir{-};"T2"+(0,18); "B2"+(0,18) **\dir{-};"T2"+(0,18); "T1"+(0,18) **\dir{-};
"T1"+(-6,12); "B1"+(-6,12) **\dir{-};"T2"+(-6,12); "B2"+(-6,12) **\dir{-};"T2"+(-6,12); "T1"+(-6,12) **\dir{-};
"T1"+(-12,6); "B1"+(-12,6) **\dir{-};"T2"+(-12,6); "B2"+(-12,6) **\dir{-};"T2"+(-12,6); "T1"+(-12,6) **\dir{-};
"T2"; "B1" **\dir{-};"T2"+(-6,0); "B1"+(-6,0) **\dir{-};"T2"+(-12,0); "B1"+(-12,0) **\dir{-};"T2"+(-18,0); "B1"+(-18,0) **\dir{-};
"T2"+(0,-0.5); "B1"+(0.5,0) **\dir{.};"T2"+(0,-1); "B1"+(1,0) **\dir{.};"T2"+(0,-1.5); "B1"+(1.5,0) **\dir{.};"T2"+(0,-2); "B1"+(2,0) **\dir{.};
"T2"+(0,-2.5); "B1"+(2.5,0) **\dir{.};"T2"+(0,-3); "B1"+(3,0) **\dir{.};"T2"+(0,-3.5); "B1"+(3.5,0) **\dir{.};"T2"+(0,-4); "B1"+(4,0) **\dir{.};
"T2"+(0,-4.5); "B1"+(4.5,0) **\dir{.};"T2"+(0,-5); "B1"+(5,0) **\dir{.};
"T2"+(-6,-0.5); "B1"+(-5.5,0) **\dir{.};"T2"+(-6,-1); "B1"+(-5,0) **\dir{.};"T2"+(-6,-1.5); "B1"+(-4.5,0) **\dir{.};"T2"+(-6,-2); "B1"+(-4,0) **\dir{.};
"T2"+(-6,-2.5); "B1"+(-3.5,0) **\dir{.};"T2"+(-6,-3); "B1"+(-3,0) **\dir{.};"T2"+(-6,-3.5); "B1"+(-2.5,0) **\dir{.};"T2"+(-6,-4); "B1"+(-2,0) **\dir{.};
"T2"+(-6,-4.5); "B1"+(-1.5,0) **\dir{.};"T2"+(-6,-5); "B1"+(-1,0) **\dir{.};
"T2"+(-12,-0.5); "B1"+(-11.5,0) **\dir{.};"T2"+(-12,-1); "B1"+(-11,0) **\dir{.};"T2"+(-12,-1.5); "B1"+(-10.5,0) **\dir{.};"T2"+(-12,-2); "B1"+(-10,0) **\dir{.};
"T2"+(-12,-2.5); "B1"+(-9.5,0) **\dir{.};"T2"+(-12,-3); "B1"+(-9,0) **\dir{.};"T2"+(-12,-3.5); "B1"+(-8.5,0) **\dir{.};"T2"+(-12,-4); "B1"+(-8,0) **\dir{.};
"T2"+(-12,-4.5); "B1"+(-7.5,0) **\dir{.};"T2"+(-12,-5); "B1"+(-7,0) **\dir{.};
"T2"+(-18,-0.5); "B1"+(-17.5,0) **\dir{.};"T2"+(-18,-1); "B1"+(-17,0) **\dir{.};"T2"+(-18,-1.5); "B1"+(-16.5,0) **\dir{.};"T2"+(-18,-2); "B1"+(-16,0) **\dir{.};
"T2"+(-18,-2.5); "B1"+(-15.5,0) **\dir{.};"T2"+(-18,-3); "B1"+(-15,0) **\dir{.};"T2"+(-18,-3.5); "B1"+(-14.5,0) **\dir{.};"T2"+(-18,-4); "B1"+(-14,0) **\dir{.};
"T2"+(-18,-4.5); "B1"+(-13.5,0) **\dir{.};"T2"+(-18,-5); "B1"+(-13,0) **\dir{.};
(4.5,2)*{6};(1.5,4)*{7};  (3,9)*{5};(3,15)*{4};(3,21)*{3};
(-1.5,2)*{7};(-4.5,4)*{6};(-3,9)*{5};(-3,15)*{4};
(-7.5,2)*{6};(-10.5,4)*{7};(-9,9)*{5};
(-13.5,2)*{7};(-16.5,4)*{6};
\endxy}}

\newcommand{\hwLfourB}
{\resizebox{1.6cm}{1.6cm}{\xy
(0,0)*{}="B1";(6,0)*{}="B2";
(0,6)*{}="T1";(6,6)*{}="T2";
"T1"; "B1" **\dir{-};"T2"; "B2" **\dir{-};"T2"; "T1" **\dir{-};"B2"; "B1" **\dir{-};
"T1"+(-6,0); "T2"+(-6,0) **\dir{-};"T1"+(-6,0); "B1"+(-6,0) **\dir{-};"B2"+(-6,0); "B1"+(-6,0) **\dir{-};
"T1"+(-12,0); "T2"+(-12,0) **\dir{-};"T1"+(-12,0); "B1"+(-12,0) **\dir{-};"B2"+(-12,0); "B1"+(-12,0) **\dir{-};
"T1"+(0,6); "B1"+(0,6) **\dir{-};"T2"+(0,6); "B2"+(0,6) **\dir{-};"T2"+(0,6); "T1"+(0,6) **\dir{-};
"T1"+(-6,6); "T2"+(-6,6) **\dir{-};"T1"+(-6,6); "B1"+(-6,6) **\dir{-};
"T1"+(0,12); "B1"+(0,12) **\dir{-};"T2"+(0,12); "B2"+(0,12) **\dir{-};"T2"+(0,12); "T1"+(0,12) **\dir{-};
"T2"; "B1" **\dir{-};"T2"+(-6,0); "B1"+(-6,0) **\dir{-};"T2"+(-12,0); "B1"+(-12,0) **\dir{-};
"T2"+(0,-0.5); "B1"+(0.5,0) **\dir{.};"T2"+(0,-1); "B1"+(1,0) **\dir{.};"T2"+(0,-1.5); "B1"+(1.5,0) **\dir{.};"T2"+(0,-2); "B1"+(2,0) **\dir{.};
"T2"+(0,-2.5); "B1"+(2.5,0) **\dir{.};"T2"+(0,-3); "B1"+(3,0) **\dir{.};"T2"+(0,-3.5); "B1"+(3.5,0) **\dir{.};"T2"+(0,-4); "B1"+(4,0) **\dir{.};
"T2"+(0,-4.5); "B1"+(4.5,0) **\dir{.};"T2"+(0,-5); "B1"+(5,0) **\dir{.};
"T2"+(-6,-0.5); "B1"+(-5.5,0) **\dir{.};"T2"+(-6,-1); "B1"+(-5,0) **\dir{.};"T2"+(-6,-1.5); "B1"+(-4.5,0) **\dir{.};"T2"+(-6,-2); "B1"+(-4,0) **\dir{.};
"T2"+(-6,-2.5); "B1"+(-3.5,0) **\dir{.};"T2"+(-6,-3); "B1"+(-3,0) **\dir{.};"T2"+(-6,-3.5); "B1"+(-2.5,0) **\dir{.};"T2"+(-6,-4); "B1"+(-2,0) **\dir{.};
"T2"+(-6,-4.5); "B1"+(-1.5,0) **\dir{.};"T2"+(-6,-5); "B1"+(-1,0) **\dir{.};
"T2"+(-12,-0.5); "B1"+(-11.5,0) **\dir{.};"T2"+(-12,-1); "B1"+(-11,0) **\dir{.};"T2"+(-12,-1.5); "B1"+(-10.5,0) **\dir{.};"T2"+(-12,-2); "B1"+(-10,0) **\dir{.};
"T2"+(-12,-2.5); "B1"+(-9.5,0) **\dir{.};"T2"+(-12,-3); "B1"+(-9,0) **\dir{.};"T2"+(-12,-3.5); "B1"+(-8.5,0) **\dir{.};"T2"+(-12,-4); "B1"+(-8,0) **\dir{.};
"T2"+(-12,-4.5); "B1"+(-7.5,0) **\dir{.};"T2"+(-12,-5); "B1"+(-7,0) **\dir{.};
(4.5,2)*{1};(1.5,4)*{0};  (3,9)*{2};(3,15)*{3};
(-1.5,2)*{0};(-4.5,4)*{1};(-3,9)*{2};
(-7.5,2)*{1};(-10.5,4)*{0};
\endxy}}

\newcommand{\youngwallextp}
{\xy
(0,0)*{}="B1";(6,0)*{}="B2";
(0,6)*{}="T1";(6,6)*{}="T2";
"T1"; "B1" **\dir{-};"T2"; "B2" **\dir{-};"T2"; "T1" **\dir{-};"B2"; "B1" **\dir{-};
"T1"+(-6,0); "T2"+(-6,0) **\dir{-};"T1"+(-6,0); "B1"+(-6,0) **\dir{-};"B2"+(-6,0); "B1"+(-6,0) **\dir{-};
"T1"+(-12,0); "T2"+(-12,0) **\dir{-};"T1"+(-12,0); "B1"+(-12,0) **\dir{-};"B2"+(-12,0); "B1"+(-12,0) **\dir{-};
"T1"+(0,6); "B1"+(0,6) **\dir{-};"T2"+(0,6); "B2"+(0,6) **\dir{-};"T2"+(0,6); "T1"+(0,6) **\dir{-};
"T1"+(-6,6); "T2"+(-6,6) **\dir{-};"T1"+(-6,6); "B1"+(-6,6) **\dir{-};
"T1"+(0,12); "B1"+(0,12) **\dir{-};"T2"+(0,12); "B2"+(0,12) **\dir{-};"T2"+(0,12); "T1"+(0,12) **\dir{-};
"T1"+(0,24); "B1"+(0,24) **\dir{-};
"T2"+(0,24); "T1"+(0,24) **\dir{-};
"T1"+(0,18); "B1"+(0,18) **\dir{-};"T2"+(0,18); "T1"+(0,18) **\dir{-};"B2"+(0,18); "T2"+(0,18) **\dir{-};
"B1"+(0,24); "T2"+(0,24) **\dir{-};
"T2"; "B1" **\dir{-};"T2"+(-6,0); "B1"+(-6,0) **\dir{-};"T2"+(-12,0); "B1"+(-12,0) **\dir{-};
"T1"+(-6,9); "T2"+(0,9) **\dir{-};"T1"+(-6,9); "T1"+(-6,6) **\dir{-};
"T2"+(0,-0.5); "B1"+(0.5,0) **\dir{.};"T2"+(0,-1); "B1"+(1,0) **\dir{.};"T2"+(0,-1.5); "B1"+(1.5,0) **\dir{.};"T2"+(0,-2); "B1"+(2,0) **\dir{.};
"T2"+(0,-2.5); "B1"+(2.5,0) **\dir{.};"T2"+(0,-3); "B1"+(3,0) **\dir{.};"T2"+(0,-3.5); "B1"+(3.5,0) **\dir{.};"T2"+(0,-4); "B1"+(4,0) **\dir{.};
"T2"+(0,-4.5); "B1"+(4.5,0) **\dir{.};"T2"+(0,-5); "B1"+(5,0) **\dir{.};
"T2"+(-6,-0.5); "B1"+(-5.5,0) **\dir{.};"T2"+(-6,-1); "B1"+(-5,0) **\dir{.};"T2"+(-6,-1.5); "B1"+(-4.5,0) **\dir{.};"T2"+(-6,-2); "B1"+(-4,0) **\dir{.};
"T2"+(-6,-2.5); "B1"+(-3.5,0) **\dir{.};"T2"+(-6,-3); "B1"+(-3,0) **\dir{.};"T2"+(-6,-3.5); "B1"+(-2.5,0) **\dir{.};"T2"+(-6,-4); "B1"+(-2,0) **\dir{.};
"T2"+(-6,-4.5); "B1"+(-1.5,0) **\dir{.};"T2"+(-6,-5); "B1"+(-1,0) **\dir{.};
"T2"+(-12,-0.5); "B1"+(-11.5,0) **\dir{.};"T2"+(-12,-1); "B1"+(-11,0) **\dir{.};"T2"+(-12,-1.5); "B1"+(-10.5,0) **\dir{.};"T2"+(-12,-2); "B1"+(-10,0) **\dir{.};
"T2"+(-12,-2.5); "B1"+(-9.5,0) **\dir{.};"T2"+(-12,-3); "B1"+(-9,0) **\dir{.};"T2"+(-12,-3.5); "B1"+(-8.5,0) **\dir{.};"T2"+(-12,-4); "B1"+(-8,0) **\dir{.};
"T2"+(-12,-4.5); "B1"+(-7.5,0) **\dir{.};"T2"+(-12,-5); "B1"+(-7,0) **\dir{.};
(4.5,2)*{1};(1.5,4)*{0};  (3,9)*{2};(3,13.5)*{3};(3,16.5)*{3};(3,21)*{2};(1.5,28)*{0};
(-1.5,2)*{0};(-4.5,4)*{1};(-3,9)*{2};(-3,13.5)*{3};
(-7.5,2)*{1};(-10.5,4)*{0};
\endxy}

\newcommand{\youngwallext}
{\xy
(0,0)*{}="B1";(6,0)*{}="B2";
(0,6)*{}="T1";(6,6)*{}="T2";
"T1"; "B1" **\dir{-};"T2"; "B2" **\dir{-};"T2"; "T1" **\dir{-};"B2"; "B1" **\dir{-};
"T1"+(-6,0); "T2"+(-6,0) **\dir{-};"T1"+(-6,0); "B1"+(-6,0) **\dir{-};"B2"+(-6,0); "B1"+(-6,0) **\dir{-};
"T1"+(-12,0); "T2"+(-12,0) **\dir{-};"T1"+(-12,0); "B1"+(-12,0) **\dir{-};"B2"+(-12,0); "B1"+(-12,0) **\dir{-};
"T1"+(0,6); "B1"+(0,6) **\dir{-};"T2"+(0,6); "B2"+(0,6) **\dir{-};"T2"+(0,6); "T1"+(0,6) **\dir{-};
"T1"+(-6,6); "T2"+(-6,6) **\dir{-};"T1"+(-6,6); "B1"+(-6,6) **\dir{-};
"T1"+(0,12); "B1"+(0,12) **\dir{-};"T2"+(0,12); "B2"+(0,12) **\dir{-};"T2"+(0,12); "T1"+(0,12) **\dir{-};
"T1"+(0,18); "B1"+(0,18) **\dir{-};"T2"+(0,18); "B2"+(0,18) **\dir{-};"T2"+(0,18); "T1"+(0,18) **\dir{-};
"T2"+(0,24); "B2"+(0,24) **\dir{-};"B1"+(0,24); "T2"+(0,24) **\dir{-};
"T2"; "B1" **\dir{-};"T2"+(-6,0); "B1"+(-6,0) **\dir{-};"T2"+(-12,0); "B1"+(-12,0) **\dir{-};
"T1"+(-6,9); "T2"+(0,9) **\dir{-};"T1"+(-6,9); "T1"+(-6,6) **\dir{-};
"T2"+(0,-0.5); "B1"+(0.5,0) **\dir{.};"T2"+(0,-1); "B1"+(1,0) **\dir{.};"T2"+(0,-1.5); "B1"+(1.5,0) **\dir{.};"T2"+(0,-2); "B1"+(2,0) **\dir{.};
"T2"+(0,-2.5); "B1"+(2.5,0) **\dir{.};"T2"+(0,-3); "B1"+(3,0) **\dir{.};"T2"+(0,-3.5); "B1"+(3.5,0) **\dir{.};"T2"+(0,-4); "B1"+(4,0) **\dir{.};
"T2"+(0,-4.5); "B1"+(4.5,0) **\dir{.};"T2"+(0,-5); "B1"+(5,0) **\dir{.};
"T2"+(-6,-0.5); "B1"+(-5.5,0) **\dir{.};"T2"+(-6,-1); "B1"+(-5,0) **\dir{.};"T2"+(-6,-1.5); "B1"+(-4.5,0) **\dir{.};"T2"+(-6,-2); "B1"+(-4,0) **\dir{.};
"T2"+(-6,-2.5); "B1"+(-3.5,0) **\dir{.};"T2"+(-6,-3); "B1"+(-3,0) **\dir{.};"T2"+(-6,-3.5); "B1"+(-2.5,0) **\dir{.};"T2"+(-6,-4); "B1"+(-2,0) **\dir{.};
"T2"+(-6,-4.5); "B1"+(-1.5,0) **\dir{.};"T2"+(-6,-5); "B1"+(-1,0) **\dir{.};
"T2"+(-12,-0.5); "B1"+(-11.5,0) **\dir{.};"T2"+(-12,-1); "B1"+(-11,0) **\dir{.};"T2"+(-12,-1.5); "B1"+(-10.5,0) **\dir{.};"T2"+(-12,-2); "B1"+(-10,0) **\dir{.};
"T2"+(-12,-2.5); "B1"+(-9.5,0) **\dir{.};"T2"+(-12,-3); "B1"+(-9,0) **\dir{.};"T2"+(-12,-3.5); "B1"+(-8.5,0) **\dir{.};"T2"+(-12,-4); "B1"+(-8,0) **\dir{.};
"T2"+(-12,-4.5); "B1"+(-7.5,0) **\dir{.};"T2"+(-12,-5); "B1"+(-7,0) **\dir{.};
(4.5,2)*{1};(1.5,4)*{0};  (3,9)*{2};(3,13.5)*{3};(3,16.5)*{3};(3,21)*{2};(4.5,26)*{1};
(-1.5,2)*{0};(-4.5,4)*{1};(-3,9)*{2};(-3,13.5)*{3};
(-7.5,2)*{1};(-10.5,4)*{0};
\endxy}

\newcommand{\youngwallex}
{\xy
(0,0)*{}="B1";(6,0)*{}="B2";
(0,24)*{}="T1";(6,24)*{}="T2";
"T1"; "B1" **\dir{-};"T2"; "B2" **\dir{-};"T1"; "T2" **\dir{-};"B2"; "B1" **\dir{-};
"B2"+(0,6);"B2"+(3,9) **\dir{-};
"B1"+(-6,0);"B1"+(-6,15) **\dir{-};
"B1"+(-6,0);"B1" **\dir{-};
"B1"+(-6,6);"B1"+(0,6) **\dir{-};
"B1"+(-6,12);"B1"+(0,12) **\dir{-};
"B1"+(-6,15);"B1"+(0,15) **\dir{-};
"B1"+(-6,15);"B1"+(-3,18) **\dir{-};
"B1"+(0,18);"B1"+(-3,18) **\dir{-};
"B1"+(-12,0);"B1"+(-12,6) **\dir{-};
"B1"+(-12,0);"B1"+(-6,0) **\dir{-};
"B1"+(-12,6);"B1"+(-6,6) **\dir{-};
"B1"+(-12,6);"B1"+(-9,9)**\dir{-};
"B1"+(-10.5,7.5);"B1"+(-23.5,7.5)**\dir{-};
"B1"+(-9,9);"B1"+(-22.5,9)**\dir{-};
"B1"+(-16.5,7.5);"B1"+(-16.5,1.5) **\dir{-};
"B1"+(-16.5,7.5);"B1"+(-15,9) **\dir{-};
"B1"+(-22.5,7.5);"B1"+(-22.5,1.5) **\dir{-}; 
"B1"+(-22.5,7.5);"B1"+(-21,9) **\dir{-};
"B1"+(-12,1.5);"B1"+(-23.5,1.5) **\dir{-};
"B1"+(-6,9); "B1"+(-9,9)**\dir{-};
"B1"+(-6,7.5);"B1"+(-10.5,7.5)**\dir{-};
"B1"+(0,6);"B2"+(0,6) **\dir{-};
"B1"+(0,12);"B2"+(0,12) **\dir{-};
"B1"+(0,15);"B2"+(0,15) **\dir{-};
"B1"+(0,18);"B2"+(0,18) **\dir{-};
"B2";"B2"+(3,3) **\dir{-};
"B2"+(1.5,1.5);"B2"+(1.5,7.5) **\dir{-};
"B1"+(-6,8);"B1"+(-23,8)**\dir{.};
"B1"+(-6,8.5);"B1"+(-22.5,8.5)**\dir{.};
"B2"+(-17,7);"B2"+(-30,7) **\dir{.};
"B2"+(-17.5,6.5);"B2"+(-30,6.5) **\dir{.};
"B2"+(-18,6);"B2"+(-30,6) **\dir{.};
"B2"+(-18.5,5.5);"B2"+(-30,5.5) **\dir{.};
"B2"+(-18.5,5);"B2"+(-30,5) **\dir{.};
"B2"+(-18.5,4.5);"B2"+(-30,4.5) **\dir{.};
"B2"+(-18.5,4);"B2"+(-30,4) **\dir{.};
"B2"+(-18.5,3.5);"B2"+(-30,3.5) **\dir{.};
"B2"+(-18.5,3);"B2"+(-30,3) **\dir{.};
"B2"+(-18.5,2.5);"B2"+(-30,2.5) **\dir{.};
"B2"+(-18.5,2);"B2"+(-30,2) **\dir{.};
"B2"+(-18.5,1.5);"B2"+(-30,1.5) **\dir{.};
"B2"+(2,2);"B2"+(2,8) **\dir{.};
"B2"+(2.5,2.5);"B2"+(2.5,8.5) **\dir{.};
"B2"+(0,12);"B2"+(3,15) **\dir{-};
"B2"+(0,15);"B2"+(3,18) **\dir{-};
"B2"+(0,18);"B2"+(3,21) **\dir{-};
"B2"+(0,24);"B2"+(3,27) **\dir{-};
"B2"+(3,3);"B2"+(3,33) **\dir{-};
"B2"+(1.5,31.5);"B2"+(3,33) **\dir{-};
"B2"+(1.5,31.5);"B2"+(1.5,25.5) **\dir{-};
"B2"+(-4.5,25.5);"B2"+(1.5,25.5) **\dir{-};
"B2"+(-4.5,25.5);"B2"+(-4.5,31.5) **\dir{-};
"B2"+(-4.5,31.5);"B2"+(1.5,31.5) **\dir{-};
"B2"+(-4.5,25.5);"B2"+(-6,24) **\dir{-};
"B2"+(-4.5,31.5);"B2"+(-3,33) **\dir{-};
"B2"+(3,33);"B2"+(-3,33) **\dir{-};
(3,3)*{0}; (3,9)*{2};(3,13.5)*{3};(3,16.5)*{3};(3,21)*{2};
(4.2,28.5)*{1};
(-3,3)*{1}; (-3,9)*{2};(-3,13.5)*{3};
(-9,3)*{0};
(-13.5,4.5)*{0};(-19.5,4.5)*{1};
\endxy}

\nc{\gwallht}[2]
{\xy (0,0)*++{\trbox{#1}}; (-6,0)*++{\trbox{#2}}; (-12,0)*++{\trbox{#1}}; (-18,0)*++{\trbox{#2}}; (-24,0)*++{\trbox{#1}}; (-28.5,0)*++{\trhbox}; \endxy}

\nc{\gwalltdhh}[1]
{\xy (0,0)*++{\hhbox{_{#1}}}; (-6,0)*++{\hhboxp{_{#1}}}; (-12,0)*++{\hhboxp{_{#1}}}; (-28,0)*++{\hhboxpp{_{#1}}};
(4.5,3)*{};(4.5,6)*{} **\dir{-};(-35.5,6)*{};(4.5,6)*{} **\dir{-};
(-1.5,3)*{};(-1.5,6)*{} **\dir{-};(-7.5,3)*{};(-7.5,6)*{} **\dir{-};
(-13.5,3)*{};(-13.5,6)*{} **\dir{-};(-19.5,3)*{};(-19.5,6)*{} **\dir{-};
(1.1,-2.5)*{};(-42,-2.5)*{} **\dir{.}; (1.1,-2)*{};(-42,-2)*{} **\dir{.};
(1.1,-1.5)*{};(-42,-1.5)*{} **\dir{.}; (1.1,-1)*{};(-42,-1)*{} **\dir{.};
(1.1,-0.5)*{};(-42,-0.5)*{} **\dir{.}; (1.1,0)*{};(-42,0)*{} **\dir{.};
(1.4,0.5)*{};(-41,0.5)*{} **\dir{.}; (1.9,1)*{};(-40.5,1)*{} **\dir{.};
(2.3,1.5)*{};(-40,1.5)*{} **\dir{.}; (2.8,2)*{};(-39.5,2)*{} **\dir{.};
(3.3,2.5)*{};(-39,2.5)*{} **\dir{.}; %
(1.5,-2.5)*{};(4.5,0.5)*{} **\dir{.}; (1.5,-2)*{};(4.5,1)*{} **\dir{.};(1.5,-1.5)*{};(4.5,1.5)*{} **\dir{.};(1.5,-1)*{};(4.5,2)*{} **\dir{.};
 \endxy}

\nc{\gwalltdht}[2]
{\xy (0,0)*++{\htbox{#1}}; (-6,0)*++{\htboxp{#2}}; (-12,0)*++{\htboxp{#1}}; (-28,0)*++{\htboxpp{#2}};
(-20.3,-4.5)*++{\shade};
(1.8,-3)*{};(-42,-3)*{} **\dir{.}; (1.8,-2.5)*{};(-42,-2.5)*{} **\dir{.};
(1.8,-2)*{};(-42,-2)*{} **\dir{.}; (1.8,-1.5)*{};(-42,-1.5)*{} **\dir{.};
(1.8,-1)*{};(-42,-1)*{} **\dir{.}; (1.8,-0.5)*{};(-42,-0.5)*{} **\dir{.};
(1.8,0)*{};(-42,0)*{} **\dir{.}; (1.8,0.5)*{};(-42,0.5)*{} **\dir{.};
(1.8,1)*{};(-42,1)*{} **\dir{.}; (1.8,1.5)*{};(-42,1.5)*{} **\dir{.};
(2.2,2.7)*{};(-41,2.7)*{} **\dir{.}; (2.3,3.2)*{};(-40.5,3.2)*{} **\dir{.};
(2.8,-2.5)*{};(2.8,2.8)*{} **\dir{.};
(3.2,-2.1)*{};(3.2,3.2)*{} **\dir{.};
(3.6,-1.7)*{};(3.6,3.6)*{} **\dir{.};
 \endxy}

\nc{\groundwall}[1]
{\xy
(0,-0.5)*{};(33,-0.5)*{} **\dir{.};
(0,-1)*{};(33,-1)*{} **\dir{.};
(0,-1.5)*{};(33,-1.5)*{} **\dir{.};
(0,-2)*{};(33,-2)*{} **\dir{.};
(0,-2.5)*{};(33,-2.5)*{} **\dir{.};
(0,-3)*{};(33,-3)*{} **\dir{-};
(0,0)*{};(33,0)*{} **\dir{-};
(33,-3)*{};(33,0)*{} **\dir{-};
(27,-3)*{};(27,0)*{} **\dir{-};
(21,-3)*{};(21,0)*{} **\dir{-};
(15,-3)*{};(15,0)*{} **\dir{-};
(9,-3)*{};(9,0)*{} **\dir{-};
(3,-3)*{};(3,0)*{} **\dir{-};
(30.2,-1.5)*{_{#1}}; (24.2,-1.5)*{_{#1}}; (18.2,-1.5)*{_{#1}};
(12.2,-1.5)*{_{#1}}; (6.2,-1.5)*{_{#1}};
\endxy}

\newcommand{\trbox}[1]
{\xy
(0,0)*{};(-6,-6)*{} **\dir{-}; (0,0)*{};(0,-6)*{} **\dir{-};
(0,-6)*{};(-6,-6)*{} **\dir{-};
(0,-5.5)*{};(-0.5,-6)*{} **\dir{.};
(0,-4.5)*{};(-1.5,-6)*{} **\dir{.};
(0,-3.5)*{};(-2.5,-6)*{} **\dir{.};
(0,-2.5)*{};(-3.5,-6)*{} **\dir{.};
(0,-1.5)*{};(-4.5,-6)*{} **\dir{.};
(0,-0.5)*{};(-5.5,-6)*{} **\dir{.};
(-2,-4.3)*{#1};
\endxy}

\newcommand{\trhbox}
{\xy
(0,0)*{};(-3,-3)*{} **\dir{-}; (0,0)*{};(0,-6)*{} **\dir{-};
(0,-6)*{};(-3,-6)*{} **\dir{-};
(0,-5.5)*{};(-0.5,-6)*{} **\dir{.};
(0,-4.5)*{};(-1.5,-6)*{} **\dir{.};
(0,-3.5)*{};(-2.5,-6)*{} **\dir{.};
(0,-2.5)*{};(-3,-5.5)*{} **\dir{.};
(0,-1.5)*{};(-3,-4.5)*{} **\dir{.};
(0,-0.5)*{};(-3,-3.5)*{} **\dir{.};
\endxy}

\newcommand{\htbox}[1]
{\xy
(0,0)*{}="T1";(6,0)*{}="T2";
(0,-6)*{}="B1";(6,-6)*{}="B2";
(1.5,1.5)*{}="TR1";(7.5,1.5)*{}="TR2";
(7.5,-4.5)*{}="BR2";
"T1"; "T2" **\dir{-}; "T1"; "B1" **\dir{-};
"T1"; "TR1" **\dir{-}; "T2"; "TR2" **\dir{-};
"B2"; "BR2" **\dir{-}; "BR2"; "TR2" **\dir{-};
"TR1"; "TR2" **\dir{-};
"B1"; "B2" **\dir{-}; "T2"; "B2" **\dir{-};
(3,-3)*{#1};
\endxy}

\newcommand{\hhbox}[1]
{\xy
(0,0)*{}="T1";(6,0)*{}="T2";
(0,-3)*{}="B1";(6,-3)*{}="B2";
(3,3)*{}="TR1";(9,3)*{}="TR2";
(9,0)*{}="BR2";
"T1"; "T2" **\dir{-}; "T1"; "B1" **\dir{-};
"T1"; "TR1" **\dir{-}; "T2"; "TR2" **\dir{-};
"B2"; "BR2" **\dir{-}; "BR2"; "TR2" **\dir{-};
"TR1"; "TR2" **\dir{-};
"B1"; "B2" **\dir{-}; "T2"; "B2" **\dir{-};
(3,-1.5)*{#1};
\endxy}

\newcommand{\htboxp}[1]
{\xy
(0,0)*{}="T1";(6,0)*{}="T2";
(0,-6)*{}="B1";(6,-6)*{}="B2";
(1.5,1.5)*{}="TR1";(7.5,1.5)*{}="TR2";
(7.5,-4.5)*{}="BR2";
"T1"; "T2" **\dir{-}; "T1"; "B1" **\dir{-};
"T1"; "TR1" **\dir{-}; "T2"; "TR2" **\dir{-};
"TR1"; "TR2" **\dir{-};
"B1"; "B2" **\dir{-}; "T2"; "B2" **\dir{-};
(3,-3)*{#1};
\endxy}

\newcommand{\hhboxp}[1]
{\xy
(0,0)*{}="T1";(6,0)*{}="T2";
(0,-3)*{}="B1";(6,-3)*{}="B2";
(3,3)*{}="TR1";(9,3)*{}="TR2";
(9,0)*{}="BR2";
"T1"; "T2" **\dir{-}; "T1"; "B1" **\dir{-};
"T1"; "TR1" **\dir{-}; "T2"; "TR2" **\dir{-};
"TR1"; "TR2" **\dir{-};
"B1"; "B2" **\dir{-}; "T2"; "B2" **\dir{-};
(3,-1.5)*{#1};
\endxy}

\newcommand{\htboxpp}[1]
{\xy
(-20,0)*{}="T0";(0,0)*{}="T1";(6,0)*{}="T2";
(-20,-6)*{}="B0";(0,-6)*{}="B1";(6,-6)*{}="B2";
(-18.5,1.5)*{}="TR0";(1.5,1.5)*{}="TR1";(7.5,1.5)*{}="TR2";
(7.5,-4.5)*{}="BR2";
"TR1"; "TR0" **\dir{-};
"T1"; "T0" **\dir{-};"B1"; "B0" **\dir{-};
"T1"; "T2" **\dir{-}; "T1"; "B1" **\dir{-};
"T1"; "TR1" **\dir{-}; "T2"; "TR2" **\dir{-};
"TR1"; "TR2" **\dir{-};
"B1"; "B2" **\dir{-}; "T2"; "B2" **\dir{-};
(3,-3)*{#1};
\endxy}

\newcommand{\hhboxpp}[1]
{\xy
(-20,0)*{}="T0";(0,0)*{}="T1";(6,0)*{}="T2";
(-20,-3)*{}="B0";(0,-3)*{}="B1";(6,-3)*{}="B2";
(-16,3)*{}="TR0";(3,3)*{}="TR1";(9,3)*{}="TR2";
(9,0)*{}="BR2";
"TR1"; "TR0" **\dir{-};
"T1"; "T0" **\dir{-};"B1"; "B0" **\dir{-};
"T1"; "T2" **\dir{-}; "T1"; "B1" **\dir{-};
"T1"; "TR1" **\dir{-}; "T2"; "TR2" **\dir{-};
"TR1"; "TR2" **\dir{-};
"B1"; "B2" **\dir{-}; "T2"; "B2" **\dir{-};
(3,-1.5)*{#1};
\endxy}

\newcommand{\shade}
{\xy
(0,0)*{}="B0";(-6,0)*{}="B1";(-12,0)*{}="B2";(-18,0)*{}="B3";(-24,0)*{}="B4";
(-1.5,-1.5)*{}="BB0";(-7.5,-1.5)*{}="BB1";(-13.5,-1.5)*{}="BB2";(-19.5,-1.5)*{}="BB3";(-25.5,-1.5)*{}="BB4";
"B0"; "BB0" **\dir{-};"B1"; "BB1" **\dir{-};"B2"; "BB2" **\dir{-};"B3"; "BB3" **\dir{-};"B4"; "BB4" **\dir{-};
"BB0"; "BB4"+(-20,0) **\dir{-};
\endxy}

\newcommand{\ublock}{
\xy  (-3,3)*{}="T1"; (3,3)*{}="T2"; (-3,-3)*{}="B1";(3,-3)*{}="B2";
(-1.5,6)*{}="TB1"; (4.5,6)*{}="TB2";
(-1.5,-0)*{}="BB1"; (4.5,-0)*{}="BB2";
"T1";"T2" **\dir{-};"B1";"B2" **\dir{-};
"T1";"B1" **\dir{-};"T2";"B2" **\dir{-};
"TB1";"TB2" **\dir{-};"TB2";"BB2" **\dir{-};
"T1";"TB1" **\dir{-};"T2";"TB2" **\dir{-};
"B2";"BB2" **\dir{-};
(7,0)*{=};
"T1"+(13,0);"T2"+(13,0) **\dir{-};"B1"+(13,0);"B2"+(13,0) **\dir{-};
"T1"+(13,0);"B1"+(13,0) **\dir{-};"T2"+(13,0);"B2"+(13,0) **\dir{-};
\endxy
}

\newcommand{\uhh}{
\scalebox{0.4}{
\xy  (-3,1.5)*{}="T1"; (3,1.5)*{}="T2"; (-3,-1.5)*{}="B1";(3,-1.5)*{}="B2";
"T1";"T2" **\dir{-};"B1";"B2" **\dir{-};
"T1";"B1" **\dir{-};"T2";"B2" **\dir{-};
\endxy
}}

\newcommand{\uhb}{
\scalebox{0.4}{
\xy  (-3,3)*{}="T1"; (3,3)*{}="T2"; (-3,-3)*{}="B1";(3,-3)*{}="B2";
(-1.5,4.5)*{}="TB1"; (4.5,4.5)*{}="TB2";
(-1.5,-1.5)*{}="BB1"; (4.5,-1.5)*{}="BB2";
(-4.5,-4.5)*{}="BF1";(1.5,-4.5)*{}="BF2";
"T1";"T2" **\dir{.};"B1";"B2" **\dir{-};
"T1";"B1" **\dir{.};"T2";"B2" **\dir{-};
"B1";"T2" **\dir{-};
\endxy
}}

\newcommand{\uhblock}{
\xy  (-3,1.5)*{}="T1"; (3,1.5)*{}="T2"; (-3,-1.5)*{}="B1";(3,-1.5)*{}="B2";
(-1.5,4.5)*{}="TB1"; (4.5,4.5)*{}="TB2";
(-1.5,1.5)*{}="BB1"; (4.5,1.5)*{}="BB2";
"T1";"T2" **\dir{-};"B1";"B2" **\dir{-};
"T1";"B1" **\dir{-};"T2";"B2" **\dir{-};
"TB1";"TB2" **\dir{-};"TB2";"BB2" **\dir{-};
"T1";"TB1" **\dir{-};"T2";"TB2" **\dir{-};
"B2";"BB2" **\dir{-};
(7,0)*{=};
"T1"+(13,0);"T2"+(13,0) **\dir{-};"B1"+(13,0);"B2"+(13,0) **\dir{-};
"T1"+(13,0);"B1"+(13,0) **\dir{-};"T2"+(13,0);"B2"+(13,0) **\dir{-};
\endxy
}

\newcommand{\uhtbblock}{
\xy  (-3,3)*{}="T1"; (3,3)*{}="T2"; (-3,-3)*{}="B1";(3,-3)*{}="B2";
(-1.5,4.5)*{}="TB1"; (4.5,4.5)*{}="TB2";
(-1.5,-1.5)*{}="BB1"; (4.5,-1.5)*{}="BB2";
(-4.5,-4.5)*{}="BF1";(1.5,-4.5)*{}="BF2";
"T1";"T2" **\dir{-};"B1";"B2" **\dir{-};"BF1";"BF2" **\dir{.};
"B1";"BF1" **\dir{.};"B2";"BF2" **\dir{.};
"T1";"B1" **\dir{-};"T2";"B2" **\dir{-};
"TB1";"TB2" **\dir{-};"TB2";"BB2" **\dir{-};
"T1";"TB1" **\dir{-};"T2";"TB2" **\dir{-};
"B2";"BB2" **\dir{-};
(7,0)*{=};
"T1"+(13,0);"T2"+(13,0) **\dir{.};"B1"+(13,0);"B2"+(13,0) **\dir{-};
"T1"+(13,0);"B1"+(13,0) **\dir{.};"T2"+(13,0);"B2"+(13,0) **\dir{-};
"B1"+(13,0);"T2"+(13,0) **\dir{-};
"TB1"+(21,0);"TB2"+(21,0) **\dir{-};
"T1"+(21,0);"TB1"+(21,0) **\dir{-};"T2"+(21,0);"TB2"+(21,0) **\dir{-};
"T1"+(21,0);"T2"+(21,0) **\dir{-};"B1"+(21,0);"B2"+(21,0) **\dir{-};
"T1"+(21,0);"B1"+(21,0) **\dir{-};"T2"+(21,0);"B2"+(21,0) **\dir{-};
"TB1"+(21,0);"TB2"+(21,0) **\dir{-};"TB2"+(21,0);"BB2"+(21,0) **\dir{-};
"B2"+(21,0);"BB2"+(21,0) **\dir{-};
"BB2"+(21,0);"BB2"+(22.5,1.5) **\dir{.};
"BB2"+(21,1.5);"BB2"+(22.5,1.5) **\dir{.};
(30,0)*{=};
"T1"+(36,0);"T2"+(36,0) **\dir{-};"B1"+(36,0);"B2"+(36,0) **\dir{.};
"T1"+(36,0);"B1"+(36,0) **\dir{-};"T2"+(36,0);"B2"+(36,0) **\dir{.};
"B1"+(36,0);"T2"+(36,0) **\dir{-};
\endxy
}

\nc{\YWw}{\mathsf{Y}_\ofw}
\newcommand{\uofw}{ {\underline{\ofw}}}

\tikzset{tab/.style={matrix of math nodes,column sep=-.4, row
sep=-.4,text height=8pt,text width=8pt,align=center}}

\title[Weight multiplicities and Young tableaux]
{Weight multiplicities and Young tableaux \\ through affine crystals}

\author[J.S. Kim, K.-H. Lee, S.-j. Oh]{Jang Soo Kim$^\ddagger$, Kyu-Hwan Lee$^{\star}$ and Se-jin Oh$^{\dagger}$}

\thanks{$^\ddagger$This work was supported by NRF Grants \#2016R1D1A1A09917506 and \#2016R1A5A1008055.}
\thanks{$^{\star}$This work was partially supported by a grant from the Simons Foundation (\#324706).}
\thanks{$^{\dagger}$This work was supported by NRF Grant \#2016R1C1B2013135.}

\address{Department of Mathematics, Sungkyunkwan University, Suwon, South Korea}
\email{jangsookim@skku.edu}

\address{Department of Mathematics, University of Connecticut, Storrs, CT 06269, U.S.A.}
\email{khlee@math.uconn.edu}

\address{Department of Mathematics, Ewha Womans University, Seoul 120-750, South Korea}
\email{sejin092@gmail.com}

\date{\today}

\begin{document}

\begin{abstract}
The weight multiplicities  of finite dimensional simple Lie algebras can be computed individually using various methods. Still, it is hard to derive explicit closed formulas. Similarly, explicit closed formulas for the multiplicities of maximal weights of affine Kac--Moody algebras are not known in most cases.
In this paper, we study weight multiplicities for both finite and affine cases of classical types for certain infinite families of highest weights modules. We introduce new classes of Young tableaux, called the {\em $($spin$)$ rigid tableaux}, and prove that they are equinumerous
to the weight multiplicities of the highest weight modules under our consideration.
These new classes of Young tableaux arise from crystal basis elements for dominant maximal weights of the integrable highest weight modules over affine Kac--Moody algebras.
By applying combinatorics of tableaux such as the Robinson--Schensted algorithm and new insertion schemes, and using integrals over orthogonal groups,
we reveal hidden structures in the sets of weight multiplicities and obtain explicit closed formulas for  the weight multiplicities.
In particular we show that some special families of  weight multiplicities  form the Pascal, Catalan, Motzkin, Riordan and Bessel triangles.

\end{abstract}

\maketitle

\tableofcontents

\section*{Introduction}

The irreducible representations $L(\omega)$ of finite dimensional complex simple Lie algebras are fundamental objects in mathematics. We understand their structures through the generating functions of weight multiplicities, i.e. the characters of the  representations, which can be computed by the celebrated Weyl's character formula. Individual weight multiplicities can be computed using Kostant's formula or Freudenthal's recursive formula. One can also exploit the crystal basis theory, initiated by Kashiwara \cite{Kas91}, and its realizations such as Kashiwara--Nakashima tableaux \cite{KN94}, Littelmann paths \cite{Lit} and Mirkovic--Vilonen polytopes \cite{Kam} to name a few.


Nonetheless there are only a few explicit closed formulas in the literature for weight multiplicities. Kostant's formula involves a summation over the Weyl group whose size becomes enormous as the rank increases, and Freudenthal's formula is recursive, and realizations of crystals convert computing weight multiplicities into  challenging combinatorial problems.

The theory of finite-dimensional simple Lie algebras was generalized to that of Kac--Moody algebras in 1960's, and the first family of infinite dimensional Lie algebras is called {\em affine} Kac--Moody algebras. Representations of affine Kac--Moody algebras have been studied extensively  as their applications have been found throughout mathematics and mathematical physics. In particular, weight multiplicities of an integrable highest weight module $V(\Lambda)$ over an affine Kac--Moody algebra are of great interests as they can be interpreted in several different ways such as generalized partition numbers \cite{LM}, Fourier coefficients of certain modular forms \cite{KP}, and numbers of isomorphism classes of irreducible modules over Hecke-type algebras \cite{Ari,LLT}. However, our understanding of weight multiplicities is, in general, very limited though we can compute them individually through generalizations of classical formulas and constructions, e.g. \cite{KMPS}.

The set of weights of $V(\Lambda)$ can be divided into $\delta$-strings and the first weight of each string is called a {\it maximal weight}. Maximal weights and their multiplicities are fundamental in understanding the structure of $V(\Lambda)$. Since weight multiplicities are invariant under the Weyl group action, it is enough to consider {\it dominant} maximal weights, and it is well-known that the set of dominant maximal weights for each highest weight $\Lambda$ is finite. Nevertheless, we do not have any explicit description of dominant maximal weights and their multiplicities in most cases. Except for trivial cases, only level $2$ maximal weights of type $A_n^{(1)}$ and their multiplicities are completely known \cite{Ts}, and recently, some maximal weights of $V(k \Lambda_0 +\Lambda_s)$, $k \in \mathbb Z_{>0}$, $s=0,1, \dots, n$, of type $A_n^{(1)}$   have been studied \cite{JM,JM1,TW}, where $\Lambda_s$ are the fundamental weights. Other than type $A_n^{(1)}$, little is known about descriptions of dominant maximal weights and their multiplicities.

In this paper, we study the multiplicities of dominant weights for finite types and those of dominant maximal weights for affine types {\em at the same time}.
We introduce special subsets of Young tableaux, called (spin) rigid Young tableaux, which are equinumerous to the weight multiplicities of the certain highest weight modules for finite and affine types
simultaneously, and we derive explicit closed formulas for the weight multiplicities when they are of level $k \le 6$ or $k \gg 0$. Our closed formulas are practically computable, and related to binomial coefficients, Catalan numbers and Motzkin numbers. We consider all  classical finite types  and affine types, but more focus will be made on finite types $B_n$ and $D_n$ and
affine types $B_n^{(1)}, D_n^{(1)},A_{2n-1}^{(2)},A_{2n}^{(2)}$ and $D_{n+1}^{(2)}$.

We summarize the results of this paper in three main parts as follows.

First, we consider some families of highest weights $\Lambda$ over affine Kac--Moody algebras of classical types, including all highest weights of levels $2$ and $3$, and determine dominant maximal weights. See, e.g., Theorems \ref{thm: B level 2}, \ref{thm: B level 3} and \ref{thm: D level 2}.
We observe that a majority of  dominant maximal weights  are {\em essentially finite}  and can be associated with  pairs of staircase partitions. We will denote the set of level $k$ (essentially finite) maximal dominant weights, associated with pairs of staircase partitions, by $\hmaxp(\Lambda|k)$ or $\tmaxp(\Lambda|k)$, depending on the corresponding finite types. Each $\eta \in \hmaxp(\Lambda|k)$ or $\tmaxp(\Lambda|k)$ is given an {\em index} $(m,s)$ recording the sizes of the associated staircase partitions.

A simple, yet crucial fact we prove is that  two essentially finite dominant maximal weights of the same finite type with  the same index $(m,s)$ have the same  weight multiplicity {\it even if their affine types are different.} This fact is related to a classification of the zero nodes of affine Dynkin diagrams (cf. \cite{LS}).
Furthermore, for essentially finite weights, the weight multiplicities of affine Kac--Moody algebras are actually the same as those of the corresponding finite dimensional simple Lie algebras, and we may use the theory of finite dimensional simple Lie algebras. However, as indicated at the beginning of this introduction, explicit closed formulas are not available even for weight multiplicities of finite dimensional simple Lie algebras. Therefore, we utilize a realization of affine crystals to determine weight multiplicities.

Second, the realization of affine crystals we use is {\em Young walls} introduced by Kang \cite{Kang03} which are visualization of Kyoto paths. We first embed the crystals of $V(\Lambda)$ into a tensor product of Young walls of level $1$ fundamental representations and investigate the sets of Young walls in the spaces of dominant maximal weights. A careful analysis of the patterns of the Young walls leads to  new classes of skew standard Young tableaux that realize crystal basis elements of dominant maximal weights in the tensor product of crystals. Namely, we define the set ${}_s\Ss_m^{(k)}$ of {\em rigid} Young tableaux and the set ${}_s\sS_m^{(k)}$ of {\em spin rigid} Young tableaux for any $k \ge 2$  and  $0 \le s \le m$. Roughly speaking, a rigid Young tableau is a skew tableau for which a shift of the last row to the right by $1$ makes the tableaux violate column-strictness. For example, the following are rigid tableaux:
\begin{equation*} \label{eqn-ex-into}
\young(\cdot754,\cdot31,62), \qquad \young(\cdot\cdot\cdot\ot\te87,\cdot\cdot\cdot\oo91,65432) .
\end{equation*}
Here we are using reverse standard Young tableaux and so the rows and columns are decreasing. Similarly, a spin rigid Young tableau is a skew tableau for which a shift of the last row to the right by 2 makes the tableaux violate column-strictness and whose shape satisfies certain conditions. For example, the following are spin rigid Young tableaux:
\[ \ytableausetup{boxsize=1em}\begin{ytableau}
 \cdot& *(gray!40) \cdot&  4& *(gray!40) 1\\
 \cdot& *(gray!40) \cdot&  2\\
 3
\end{ytableau},
\qquad
\begin{ytableau}
 \cdot& *(gray!40) \cdot&  4 & *(gray!40) 3&  2\\
 5& *(gray!40) 1
\end{ytableau}.
\]

Using combinatorics of Young walls, we prove that
the sets ${}_s\Ss_m^{(k)}$ and ${}_s\sS_m^{(k)}$ are
equinumerous to the weight multiplicities of  highest weight modules
of finite and affine types simultaneously (Theorems \ref{thm: level k hh rigid} and \ref{thm: level k ht spin rigid}).

\begin{theorem}
Let $k \ge 2$ and $0 \le s \le m \le n$.

{\rm (1)} For $\eta \in \hmaxp(\Lambda|k)$  of index $(m,s)$, we have
$$ \dim V(\Lambda)_\eta = |{}_s\Ss_m^{(k)}| = \dim  L((k-2)\omega_n+\tilde{\omega}_{n-s})_{(k-2)\omega_n+\tilde{\omega}_{n-m}}, $$
where $L({\omega})$ is of type $B_n$, 
$\omega_t$ are the fundamental weights,
and $\tilde{\omega}_t$ are defined by
\begin{align*}
\tilde{\omega}_t \seteq \begin{cases} 2 \omega_n & \text{ if } t=n ,\\ \omega_t & \text{ otherwise.} \end{cases}
\end{align*}

{\rm (2)} For $\eta \in \tmaxp(\Lambda|k)$ of index $(m,s-1)$, we have
$$ \dim V(\Lambda)_\eta = |{}_s\sS_m^{(k)}|=\dim L \big( (k-2)\omega_n+\tilde{\omega}_{n-s} \big)_\mu,$$
where $L({\omega})$ is  of type $D_n$,  and $\tilde{\omega}_t$ are defined by
\begin{align*}
\tilde{\omega}_t = \begin{cases}
\omega_t  & \text{ if } 1 \le t <n-1,  \\
\omega_{n-1}+\omega_n & \text{ if } t=n-1, \\
2\omega_n & \text{ if } t=n,
\end{cases}
\end{align*}
and the weights $\mu$ are given by
\[  \mu = \begin{cases}
(k-2)\omega_n + \tilde{\omega}_{n-m-1} & \text{if } k=2, \text{ or } k \ge 3 \text{ and } m \not \equiv_2 s, \\
(k-3)\omega_n + \omega_{n-1}+ \tilde{\omega}_{n-m-1} & \text{if } k \ge 3 \text{ and } m \equiv_2 s.
\end{cases} \]
\end{theorem}

Our methods unexpectedly reveal hidden structures of weight multiplicities. We consider highest weights in a family at the same time and form a triangular array consisting of $|{}_s\Ss_m^{(k)}|$ or $|{}_{s}\sS_m^{(k)}|$ as highest weights varies in the family. Interestingly, the entries of the resulting triangular arrays count the number of certain lattice paths and we construct bijections between the sets of lattice paths and the corresponding sets of tableaux.  These arrays are the Pascal, Catalan, Motzkin and Riordan triangles for various families of highest weights.  See the triangular arrays in \eqref{MT} and \eqref{RT} for the Motzkin and Riordan triangles, respectively.  See Example \ref{exa-exa} for the case of generalized Motzkin paths. Moreover, the entries of the triangular arrays also represent some decomposition multiplicities of tensor products of $\mathfrak{sl}_2$-modules, invoking Schur--Weyl type dualities (\cite{BH,FK}) into the structures of weight multiplicities.

Third, we use various combinatorial methods to find explicit formulas for the numbers $|{}_s\Ss_m^{(k)}|$ and $|{}_{s}\sS_m^{(k)}|$ for $k=2$ (Theorems~\ref{them: pascal hh}, \ref{them: pascal ht}), for $k=3$ (Theorems~\ref{thm: motzkin}, \ref{thm: riordan}), and for the number $|{}_{0}\sS_m^{(k)}|$ for $2 \le k \le 5$ (Theorem~\ref{thm:card_S}). In particular, we use the Robinson--Schensted algorithm and a new insertion scheme for the (spin) rigid tableaux, see Algorithm \ref{Alg: JdT1}.  We also use integrals over orthogonal groups to derive explicit formulas for $|_0\mathfrak B_m^{(k)}|$ (Theorem~\ref{thm:Selberg}).  The set ${}_{0}\Ss_m^{(k)}$ is nothing but the set of (reverse) standard Young tableaux with $m$ cells and at most $k$ rows. In the literature an explicit formula for its cardinality is known only  for $k \le 5$ (\cite{GB89,Re91}). We summarize our formulas as follows.

\begin{theorem} For $0 \le s \le m$, we have
\begin{align*}
|{}_{s}\Ss^{(2)}_{m}| &= \binom{m}{\lfloor \frac{m-s}{2} \rfloor},\qquad |{}_{s}\sS^{(2)}_{2u-1+s}| = \matr{2u+s-\delta_{s,0}}{u} \quad (u \ge 0),
\allowdisplaybreaks\\
|{}_s\Ss_m^{(3)}| & =\sum_{i=0}^{\lfloor s/2 \rfloor}\binom{m}{2i+m-s}\left(\binom{2i+m-s}{i} - \binom{2i+m-s}{i-1}\right), \allowdisplaybreaks\\
|{}_s\sS_m^{(3)}| & = \sum_{i=0}^{m+1-\delta_{s,0}-s} (-1)^i \left (|{}_s\Ss_{m-\delta_{s,0}-i}^{(3)}|+|{}_{s-1}\Ss_{m-\delta_{s,0}-i}^{(3)}| \right ), \allowdisplaybreaks\\
|{}_0\Ae_{2m-1}^{(4)}| & = \binom{\mathsf{C}_m+1}2, \qquad |{}_0\Ae_{2m}^{(4)}|=\mathsf{C}_m\mathsf{C}_{m+1} - \mathsf{C}_{m}^2, \allowdisplaybreaks\\
|{}_0\Ae_{2m-1}^{(5)}| &=\sum_{i=0}^{m} \binom{2m}{2i} \mathsf{C}_i\mathsf{C}_{i+1} - \sum_{i=0}^{m-1} \binom{2m}{2i+1}\mathsf{C}_{i+1}^2,
\\ |{}_0\Ae_{2m}^{(5)}|&=\sum_{i=0}^{m} \frac{2i}{i+3}\binom{2m}{2i} \mathsf{C}_i\mathsf{C}_{i+1}- \sum_{i=0}^{m-1} \frac{2i}{i+3}\binom{2m}{2i+1} \mathsf{C}_{i+1}^2,
\end{align*}
where $\mathsf{C}_i = \frac 1 {i+1} \binom{2i}{i}$ is the $i$-th Catalan number.

For integers $k\ge1$ and $m\ge0$, we have
\[
|{}_{0}\Ss^{(2k)}_{m}|
=\sum_{t_1+\dots+t_k=m}\binom{m}{t_1,\dots,t_k}
\det\left(
\binom{t_i+2k-i-j}{\lfloor \frac{t_i+2k-i-j}2 \rfloor}
\right)_{i,j=1}^k,
\]
\[
|{}_{0}\Ss^{(2k+1)}_{m}|
=\sum_{t_0+t_1+\dots+t_k=m}\binom{m}{t_0,t_1,\dots,t_k}
\det\left(
C\left(\frac{t_i+2k-i-j}2\right)
\right)_{i,j=1}^k,
\]
where $C(x)=\mathsf{C}_x=\frac{1}{x+1}\binom{2x}x$ if $x$ is an integer and $C(x )=0$ otherwise.
\end{theorem}

When  $k$ increases, the numbers $|{}_s\Ss_m^{(k)}|$ and $|{}_{s}\sS_m^{(k)}|$ (and thus the weight multiplicities) stabilize and we find their closed formulas (Corollary~\ref{cor:B_infty} and Theorem~\ref{thm: k infty ht not mod 2}). In particular, from $\lim_{k \rightarrow \infty} |{}_{s}\sS_m^{(k)}|$, we obtain a triangular array of numbers, called {\em Bessel triangle},  consisting of the coefficients of Bessel polynomials, see \eqref{BT}.


The organization of this paper is as follows. In Section \ref{sec:KM}, we fix notations and present basic definitions for affine Kac--Moody algebras and quantum affine algebras. Throughout this paper we mainly use the notations of affine types, even though we study finite types together.
The relationship between weights of finite types and affine types is explained in Section \ref{subsec-finite}. In Section \ref{sec:CW}, after the theory of crystals is reviewed briefly, we describe constructions of Young walls and explain embeddings of highest weight crystals into tensor products of level $1$ crystals. A connection between affine crystals and finite crystals is pointed out in Section \ref{subsec-fini-crys}.
In Section \ref{sec:YT}, we explain a correspondence between Young walls and Young tableaux, and introduce some families of Young tableaux that will be used later. Section \ref{sec:LP} is devoted to lattice paths and triangular arrays of numbers. The  entries of the triangular arrays are the numbers of certain types of lattice paths and also the decomposition multiplicities of tensor products of $\mathfrak{sl}_2$-modules. All the entries of the triangular arrays are also to appear as weight multiplicities.

In Section \ref{sec:Rep}, we determine dominant maximal weights for certain families of highest weight modules. These families include all highest weight modules of levels $2$ and $3$ except for types $A_n^{(1)}$ and $C_n^{(1)}$. A conjecture is made for the numbers of the dominant maximal weights for type $B_n^{(1)}$. We classify staircase dominant maximal weights according to their finite types.
In Section \ref{sec:WM}, we investigate the Young walls of dominant maximal weights and define (spin) rigid Young tableaux. Using combinatorics of Young walls, we prove that the sets of (spin) rigid tableaux are equinumerous to weight multiplicities.

Section \ref{sec:level 2} is concerned about the level $2$ cases. We prove that the weight multiplicities form the Catalan triangle and the Pascal triangle. The main tool is an insertion scheme for tableaux. We also construct bijections between the set of lattice paths and the set of rigid Young tableaux in ${}_s\Ss_m^{(2)}$. In Section \ref{sec-level3}, we consider the level $3$ cases and prove that the weight multiplicities form the Motzkin triangle for rigid Young tableaux and the Riordan triangle for spin rigid Young tableaux. We prove both cases using the Robinson--Schensted algorithm and provide a different proof for the Motzkin case using an insertion scheme which naturally realizes the Motzkin triangle through combinatorics of tableaux. An explicit bijection from the set of rigid tableaux in ${}_s\Ss_m^{(3)}$ to the set of generalized Motzkin paths is also given.

In Section \ref{sec:Bessel}, we investigate the limits of weight multiplicities of level $k$ as $k$ increases. We observe that the weight multiplicities given by the numbers of (spin) rigid Young tableaux stabilize as $k$ increases, and compute the limits explicitly. The computation uses formulas for the numbers of involutions in the symmetric groups.

In the final section, we consider the set $\SYT_m^{(k,t)}$ of standard Young tableaux with $m$ cells, at most $k$ rows and exactly $t$ rows of odd length. Both ${}_0\Ss_m^{(k)}$ and ${}_0\sS_m^{(k)}$ can be considered as special cases of the set $\SYT_m^{(k,t)}$. Using the Robinson--Schensted algorithm, we find a relation between $|\SYT_m^{(k,0)}|$, $|\SYT_m^{(k,k)}|$ and $|_0\mathfrak B^{(k-1)}_m|$. Using this relation and some known results, we find an explicit formula for $|\sfS_{m}^{(k,t)}|$ for every $0\le t\le k\le 5$.  We then express $|_0\mathfrak B^{(k)}_m|$ as an integral over the orthogonal group $\OG(k)$.  By evaluating this integral we obtain an explicit formula for $|_0\mathfrak B^{(k)}_m|$.

\subsection*{Acknowledgments}
We are grateful to Daniel Bump and Kailash Misra for stimulating discussions. We thank Georgia Benkart, James Humphreys and Anne Schilling for helpful comments on an earlier version of this paper. We also thank Ole Warnaar for his helpful comments, which greatly improved Theorem~\ref{thm:Selberg}. K.-H.~L. gratefully acknowledges support from the Simons Center for Geometry and Physics at which some of the research for this paper was performed.  S.-j.~O. gratefully acknowledges hospitality of the University of Connecticut during his visits in 2016 and 2017.

\medskip

\section{Affine Kac--Moody algebras} \label{sec:KM}

\subsection{Preliminaries} Let $I=\{ 0,1,...,n \}$ be an index set. The \emph{affine Cartan
datum} $(\cmA,P^{\vee},P,\Pi^{\vee},\Pi)$ consists of
\begin{enumerate}
\item[({\rm a})] a matrix $\cmA=(a_{ij})$ of corank $1$, called the {\it affine Cartan matrix} satisfying, for $i,j \in I$,
$$ ({\rm i}) \ a_{ii}=2 , \quad ({\rm ii}) \ a_{ij}  \in \Z_{\le 0} \text{ for $i \ne j \in I$,} \quad  ({\rm iii}) \ a_{ij}=0 \text{ if } a_{ji}=0,$$
\item[({\rm b})] a free abelian group $P^{\vee}=\bigoplus_{i=0}^{n} \Z h_i \oplus \Z d $, the \emph{dual weight lattice},  with $\h:= \C \otimes_\Z P^{\vee}$,
\item[({\rm c})] a free abelian group $P=\bigoplus_{i=0}^{n} \Z \Lambda_i \oplus \Z \updelta \subset \h^*$, the \emph{weight lattice},
\item[({\rm d})] a linearly independent set $\Pi^{\vee} = \{ h_i \mid i\in I\} \subset P^{\vee}$, the set of {\it simple coroots},
\item[({\rm e})] a linearly independent set $\Pi = \{ \alpha_i \mid i\in I \} \subset P$, the set of {\it simple roots},
\end{enumerate}
which satisfy
$$\text{$\langle h_i, \alpha_j \rangle  = a_{ij}$ and $\langle h_i, \Lambda_j \rangle =\delta_{i,j}$ for all $i,j\in I$,}$$
where $\Lambda_i$ denotes the {\it $i$-th fundamental weight},  $\updelta=\sum_{i \in I}a_i\alpha_i$  the {\it null root}
and $d$  the {\it degree derivation}:
$$ \langle h_i,\updelta\rangle=0, \quad \langle d,\updelta\rangle=1 \quad \text{ and } \quad \langle d, \alpha_i\rangle=\delta_{i,0}.$$

Let $c=\sum_{i \in I}a^\vee_ih_i$ be the unique element such that $a_i^\vee \in \Z_{\ge 0}$ and
$$\Z c =\{ h \in \bigoplus_{i \in I} \Z h_i \ | \ \langle h,\alpha_i \rangle =0 \text{ for any } i \in I \}.$$
Recall that $\cmA$ is {\it symmetrizable} in the sense that $D\cmA$ is symmetric where
$$D= {\rm diag}(\mathsf{d}_i \seteq a^\vee_ia_i^{-1} \mid i \in I).$$

We say that a weight $\Lambda \in P$ is of {\it level} $k$ if $\Lambda(c)=k$.
There exists a non-degenerate symmetric bilinear form $( \ | \ )$ on $\h^*$ (\cite[(6.2.2)]{Kac}) such that
$$ (\alpha_i|\alpha_j)= \mathsf{d}_ia_{ij}\quad \text{ for any } i,j\in I.$$

We denote by $P^{+} \seteq \{\Lambda \in P \mid  \langle h_i,\Lambda \rangle \in \Z_{\ge 0},\  i \in I \}$ the set of \emph{dominant integral weights}.
The free abelian group $\rl\seteq\bigoplus_{i \in I} \Z\alpha_i$ is called the \emph{root lattice} and we set
$\rl^{+}\seteq\bigoplus_{i \in I}\Z_{\ge 0}\alpha_i$.
For an element $\be = \sum_{i \in I}k_i\al_i \in \rl^+$ and $i \in I$, we define
the integer $\het(\be): = \sum_{i \in I}k_i$, called the {\it height} of $\be$, a subset $\Supp(\be)=\{ i \ | \ k_i \ne 0 \}$ of $I$, called the {\it support} of $\beta$.

\begin{definition}
The {\em affine Kac-Moody algebra} $\g$ associated with an affine Cartan datum $(\cmA,P^{\vee},P,\Pi^{\vee},\Pi)$
is the Lie algebra over $\C$ generated by $e_i,f_i \ ( i \in I)$ and $h \in P^{\vee}$
satisfying following relations:
\begin{enumerate}
  \item $[h,h']=0, \ [h,e_i]=\alpha_{i}(h)e_i, \ \ [h,f_i]=-\alpha_{i}(h)f_i \ $ for all $ h,h' \in P^{\vee}$,
  \item $[e_i,f_j]= \delta_{i,j}h_i\ $ for $ i,j \in I$,
  \item $({\rm ad} e_i)^{1-a_{ij}}(e_j)=({\rm ad}f_i)^{1-a_{ij}}(f_j)=0$ if $i \ne j$.
\end{enumerate}
\end{definition}

A $\g$-module $V$  is called a \emph{weight module} if it admits a \emph{weight space decomposition}
$$V= \bigoplus_{\mu \in P} V_{\mu}$$ where $V_{\mu}=\{ v \in V | \ [h, v] =\langle h, \mu \rangle  \, v \text{ for all }  h \in P^{\vee} \}.$
A weight module $V$ over $\g$ is {\it integrable} if all $e_i$ and $f_i$ ($i \in I$) are locally nilpotent on $V$.

\begin{definition}
The category $\mathcal{O}_{{\rm int}}$ consists of integrable $\g$-modules $V$ satisfying the following conditions:
\begin{enumerate}
\item $V$ admits a weight space decomposition $V=\bigoplus_{\mu \in P} V_{\mu}$ and $\dim_\C(V_\mu) < \infty$ for each weight $\mu$.
\item There exists a finite number of elements $\lambda_1,\ldots,\lambda_s \in P$ such that
$$ {\rm wt}(V) \subset D(\lambda_1) \cup \cdots \cup D(\lambda_s).$$
Here ${\rm wt}(V): =\{ \mu \in P \ | \ V_\mu \ne 0 \}$ and $D(\lambda):= \{ \lambda - \sum_{i \in I} k_i\alpha_i \ | \ k_i \in \Z_{\ge 0} \}$.
\end{enumerate}
\end{definition}

It is well-known that the
category $\mathcal{O}_{{\rm int}}$ is a semisimple tensor category with its irreducible
objects being isomorphic to {\it the highest weight modules $V(\Lambda)$} ($\Lambda \in P^+$), each of which is generated by a {\it highest weight vector} $v_\Lambda$.
Recall, e.g. from \cite[Chapter 10]{Kac}, that if $M$, $N \in \mathcal{O}_{{\rm int}}$, then
\begin{eqnarray} &&
\parbox{90ex}{
\begin{itemize}
\item[{\rm (a)}] $M \simeq N$ if and only if ${\rm ch}(M)={\rm ch}(N)$,
\item[{\rm (b)}] ${\rm ch}(V(\Lambda)) = e^{-t\updelta}{\rm ch}(V(\Lambda+t\updelta))$ for $\Lambda \in P^+$ and $t \in \Z$,
\end{itemize}
}\label{note: character}
\end{eqnarray}
where ${\rm ch}(M):= \sum_{\mu \in P} (\dim_{\C} M_{\mu}) e^\mu$ is the {\em character} of $M$.

For $\eta \in  {\rm wt}(V(\Lambda))$, we define
$$ \Supp_{\Lambda}(\eta) \seteq  \Supp(\Lambda-\eta).$$ 

The dimension of the $\mu$-weight space $V(\Lambda)_\mu$ is called the {\it multiplicity} of $\mu$ in $V(\Lambda)$.
A weight $\mu$ is {\it maximal} if $\mu+\updelta \not \in \wt(V(\Lambda))$. The set of all maximal weights of $V(\Lambda)$ of level $k$ is denoted by $\mx(\Lambda|k)$.

\begin{proposition} $($\cite[Chapter 12.6]{Kac}$)$ For each $\Lambda \in P^+$ of level $k$, we have
$$ {\rm wt}(V(\Lambda)) = \bigsqcup_{\mu \in \mx(\Lambda|k)} \{ \mu-s\updelta \ | \ s \in \Z_{\ge 0} \}.$$
\end{proposition}
We denote by $\mx^+(\Lambda|k)$ the set of all dominant maximal weights of level $k$ in $V(\Lambda)$, i.e., $$\mx^+(\Lambda|k) \seteq \mx(\Lambda|k) \cap P^+.$$
It is well-known that
$$  \mx(\Lambda|k) = W \cdot \mx^+(\Lambda|k)  \quad \text{ where $W$ is the Weyl group of $\g$}.$$

Let $\h_0$ be the $\C$-vector space spanned by $\{ h_i \ | \ i \in I_0\}$ for $I_0 \seteq I \setminus \{0\}$.
Define the {\it orthogonal projection} $\bar{ \ } \colon \h^* \to \h_0^*$ (\cite[(6.2.7)]{Kac}) by
$$  \mu \longmapsto \overline{\mu} =\mu - \mu(c)\Lambda_0 - (\mu | \Lambda_0) \updelta.$$
We denote by $\overline{\rl}$ the image of $\rl$ under the orthogonal projection $\bar{ \ }$. Define
\begin{align}\label{eq: KCaf}
k \mathcal{C}_{{\rm af}}= \{ \mu \in \h_0^* \ | \ \mu(h_i) \ge 0, \ (\mu | \theta) \le k \} \quad \text{ where }
\theta \seteq \updelta-a_0\alpha_0.
\end{align}

\begin{proposition} $($\cite[Proposition 12.6]{Kac}$)$ \label{prop: Kac bijection}
The map $\mu \longmapsto \overline{\mu}$ defines a bijection from $\mx^+(\Lambda|k)$ onto $k \mathcal{C}_{{\rm af}} \cap (\overline{\Lambda}+\overline{\rl})$
where $\Lambda$ is of level $k$. In particular, the set $\mx^+(\Lambda|k)$ is finite.
\end{proposition}

For later use, we present
 the Dynkin diagrams of classical 
 affine types. 

\begin{equation} \label{aff dyn}
\begin{aligned}
\bigtriangleup_{A_{n}^{(1)}}: &  \xymatrix@R=3ex{
*{\bullet}<3pt> \ar@{-}[r]_<{0}  \ar@/^1pc/@{-}[rrrrr]  &*{\bullet}<3pt>
\ar@{-}[r]_<{1}  &
  *{\bullet}<3pt> \ar@{-}[r]_<{2} & {}
\ar@{.}[r]_>{\,\,\,\ n-1} & *{\bullet}<3pt>\ar@{-}[r]_>{\,\,\,\, n} &*{\bullet}<3pt> },  \ \
\bigtriangleup_{B_{n}^{(1)}}:  \raisebox{1em}{\xymatrix@R=3ex{ & *{\bullet}<3pt> \ar@{-}[d]^<{0} \\
*{\bullet}<3pt> \ar@{-}[r]_<{1}  &*{\circ}<3pt>
\ar@{-}[r]_<{2}  &
  *{\circ}<3pt> \ar@{-}[r]_<{3} & {}
\ar@{.}[r]_>{\,\,\,\ n-1} &*{\circ}<3pt>\ar@{=>}[r]_>{\,\,\,\, n} &*{\bullet}<3pt> }},\allowdisplaybreaks\\
 \bigtriangleup_{C_{n}^{(1)}}: & \xymatrix@R=3ex{
*{\bullet}<3pt> \ar@{=>}[r]_<{0}   & *{\circ}<3pt> \ar@{-}[r]_<{1} &
  *{\circ}<3pt> \ar@{-}[r]_<{2} & {}
\ar@{.}[r]_>{\,\,\,\ n-1}
&*{\circ}<3pt>\ar@{<=}[r]_>{\,\,\,\, n} &*{\bullet}<3pt> }, \ \
\bigtriangleup_{D_{n}^{(1)}}:  \raisebox{1em}{\xymatrix@R=3ex{ & *{\bullet}<3pt> \ar@{-}[d]^<{0} & & & *{\bullet}<3pt> \ar@{-}[d]^<{n} \\
*{\bullet}<3pt> \ar@{-}[r]_<{1}  &*{\circ}<3pt>
\ar@{-}[r]_<{2}  &
  *{\circ}<3pt> \ar@{-}[r]_<{3} & {}
\ar@{.}[r] & *{\circ}<3pt> \ar@{-}[r]_<{n-2} & *{\bullet}<3pt> \ar@{-}[l]^<{\ \ \ n-1} }}, \allowdisplaybreaks\\
 \bigtriangleup_{A_{2n-1}^{(2)}}: & \raisebox{1em}{ \xymatrix@R=3ex{ & *{\bullet}<3pt> \ar@{-}[d]^<{0}\\
*{\bullet}<3pt> \ar@{-}[r]_<{1}  &*{\circ}<3pt>
\ar@{-}[r]_<{2}  &
  *{\circ}<3pt> \ar@{-}[r]_<{3} & {}
\ar@{.}[r]_>{\,\,\,\ n-1} &*{\circ}<3pt>\ar@{<=}[r]_>{\,\,\,\, n} &*{\bullet}<3pt> }}, \allowdisplaybreaks \\
\bigtriangleup_{A_{2n}^{(2)}}: &  \xymatrix@R=3ex{
*{\bullet}<3pt> \ar@{<=}[r]_<{0}  &*{\circ}<3pt>
\ar@{-}[r]_<{1}   & *{\circ}<3pt> \ar@{-}[r]_<{2} & {}
\ar@{.}[r]_>{\,\,\,\ n-1}
&*{\circ}<3pt>\ar@{<=}[r]_>{\,\,\,\, n} &*{\bullet}<3pt> }, \
\bigtriangleup_{D_{n+1}^{(2)}}:  \xymatrix@R=3ex{
*{\bullet}<3pt> \ar@{<=}[r]_<{0}  &*{\circ}<3pt>
\ar@{-}[r]_<{1}   & *{\circ}<3pt> \ar@{-}[r]_<{2} & {}
\ar@{.}[r]_>{\,\,\,\ n-1}
&*{\circ}<3pt>\ar@{=>}[r]_>{\,\,\,\, n} &*{\bullet}<3pt> }.
\end{aligned}
\end{equation}

For an affine Dynkin diagram $\bigtriangleup_\g$ and a subset $J \subsetneq I$, we denote by $\bigtriangleup_\g|_{J}$ the full-subdiagram of $\bigtriangleup_\g$ whose vertices are in $J$.
We call a vertex $s$ in $\bigtriangleup_\g$ {\it extremal} if $\bigtriangleup_\g|_{I_s}$ for $I_s \seteq I \setminus \{s\}$ is a connected Dynkin diagram of finite type.
For example, every vertex in $\bigtriangleup_{A^{(1)}_{n}}$ is extremal, while $0$, $1$ and $n$ are all the extremal vertices of $\bigtriangleup_{B^{(1)}_{n}}$.
In \eqref{aff dyn}, the solid dot $\bullet$ denotes an extremal vertex.

Let $\g_s$ be the finite dimensional subalgebra of $\g$ corresponding to $\bigtriangleup_\g|_{I_s}$ for an extremal vertex $s$. Then each finite dimensional simple Lie algebra $\g_\text{fin}$ of classical types appears as the subalgebra $\g_s$ of an affine $\g$ as follows:

\begin{table}[ht]
\renewcommand{\arraystretch}{1.3}
\centering
\begin{tabular}[c]{|c|c|} \hline
$\g_\text{fin}$ & $\g$  \\ \hline
$A_{n}$ & $A_n^{(1)}$   \\ \hline
$B_{n}$ & $B_n^{(1)}$, $A_{2n}^{(2)}$, $D_{n+1}^{(2)}$  \\ \hline
$C_{n}$ & $C_n^{(1)}$, $A_{2n}^{(2)}$, $A_{2n-1}^{(2)}$  \\ \hline
$D_n$ & $B_n^{(1)}$, $A_{2n-1}^{(2)}$, $D^{(1)}_n$ \\ \hline
\end{tabular}
\caption{ Relationship between $\g_\text{fin}$ and $\g$  }
\label{table: g_s}
\end{table}

\begin{convention} We denote an arbitrary fundamental weight of level $1$ by boldfaced $\ofw$ to distinguish them from other (fundamental) weights.
\end{convention}

\subsection{Connection to finite types} \label{subsec-finite}

Let $\Lambda=\sum_{i \in I} m_i \Lambda_i \in P^+$ and $\mu =\Lambda - \sum_{i \in I} k_i \alpha_i \in \mx^+(\Lambda|k)$ for some $k$. Assume $k_s=0$ for a fixed $s \in I$. We consider the finite dimensional subalgebra $\mathfrak g_s$ generated by $e_i, h_i, f_i$ for $i \in I_s$.
Assume that $\mathfrak g_s$ is simple, or equivalently that the Dynkin diagram of $\mathfrak g_s$ is connected. Then $s$ corresponds to an extremal vertex.
We denote by $\omega$ the weight of $\mathfrak g_s$ corresponding to $\Lambda$ via $\omega(h_i) =m_i$ for $i \in I_s$ and let $\eta = \omega - \sum_{i \in I_s} k_i \alpha_i$. Then $\eta$ is clearly a dominant weight of $\mathfrak g_s$. A highest weight vector $v_\Lambda$ in $V(\Lambda)$ generates the highest weight module $L(\omega)$ as $\mathfrak g_s$-module. Moreover, since $k_s =0$, we have
\begin{equation} \label{ena-sam} \dim (V(\Lambda)_\mu) = \dim ( L(\omega)_\eta) . \end{equation}

Conversely, consider a finite dimensional simple Lie algebra $\mathfrak g_\text{fin}$ of type $A_n, B_n, C_n$ or $D_n$, and identify it with the subalgebra $\mathfrak g_s$ of an affine Kac--Moody algebra $\mathfrak g$ for some $s \in I$. Let $\omega =\sum_{i \in J} m_i \omega_i$ be a dominant integral weight of $\mathfrak g_\text{fin}$ and $L(\omega)$ be the highest weight module with highest weight $\omega$, where $\omega_i$ are the fundamental weights of $\mathfrak g_\text{fin}$. Consider a dominant weight $\eta=\omega - \sum_{i \in J} k_i \alpha_i$ of $L(\omega)$.
We may assume $J=I_s$ according to the identification of $\mathfrak g_\text{fin}$ with $\mathfrak g_s$.
If we let $\Lambda = \sum_{i \in J} m_i \Lambda_i$ and $\mu = \Lambda - \sum_{i \in J} k_i \alpha_i$,
then $\mu \in \mx^+(\Lambda|k)$ for some $k$. In this case, the equation \eqref{ena-sam} also holds.

Motivated by the above observation, we make the following definition.
\begin{definition} \label{def-ess-fin}
Let $\Lambda=\sum_{i \in I} m_i \Lambda_i \in P^+$.
A dominant maximal weight $\mu =\Lambda - \sum_{i \in I} k_i \alpha_i \in \mx^+(\Lambda|k)$ is called {\em essentially finite of type $X_n$} if there is an $s \in I$ such that $k_s=0$ and  $\mathfrak g_s$  is of finite type $X_n$ with $X=A,B,C$ or $D$.
\end{definition}

In Section \ref{sec:Rep}, we will see that most of the dominant maximal weights are essentially finite.

\subsection{Quantum affine algebras}
Let $q$ be an indeterminate and $m,n \in \Z_{\ge 0}$. For $i\in I$,
let $q_i = q^{ \mathsf{d}_i}$ and
\begin{equation*}
 \begin{aligned}
 \ &[n]_i =\frac{ q^n_{i} - q^{-n}_{i} }{ q_{i} - q^{-1}_{i} },
 & \quad   [n]_i! = \prod^{n}_{k=1} [k]_i ,
 \ & \quad \left[\begin{matrix}m \\ n \end{matrix} \right]_i=  \frac{ [m]_i! }{[m-n]_i! [n]_i! }.
 \end{aligned}
\end{equation*}

\begin{definition} The {\em quantum affine algebra} $U_q(\g)$ associated with an affine Cartan datum $(\cmA,P^{\vee},P,\Pi^{\vee},\Pi)$
is the associative algebra over $\Q(q)$ with ${\bf 1}$ generated by $e_i,f_i$ $(i \in I)$ and
$q^{h}$ $(h \in P^{\vee})$ satisfying the following relations:
\begin{enumerate}
\item  $q^0=1,\ q^{h} q^{h'}=q^{h+h'}, \ q^{h}e_i q^{-h}= q^{ \langle h,\alpha_i \rangle} e_i,
          \ q^{h}f_i q^{-h} = q^{- \langle h,\alpha_i \rangle }f_i$ for $ h,h' \in P^{\vee},$
\item  $e_if_j - f_je_i = \delta_{i,j} \dfrac{K_i -K^{-1}_i}{q_i- q^{-1}_i }, \ \ \mbox{ where } K_i=q_i^{ h_i},$
\item  $\displaystyle \sum^{1-a_{ij}}_{k=0}(-1)^k  
e^{(1-a_{ij}-k)}_i e_j e^{(k)}_i = \displaystyle \sum^{1-a_{ij}}_{k=0} (-1)^k 
f^{(1-a_{ij}-k)}_if_j f^{(k)}_i=0 \ \ \text{ if }  i \ne j. $
\end{enumerate}
Here we set
$$  e_i^{(n)} \seteq e_i^n/[n]_{i}! \quad \text{ and } \quad f_i^{(n)} \seteq f_i^n/[n]_{i}!.$$
\end{definition}

We define integrable $U_q(\g)$-modules,
the category $\Oqint$, the character for $M \in \Oqint$ and highest weight modules $V^q(\Lambda)$ for $\Lambda \in P^+$ in the standard way (\cite{HK}). It is well-known that
 $\Oqint$ is a semisimple tensor category with its irreducible object being isomorphic to $V^q(\Lambda)$ for some $\Lambda \in P^+$ and
\begin{align} \label{eq: ch=qch}
{\rm ch}\big(V(\Lambda)\big) =  {\rm ch}\big(V^q(\Lambda)\big) \quad \text{ and hence } \quad {\rm dim}_\Q(V(\Lambda)_\mu)={\rm dim}_{\Q(q)}(V^q(\Lambda)_\mu)
\end{align}
for any $\mu \in P$.

\section{Crystals and Young walls} \label{sec:CW}
In this section, we briefly review the theory of {\it crystals} developed by Kashiwara (\cite{Kas90,Kas91}). Then we recall the
combinatorial realization of affine crystals, called the {\it Young walls}, due to Kang (\cite{Kang03}).

\subsection{Crystals} For an index $i \in I$ and $M=\bigoplus_{\mu \in P} M_{\mu} \in \Oqint$, every element $v \in M_\mu$ can
be uniquely expressed as
$$ v = \sum_{k \ge 0} f_i^{(k)}v_k,$$
where $\mu(h_i) + k \ge 0$ and $v_k \in {\rm Ker} \ e_i \cap M_{\mu+k\al_i}$. The {\it Kashiwara operators} $\eit$ and $\fit$ are defined by
\begin{align} \label{eq: Kashiwara operator}
\eit v = \sum_{k \ge 1} f_i^{(k-1)}v_k, \quad \fit v = \sum_{k \ge 0} f_i^{(k+1)}v_k.
\end{align}

Let $\A_0=\{f/g \in \Q(q) \mid f,g \in \Q[q], \ g(0) \neq 0 \}$ and $M$ a weight $U_q(\g)$-module.

\begin{definition} A {\it crystal basis} of $M$ consists of a pair $(L,B)$ with the Kashiwara operators
$\tilde{e}_i$ and $\tilde{f}_i$ ($i \in I$) as follows:
\begin{enumerate}
\item $L= \bigoplus_\mu L_\mu$ is a free $\A_0$-submodule of $M$ such that
$$ M \simeq \Q(q) \otimes_{\A_0} L \quad\text{ and }\quad L_\mu = L \cap M_\mu,$$
\item $B= \bigsqcup_\mu B_\mu$ is a basis of the $\Q$-vector space $L/qL$, where $B_\mu=B \cap (L_\mu/qL_\mu)$,
\item $\tilde{e}_i$ and $\tilde{f}_i$ ($i \in I$) are defined on $L$, i.e., $\tilde{e}_iL, \tilde{f}_iL \subset L$,
\item the induced maps $\tilde{e}_i$ and $\tilde{f}_i$ on $L/qL$ satisfy
$$ \tilde{e}_iB, \tilde{f}_iB \subset B \sqcup \{0\}, \quad \text{ and } \quad  \tilde{f}_ib=b' \quad \text{ if and only if } \quad b=\tilde{e}_ib' \quad  \text{ for } b,b' \in B.$$
\end{enumerate}
\end{definition}

The set $B$ has a colored oriented graph structure as follows:
$$ b \overset{i}{\longrightarrow} b' \quad\text{ if and only if }\quad \tilde{f}_ib=b'.$$
The graph structure encodes information on the structure of $M$. For example,
\begin{itemize}
\item $|B_\mu| =\dim_{\Q(q)}M_\mu$ for all $\mu \in {\rm wt}(M)$,
\item the graph of $B$ is connected if and only if $M$ is irreducible.
\end{itemize}

\begin{theorem}[\cite{Kas91}] \label{thm: existence crystal}
For $\Lambda \in P^+$, the module $V^q(\Lambda)$ has a crystal basis $(\mathbf{L}(\Lambda),\mathbf{B}(\Lambda))$ given
 as follows:
\begin{enumerate}
\item $\mathbf{L}(\Lambda)$ is the $\A_0$-submodule generated by $\{\tilde{f}_{i_1} \cdots \tilde{f}_{i_r}v_\Lambda \mid r \ge 0 , i_k \in I\}$,
\item $\mathbf{B}(\Lambda) = \{\tilde{f}_{i_1} \cdots \tilde{f}_{i_r}v_\Lambda+qL(\Lambda) \mid r \ge 0 , i_k \in I\}\setminus\{0\}$.
\end{enumerate}
\end{theorem}

By \eqref{note: character}, \eqref{eq: ch=qch} and the above theorem, we have that for $k \in \Z$
\begin{align*}
{\rm ch}( V(\Lambda) )= \sum_{\mu \in P} |\mathbf{B}(\Lambda)_\mu|e^\mu \quad \text{ and } \quad |\mathbf{B}(\Lambda)_\mu| = |\mathbf{B}(\Lambda+k\updelta)_{\mu+k\updelta}|.
\end{align*}
In particular,
\begin{eqnarray} &&
\parbox{85ex}{
\begin{enumerate}
\item[{\rm (a)}] $|\mathbf{B}(\Lambda)_\mu| = |\mathbf{B}(\Lambda+k\updelta)_{\mu+k\updelta}| \text{ for any } \mu \in \mx^+(\Lambda)$,
\item[{\rm (b)}] ${\rm ch}\big(V(\Lambda)\big) = \sum_{\mu \in P} |\mathbf{B}(\Lambda+k\updelta)_\mu|e^{\mu-k\updelta}$.
\end{enumerate}
}\label{note: kdelta+}
\end{eqnarray}

\begin{definition}
An {\it $($affine$)$ crystal} associated to an affine Cartan datum $(\cmA,P^{\vee},P,\Pi^{\vee},\Pi)$ is the set $B$ together with maps
$$ \wt: B \to P, \ \ \ve_i,\vp_i: B \to \Z \sqcup \{ -\infty \} \ \text{ and } \  \eit,\fit : B \to B \sqcup \{0\} \ \ (i \in I)$$
satisfying the following conditions:
\begin{enumerate}
\item[{\rm (i)}] For $i \in I$, $b \in B$, we have
$$ \vp_i(b)=\ve_i(b)+ \langle \wt(b), h_i \rangle,  \ \ \wt(\eit b)=\wt(b)+\al_i, \ \ \wt(\fit b)=\wt(b)-\al_i,$$
\item[{\rm (ii)}] if $\eit b \in B$, then $\ve_i(\eit b)=\ve_i(b)-1$ and $\vp_i(\eit b)=\vp_i(b)+1$,
\item[{\rm (iii)}] if $\fit b \in B$, then $\ve_i(\fit b)=\ve_i(b)+1$ and $\vp_i(\fit b)=\vp_i(b)-1$,
\item[{\rm (iv)}] $\fit b=b'$ if and only if $b=\eit b'$ for all $i \in I$, $b,b'\in B$,
\item[{\rm (v)}] if $\ve_i(b)=-\infty$, then $\eit b=\fit b=0$.
\end{enumerate}
\end{definition}

\begin{definition} The {\it tensor product} $B_1 \otimes B_2$ of crystals $B_1$ and $B_2$ is defined to be the set $B_1 \times B_2$ whose crystal structure is given by
\begin{enumerate}
\item[{\rm (i)}] $\wt(b_1 \otimes b_2) = \wt(b_1) + \wt(b_2)$,
\item[{\rm (ii)}] $\ve_i(b_1 \otimes b_2)=\max(\ve_i(b_1),\ve_i(b_2)- \langle \wt(b_1), h_i \rangle )$,
$\vp_i(b_1 \otimes b_2)=\max(\vp_i(b_2),\vp_i(b_1)+ \langle  \wt(b_2), h_i \rangle )$,
\item[{\rm (iii)}] $\eit(b_1 \otimes b_2) = \begin{cases} \eit b_1 \otimes b_2 &\text{if } \vp_i(b_1) \ge \ve_i(b_2),\\
b_1 \otimes \eit b_2 &\text{if } \vp_i(b_1) < \ve_i(b_2), \end{cases}$
$\fit(b_1 \otimes b_2) = \begin{cases} \fit b_1 \otimes b_2 &\text{if } \vp_i(b_1) > \ve_i(b_2),\\
b_1 \otimes \fit b_2 &\text{if } \vp_i(b_1) \le \ve_i(b_2). \end{cases}$
\end{enumerate}
\end{definition}

\begin{theorem} \cite{Kas90,Kas91}
For $M$ and $N \in \Oqint$ with crystals $B_M$ and $B_N$, the tensor product $B_M \otimes B_N$ is the crystal of $M \otimes N \in \Oqint$.
\end{theorem}

\subsection{Connection to finite type crystals} \label{subsec-fini-crys}
Now we interpret the arguments in Section \ref{subsec-finite} from the viewpoint of crystals. As we mentioned above, $\mathbf{B}(\Lambda)$
can be understood as a colored oriented graph. For an extremal vertex $s \in I$, we denoted by $\mathbf{B}(\Lambda)|_{I_s}$ the graph obtained by removing the arrows $\overset{s}{\longrightarrow}$ of color $s$. Then we have
$$  \mathbf{B}(\Lambda)|_{I_s} = \bigsqcup_{\omega'}   \mathbf{B}(\omega') \qquad \text{ as $\g_s$-crystals.} $$
Here $\mathbf{B}(\omega')$ is a connected component of $\mathbf{B}(\Lambda)|_{I_s}$, which is a crystal of some irreducible module $L(\omega')$ over $U_q(\g_s)$.

For a highest weight $\Lambda= \sum_{i\in I} m_i \Lambda_i$ and an essentially finite dominant maximal weight $\mu=\Lambda - \sum_{i \in I} k_i \alpha_i \in \mx^+(\Lambda|k)$, we have
$$  \mathbf{B}(\Lambda)_\mu = \mathbf{B}(\omega)_\eta \quad
\text{where } \ \omega = \hspace{-1.5ex} \sum_{i \in \Supp_\Lambda(\mu)} \hspace{-1.5ex}  m_i \omega_i  \ \text{ and } \  \eta=\omega - \hspace{-1.5ex}
\sum_{i \in \Supp_\Lambda(\mu)} \hspace{-1.5ex} k_i\alpha_i.$$

\begin{definition} For an extremal $s \in I$ and a highest weight $\Lambda \in P^+$,
we denote by $\mathbf{B}^0(\Lambda)|_{I_s}$ the connected component of $\mathbf{B}(\Lambda)|_{I_s}$ originated from the highest weight element $v_{\Lambda}$.
\end{definition}

The following lemma is obvious.

\begin{lemma} \label{Thm: depend on finite}
We assume the following conditions:
\begin{enumerate}
\item  For $V(\Lambda)$ over an affine $\g$ and  $\eta \in \mx^+(\Lambda|k)$, there exists an extremal $s \not \in \Supp_\Lambda(\eta)$ such that $\bigtriangleup_{\g}|_{I_s}$ is of finite type $X_n$.
\item  For $V(\Lambda')$ over another affine $\g'$ and $\mu \in \mx^+(\Lambda'|k')$, there exists an extremal $s' \not \in \Supp_{\Lambda'}(\mu)$ such that $\bigtriangleup_{\g'}|_{I_{s'}}$ is of the same finite type $X_n$.
\item We have $\mathbf{B}(\Lambda)^0|_{I_s} \simeq \mathbf{B}(\Lambda')^0|_{I_{s'}}$ and
$\eta \simeq \mu$ via a bijection $\sigma:I_{s} \to I_{s'}$ which induces a diagram isomorphism $\bigtriangleup_\g|_{I_s} \simeq \bigtriangleup_{\g'}|_{I_{s'}}$; that is,
$$ \eta=\Lambda - \sum_{i \in I_s}  m_i \alpha_i   \quad \text{ and } \quad  \mu=\Lambda' -  \sum_{i \in I_s} m_{\sigma(i)} \alpha_{\sigma(i)}.$$
\end{enumerate}
Then we have
$$ {\rm dim}V(\Lambda)_\eta ={\rm dim}V(\Lambda')_\mu.$$
\end{lemma}

\subsection{Young walls for level $1$ representations} \label{subsec: Young wall}
In \cite{Kang03}, Kang constructed realizations of level $1$ highest weight crystals $B(\ofw)$ for all classical
quantum affine algebras except $C_n^{(1)}$ in terms of \emph{reduced Young walls}.
For the rest of this section, we assume that $\g$ is an affine Kac-Moody algebra of type $A^{(2)}_{2n-1}$, $A^{(2)}_{2n}$, $B^{(1)}_{n}$, $D^{(1)}_{n}$ or  $D^{(2)}_{n+1}$.

Young walls are built from colored blocks.
There are three types of blocks whose shapes are different and which appear depending on affine Cartan types as follows:

\begin{center}
\begin{tabular}{|c c c c |c|} \hline
Shape & Width & Thickness & Height & Type \\ \hline & & & &  \vspace*{-0.35 cm} \\
\scalebox{0.7}{\ublock} & 1 & 1 & 1 &  all types \\ [1ex]
\scalebox{0.7}{\uhblock} & 1 & 1 & 1/2 & $A^{(2)}_{2n},B^{(1)}_{n},D^{(2)}_{n+1}$ \\ [1ex]
\scalebox{0.7}{\uhtbblock} & 1 & 1/2 & 1 & $A^{(2)}_{2n-1},B^{(1)}_{n},D^{(1)}_{n}$\\ [1.5ex] \hline
\end{tabular}
\end{center}

The walls are built on the ground-state wall $\boxed{\ofw}$, which is given below as the shaded part in \eqref{eq: gwalls}, by the following rules:
\begin{enumerate}
\item Blocks should be built in the pattern given below in \eqref{eq: YW pat hh}, \eqref{eq: YW pat ht} or \eqref{eq: YW pat ht2}.
\item No block can be placed on top of a column of half-unit thickness.
\item There should be no free space to the right of any block except the rightmost column.
\end{enumerate}

Ground-state Young walls $\boxed{\ofw}$ corresponding to $\ofw$ are given as follows:
\begin{equation}\label{eq: gwalls}
\begin{aligned}
& \scalebox{0.7}{\boxed{\ofw^{\uhb}_1}}\seteq\scalebox{0.8}{\gwalltdht{0}{1}}, \quad
 \scalebox{0.7}{\boxed{\ofw^{\uhb}_0}}\seteq\scalebox{0.8}{\gwalltdht{1}{0}}, \allowdisplaybreaks\\
& \scalebox{0.7}{\boxed{\ofw^{\uhb}_{n-1}}}\seteq\scalebox{0.8}{\gwalltdht{_n}{_{n-1}}},  \quad
 \scalebox{0.7}{\boxed{\ofw^{\uhb}_{n}}}\seteq\scalebox{0.8}{\gwalltdht{_{n-1}}{_n}},  \allowdisplaybreaks\\
& \scalebox{0.7}{\boxed{\ofw^{\uhh}_{0}}}\seteq\scalebox{0.8}{\gwalltdhh{0}}, \quad
 \scalebox{0.7}{\boxed{\ofw^{\uhh}_{n}}}\seteq\scalebox{0.8}{\gwalltdhh{n}}.
 \end{aligned}
\end{equation}

Now we give the patterns mentioned above:
\begin{align} \label{eq: YW pat hh}
\scalebox{0.75}{\xy
(0,-0.5)*{};(33,-0.5)*{} **\dir{.};
(0,-1)*{};(33,-1)*{} **\dir{.};
(0,-1.5)*{};(33,-1.5)*{} **\dir{.};
(0,-2)*{};(33,-2)*{} **\dir{.};
(0,-2.5)*{};(33,-2.5)*{} **\dir{.};
(0,-3)*{};(33,-3)*{} **\dir{-};
(0,0)*{};(33,0)*{} **\dir{-};
(0,3)*{};(33,3)*{} **\dir{-};
(0,9)*{};(33,9)*{} **\dir{-};
(0,19)*{};(33,19)*{} **\dir{-};
(0,25)*{};(33,25)*{} **\dir{-};
(0,35)*{};(33,35)*{} **\dir{-};
(0,41)*{};(33,41)*{} **\dir{-};
(0,44)*{};(33,44)*{} **\dir{-};
(0,47)*{};(33,47)*{} **\dir{-};
(0,53)*{};(33,53)*{} **\dir{-};
(0,59)*{};(33,59)*{} **\dir{-};
(33,-3)*{};(33,64)*{} **\dir{-};
(27,-3)*{};(27,64)*{} **\dir{-};
(21,-3)*{};(21,64)*{} **\dir{-};
(15,-3)*{};(15,64)*{} **\dir{-};
(9,-3)*{};(9,64)*{} **\dir{-};
(3,-3)*{};(3,64)*{} **\dir{-};
(30.2,-1.5)*{_0}; (24.2,-1.5)*{_0}; (24.2,-1.5)*{_0};(18.2,-1.5)*{_0};(12.2,-1.5)*{_0}; (6.2,-1.5)*{_0};
(30.2,1.5)*{_0}; (30.2,6)*{_1}; (30.2,15)*{\vdots}; (30.2,22)*{_n}; (30.2,31)*{\vdots};
(30.2,38)*{_1}; (30.2,42.5)*{_0}; (30.2,45.5)*{_0}; (30.2,50)*{_1}; (30.2,56)*{_2};
(24.2,1.5)*{_0}; (24.2,6)*{_1}; (24.2,15)*{\vdots}; (24.2,22)*{_n}; (24.2,31)*{\vdots};
(24.2,38)*{_1}; (24.2,42.5)*{_0}; (24.2,45.5)*{_0}; (24.2,50)*{_1}; (24.2,56)*{_2};
(18.2,1.5)*{_0}; (18.2,6)*{_1}; (18.2,15)*{\vdots}; (18.2,22)*{_n}; (18.2,31)*{\vdots};
(18.2,38)*{_1}; (18.2,42.5)*{_0}; (18.2,45.5)*{_0}; (18.2,50)*{_1}; (18.2,56)*{_2};
(12.2,1.5)*{_0}; (12.2,6)*{_1}; (12.2,15)*{\vdots}; (12.2,22)*{_n}; (12.2,31)*{\vdots};
(12.2,38)*{_1}; (12.2,42.5)*{_0}; (12.2,45.5)*{_0}; (12.2,50)*{_1}; (12.2,56)*{_2};
(6.2,1.5)*{_0}; (6.2,6)*{_1}; (6.2,15)*{\vdots}; (6.2,22)*{_n}; (6.2,31)*{\vdots};
(6.2,38)*{_1}; (6.2,42.5)*{_0}; (6.2,45.5)*{_0}; (6.2,50)*{_1}; (6.2,56)*{_2};
(15,-6)*{A_{2n}^{(2)}, \ \ofw_0};
(40,-0.5)*{};(73,-0.5)*{} **\dir{.};
(40,-1)*{};(73,-1)*{} **\dir{.};
(40,-1.5)*{};(73,-1.5)*{} **\dir{.};
(40,-2)*{};(73,-2)*{} **\dir{.};
(40,-2.5)*{};(73,-2.5)*{} **\dir{.};
(40,-3)*{};(73,-3)*{} **\dir{-};
(40,0)*{};(73,0)*{} **\dir{-};
(40,3)*{};(73,3)*{} **\dir{-};
(40,9)*{};(73,9)*{} **\dir{-};
(40,19)*{};(73,19)*{} **\dir{-};
(40,25)*{};(73,25)*{} **\dir{-};
(40,28)*{};(73,28)*{} **\dir{-};
(40,31)*{};(73,31)*{} **\dir{-};
(40,37)*{};(73,37)*{} **\dir{-};
(40,47)*{};(73,47)*{} **\dir{-};
(40,53)*{};(73,53)*{} **\dir{-};
(40,56)*{};(73,56)*{} **\dir{-};
(40,59)*{};(73,59)*{} **\dir{-};
(40,65)*{};(73,65)*{} **\dir{-};
(73,-3)*{};(73,66)*{} **\dir{-};
(67,-3)*{};(67,66)*{} **\dir{-};
(61,-3)*{};(61,66)*{} **\dir{-};
(55,-3)*{};(55,66)*{} **\dir{-};
(49,-3)*{};(49,66)*{} **\dir{-};
(43,-3)*{};(43,66)*{} **\dir{-};
(70.2,-1.5)*{_0}; (64.2,-1.5)*{_0}; (58.2,-1.5)*{_0};(52.2,-1.5)*{_0}; (46.2,-1.5)*{_0};
(70.2,1.5)*{_0}; (70.2,6)*{_1}; (70.2,15)*{\vdots}; (70.2,22)*{_{n-1}}; (70.2,26.5)*{_{n}}; (70.2,29.5)*{_{n}}; (70.2,34)*{_{n-1}}; (70.2,43)*{\vdots};
(70.2,50)*{_1}; (70.2,54.5)*{_0}; (70.2,57.5)*{_0}; (70.2,62)*{_1};
(64.2,1.5)*{_0}; (64.2,6)*{_1}; (64.2,15)*{\vdots}; (64.2,22)*{_{n-1}}; (64.2,26.5)*{_{n}}; (64.2,29.5)*{_{n}}; (64.2,34)*{_{n-1}}; (64.2,43)*{\vdots};
(64.2,50)*{_1}; (64.2,54.5)*{_0}; (64.2,57.5)*{_0}; (64.2,62)*{_1};
(58.2,1.5)*{_0}; (58.2,6)*{_1}; (58.2,15)*{\vdots}; (58.2,22)*{_{n-1}}; (58.2,26.5)*{_{n}}; (58.2,29.5)*{_{n}}; (58.2,34)*{_{n-1}}; (58.2,43)*{\vdots};
(58.2,50)*{_1}; (58.2,54.5)*{_0}; (58.2,57.5)*{_0}; (58.2,62)*{_1};
(52.2,1.5)*{_0}; (52.2,6)*{_1}; (52.2,15)*{\vdots}; (52.2,22)*{_{n-1}}; (52.2,26.5)*{_{n}}; (52.2,29.5)*{_{n}}; (52.2,34)*{_{n-1}}; (52.2,43)*{\vdots};
(52.2,50)*{_1}; (52.2,54.5)*{_0}; (52.2,57.5)*{_0}; (52.2,62)*{_1};
(46.2,1.5)*{_0}; (46.2,6)*{_1}; (46.2,15)*{\vdots}; (46.2,22)*{_{n-1}}; (46.2,26.5)*{_{n}}; (46.2,29.5)*{_{n}}; (46.2,34)*{_{n-1}}; (46.2,43)*{\vdots};
(46.2,50)*{_1}; (46.2,54.5)*{_0}; (46.2,57.5)*{_0}; (46.2,62)*{_1};
(55,-6)*{D_{n+1}^{(2)}, \ \ofw_0};
(80,-0.5)*{};(113,-0.5)*{} **\dir{.};
(80,-1)*{};(113,-1)*{} **\dir{.};
(80,-1.5)*{};(113,-1.5)*{} **\dir{.};
(80,-2)*{};(113,-2)*{} **\dir{.};
(80,-2.5)*{};(113,-2.5)*{} **\dir{.};
(80,-3)*{};(113,-3)*{} **\dir{-};
(80,0)*{};(113,0)*{} **\dir{-};
(80,3)*{};(113,3)*{} **\dir{-};
(80,9)*{};(113,9)*{} **\dir{-};
(80,19)*{};(113,19)*{} **\dir{-};
(80,25)*{};(113,25)*{} **\dir{-};
(80,28)*{};(113,28)*{} **\dir{-};
(80,31)*{};(113,31)*{} **\dir{-};
(80,37)*{};(113,37)*{} **\dir{-};
(80,47)*{};(113,47)*{} **\dir{-};
(80,53)*{};(113,53)*{} **\dir{-};
(80,56)*{};(113,56)*{} **\dir{-};
(80,59)*{};(113,59)*{} **\dir{-};
(80,65)*{};(113,65)*{} **\dir{-};
(113,-3)*{};(113,66)*{} **\dir{-};
(107,-3)*{};(107,66)*{} **\dir{-};
(101,-3)*{};(101,66)*{} **\dir{-};
(95,-3)*{};(95,66)*{} **\dir{-};
(89,-3)*{};(89,66)*{} **\dir{-};
(83,-3)*{};(83,66)*{} **\dir{-};
(110.2,-1.5)*{_n}; (104.2,-1.5)*{_n}; (98.2,-1.5)*{_n};(92.2,-1.5)*{_n}; (86.2,-1.5)*{_n};
(110.2,1.5)*{_n}; (110.2,6)*{_{n-1}}; (110.2,15)*{\vdots}; (110.2,22)*{_1}; (110.2,26.5)*{_{0}}; (110.2,29.5)*{_{0}}; (110.2,34)*{_1}; (110.2,43)*{\vdots};
(110.2,50)*{_{n-1}}; (110.2,54.5)*{_n}; (110.2,57.5)*{_n}; (110.2,62)*{_{n-1}};
(104.2,1.5)*{_n}; (104.2,6)*{_{n-1}}; (104.2,15)*{\vdots}; (104.2,22)*{_1}; (104.2,26.5)*{_{0}}; (104.2,29.5)*{_{0}}; (104.2,34)*{_1}; (104.2,43)*{\vdots};
(104.2,50)*{_{n-1}}; (104.2,54.5)*{_n}; (104.2,57.5)*{_n}; (104.2,62)*{_{n-1}};
(98.2,1.5)*{_n}; (98.2,6)*{_{n-1}}; (98.2,15)*{\vdots}; (98.2,22)*{_1}; (98.2,26.5)*{_{0}}; (98.2,29.5)*{_{0}}; (98.2,34)*{_1}; (98.2,43)*{\vdots};
(98.2,50)*{_{n-1}}; (98.2,54.5)*{_n}; (98.2,57.5)*{_n}; (98.2,62)*{_{n-1}};
(92.2,1.5)*{_n}; (92.2,6)*{_{n-1}}; (92.2,15)*{\vdots}; (92.2,22)*{_1}; (92.2,26.5)*{_{0}}; (92.2,29.5)*{_{0}}; (92.2,34)*{_1}; (92.2,43)*{\vdots};
(92.2,50)*{_{n-1}}; (92.2,54.5)*{_n}; (92.2,57.5)*{_n}; (92.2,62)*{_{n-1}};
(86.2,1.5)*{_n}; (86.2,6)*{_{n-1}}; (86.2,15)*{\vdots}; (86.2,22)*{_1}; (86.2,26.5)*{_{0}}; (86.2,29.5)*{_{0}}; (86.2,34)*{_1}; (86.2,43)*{\vdots};
(86.2,50)*{_{n-1}}; (86.2,54.5)*{_n}; (86.2,57.5)*{_n}; (86.2,62)*{_{n-1}};
(95,-6)*{D_{n+1}^{(2)}, \ \ofw_n};
(120,-0.5)*{};(153,-0.5)*{} **\dir{.};
(120,-1)*{};(153,-1)*{} **\dir{.};
(120,-1.5)*{};(153,-1.5)*{} **\dir{.};
(120,-2)*{};(153,-2)*{} **\dir{.};
(120,-2.5)*{};(153,-2.5)*{} **\dir{.};
(120,-3)*{};(153,-3)*{} **\dir{-};
(120,0)*{};(153,0)*{} **\dir{-};
(120,3)*{};(153,3)*{} **\dir{-};
(120,9)*{};(153,9)*{} **\dir{-};
(120,19)*{};(153,19)*{} **\dir{-};
(120,25)*{};(153,25)*{} **\dir{-};
(123,25)*{};(129,31)*{} **\dir{-}; %
(129,25)*{};(135,31)*{} **\dir{-}; %
(135,25)*{};(141,31)*{} **\dir{-}; %
(141,25)*{};(147,31)*{} **\dir{-}; %
(147,25)*{};(153,31)*{} **\dir{-}; %
(120,31)*{};(153,31)*{} **\dir{-};
(120,37)*{};(153,37)*{} **\dir{-};
(120,47)*{};(153,47)*{} **\dir{-};
(120,53)*{};(153,53)*{} **\dir{-};
(120,56)*{};(153,56)*{} **\dir{-};
(120,59)*{};(153,59)*{} **\dir{-};
(120,65)*{};(153,65)*{} **\dir{-};
(153,-3)*{};(153,66)*{} **\dir{-};
(147,-3)*{};(147,66)*{} **\dir{-};
(141,-3)*{};(141,66)*{} **\dir{-};
(135,-3)*{};(135,66)*{} **\dir{-};
(129,-3)*{};(129,66)*{} **\dir{-};
(123,-3)*{};(123,66)*{} **\dir{-};
(150.2,-1.5)*{_n}; (144.2,-1.5)*{_n}; (138.2,-1.5)*{_n};(132.2,-1.5)*{_n}; (126.2,-1.5)*{_n};
(150.2,1.5)*{_n}; (150.2,6)*{_{ n-1}}; (150.2,15)*{\vdots}; (150.2,22)*{_2}; (151.2,26.5)*{_{0}}; (149.2,29.5)*{_{1}}; (150.2,34)*{_2}; (150.2,43)*{\vdots};
(150.2,50)*{_{ n-1}}; (150.2,54.5)*{_n}; (150.2,57.5)*{_n}; (150.2,62)*{_{ n-1}};
(144.2,1.5)*{_n}; (144.2,6)*{_{ n-1}}; (144.2,15)*{\vdots}; (144.2,22)*{_2}; (145.2,26.5)*{_{1}}; (143.2,29.5)*{_{0}}; (144.2,34)*{_2}; (144.2,43)*{\vdots};
(144.2,50)*{_{ n-1}}; (144.2,54.5)*{_n}; (144.2,57.5)*{_n}; (144.2,62)*{_{ n-1}};
(138.2,1.5)*{_n}; (138.2,6)*{_{ n-1}}; (138.2,15)*{\vdots}; (138.2,22)*{_2}; (139.2,26.5)*{_{0}}; (137.2,29.5)*{_{1}}; (138.2,34)*{_2}; (138.2,43)*{\vdots};
(138.2,50)*{_{ n-1}}; (138.2,54.5)*{_n}; (138.2,57.5)*{_n}; (138.2,62)*{_{ n-1}};
(132.2,1.5)*{_n}; (132.2,6)*{_{ n-1}}; (132.2,15)*{\vdots}; (132.2,22)*{_2}; (133.2,26.5)*{_{1}}; (131.2,29.5)*{_{0}}; (132.2,34)*{_2}; (132.2,43)*{\vdots};
(132.2,50)*{_{ n-1}}; (132.2,54.5)*{_n}; (132.2,57.5)*{_n}; (132.2,62)*{_{ n-1}};
(126.2,1.5)*{_n}; (126.2,6)*{_{n-1}}; (126.2,15)*{\vdots}; (126.2,22)*{_{2}}; (127.2,26.5)*{_{0}}; (125.2,29.5)*{_{1}}; (126.2,34)*{_{2}}; (126.2,43)*{\vdots};
(126.2,50)*{_{n-1}}; (126.2,54.5)*{_n}; (126.2,57.5)*{_n}; (126.2,62)*{_{n-1}};
(135,-6)*{B_{n}^{(1)}, \ \ofw_n};
\endxy}
\end{align}
\vskip -1em
\begin{align} \label{eq: YW pat ht}
& \scalebox{0.75}{\xy
(3.5,-2.5)*{};(9,-2.5)*{} **\dir{.}; (4,-2)*{};(9,-2)*{} **\dir{.};
(4.5,-1.5)*{};(9,-1.5)*{} **\dir{.}; (5,-1)*{};(9,-1)*{} **\dir{.};
(5.5,-0.5)*{};(9,-0.5)*{} **\dir{.}; (6,0)*{};(9,0)*{} **\dir{.};
(6.5,0.5)*{};(9,0.5)*{} **\dir{.}; (7,1)*{};(9,1)*{} **\dir{.};
(7.5,1.5)*{};(9,1.5)*{} **\dir{.}; (8,2)*{};(9,2)*{} **\dir{.};
(8.5,2.5)*{};(9,2.5)*{} **\dir{.};
(9.5,-2.5)*{};(15,-2.5)*{} **\dir{.}; (10,-2)*{};(15,-2)*{} **\dir{.};
(10.5,-1.5)*{};(15,-1.5)*{} **\dir{.}; (11,-1)*{};(15,-1)*{} **\dir{.};
(11.5,-0.5)*{};(15,-0.5)*{} **\dir{.}; (12,0)*{};(15,0)*{} **\dir{.};
(12.5,0.5)*{};(15,0.5)*{} **\dir{.}; (13,1)*{};(15,1)*{} **\dir{.};
(13.5,1.5)*{};(15,1.5)*{} **\dir{.}; (14,2)*{};(15,2)*{} **\dir{.};
(14.5,2.5)*{};(15,2.5)*{} **\dir{.};
(15.5,-2.5)*{};(21,-2.5)*{} **\dir{.}; (16,-2)*{};(21,-2)*{} **\dir{.};
(16.5,-1.5)*{};(21,-1.5)*{} **\dir{.}; (17,-1)*{};(21,-1)*{} **\dir{.};
(17.5,-0.5)*{};(21,-0.5)*{} **\dir{.}; (24,0)*{};(21,0)*{} **\dir{.};
(24.5,0.5)*{};(21,0.5)*{} **\dir{.}; (19,1)*{};(21,1)*{} **\dir{.};
(19.5,1.5)*{};(21,1.5)*{} **\dir{.}; (20,2)*{};(21,2)*{} **\dir{.};
(20.5,2.5)*{};(21,2.5)*{} **\dir{.};
(21.5,-2.5)*{};(27,-2.5)*{} **\dir{.}; (22,-2)*{};(27,-2)*{} **\dir{.};
(22.5,-1.5)*{};(27,-1.5)*{} **\dir{.}; (23,-1)*{};(27,-1)*{} **\dir{.};
(23.5,-0.5)*{};(27,-0.5)*{} **\dir{.}; (24,0)*{};(27,0)*{} **\dir{.};
(24.5,0.5)*{};(27,0.5)*{} **\dir{.}; (25,1)*{};(27,1)*{} **\dir{.};
(25.5,1.5)*{};(27,1.5)*{} **\dir{.}; (26,2)*{};(27,2)*{} **\dir{.};
(26.5,2.5)*{};(27,2.5)*{} **\dir{.};
(27.5,-2.5)*{};(33,-2.5)*{} **\dir{.}; (28,-2)*{};(33,-2)*{} **\dir{.};
(28.5,-1.5)*{};(33,-1.5)*{} **\dir{.}; (29,-1)*{};(33,-1)*{} **\dir{.};
(29.5,-0.5)*{};(33,-0.5)*{} **\dir{.}; (30,0)*{};(33,0)*{} **\dir{.};
(30.5,0.5)*{};(33,0.5)*{} **\dir{.}; (31,1)*{};(33,1)*{} **\dir{.};
(31.5,1.5)*{};(33,1.5)*{} **\dir{.}; (32,2)*{};(33,2)*{} **\dir{.};
(32.5,2.5)*{};(33,2.5)*{} **\dir{.};
(0,-3)*{};(33,-3)*{} **\dir{-};
(0,3)*{};(33,3)*{} **\dir{-};
(3,-3)*{};(9,3)*{} **\dir{-};%
(9,-3)*{};(15,3)*{} **\dir{-};%
(15,-3)*{};(21,3)*{} **\dir{-};%
(21,-3)*{};(27,3)*{} **\dir{-};%
(27,-3)*{};(33,3)*{} **\dir{-};%
(0,9)*{};(33,9)*{} **\dir{-};
(0,19)*{};(33,19)*{} **\dir{-};
(0,25)*{};(33,25)*{} **\dir{-};
(0,28)*{};(33,28)*{} **\dir{-};
(0,31)*{};(33,31)*{} **\dir{-};
(0,37)*{};(33,37)*{} **\dir{-};
(0,47)*{};(33,47)*{} **\dir{-};
(0,53)*{};(33,53)*{} **\dir{-};
(3,53)*{};(9,59)*{} **\dir{-};%
(9,53)*{};(15,59)*{} **\dir{-};%
(15,53)*{};(21,59)*{} **\dir{-};%
(21,53)*{};(27,59)*{} **\dir{-};%
(27,53)*{};(33,59)*{} **\dir{-};%
(0,59)*{};(33,59)*{} **\dir{-}; (0,65)*{};(33,65)*{} **\dir{-};
(33,-3)*{};(33,66)*{} **\dir{-}; (27,-3)*{};(27,66)*{} **\dir{-};
(21,-3)*{};(21,66)*{} **\dir{-}; (15,-3)*{};(15,66)*{} **\dir{-};
(9,-3)*{};(9,66)*{} **\dir{-}; (3,-3)*{};(3,66)*{} **\dir{-};
(31.2,-1.5)*{_0}; (25.2,-1.5)*{_1}; (19.2,-1.5)*{_0};(13.2,-1.5)*{_1}; (7.2,-1.5)*{_0};
(29.2,1.5)*{_1}; (30.2,6)*{_2}; (30.2,15)*{\vdots}; (30.2,22)*{_{n-1}}; (30.2,26.5)*{_{n}}; (30.2,29.5)*{_{n}}; (30.2,34)*{_{n-1}}; (30.2,43)*{\vdots};
(30.2,50)*{_2}; (31.2,54.5)*{_0};(29.2,57.5)*{_1}; (30.2,62)*{_2};
(23.2,1.5)*{_0}; (24.2,6)*{_2}; (24.2,15)*{\vdots}; (24.2,22)*{_{n-1}}; (24.2,26.5)*{_{n}};(24.2,29.5)*{_{n}}; (24.2,34)*{_{n-1}}; (24.2,43)*{\vdots};
(24.2,50)*{_2}; (25.2,54.5)*{_1}; (23.2,57.5)*{_0}; (24.2,62)*{_2};
(17.2,1.5)*{_1}; (18.2,6)*{_2}; (18.2,15)*{\vdots}; (18.2,22)*{_{n-1}}; (18.2,26.5)*{_{n}}; (18.2,29.5)*{_{n}};(18.2,34)*{_{n-1}}; (18.2,43)*{\vdots};
(18.2,50)*{_2}; (19.2,54.5)*{_0}; (17.2,57.5)*{_1}; (18.2,62)*{_2};
(11.2,1.5)*{_0};(12.2,6)*{_2}; (12.2,15)*{\vdots}; (12.2,22)*{_{n-1}}; (12.2,26.5)*{_{n}}; (12.2,29.5)*{_{n}}; (12.2,34)*{_{n-1}};(12.2,43)*{\vdots};
(12.2,50)*{_2}; (13.2,54.5)*{_1}; (11.2,57.5)*{_0}; (12.2,62)*{_2};
(5.2,1.5)*{_1}; (6.2,6)*{_2}; (6.2,15)*{\vdots}; (6.2,22)*{_{n-1}}; (6.2,26.5)*{_{n}}; (6.2,29.5)*{_{n}}; (6.2,34)*{_{n-1}}; (6.2,43)*{\vdots};
(6.2,50)*{_2}; (7.2,54.5)*{_0}; (5.2,57.5)*{_1}; (6.2,62)*{_2};
(15,-6)*{B_{n}^{(1)}, \ \ofw_1};
(43.5,-2.5)*{};(49,-2.5)*{} **\dir{.}; (44,-2)*{};(49,-2)*{} **\dir{.};
(44.5,-1.5)*{};(49,-1.5)*{} **\dir{.}; (45,-1)*{};(49,-1)*{} **\dir{.};
(45.5,-0.5)*{};(49,-0.5)*{} **\dir{.}; (46,0)*{};(49,0)*{} **\dir{.};
(46.5,0.5)*{};(49,0.5)*{} **\dir{.}; (47,1)*{};(49,1)*{} **\dir{.};
(47.5,1.5)*{};(49,1.5)*{} **\dir{.}; (48,2)*{};(49,2)*{} **\dir{.};
(48.5,2.5)*{};(49,2.5)*{} **\dir{.};
(49.5,-2.5)*{};(55,-2.5)*{} **\dir{.}; (50,-2)*{};(55,-2)*{} **\dir{.};
(50.5,-1.5)*{};(55,-1.5)*{} **\dir{.}; (51,-1)*{};(55,-1)*{} **\dir{.};
(51.5,-0.5)*{};(55,-0.5)*{} **\dir{.}; (52,0)*{};(55,0)*{} **\dir{.};
(52.5,0.5)*{};(55,0.5)*{} **\dir{.}; (53,1)*{};(55,1)*{} **\dir{.};
(53.5,1.5)*{};(55,1.5)*{} **\dir{.}; (54,2)*{};(55,2)*{} **\dir{.};
(54.5,2.5)*{};(55,2.5)*{} **\dir{.};
(55.5,-2.5)*{};(61,-2.5)*{} **\dir{.}; (56,-2)*{};(61,-2)*{} **\dir{.};
(56.5,-1.5)*{};(61,-1.5)*{} **\dir{.}; (57,-1)*{};(61,-1)*{} **\dir{.};
(57.5,-0.5)*{};(61,-0.5)*{} **\dir{.}; (58,0)*{};(61,0)*{} **\dir{.};
(58.5,0.5)*{};(61,0.5)*{} **\dir{.}; (59,1)*{};(61,1)*{} **\dir{.};
(59.5,1.5)*{};(61,1.5)*{} **\dir{.}; (60,2)*{};(61,2)*{} **\dir{.};
(60.5,2.5)*{};(61,2.5)*{} **\dir{.};
(61.5,-2.5)*{};(67,-2.5)*{} **\dir{.}; (62,-2)*{};(67,-2)*{} **\dir{.};
(62.5,-1.5)*{};(67,-1.5)*{} **\dir{.}; (63,-1)*{};(67,-1)*{} **\dir{.};
(63.5,-0.5)*{};(67,-0.5)*{} **\dir{.}; (64,0)*{};(67,0)*{} **\dir{.};
(64.5,0.5)*{};(67,0.5)*{} **\dir{.}; (65,1)*{};(67,1)*{} **\dir{.};
(65.5,1.5)*{};(67,1.5)*{} **\dir{.}; (66,2)*{};(67,2)*{} **\dir{.};
(66.5,2.5)*{};(67,2.5)*{} **\dir{.};
(67.5,-2.5)*{};(73,-2.5)*{} **\dir{.}; (68,-2)*{};(73,-2)*{} **\dir{.};
(68.5,-1.5)*{};(73,-1.5)*{} **\dir{.}; (69,-1)*{};(73,-1)*{} **\dir{.};
(69.5,-0.5)*{};(73,-0.5)*{} **\dir{.}; (70,0)*{};(73,0)*{} **\dir{.};
(70.5,0.5)*{};(73,0.5)*{} **\dir{.}; (71,1)*{};(73,1)*{} **\dir{.};
(71.5,1.5)*{};(73,1.5)*{} **\dir{.}; (72,2)*{};(73,2)*{} **\dir{.};
(72.5,2.5)*{};(73,2.5)*{} **\dir{.};
(40,-3)*{};(73,-3)*{} **\dir{-};
(40,3)*{};(73,3)*{} **\dir{-};
(43,-3)*{};(49,3)*{} **\dir{-};%
(49,-3)*{};(55,3)*{} **\dir{-};%
(55,-3)*{};(61,3)*{} **\dir{-};%
(61,-3)*{};(67,3)*{} **\dir{-};%
(67,-3)*{};(73,3)*{} **\dir{-};%
(40,9)*{};(73,9)*{} **\dir{-};
(40,19)*{};(73,19)*{} **\dir{-};
(40,25)*{};(73,25)*{} **\dir{-};
(40,28)*{};(73,28)*{} **\dir{-};
(40,31)*{};(73,31)*{} **\dir{-};
(40,37)*{};(73,37)*{} **\dir{-};
(40,47)*{};(73,47)*{} **\dir{-};
(40,53)*{};(73,53)*{} **\dir{-};
(43,53)*{};(49,59)*{} **\dir{-};%
(49,53)*{};(55,59)*{} **\dir{-};%
(55,53)*{};(61,59)*{} **\dir{-};%
(61,53)*{};(67,59)*{} **\dir{-};%
(67,53)*{};(73,59)*{} **\dir{-};%
(40,59)*{};(73,59)*{} **\dir{-};
(40,65)*{};(73,65)*{} **\dir{-};
(73,-3)*{};(73,66)*{} **\dir{-};
(67,-3)*{};(67,66)*{} **\dir{-};
(61,-3)*{};(61,66)*{} **\dir{-};
(55,-3)*{};(55,66)*{} **\dir{-};
(49,-3)*{};(49,66)*{} **\dir{-};
(43,-3)*{};(43,66)*{} **\dir{-};
(71.2,-1.5)*{_1}; (65.2,-1.5)*{_0}; (59.2,-1.5)*{_1};(53.2,-1.5)*{_0}; (47.2,-1.5)*{_1};
(69.2,1.5)*{_0}; (70.2,6)*{_2}; (70.2,15)*{\vdots}; (70.2,22)*{_{n-1}}; (70.2,26.5)*{_{n}}; (70.2,29.5)*{_{n}}; (70.2,34)*{_{n-1}}; (70.2,43)*{\vdots};
(70.2,50)*{_2}; (71.2,54.5)*{_1};(69.2,57.5)*{_0}; (70.2,62)*{_2};
(63.2,1.5)*{_1}; (64.2,6)*{_2}; (64.2,15)*{\vdots}; (64.2,22)*{_{n-1}}; (64.2,26.5)*{_{n}};(64.2,29.5)*{_{n}}; (64.2,34)*{_{n-1}}; (64.2,43)*{\vdots};
(64.2,50)*{_2}; (65.2,54.5)*{_0}; (63.2,57.5)*{_1}; (64.2,62)*{_2};
(57.2,1.5)*{_0}; (58.2,6)*{_2}; (58.2,15)*{\vdots}; (58.2,22)*{_{n-1}}; (58.2,26.5)*{_{n}}; (58.2,29.5)*{_{n}};(58.2,34)*{_{n-1}}; (58.2,43)*{\vdots};
(58.2,50)*{_2}; (59.2,54.5)*{_1}; (57.2,57.5)*{_0}; (58.2,62)*{_2};
(51.2,1.5)*{_1};(52.2,6)*{_2}; (52.2,15)*{\vdots}; (52.2,22)*{_{n-1}}; (52.2,26.5)*{_{n}}; (52.2,29.5)*{_{n}}; (52.2,34)*{_{n-1}};(52.2,43)*{\vdots};
(52.2,50)*{_2}; (53.2,54.5)*{_0}; (51.2,57.5)*{_1}; (52.2,62)*{_2};
(45.2,1.5)*{_0}; (46.2,6)*{_2}; (46.2,15)*{\vdots}; (46.2,22)*{_{n-1}}; (46.2,26.5)*{_{n}}; (46.2,29.5)*{_{n}}; (46.2,34)*{_{n-1}}; (46.2,43)*{\vdots};
(46.2,50)*{_2}; (47.2,54.5)*{_1}; (45.2,57.5)*{_0}; (46.2,62)*{_2};
(55,-6)*{B_{n}^{(1)}, \ \ofw_0};
(83.5,-2.5)*{};(89,-2.5)*{} **\dir{.}; (84,-2)*{};(89,-2)*{} **\dir{.};
(84.5,-1.5)*{};(89,-1.5)*{} **\dir{.}; (85,-1)*{};(89,-1)*{} **\dir{.};
(85.5,-0.5)*{};(89,-0.5)*{} **\dir{.}; (86,0)*{};(89,0)*{} **\dir{.};
(86.5,0.5)*{};(89,0.5)*{} **\dir{.}; (87,1)*{};(89,1)*{} **\dir{.};
(87.5,1.5)*{};(89,1.5)*{} **\dir{.}; (88,2)*{};(89,2)*{} **\dir{.};
(88.5,2.5)*{};(89,2.5)*{} **\dir{.};
(89.5,-2.5)*{};(95,-2.5)*{} **\dir{.}; (90,-2)*{};(95,-2)*{} **\dir{.};
(90.5,-1.5)*{};(95,-1.5)*{} **\dir{.}; (91,-1)*{};(95,-1)*{} **\dir{.};
(91.5,-0.5)*{};(95,-0.5)*{} **\dir{.}; (92,0)*{};(95,0)*{} **\dir{.};
(92.5,0.5)*{};(95,0.5)*{} **\dir{.}; (93,1)*{};(95,1)*{} **\dir{.};
(93.5,1.5)*{};(95,1.5)*{} **\dir{.}; (94,2)*{};(95,2)*{} **\dir{.};
(94.5,2.5)*{};(95,2.5)*{} **\dir{.};
(95.5,-2.5)*{};(101,-2.5)*{} **\dir{.}; (96,-2)*{};(101,-2)*{} **\dir{.};
(96.5,-1.5)*{};(101,-1.5)*{} **\dir{.}; (97,-1)*{};(101,-1)*{} **\dir{.};
(97.5,-0.5)*{};(101,-0.5)*{} **\dir{.}; (98,0)*{};(101,0)*{} **\dir{.};
(98.5,0.5)*{};(101,0.5)*{} **\dir{.}; (99,1)*{};(101,1)*{} **\dir{.};
(99.5,1.5)*{};(101,1.5)*{} **\dir{.}; (100,2)*{};(101,2)*{} **\dir{.};
(100.5,2.5)*{};(101,2.5)*{} **\dir{.};
(101.5,-2.5)*{};(107,-2.5)*{} **\dir{.}; (102,-2)*{};(107,-2)*{} **\dir{.};
(102.5,-1.5)*{};(107,-1.5)*{} **\dir{.}; (103,-1)*{};(107,-1)*{} **\dir{.};
(103.5,-0.5)*{};(107,-0.5)*{} **\dir{.}; (104,0)*{};(107,0)*{} **\dir{.};
(104.5,0.5)*{};(107,0.5)*{} **\dir{.}; (105,1)*{};(107,1)*{} **\dir{.};
(105.5,1.5)*{};(107,1.5)*{} **\dir{.}; (106,2)*{};(107,2)*{} **\dir{.};
(106.5,2.5)*{};(107,2.5)*{} **\dir{.};
(107.5,-2.5)*{};(113,-2.5)*{} **\dir{.}; (108,-2)*{};(113,-2)*{} **\dir{.};
(108.5,-1.5)*{};(113,-1.5)*{} **\dir{.}; (109,-1)*{};(113,-1)*{} **\dir{.};
(109.5,-0.5)*{};(113,-0.5)*{} **\dir{.}; (110,0)*{};(113,0)*{} **\dir{.};
(110.5,0.5)*{};(113,0.5)*{} **\dir{.}; (111,1)*{};(113,1)*{} **\dir{.};
(111.5,1.5)*{};(113,1.5)*{} **\dir{.}; (112,2)*{};(113,2)*{} **\dir{.};
(112.5,2.5)*{};(113,2.5)*{} **\dir{.};
(80,-3)*{};(113,-3)*{} **\dir{-};
(80,3)*{};(113,3)*{} **\dir{-};
(83,-3)*{};(89,3)*{} **\dir{-};%
(89,-3)*{};(95,3)*{} **\dir{-};%
(95,-3)*{};(101,3)*{} **\dir{-};%
(101,-3)*{};(107,3)*{} **\dir{-};%
(107,-3)*{};(113,3)*{} **\dir{-};%
(80,9)*{};(113,9)*{} **\dir{-};
(80,19)*{};(113,19)*{} **\dir{-};
(80,25)*{};(113,25)*{} **\dir{-};
(83,25)*{};(89,31)*{} **\dir{-};%
(89,25)*{};(95,31)*{} **\dir{-};%
(95,25)*{};(101,31)*{} **\dir{-};%
(101,25)*{};(107,31)*{} **\dir{-};%
(107,25)*{};(113,31)*{} **\dir{-};%
(80,31)*{};(113,31)*{} **\dir{-};
(80,37)*{};(113,37)*{} **\dir{-};
(80,47)*{};(113,47)*{} **\dir{-};
(80,53)*{};(113,53)*{} **\dir{-};
(83,53)*{};(89,59)*{} **\dir{-};%
(89,53)*{};(95,59)*{} **\dir{-};%
(95,53)*{};(101,59)*{} **\dir{-};%
(101,53)*{};(107,59)*{} **\dir{-};%
(107,53)*{};(113,59)*{} **\dir{-};%
(80,59)*{};(113,59)*{} **\dir{-};
(80,65)*{};(113,65)*{} **\dir{-};
(113,-3)*{};(113,66)*{} **\dir{-};
(107,-3)*{};(107,66)*{} **\dir{-};
(101,-3)*{};(101,66)*{} **\dir{-};
(95,-3)*{};(95,66)*{} **\dir{-};
(89,-3)*{};(89,66)*{} **\dir{-};
(83,-3)*{};(83,66)*{} **\dir{-};
(111.2,-1.5)*{_0}; (105.2,-1.5)*{_1}; (99.2,-1.5)*{_0};(93.2,-1.5)*{_1}; (87.2,-1.5)*{_0};
(109.2,1.5)*{_1}; (110.2,6)*{_2}; (110.2,15)*{\vdots}; (110.2,22)*{_{n-2}}; (111.2,26.5)*{_{n}}; (109.2,29.5)*{_{_{n-1}}}; (110.2,34)*{_{n-2}}; (110.2,43)*{\vdots};
(110.2,50)*{_2}; (111.2,54.5)*{_0};(109.2,57.5)*{_1}; (110.2,62)*{_2};
(103.2,1.5)*{_0}; (104.2,6)*{_2}; (104.2,15)*{\vdots}; (104.2,22)*{_{n-2}}; (105.2,26.5)*{_{_{n-1}}};(103.2,29.5)*{_{n}}; (104.2,34)*{_{n-2}}; (104.2,43)*{\vdots};
(104.2,50)*{_2}; (105.2,54.5)*{_1}; (103.2,57.5)*{_0}; (104.2,62)*{_2};
(97.2,1.5)*{_1}; (98.2,6)*{_2}; (98.2,15)*{\vdots}; (98.2,22)*{_{n-2}}; (99.2,26.5)*{_{n}}; (97.2,29.5)*{_{_{n-1}}};(98.2,34)*{_{n-2}}; (98.2,43)*{\vdots};
(98.2,50)*{_2}; (99.2,54.5)*{_0}; (97.2,57.5)*{_1}; (98.2,62)*{_2};
(91.2,1.5)*{_0};(92.2,6)*{_2}; (92.2,15)*{\vdots}; (92.2,22)*{_{n-2}}; (93.2,26.5)*{_{_{n-1}}}; (91.2,29.5)*{_{n}}; (92.2,34)*{_{n-2}};(92.2,43)*{\vdots};
(92.2,50)*{_2}; (93.2,54.5)*{_1}; (91.2,57.5)*{_0}; (92.2,62)*{_2};
(85.2,1.5)*{_1}; (86.2,6)*{_2}; (86.2,15)*{\vdots}; (86.2,22)*{_{n-2}}; (87.2,26.5)*{_{n}}; (85.2,29.5)*{_{_{n-1}}}; (86.2,34)*{_{n-2}}; (86.2,43)*{\vdots};
(86.2,50)*{_2}; (87.2,54.5)*{_0}; (85.2,57.5)*{_1}; (86.2,62)*{_2};
(95,-6)*{D_{n}^{(1)}, \ \ofw_1};
(123.5,-2.5)*{};(129,-2.5)*{} **\dir{.}; (124,-2)*{};(129,-2)*{} **\dir{.};
(124.5,-1.5)*{};(129,-1.5)*{} **\dir{.}; (125,-1)*{};(129,-1)*{} **\dir{.};
(125.5,-0.5)*{};(129,-0.5)*{} **\dir{.}; (126,0)*{};(129,0)*{} **\dir{.};
(126.5,0.5)*{};(129,0.5)*{} **\dir{.}; (127,1)*{};(129,1)*{} **\dir{.};
(127.5,1.5)*{};(129,1.5)*{} **\dir{.}; (128,2)*{};(129,2)*{} **\dir{.};
(128.5,2.5)*{};(129,2.5)*{} **\dir{.};
(129.5,-2.5)*{};(135,-2.5)*{} **\dir{.}; (130,-2)*{};(135,-2)*{} **\dir{.};
(130.5,-1.5)*{};(135,-1.5)*{} **\dir{.}; (131,-1)*{};(135,-1)*{} **\dir{.};
(131.5,-0.5)*{};(135,-0.5)*{} **\dir{.}; (132,0)*{};(135,0)*{} **\dir{.};
(132.5,0.5)*{};(135,0.5)*{} **\dir{.}; (133,1)*{};(135,1)*{} **\dir{.};
(133.5,1.5)*{};(135,1.5)*{} **\dir{.}; (134,2)*{};(135,2)*{} **\dir{.};
(134.5,2.5)*{};(135,2.5)*{} **\dir{.};
(135.5,-2.5)*{};(141,-2.5)*{} **\dir{.}; (136,-2)*{};(141,-2)*{} **\dir{.};
(136.5,-1.5)*{};(141,-1.5)*{} **\dir{.}; (137,-1)*{};(141,-1)*{} **\dir{.};
(137.5,-0.5)*{};(141,-0.5)*{} **\dir{.}; (138,0)*{};(141,0)*{} **\dir{.};
(138.5,0.5)*{};(141,0.5)*{} **\dir{.}; (139,1)*{};(141,1)*{} **\dir{.};
(139.5,1.5)*{};(141,1.5)*{} **\dir{.}; (140,2)*{};(141,2)*{} **\dir{.};
(140.5,2.5)*{};(141,2.5)*{} **\dir{.};
(141.5,-2.5)*{};(147,-2.5)*{} **\dir{.}; (142,-2)*{};(147,-2)*{} **\dir{.};
(142.5,-1.5)*{};(147,-1.5)*{} **\dir{.}; (143,-1)*{};(147,-1)*{} **\dir{.};
(143.5,-0.5)*{};(147,-0.5)*{} **\dir{.}; (144,0)*{};(147,0)*{} **\dir{.};
(144.5,0.5)*{};(147,0.5)*{} **\dir{.}; (145,1)*{};(147,1)*{} **\dir{.};
(145.5,1.5)*{};(147,1.5)*{} **\dir{.}; (146,2)*{};(147,2)*{} **\dir{.};
(146.5,2.5)*{};(147,2.5)*{} **\dir{.};
(147.5,-2.5)*{};(153,-2.5)*{} **\dir{.}; (148,-2)*{};(153,-2)*{} **\dir{.};
(148.5,-1.5)*{};(153,-1.5)*{} **\dir{.}; (149,-1)*{};(153,-1)*{} **\dir{.};
(149.5,-0.5)*{};(153,-0.5)*{} **\dir{.}; (150,0)*{};(153,0)*{} **\dir{.};
(150.5,0.5)*{};(153,0.5)*{} **\dir{.}; (151,1)*{};(153,1)*{} **\dir{.};
(151.5,1.5)*{};(153,1.5)*{} **\dir{.}; (152,2)*{};(153,2)*{} **\dir{.};
(152.5,2.5)*{};(153,2.5)*{} **\dir{.};
(120,-3)*{};(153,-3)*{} **\dir{-};
(120,3)*{};(153,3)*{} **\dir{-};
(123,-3)*{};(129,3)*{} **\dir{-};%
(129,-3)*{};(135,3)*{} **\dir{-};%
(135,-3)*{};(141,3)*{} **\dir{-};%
(141,-3)*{};(147,3)*{} **\dir{-};%
(147,-3)*{};(153,3)*{} **\dir{-};%
(120,9)*{};(153,9)*{} **\dir{-};
(120,19)*{};(153,19)*{} **\dir{-};
(120,25)*{};(153,25)*{} **\dir{-};
(123,25)*{};(129,31)*{} **\dir{-};%
(129,25)*{};(135,31)*{} **\dir{-};%
(135,25)*{};(141,31)*{} **\dir{-};%
(141,25)*{};(147,31)*{} **\dir{-};%
(147,25)*{};(153,31)*{} **\dir{-};%
(120,31)*{};(153,31)*{} **\dir{-};
(120,37)*{};(153,37)*{} **\dir{-};
(120,47)*{};(153,47)*{} **\dir{-};
(120,53)*{};(153,53)*{} **\dir{-};
(123,53)*{};(129,59)*{} **\dir{-};%
(129,53)*{};(135,59)*{} **\dir{-};%
(135,53)*{};(141,59)*{} **\dir{-};%
(141,53)*{};(147,59)*{} **\dir{-};%
(147,53)*{};(153,59)*{} **\dir{-};%
(120,59)*{};(153,59)*{} **\dir{-};
(120,65)*{};(153,65)*{} **\dir{-};
(153,-3)*{};(153,66)*{} **\dir{-};
(147,-3)*{};(147,66)*{} **\dir{-};
(141,-3)*{};(141,66)*{} **\dir{-};
(135,-3)*{};(135,66)*{} **\dir{-};
(129,-3)*{};(129,66)*{} **\dir{-};
(123,-3)*{};(123,66)*{} **\dir{-};
(151.2,-1.5)*{_1}; (145.2,-1.5)*{_0}; (139.2,-1.5)*{_1};(133.2,-1.5)*{_0}; (127.2,-1.5)*{_1};
(149.2,1.5)*{_0}; (150.2,6)*{_2}; (150.2,15)*{\vdots}; (150.2,22)*{_{n-2}}; (151.2,26.5)*{_{n}}; (149.2,29.5)*{_{_{n-1}}}; (150.2,34)*{_{n-2}}; (150.2,43)*{\vdots};
(150.2,50)*{_2}; (151.2,54.5)*{_1};(149.2,57.5)*{_0}; (150.2,62)*{_2};
(143.2,1.5)*{_1}; (144.2,6)*{_2}; (144.2,15)*{\vdots}; (144.2,22)*{_{n-2}}; (145.2,26.5)*{_{_{n-1}}};(143.2,29.5)*{_{n}}; (144.2,34)*{_{n-2}}; (144.2,43)*{\vdots};
(144.2,50)*{_2}; (145.2,54.5)*{_0}; (143.2,57.5)*{_1}; (144.2,62)*{_2};
(137.2,1.5)*{_0}; (138.2,6)*{_2}; (138.2,15)*{\vdots}; (138.2,22)*{_{n-2}}; (139.2,26.5)*{_{n}}; (137.2,29.5)*{_{_{n-1}}};(138.2,34)*{_{n-2}}; (138.2,43)*{\vdots};
(138.2,50)*{_2}; (139.2,54.5)*{_1}; (137.2,57.5)*{_0}; (138.2,62)*{_2};
(131.2,1.5)*{_1};(132.2,6)*{_2}; (132.2,15)*{\vdots}; (132.2,22)*{_{n-2}}; (133.2,26.5)*{_{_{n-1}}}; (131.2,29.5)*{_{n}}; (132.2,34)*{_{n-2}};(132.2,43)*{\vdots};
(132.2,50)*{_2}; (133.2,54.5)*{_0}; (131.2,57.5)*{_1}; (132.2,62)*{_2};
(125.2,1.5)*{_0}; (126.2,6)*{_2}; (126.2,15)*{\vdots}; (126.2,22)*{_{n-2}}; (127.2,26.5)*{_{n}}; (125.2,29.5)*{_{_{n-1}}}; (126.2,34)*{_{n-2}}; (126.2,43)*{\vdots};
(126.2,50)*{_2}; (127.2,54.5)*{_1}; (125.2,57.5)*{_0}; (126.2,62)*{_2};
(135,-6)*{D_{n}^{(1)}, \ \ofw_0};
\endxy}
\allowdisplaybreaks\\
&\scalebox{0.75}{\xy
(3.5,-2.5)*{};(9,-2.5)*{} **\dir{.}; (4,-2)*{};(9,-2)*{} **\dir{.};
(4.5,-1.5)*{};(9,-1.5)*{} **\dir{.}; (5,-1)*{};(9,-1)*{} **\dir{.};
(5.5,-0.5)*{};(9,-0.5)*{} **\dir{.}; (6,0)*{};(9,0)*{} **\dir{.};
(6.5,0.5)*{};(9,0.5)*{} **\dir{.}; (7,1)*{};(9,1)*{} **\dir{.};
(7.5,1.5)*{};(9,1.5)*{} **\dir{.}; (8,2)*{};(9,2)*{} **\dir{.};
(8.5,2.5)*{};(9,2.5)*{} **\dir{.};
(9.5,-2.5)*{};(15,-2.5)*{} **\dir{.}; (10,-2)*{};(15,-2)*{} **\dir{.};
(10.5,-1.5)*{};(15,-1.5)*{} **\dir{.}; (11,-1)*{};(15,-1)*{} **\dir{.};
(11.5,-0.5)*{};(15,-0.5)*{} **\dir{.}; (12,0)*{};(15,0)*{} **\dir{.};
(12.5,0.5)*{};(15,0.5)*{} **\dir{.}; (13,1)*{};(15,1)*{} **\dir{.};
(13.5,1.5)*{};(15,1.5)*{} **\dir{.}; (14,2)*{};(15,2)*{} **\dir{.};
(14.5,2.5)*{};(15,2.5)*{} **\dir{.};
(15.5,-2.5)*{};(21,-2.5)*{} **\dir{.}; (16,-2)*{};(21,-2)*{} **\dir{.};
(16.5,-1.5)*{};(21,-1.5)*{} **\dir{.}; (17,-1)*{};(21,-1)*{} **\dir{.};
(17.5,-0.5)*{};(21,-0.5)*{} **\dir{.}; (24,0)*{};(21,0)*{} **\dir{.};
(24.5,0.5)*{};(21,0.5)*{} **\dir{.}; (19,1)*{};(21,1)*{} **\dir{.};
(19.5,1.5)*{};(21,1.5)*{} **\dir{.}; (20,2)*{};(21,2)*{} **\dir{.};
(20.5,2.5)*{};(21,2.5)*{} **\dir{.};
(21.5,-2.5)*{};(27,-2.5)*{} **\dir{.}; (22,-2)*{};(27,-2)*{} **\dir{.};
(22.5,-1.5)*{};(27,-1.5)*{} **\dir{.}; (23,-1)*{};(27,-1)*{} **\dir{.};
(23.5,-0.5)*{};(27,-0.5)*{} **\dir{.}; (24,0)*{};(27,0)*{} **\dir{.};
(24.5,0.5)*{};(27,0.5)*{} **\dir{.}; (25,1)*{};(27,1)*{} **\dir{.};
(25.5,1.5)*{};(27,1.5)*{} **\dir{.}; (26,2)*{};(27,2)*{} **\dir{.};
(26.5,2.5)*{};(27,2.5)*{} **\dir{.};
(27.5,-2.5)*{};(33,-2.5)*{} **\dir{.}; (28,-2)*{};(33,-2)*{} **\dir{.};
(28.5,-1.5)*{};(33,-1.5)*{} **\dir{.}; (29,-1)*{};(33,-1)*{} **\dir{.};
(29.5,-0.5)*{};(33,-0.5)*{} **\dir{.}; (30,0)*{};(33,0)*{} **\dir{.};
(30.5,0.5)*{};(33,0.5)*{} **\dir{.}; (31,1)*{};(33,1)*{} **\dir{.};
(31.5,1.5)*{};(33,1.5)*{} **\dir{.}; (32,2)*{};(33,2)*{} **\dir{.};
(32.5,2.5)*{};(33,2.5)*{} **\dir{.};
(0,-3)*{};(33,-3)*{} **\dir{-};
(0,3)*{};(33,3)*{} **\dir{-};
(3,-3)*{};(9,3)*{} **\dir{-};%
(9,-3)*{};(15,3)*{} **\dir{-};%
(15,-3)*{};(21,3)*{} **\dir{-};%
(21,-3)*{};(27,3)*{} **\dir{-};%
(27,-3)*{};(33,3)*{} **\dir{-};%
(0,9)*{};(33,9)*{} **\dir{-};
(0,19)*{};(33,19)*{} **\dir{-};
(0,25)*{};(33,25)*{} **\dir{-};
(0,31)*{};(33,31)*{} **\dir{-};
(0,37)*{};(33,37)*{} **\dir{-};
(0,47)*{};(33,47)*{} **\dir{-};
(0,53)*{};(33,53)*{} **\dir{-};
(3,53)*{};(9,59)*{} **\dir{-};%
(9,53)*{};(15,59)*{} **\dir{-};%
(15,53)*{};(21,59)*{} **\dir{-};%
(21,53)*{};(27,59)*{} **\dir{-};%
(27,53)*{};(33,59)*{} **\dir{-};%
(0,59)*{};(33,59)*{} **\dir{-};
(0,65)*{};(33,65)*{} **\dir{-};
(33,-3)*{};(33,66)*{} **\dir{-};
(27,-3)*{};(27,66)*{} **\dir{-};
(21,-3)*{};(21,66)*{} **\dir{-};
(15,-3)*{};(15,66)*{} **\dir{-};
(9,-3)*{};(9,66)*{} **\dir{-};
(3,-3)*{};(3,66)*{} **\dir{-};
(31.2,-1.5)*{_0}; (25.2,-1.5)*{_1}; (19.2,-1.5)*{_0};(13.2,-1.5)*{_1}; (7.2,-1.5)*{_0};
(29.2,1.5)*{_1}; (30.2,6)*{_2}; (30.2,15)*{\vdots}; (30.2,22)*{_{n-1}}; (30.2,28)*{_{n}}; (30.2,34)*{_{n-1}}; (30.2,43)*{\vdots};
(30.2,50)*{_2}; (31.2,54.5)*{_0};(29.2,57.5)*{_1}; (30.2,62)*{_2};
(23.2,1.5)*{_0}; (24.2,6)*{_2}; (24.2,15)*{\vdots}; (24.2,22)*{_{n-1}}; (24.2,28)*{_{n}}; (24.2,34)*{_{n-1}}; (24.2,43)*{\vdots};
(24.2,50)*{_2}; (25.2,54.5)*{_1}; (23.2,57.5)*{_0}; (24.2,62)*{_2};
(17.2,1.5)*{_1}; (18.2,6)*{_2}; (18.2,15)*{\vdots}; (18.2,22)*{_{n-1}}; (18.2,28)*{_{n}}; (18.2,34)*{_{n-1}}; (18.2,43)*{\vdots};
(18.2,50)*{_2}; (19.2,54.5)*{_0}; (17.2,57.5)*{_1}; (18.2,62)*{_2};
(11.2,1.5)*{_0};(12.2,6)*{_2}; (12.2,15)*{\vdots}; (12.2,22)*{_{n-1}}; (12.2,28)*{_{n}}; (12.2,34)*{_{n-1}};(12.2,43)*{\vdots};
(12.2,50)*{_2}; (13.2,54.5)*{_1}; (11.2,57.5)*{_0}; (12.2,62)*{_2};
(5.2,1.5)*{_1}; (6.2,6)*{_2}; (6.2,15)*{\vdots}; (6.2,22)*{_{n-1}}; (6.2,28)*{_{n}}; (6.2,34)*{_{n-1}}; (6.2,43)*{\vdots};
(6.2,50)*{_2}; (7.2,54.5)*{_0}; (5.2,57.5)*{_1}; (6.2,62)*{_2};
(15,-6)*{A_{2n-1}^{(2)}, \ \ofw_1};
(43.5,-2.5)*{};(49,-2.5)*{} **\dir{.}; (44,-2)*{};(49,-2)*{} **\dir{.};
(44.5,-1.5)*{};(49,-1.5)*{} **\dir{.}; (45,-1)*{};(49,-1)*{} **\dir{.};
(45.5,-0.5)*{};(49,-0.5)*{} **\dir{.}; (46,0)*{};(49,0)*{} **\dir{.};
(46.5,0.5)*{};(49,0.5)*{} **\dir{.}; (47,1)*{};(49,1)*{} **\dir{.};
(47.5,1.5)*{};(49,1.5)*{} **\dir{.}; (48,2)*{};(49,2)*{} **\dir{.};
(48.5,2.5)*{};(49,2.5)*{} **\dir{.};
(49.5,-2.5)*{};(55,-2.5)*{} **\dir{.}; (50,-2)*{};(55,-2)*{} **\dir{.};
(50.5,-1.5)*{};(55,-1.5)*{} **\dir{.}; (51,-1)*{};(55,-1)*{} **\dir{.};
(51.5,-0.5)*{};(55,-0.5)*{} **\dir{.}; (52,0)*{};(55,0)*{} **\dir{.};
(52.5,0.5)*{};(55,0.5)*{} **\dir{.}; (53,1)*{};(55,1)*{} **\dir{.};
(53.5,1.5)*{};(55,1.5)*{} **\dir{.}; (54,2)*{};(55,2)*{} **\dir{.};
(54.5,2.5)*{};(55,2.5)*{} **\dir{.};
(55.5,-2.5)*{};(61,-2.5)*{} **\dir{.}; (56,-2)*{};(61,-2)*{} **\dir{.};
(56.5,-1.5)*{};(61,-1.5)*{} **\dir{.}; (57,-1)*{};(61,-1)*{} **\dir{.};
(57.5,-0.5)*{};(61,-0.5)*{} **\dir{.}; (58,0)*{};(61,0)*{} **\dir{.};
(58.5,0.5)*{};(61,0.5)*{} **\dir{.}; (59,1)*{};(61,1)*{} **\dir{.};
(59.5,1.5)*{};(61,1.5)*{} **\dir{.}; (60,2)*{};(61,2)*{} **\dir{.};
(60.5,2.5)*{};(61,2.5)*{} **\dir{.};
(61.5,-2.5)*{};(67,-2.5)*{} **\dir{.}; (62,-2)*{};(67,-2)*{} **\dir{.};
(62.5,-1.5)*{};(67,-1.5)*{} **\dir{.}; (63,-1)*{};(67,-1)*{} **\dir{.};
(63.5,-0.5)*{};(67,-0.5)*{} **\dir{.}; (64,0)*{};(67,0)*{} **\dir{.};
(64.5,0.5)*{};(67,0.5)*{} **\dir{.}; (65,1)*{};(67,1)*{} **\dir{.};
(65.5,1.5)*{};(67,1.5)*{} **\dir{.}; (66,2)*{};(67,2)*{} **\dir{.};
(66.5,2.5)*{};(67,2.5)*{} **\dir{.};
(67.5,-2.5)*{};(73,-2.5)*{} **\dir{.}; (68,-2)*{};(73,-2)*{} **\dir{.};
(68.5,-1.5)*{};(73,-1.5)*{} **\dir{.}; (69,-1)*{};(73,-1)*{} **\dir{.};
(69.5,-0.5)*{};(73,-0.5)*{} **\dir{.}; (70,0)*{};(73,0)*{} **\dir{.};
(70.5,0.5)*{};(73,0.5)*{} **\dir{.}; (71,1)*{};(73,1)*{} **\dir{.};
(71.5,1.5)*{};(73,1.5)*{} **\dir{.}; (72,2)*{};(73,2)*{} **\dir{.};
(72.5,2.5)*{};(73,2.5)*{} **\dir{.};
(40,-3)*{};(73,-3)*{} **\dir{-};
(40,3)*{};(73,3)*{} **\dir{-};
(43,-3)*{};(49,3)*{} **\dir{-};%
(49,-3)*{};(55,3)*{} **\dir{-};%
(55,-3)*{};(61,3)*{} **\dir{-};%
(61,-3)*{};(67,3)*{} **\dir{-};%
(67,-3)*{};(73,3)*{} **\dir{-};%
(40,9)*{};(73,9)*{} **\dir{-};
(40,19)*{};(73,19)*{} **\dir{-};
(40,25)*{};(73,25)*{} **\dir{-};
(40,31)*{};(73,31)*{} **\dir{-};
(40,37)*{};(73,37)*{} **\dir{-};
(40,47)*{};(73,47)*{} **\dir{-};
(40,53)*{};(73,53)*{} **\dir{-};
(43,53)*{};(49,59)*{} **\dir{-};%
(49,53)*{};(55,59)*{} **\dir{-};%
(55,53)*{};(61,59)*{} **\dir{-};%
(61,53)*{};(67,59)*{} **\dir{-};%
(67,53)*{};(73,59)*{} **\dir{-};%
(40,59)*{};(73,59)*{} **\dir{-};
(40,65)*{};(73,65)*{} **\dir{-};
(73,-3)*{};(73,66)*{} **\dir{-};
(67,-3)*{};(67,66)*{} **\dir{-};
(61,-3)*{};(61,66)*{} **\dir{-};
(55,-3)*{};(55,66)*{} **\dir{-};
(49,-3)*{};(49,66)*{} **\dir{-};
(43,-3)*{};(43,66)*{} **\dir{-};
(71.2,-1.5)*{_1}; (65.2,-1.5)*{_0}; (59.2,-1.5)*{_1};(53.2,-1.5)*{_0}; (47.2,-1.5)*{_1};
(69.2,1.5)*{_0}; (70.2,6)*{_2}; (70.2,15)*{\vdots}; (70.2,22)*{_{n-1}}; (70.2,28)*{_{n}}; (70.2,34)*{_{n-1}}; (70.2,43)*{\vdots};
(70.2,50)*{_2}; (71.2,54.5)*{_1};(69.2,57.5)*{_0}; (70.2,62)*{_2};
(63.2,1.5)*{_1}; (64.2,6)*{_2}; (64.2,15)*{\vdots}; (64.2,22)*{_{n-1}}; (64.2,28)*{_{n}}; (64.2,34)*{_{n-1}}; (64.2,43)*{\vdots};
(64.2,50)*{_2}; (65.2,54.5)*{_0}; (63.2,57.5)*{_1}; (64.2,62)*{_2};
(57.2,1.5)*{_0}; (58.2,6)*{_2}; (58.2,15)*{\vdots}; (58.2,22)*{_{n-1}}; (58.2,28)*{_{n}};(58.2,34)*{_{n-1}}; (58.2,43)*{\vdots};
(58.2,50)*{_2}; (59.2,54.5)*{_1}; (57.2,57.5)*{_0}; (58.2,62)*{_2};
(51.2,1.5)*{_1};(52.2,6)*{_2}; (52.2,15)*{\vdots}; (52.2,22)*{_{n-1}}; (52.2,28)*{_{n}}; (52.2,34)*{_{n-1}};(52.2,43)*{\vdots};
(52.2,50)*{_2}; (53.2,54.5)*{_0}; (51.2,57.5)*{_1}; (52.2,62)*{_2};
(45.2,1.5)*{_0}; (46.2,6)*{_2}; (46.2,15)*{\vdots}; (46.2,22)*{_{n-1}}; (46.2,28)*{_{n}};  (46.2,34)*{_{n-1}}; (46.2,43)*{\vdots};
(46.2,50)*{_2}; (47.2,54.5)*{_1}; (45.2,57.5)*{_0}; (46.2,62)*{_2};
(55,-6)*{A_{2n-1}^{(2)}, \ \ofw_0};
(83.5,-2.5)*{};(89,-2.5)*{} **\dir{.}; (84,-2)*{};(89,-2)*{} **\dir{.};
(84.5,-1.5)*{};(89,-1.5)*{} **\dir{.}; (85,-1)*{};(89,-1)*{} **\dir{.};
(85.5,-0.5)*{};(89,-0.5)*{} **\dir{.}; (86,0)*{};(89,0)*{} **\dir{.};
(86.5,0.5)*{};(89,0.5)*{} **\dir{.}; (87,1)*{};(89,1)*{} **\dir{.};
(87.5,1.5)*{};(89,1.5)*{} **\dir{.}; (88,2)*{};(89,2)*{} **\dir{.};
(88.5,2.5)*{};(89,2.5)*{} **\dir{.};
(89.5,-2.5)*{};(95,-2.5)*{} **\dir{.}; (90,-2)*{};(95,-2)*{} **\dir{.};
(90.5,-1.5)*{};(95,-1.5)*{} **\dir{.}; (91,-1)*{};(95,-1)*{} **\dir{.};
(91.5,-0.5)*{};(95,-0.5)*{} **\dir{.}; (92,0)*{};(95,0)*{} **\dir{.};
(92.5,0.5)*{};(95,0.5)*{} **\dir{.}; (93,1)*{};(95,1)*{} **\dir{.};
(93.5,1.5)*{};(95,1.5)*{} **\dir{.}; (94,2)*{};(95,2)*{} **\dir{.};
(94.5,2.5)*{};(95,2.5)*{} **\dir{.};
(95.5,-2.5)*{};(101,-2.5)*{} **\dir{.}; (96,-2)*{};(101,-2)*{} **\dir{.};
(96.5,-1.5)*{};(101,-1.5)*{} **\dir{.}; (97,-1)*{};(101,-1)*{} **\dir{.};
(97.5,-0.5)*{};(101,-0.5)*{} **\dir{.}; (98,0)*{};(101,0)*{} **\dir{.};
(98.5,0.5)*{};(101,0.5)*{} **\dir{.}; (99,1)*{};(101,1)*{} **\dir{.};
(99.5,1.5)*{};(101,1.5)*{} **\dir{.}; (100,2)*{};(101,2)*{} **\dir{.};
(100.5,2.5)*{};(101,2.5)*{} **\dir{.};
(101.5,-2.5)*{};(107,-2.5)*{} **\dir{.}; (102,-2)*{};(107,-2)*{} **\dir{.};
(102.5,-1.5)*{};(107,-1.5)*{} **\dir{.}; (103,-1)*{};(107,-1)*{} **\dir{.};
(103.5,-0.5)*{};(107,-0.5)*{} **\dir{.}; (104,0)*{};(107,0)*{} **\dir{.};
(104.5,0.5)*{};(107,0.5)*{} **\dir{.}; (105,1)*{};(107,1)*{} **\dir{.};
(105.5,1.5)*{};(107,1.5)*{} **\dir{.}; (106,2)*{};(107,2)*{} **\dir{.};
(106.5,2.5)*{};(107,2.5)*{} **\dir{.};
(107.5,-2.5)*{};(113,-2.5)*{} **\dir{.}; (108,-2)*{};(113,-2)*{} **\dir{.};
(108.5,-1.5)*{};(113,-1.5)*{} **\dir{.}; (109,-1)*{};(113,-1)*{} **\dir{.};
(109.5,-0.5)*{};(113,-0.5)*{} **\dir{.}; (110,0)*{};(113,0)*{} **\dir{.};
(110.5,0.5)*{};(113,0.5)*{} **\dir{.}; (111,1)*{};(113,1)*{} **\dir{.};
(111.5,1.5)*{};(113,1.5)*{} **\dir{.}; (112,2)*{};(113,2)*{} **\dir{.};
(112.5,2.5)*{};(113,2.5)*{} **\dir{.};
(80,-3)*{};(113,-3)*{} **\dir{-};
(80,3)*{};(113,3)*{} **\dir{-};
(83,-3)*{};(89,3)*{} **\dir{-};%
(89,-3)*{};(95,3)*{} **\dir{-};%
(95,-3)*{};(101,3)*{} **\dir{-};%
(101,-3)*{};(107,3)*{} **\dir{-};%
(107,-3)*{};(113,3)*{} **\dir{-};%
(80,9)*{};(113,9)*{} **\dir{-};
(80,19)*{};(113,19)*{} **\dir{-};
(80,25)*{};(113,25)*{} **\dir{-};
(83,25)*{};(89,31)*{} **\dir{-};%
(89,25)*{};(95,31)*{} **\dir{-};%
(95,25)*{};(101,31)*{} **\dir{-};%
(101,25)*{};(107,31)*{} **\dir{-};%
(107,25)*{};(113,31)*{} **\dir{-};%
(80,31)*{};(113,31)*{} **\dir{-};
(80,37)*{};(113,37)*{} **\dir{-};
(80,47)*{};(113,47)*{} **\dir{-};
(80,53)*{};(113,53)*{} **\dir{-};
(83,53)*{};(89,59)*{} **\dir{-};%
(89,53)*{};(95,59)*{} **\dir{-};%
(95,53)*{};(101,59)*{} **\dir{-};%
(101,53)*{};(107,59)*{} **\dir{-};%
(107,53)*{};(113,59)*{} **\dir{-};%
(80,59)*{};(113,59)*{} **\dir{-};
(80,65)*{};(113,65)*{} **\dir{-};
(113,-3)*{};(113,66)*{} **\dir{-};
(107,-3)*{};(107,66)*{} **\dir{-};
(101,-3)*{};(101,66)*{} **\dir{-};
(95,-3)*{};(95,66)*{} **\dir{-};
(89,-3)*{};(89,66)*{} **\dir{-};
(83,-3)*{};(83,66)*{} **\dir{-};
(111.2,-1.5)*{_n}; (105.2,-1.5)*{_{n-1}}; (99.2,-1.5)*{_n};(93.2,-1.5)*{_{n-1}}; (87.2,-1.5)*{_n};
(109.2,1.5)*{_{n-1}}; (110.2,6)*{_{n-2}}; (110.2,15)*{\vdots}; (110.2,22)*{_2}; (111.2,26.5)*{_1}; (109.2,29.5)*{_{_0}}; (110.2,34)*{_2}; (110.2,43)*{\vdots};
(110.2,50)*{_{n-2}}; (111.2,54.5)*{_n};(109.2,57.5)*{_{n-1}}; (110.2,62)*{_{n-2}};
(103.2,1.5)*{_n}; (104.2,6)*{_{n-2}}; (104.2,15)*{\vdots}; (104.2,22)*{_2}; (105.2,26.5)*{_{_0}};(103.2,29.5)*{_1}; (104.2,34)*{_2}; (104.2,43)*{\vdots};
(104.2,50)*{_{n-2}}; (105.2,54.5)*{_{n-1}}; (103.2,57.5)*{_n}; (104.2,62)*{_{n-2}};
(97.2,1.5)*{_{n-1}}; (98.2,6)*{_{n-2}}; (98.2,15)*{\vdots}; (98.2,22)*{_2}; (99.2,26.5)*{_1}; (97.2,29.5)*{_{_0}};(98.2,34)*{_2}; (98.2,43)*{\vdots};
(98.2,50)*{_{n-2}}; (99.2,54.5)*{_n}; (97.2,57.5)*{_{n-1}}; (98.2,62)*{_{n-2}};
(91.2,1.5)*{_n};(92.2,6)*{_{n-2}}; (92.2,15)*{\vdots}; (92.2,22)*{_2}; (93.2,26.5)*{_{_0}}; (91.2,29.5)*{_1}; (92.2,34)*{_2};(92.2,43)*{\vdots};
(92.2,50)*{_{n-2}}; (93.2,54.5)*{_{n-1}}; (91.2,57.5)*{_n}; (92.2,62)*{_{n-2}};
(85.2,1.5)*{_{n-1}}; (86.2,6)*{_{n-2}}; (86.2,15)*{\vdots}; (86.2,22)*{_2}; (87.2,26.5)*{_1}; (85.2,29.5)*{_{_0}}; (86.2,34)*{_2}; (86.2,43)*{\vdots};
(86.2,50)*{_{n-2}}; (87.2,54.5)*{_n}; (85.2,57.5)*{_{n-1}}; (86.2,62)*{_{n-2}};
(95,-6)*{D_n^{(1)}, \ \ofw_{n-1}};
(123.5,-2.5)*{};(129,-2.5)*{} **\dir{.}; (124,-2)*{};(129,-2)*{} **\dir{.};
(124.5,-1.5)*{};(129,-1.5)*{} **\dir{.}; (125,-1)*{};(129,-1)*{} **\dir{.};
(125.5,-0.5)*{};(129,-0.5)*{} **\dir{.}; (126,0)*{};(129,0)*{} **\dir{.};
(126.5,0.5)*{};(129,0.5)*{} **\dir{.}; (127,1)*{};(129,1)*{} **\dir{.};
(127.5,1.5)*{};(129,1.5)*{} **\dir{.}; (128,2)*{};(129,2)*{} **\dir{.};
(128.5,2.5)*{};(129,2.5)*{} **\dir{.};
(129.5,-2.5)*{};(135,-2.5)*{} **\dir{.}; (130,-2)*{};(135,-2)*{} **\dir{.};
(130.5,-1.5)*{};(135,-1.5)*{} **\dir{.}; (131,-1)*{};(135,-1)*{} **\dir{.};
(131.5,-0.5)*{};(135,-0.5)*{} **\dir{.}; (132,0)*{};(135,0)*{} **\dir{.};
(132.5,0.5)*{};(135,0.5)*{} **\dir{.}; (133,1)*{};(135,1)*{} **\dir{.};
(133.5,1.5)*{};(135,1.5)*{} **\dir{.}; (134,2)*{};(135,2)*{} **\dir{.};
(134.5,2.5)*{};(135,2.5)*{} **\dir{.};
(135.5,-2.5)*{};(141,-2.5)*{} **\dir{.}; (136,-2)*{};(141,-2)*{} **\dir{.};
(136.5,-1.5)*{};(141,-1.5)*{} **\dir{.}; (137,-1)*{};(141,-1)*{} **\dir{.};
(137.5,-0.5)*{};(141,-0.5)*{} **\dir{.}; (138,0)*{};(141,0)*{} **\dir{.};
(138.5,0.5)*{};(141,0.5)*{} **\dir{.}; (139,1)*{};(141,1)*{} **\dir{.};
(139.5,1.5)*{};(141,1.5)*{} **\dir{.}; (140,2)*{};(141,2)*{} **\dir{.};
(140.5,2.5)*{};(141,2.5)*{} **\dir{.};
(141.5,-2.5)*{};(147,-2.5)*{} **\dir{.}; (142,-2)*{};(147,-2)*{} **\dir{.};
(142.5,-1.5)*{};(147,-1.5)*{} **\dir{.}; (143,-1)*{};(147,-1)*{} **\dir{.};
(143.5,-0.5)*{};(147,-0.5)*{} **\dir{.}; (144,0)*{};(147,0)*{} **\dir{.};
(144.5,0.5)*{};(147,0.5)*{} **\dir{.}; (145,1)*{};(147,1)*{} **\dir{.};
(145.5,1.5)*{};(147,1.5)*{} **\dir{.}; (146,2)*{};(147,2)*{} **\dir{.};
(146.5,2.5)*{};(147,2.5)*{} **\dir{.};
(147.5,-2.5)*{};(153,-2.5)*{} **\dir{.}; (148,-2)*{};(153,-2)*{} **\dir{.};
(148.5,-1.5)*{};(153,-1.5)*{} **\dir{.}; (149,-1)*{};(153,-1)*{} **\dir{.};
(149.5,-0.5)*{};(153,-0.5)*{} **\dir{.}; (150,0)*{};(153,0)*{} **\dir{.};
(150.5,0.5)*{};(153,0.5)*{} **\dir{.}; (151,1)*{};(153,1)*{} **\dir{.};
(151.5,1.5)*{};(153,1.5)*{} **\dir{.}; (152,2)*{};(153,2)*{} **\dir{.};
(152.5,2.5)*{};(153,2.5)*{} **\dir{.};
(120,-3)*{};(153,-3)*{} **\dir{-};
(120,3)*{};(153,3)*{} **\dir{-};
(123,-3)*{};(129,3)*{} **\dir{-};%
(129,-3)*{};(135,3)*{} **\dir{-};%
(135,-3)*{};(141,3)*{} **\dir{-};%
(141,-3)*{};(147,3)*{} **\dir{-};%
(147,-3)*{};(153,3)*{} **\dir{-};%
(120,9)*{};(153,9)*{} **\dir{-};
(120,19)*{};(153,19)*{} **\dir{-};
(120,25)*{};(153,25)*{} **\dir{-};
(123,25)*{};(129,31)*{} **\dir{-};%
(129,25)*{};(135,31)*{} **\dir{-};%
(135,25)*{};(141,31)*{} **\dir{-};%
(141,25)*{};(147,31)*{} **\dir{-};%
(147,25)*{};(153,31)*{} **\dir{-};%
(120,31)*{};(153,31)*{} **\dir{-};
(120,37)*{};(153,37)*{} **\dir{-};
(120,47)*{};(153,47)*{} **\dir{-};
(120,53)*{};(153,53)*{} **\dir{-};
(123,53)*{};(129,59)*{} **\dir{-};%
(129,53)*{};(135,59)*{} **\dir{-};%
(135,53)*{};(141,59)*{} **\dir{-};%
(141,53)*{};(147,59)*{} **\dir{-};%
(147,53)*{};(153,59)*{} **\dir{-};%
(120,59)*{};(153,59)*{} **\dir{-};
(120,65)*{};(153,65)*{} **\dir{-};
(153,-3)*{};(153,66)*{} **\dir{-};
(147,-3)*{};(147,66)*{} **\dir{-};
(141,-3)*{};(141,66)*{} **\dir{-};
(135,-3)*{};(135,66)*{} **\dir{-};
(129,-3)*{};(129,66)*{} **\dir{-};
(123,-3)*{};(123,66)*{} **\dir{-};
(151.2,-1.5)*{_{n-1}}; (145.2,-1.5)*{_{n}}; (139.2,-1.5)*{_{n-1}};(133.2,-1.5)*{_{n}}; (127.2,-1.5)*{_{n-1}};
(149.2,1.5)*{_{n}}; (150.2,6)*{_{n-2}}; (150.2,15)*{\vdots}; (150.2,22)*{_2}; (151.2,26.5)*{_1}; (149.2,29.5)*{_{_0}}; (150.2,34)*{_2}; (150.2,43)*{\vdots};
(150.2,50)*{_{n-2}}; (151.2,54.5)*{_{n-1}};(149.2,57.5)*{_{n}}; (150.2,62)*{_{n-2}};
(143.2,1.5)*{_{n-1}}; (144.2,6)*{_{n-2}}; (144.2,15)*{\vdots}; (144.2,22)*{_2}; (145.2,26.5)*{_{_0}};(143.2,29.5)*{_1}; (144.2,34)*{_2}; (144.2,43)*{\vdots};
(144.2,50)*{_{n-2}}; (145.2,54.5)*{_{n}}; (143.2,57.5)*{_{n-1}}; (144.2,62)*{_{n-2}};
(137.2,1.5)*{_{n}}; (138.2,6)*{_{n-2}}; (138.2,15)*{\vdots}; (138.2,22)*{_2}; (139.2,26.5)*{_1}; (137.2,29.5)*{_{_0}};(138.2,34)*{_2}; (138.2,43)*{\vdots};
(138.2,50)*{_{n-2}}; (139.2,54.5)*{_{n-1}}; (137.2,57.5)*{_{n}}; (138.2,62)*{_{n-2}};
(131.2,1.5)*{_{n-1}};(132.2,6)*{_{n-2}}; (132.2,15)*{\vdots}; (132.2,22)*{_2}; (133.2,26.5)*{_{_0}}; (131.2,29.5)*{_1}; (132.2,34)*{_2};(132.2,43)*{\vdots};
(132.2,50)*{_{n-2}}; (133.2,54.5)*{_{n}}; (131.2,57.5)*{_{n-1}}; (132.2,62)*{_{n-2}};
(125.2,1.5)*{_{n}}; (126.2,6)*{_{n-2}}; (126.2,15)*{\vdots}; (126.2,22)*{_2}; (127.2,26.5)*{_1}; (125.2,29.5)*{_{_0}}; (126.2,34)*{_2}; (126.2,43)*{\vdots};
(126.2,50)*{_{n-2}}; (127.2,54.5)*{_{n-1}}; (125.2,57.5)*{_{n}}; (126.2,62)*{_{n-2}};
(135,-6)*{D_n^{(1)}, \ \ofw_{n}};
\endxy} \label{eq: YW pat ht2}
\end{align}

According to the ground-state Young walls in \eqref{eq: gwalls}, we classify the fundamental weights $\ofw$ of level $1$ into two types:
\begin{itemize}
\item Type ${\Ss}$ : those $\ofw$ whose ground-state Young wall consists of half-height blocks,
\item Type ${\sS}$ : those $\ofw$ whose ground-state Young wall consists of half-thickness blocks.
\end{itemize}

\begin{remark}
For classifying fundamental weights $\ofw_i$ of level $1$, we use $\Ss$ and $\sS$ by the following reason:
\begin{itemize}
\item When $\boxed{\ofw_i}$ consists of half-height blocks, the vertex $i$ in the affine Dynkin diagram is an extremal vertex incident on a
doubly-laced incoming arrow, which can be identified with the extremal vertex $n$ in the Dynkin diagram $\bigtriangleup_{B_n}$.
\item When $\boxed{\ofw_i}$ consists of half-thickness blocks, the vertex $i$ in the affine Dynkin diagram is an extremal vertex incident on a
simply-laced edge, which can be identified with an extremal vertex $n$ or $n-1$ in the Dynkin diagram $\bigtriangleup_{D_n}$.
\end{itemize}

Later, we will see that this classification is closely related to finite simple Lie algebras of type $B_n$ and $D_n$.
\end{remark}

\begin{remark} \label{rmk: pattern up to one column}
For $\g=B_n^{(1)}$, the patterns of Young walls based on $\ofw_0$ and $\ofw_1$ are the same {\it up to one column};
that is, if we ignore the first column of the pattern for $\ofw_1$, then
we get the pattern for $\ofw_0$ (see \eqref{eq: YW pat ht} and \eqref{eq: YW pat ht2}).
\end{remark}

We denote by $\YWw$ a Young wall stacked on $\boxed{\ofw}$  whose type will be clear from the context.
For a Young wall $\YWw$, we write $\YWw = (y_k)_{k=1}^\infty = (\ldots, y_2, y_1) $ as a sequence of its columns from the right.
For $u \in \Z_{\ge 1}$, we define Young walls  $(\YWw)_{\ge u}$ and $(\YWw)_{\le u}$ as follows:
\begin{align*}
(\YWw)_{\ge u}=(\ldots,y_{u+2},y_{u+1},y_u), \qquad  (\YWw)_{\le u}=(y_u,y_{u-1},y_{u-2},\ldots,y_1).
\end{align*}

\begin{example} \label{ex: YW 3d 2d} For $\g=B_3^{(1)}$ and $\ofw_0$, the following is an example of a Young wall $\YW_{\ofw_0}$:
\begin{align}
\scalebox{0.8}{{\xy (0,0)*++{\youngwallex} \endxy}} = \scalebox{0.8}{{\xy (0,0)*++{\youngwallext}\endxy}}
\end{align}
\end{example}

\begin{definition} \label{def: full,proper} \hfill
\begin{enumerate}
\item[{\rm (1)}] A column of a Young wall is called a {\it full column} if its height is a multiple of the unit length
and its top is of unit thickness.
\item[{\rm (2)}] A Young wall is said to be {\it proper} if none of the full columns have the same heights.
\item[{\rm (3)}] An $i$-block of a proper Young wall $\YWw$ is called a {\it removable} $i$-block if $\YWw$ remains a proper Young wall after removing the block.
\item[{\rm (4)}] A place in a proper Young wall $\YWw$ is called an {\it admissible} or {\it addable} $i$-slot if $\YWw$ remains a proper Young wall after adding an $i$-block at the place.
\end{enumerate}
\end{definition}

A {\it partition} $\lambda$ of $m$ is a weakly decreasing sequence of positive integers
$(\lambda_1 \ge \lambda_2 \ge \cdots \ge \lambda_{k}>0)$ such that
$\|\lambda \| \seteq \sum_{i=1}^k \lambda_i=m$, and we write $\lambda \vdash m$. Each integer $\lambda_i$ is called a \emph{part} of $\lambda$. For a given partition
$\lambda=(\lambda_1,\lambda_2,\ldots,\lambda_k)$, we say that the integer $\ell(\lambda)\seteq k$ is the {\it length} of $\lambda$.
We say that a partition $\lambda$ is {\it strict} if $\lambda_i> \lambda_{i+1} >0$ for $1 \le  i \le \ell(\lambda)-1$.
We set $\lambda_i =0$ when $i > \ell(\lambda)$.

For a partition $\lambda=(\lambda_1,\lambda_2,\ldots,\lambda_k)$ and $1 \le u \le k$, we define partitions $\lambda_{\ge u}$ and $\lambda_{\le u}$ as follows:
\begin{align*}
\lambda_{\ge u}=(\lambda_{u},\lambda_{u+1},\ldots,\lambda_k),\qquad \lambda_{\le u}=(\lambda_{1},\lambda_{2},\ldots,\lambda_u).
\end{align*}

\begin{definition} \label{def: association}\hfill
\begin{enumerate}
\item[{\rm (a)}]
For a given proper Young wall $\YWw=(y_i)^{\infty}_{i=1}$, define
$|\YWw|=(|y_1|,|y_2|,\ldots)$ to be the sequence of nonnegative integers, where the
$|y_i|$ is the number of blocks in the $i$-th column of $\YWw$ above the ground-state wall $\boxed{\ofw}$, and call
$|\YWw|$ the \emph{partition associated} to $\YWw$.
\item[{\rm (b)}] For a partition $\lambda$ and a fundamental weight $\ofw$ of level $1$,
we can build a proper Young wall so that its associated partition is equal to $\lambda$. If  the Young wall is uniquely determined (see Example \ref{Ex: for uniqueness} below), we denote it by $\YWw^\lambda$ and call it the \emph{Young wall associated} to $\lambda$ and $\ofw$.
\end{enumerate}
\end{definition}

\begin{example} \label{Ex: for uniqueness}
For the proper Young wall given in \eqref{ex: YW 3d 2d}, the associated partition is $\lambda=(6,3,1)$. However, there are two proper
Young walls corresponding to the  partition  $(6,3,1)$:
$$\scalebox{0.8}
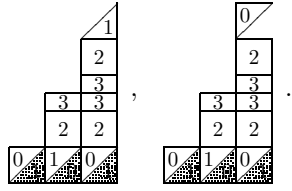
{{\xy (0,0)*++{\youngwallext}\endxy}} , \ \scalebox{0.8}{{\xy (0,0)*++{\youngwallextp}\endxy}} . $$
For the partition $(5,3,1)$, one can easily see that $\YW_{\ofw_0}^{(5,3,1)}$ is well-defined (see \cite{Oh15} also).
\end{example}

For the rest of this paper, we will always deal with partitions $\lambda$ so that the Young walls $\YWw^\lambda$ are uniquely determined, unless otherwise stated.

\medskip

We denote by $\ZZ(\ofw)$ the set of all proper Young walls on $\boxed{\ofw}$, and  define the Kashiwara operators $\eit$ and $\fit$ on $\ZZ(\ofw)$ as follows:
Fix $i \in I$ and let $\YWw=(y_u)_{u=1}^{\infty}$ be a proper Young wall.
\begin{enumerate}
\item[{\rm (a)}] To each column $y_u$ of $\YWw$, assign
$$
\begin{cases}
-- & \text{ if $y_u$ is twice $i$-removable},  \\
- & \text{ if $y_u$ is once $i$-removable}, \\
-+ & \text{ if $y_u$ is once $i$-removable and once $i$-addable}, \\
+ & \text{ if $y_u$ is once $i$-addable}, \\
++ & \text{ if $y_u$ is twice $i$-addable}, \\
\cdot & \text{ otherwise.}
\end{cases}
$$
\item[{\rm (b)}] From this sequence of $+$'s and $-$'s, we cancel out every $(+,-)$-pair to obtain a finite
sequence of $-$'s followed by $+$'s, reading from left to right. This finite sequence
$(-\cdots-,+\cdots+)$ is called the {\it $i$-signature} of $\YWw$ and is denoted by ${\rm sig}_i(\YWw)$.
\item[{\rm (c)}] We define $\eit \YWw$ to be the proper Young wall obtained from $\YWw$ by removing the $i$-block
corresponding to the right-most $-$ in the $i$-signature of $\YWw$. We define $\eit \YWw=0$ if there
is no $-$ in the $i$-signature of $\YWw$.
\item[{\rm (d)}]  We define $\fit \YWw$ to be the proper Young wall obtained from $\YWw$ by adding an $i$-block to the
column corresponding to the left-most $+$ in the $i$-signature of $\YWw$. We define $\fit \YWw=0$ if
there is no $+$ in the $i$-signature of $\YWw$.
\end{enumerate}

For the $\YW_{\ofw_0}$ in Example \ref{ex: YW 3d 2d}, one can compute that
$${\rm sig}_0(\YW_{\ofw_0})=(-,\cdot,+), \ {\rm sig}_1(\YW_{\ofw_0})=(\cdot,\cdot,-), \ {\rm sig}_2(\YW_{\ofw_0})=(+,\cdot,\cdot), \
{\rm sig}_3(\YW_{\ofw_0})=(\cdot,-+,\cdot).$$

We define
\begin{enumerate}
\item[{\rm (a)}] $ \wt(\YWw)=\ofw-\sum_{i \in I}m_i\al_i$,
\item[{\rm (b)}] $\ve_i(\YWw)\text{(resp. $\vp_i(\YWw)$)}=\text{the number of $-$'s (resp $+$'s) in ${\rm sig}_i(\YWw)$},$
\end{enumerate}
where $m_i$ is the number of $i$-blocks that have been added to the ground-state wall $\boxed{\ofw}$. We also define
$$\cont(\YWw)=\ofw-\wt(\YWw)=\sum_{i \in I}m_i\al_i$$ and call it the {\it content} of $\YWw$.

For the Young wall $\YW_{\ofw_0}$ in Example \ref{ex: YW 3d 2d}, we have
$$\wt(\YW_{\ofw_0}) =\ofw_0 -(2\alpha_0+2\alpha_1+3\alpha_2+3\alpha_3) \text{ and } \cont(\YW_{\ofw_0})=2\alpha_0+2\alpha_1+3\alpha_2+3\alpha_3.$$

\begin{definition}
Let $\YWw= (\ldots, y_2, y_1)$ and $\YW_{\ofw'}= (\ldots, y'_2, y'_1)$ be Young walls
of the same affine type. For $t,u \in \Z_{\ge 1}$, we write $$(\YWw)_{\ge t} \supset (\YW_{\ofw'})_{\ge u}$$ if, for each $s \in \Z_{\ge 0}$,
\begin{enumerate}
\item[{\rm (a)}] the ground patterns for $y_{t+s}$ and  $y'_{u+s}$ coincide with each other, 
\item[{\rm (b)}] $\cont(y_{t+s}) - \cont(y'_{u+s}) \in \rl^+$.
\end{enumerate}
\end{definition}

Recall we denote the null root by $\updelta=a_0\alpha_0+a_1\alpha_1+\cdots+a_n\alpha_n$.

\begin{definition} Set $d=2$ if $\g=D_{n+1}^{(2)}$ and $d=1$, otherwise.
\begin{enumerate}
\item[{\rm (i)}] A connected part of a column in a proper Young wall is called a \emph{$\updelta$-column} if it contains $da_0$-many $0$-blocks,  $da_1$-many $1$-blocks, \ldots, $da_n$-many $n$-blocks.
\item[{\rm (ii)}] A $\updelta$-column in a proper Young wall $\YWw$ is {\it removable} if one can remove the $\updelta$-column from $\YWw$ and the result is still a proper Young wall.
\item[{\rm (iii)}] A proper Young wall is said to be {\it reduced} if it has no removable $\updelta$-column.
\end{enumerate}
\end{definition}

We denote by $\YY(\ofw)$ the set of all reduced proper Young walls on $\boxed{\ofw}$.

\begin{theorem}[\cite{Kang03}] \hfill
\begin{enumerate}
\item[{\rm (1)}] The set $\ZZ(\ofw)$ with $\eit,\fit,\wt,\ve_i$ and $\vp_i$ is an affine crystal.
\item[{\rm (2)}] The set $\YY(\ofw)$ is an affine subcrystal which is isomorphic to $\mathbf{B}(\ofw)$, where
$\mathbf{B}(\ofw)$ is the crystal of the highest weight module $V^q(\ofw)$.
\end{enumerate}
\end{theorem}

\subsection{Higher level representations} In this subsection, we will realize the crystal $\mathbf{B}(\Uplambda)$ for $\Uplambda(c) \ge 2$ in terms of tensor products of Young walls.
To begin with, we consider the crystal $\mathbf{B}(k\ofw)$ of level $k$ and see that $\mathbf{B}(k\ofw)$ is realized as
\begin{eqnarray} &&
\parbox{85ex}{ the subcrystal of $\ZZ(\ofw)^{\otimes k}$ whose graph is the connected component of the $k$-fold tensor of ground-state Young walls, denoted by
$\boxed{k\ofw} \seteq  \underbrace{\boxed{\ofw} \otimes \cdots \otimes \boxed{\ofw}}_{\text{ k-times}}.$   }\label{eq: char B(klambda)}
\end{eqnarray}

Next we consider $\mathbf{B}(\Lambda_s)$ where $\Lambda_s$ is a fundamental weight of level $2$.   In order to embed $\mathbf{B}(\Lambda_s)$ into a tensor product $\ZZ(\ofw')\otimes\ZZ(\ofw'')$ for some $\ofw'$ and $\ofw''$ of level $1$, we first need  equations of the form
\begin{align} \label{eq: computation}
\Lambda_s - m \updelta= \ofw'+\ofw'' -\sum_{i \in I}t_i \alpha_i   \quad \text{ for some $m\in \Z$ and $t_i\in \Z$, $i\in I$.}
\end{align}
For each $\g$ and a fundamental weight $\Lambda_s$ of level $2$, an equation of the form \eqref{eq: computation}
is explicitly given in  what follows according to whether $\ofw'$ and $\ofw''$ are of type $\sS$ or $\Ss$.
Using the pairing $\langle \:\:, \:\: \rangle$,
one can compute the followings:
\begin{equation} \label{Type HT}
\begin{aligned}
\text{ Type $\sS$:} \qquad &\Lambda_{2u}-u\updelta= 2\ofw_0-\left(u\alpha_0+(u-1)\alpha_1+\displaystyle\sum_{i=2}^{2u-1}(2u-i)\alpha_i\right),\allowdisplaybreaks\\
&\Lambda_{2u+1}-u\updelta = \ofw_1+\ofw_0-\left(u\alpha_0+u\alpha_1+\displaystyle\sum_{i=2}^{2u}(2u+1-i)\alpha_i\right), \allowdisplaybreaks\\
&\Lambda_{n-2u} = 2\ofw_n - \left(u\alpha_n+(u-1)\alpha_{n-1}+\displaystyle\sum_{i=n-2}^{n-2u+1}i-(n-2u)\alpha_i\right), \allowdisplaybreaks\\
&\Lambda_{n-2u-1} = \ofw_{n-1}+\ofw_n - \left(u\alpha_n+u\alpha_{n-1}+\displaystyle\sum_{i=n-2}^{n-2u}i-(n-2u-1)\alpha_i\right). \qquad\qquad
\end{aligned}
\end{equation}
\begin{equation} \label{Type HH}
\begin{aligned}
\text{ Type $\Ss$:} \qquad & \Lambda_u  = 2\ofw_n-\displaystyle\sum_{i=u+1}^{n}(i-u)\alpha_i, \qquad
\ \Lambda_u-u\updelta = 2\ofw_0-\displaystyle\sum_{i=0}^{u-1}(u-i)\alpha_i. \qquad\qquad\qquad
\end{aligned}
\end{equation}

Here we observe that what is subtracted in the right-hand side of each of the formulas in \eqref{Type HT} and \eqref{Type HH} corresponds to a specific type of partitions. To be precise, we need the following definition.
\begin{definition} \label{def: staircase partition}
For a positive integer $m$, we denote by $\uplambda(m)$ the strict partition given by
$$\uplambda (m)=(m,m-1,\ldots,2,1),$$
and call $\uplambda (m)$ the {\it $m$-th staircase partition}. We also set $\uplambda(m)=(0)$ for any non-positive integer $m$.
\end{definition}

Now, for each $\Lambda_s$ of level $2$, the crystal $\mathbf{B}(\Lambda_s)$ is realized up to a weight shift by an element of $\Z \updelta$ as the subcrystal of $\ZZ(\ofw')\otimes\ZZ(\ofw'')$ generated by a highest weight crystal $\boxed{\ofw'} \otimes \mathsf{Y}_{\ofw''}^{\uplambda(s)}$ for some staircase partition $\uplambda(s)$.
Concretely,
we associate a tensor product of Young walls to a fundamental weight $\Lambda_s$ of level $2$  using \eqref{Type HT} and \eqref{Type HH}.
\begin{align*}
{\rm (i)} \ \text{ For type $\sS$,} \ \ \Lambda_{2u}-u\updelta = 2\ofw_0 - \left(u\alpha_0+(u-1)\alpha_1 +  \displaystyle\sum_{i=2}^{2u-1}(2u-i)\alpha_i\right)\longleftrightarrow \ &
\boxed{\Lambda_{2u}^{0,0}} \seteq \boxed{\ofw_0} \otimes \YW^{\uplambda(2u-1)}_{\ofw_0}, \allowdisplaybreaks\\
\Lambda_{2u+1}  -  u\updelta  =  \ofw_1  +  \ofw_0  - \left(u\alpha_0+  u\alpha_1  + \displaystyle\sum_{i=2}^{2u}(2u+1-i)\alpha_i  \right)\longleftrightarrow \ &
\boxed{\Lambda_{2u+1}^{1,0}} \seteq \boxed{\ofw_1} \otimes \YW^{\uplambda(2u)}_{\ofw_0}, \allowdisplaybreaks\\
  \text{ or } & \boxed{\Lambda_{2u+1}^{0,1}} \seteq  \boxed{\ofw_0} \otimes \YW^{\uplambda(2u)}_{\ofw_1}, \allowdisplaybreaks \\
\Lambda_{n-2u} = 2\ofw_n  - \left(  u\alpha_n+(u-1)\alpha_{n-1}+ \displaystyle\sum_{i=n-2}^{n-2u+1}  i-(n-2u)\alpha_i  \right) \longleftrightarrow \ &
\boxed{\Lambda_{n-2u}^{n,n}} \seteq \boxed{\ofw_n} \otimes \YW^{\uplambda(2u-1)}_{\ofw_n}, \allowdisplaybreaks \\
\Lambda_{n-2u-1} = 2\ofw_n - \left(  u\alpha_n+u\alpha_{n-1}+ \displaystyle\sum_{i=n-2}^{n-2u}  i-(n-2u-1)\alpha_i \right) \longleftrightarrow \ &
\boxed{\Lambda_{n-2u-1}^{n,n-1}} \seteq \boxed{\ofw_{n}} \otimes \YW^{\uplambda(2u)}_{\ofw_{n-1}}, \\ 
  \text{ or } & \boxed{\Lambda_{n-2u-1}^{n-1,n}} \seteq  \boxed{\ofw_{n-1}} \otimes \YW^{\uplambda(2u)}_{\ofw_{n}}.
\end{align*}

\begin{example}\label{ex:Type HT}
For $\g= D_7^{(1)}$, we describe $\boxed{\Lambda_{4}^{0,0}}$ and $\boxed{\Lambda_{2}^{6,7}}$:
\begin{align*}
& \boxed{\Lambda_{4}^{0,0}} = \boxed{\ofw_0} \otimes   \scalebox{.8}{{\xy (0,3)*++{\hwLfourB}\endxy}}, \quad
 \boxed{\Lambda_{2}^{6,7}} = \boxed{\ofw_6} \otimes   \scalebox{.8}{{\xy (0,6)*++{\hwLfiveBB}\endxy}}.
\end{align*}
\end{example}

\begin{align*}
{\rm (ii)} \ \text{ For type $\Ss$, } \ \Lambda_u= 2\ofw_n-\left(\displaystyle\sum_{i=u+1}^{n-1}(i-u)\alpha_i\right) \longleftrightarrow \ & \boxed{\Lambda_{u}^{n,n}} \seteq \boxed{\ofw_n} \otimes \YW^{\uplambda(n-u)}_{\ofw_n}, \qquad\qquad\qquad\qquad\qquad\\
\Lambda_u-u\updelta= 2\ofw_0-\left(\displaystyle\sum_{i=0}^{u-1}(u-i)\alpha_i\right) \longleftrightarrow \ & \boxed{\Lambda_{u}^{0,0}} \seteq \boxed{\ofw_0} \otimes \YW^{\uplambda(u-1)}_{\ofw_0}.
\end{align*}

\begin{example} \label{ex:Type HH}
For $\g= B_7^{(1)}$,
we have
\begin{align*}
& \boxed{\Lambda_{5}^{7,7}} = \boxed{\ofw_7} \otimes   \scalebox{.8}{{\xy (0,2)*++{\hwLfiveBfn}\endxy}} .
\end{align*}
\end{example}

The tensor products of Young walls given above will be denoted by $\boxed{\Lambda_s}$ without superscripts if there is no possible confusion. One can see that  the crystal $\mathbf{B}(\Lambda_s)$ is realized as the subcrystal of $\ZZ(\ofw')\otimes\ZZ(\ofw'')$ generated by $\boxed{\Lambda_s}$ for each fundamental weight $\Lambda_s$ of level $2$.

\medskip

Next, the crystal $\mathbf{B}((k-2)\ofw+\Lambda_s)$ of level $k$ is realized as
\begin{eqnarray} &&
\parbox{85ex}{ the subcrystal of $\ZZ(\ofw)^{\otimes k-2} \otimes \ZZ(\ofw')\otimes \ZZ(\ofw'')$ generated by the highest weight crystal
$\boxed{(k-2)\ofw} \otimes \boxed{\Lambda_{s}}$ whose weight is $(k-2)\ofw+ \Lambda_s$ up to $\Z\updelta$. Here, $\boxed{k \ofw} =  \boxed{\ofw}^{\, \otimes k}$ as defined in \eqref{eq: char B(klambda)}. }\label{eq: char B(klambda+lambda_k)}
\end{eqnarray}

\begin{remark} \hfill
\begin{enumerate}
\item There are several other possible realizations of $\mathbf{B}((k-2)\ofw+\Lambda_s)$ depending on the choice of highest weight crystals. For example,
the connected component originated from $$\boxed{a\ofw} \otimes \boxed{\Lambda_{s}} \otimes \boxed{b\ofw} \subset \ZZ(\ofw)^{\otimes a} \otimes \ZZ(\ofw')\otimes \ZZ(\ofw'') \otimes \ZZ(\ofw)^{\otimes b} \quad  (a+b=k-2)$$
is also a realization of $\mathbf{B}((k-2)\ofw+\Lambda_s)$, and we can also choose different highest weight crystals for $\Lambda_s$.
\item For each $\Uplambda\in P^+$ of level $k$ with $\Uplambda=\sum_{i=0}^{n}m_i\Lambda_i$, the crystal
$\mathbf{B}(\Uplambda)$ can be realized as the subcrystal of $\ZZ(\ofw_{i_1}) \otimes \ZZ(\ofw_{i_2}) \otimes \cdots \otimes \ZZ(\ofw_{i_k})$ for some $(i_1, i_2, \dots , i_k)$, which is generated by $\otimes_{i=0}^n \boxed{\Lambda_i}^{\, \otimes m_i}$. (Here we abuse notations a little bit and write $\boxed{\Lambda_i} =\boxed{\ofw_i}$ even if $\ofw_i$ is of level $1$.)
\end{enumerate}
\end{remark}

Throughout this paper we will use the following notational convention.

\begin{convention} For a statement $P$, the number $\delta(P)$ is equal to 1 if $P$ is true and 0 if $P$ is false. Sometimes, we will write $\delta_P$  for $\delta(P)$. \end{convention}

\section{Young tableaux and almost even tableaux}
\label{sec:YT}

In this section, we make connections between tensor products of Young walls and Young tableaux.

\subsection{Young tableaux}

For partitions $\lambda^{(1)}$ and $\lambda^{(2)}$, we define the partition $\lambda^{(1)} * \lambda^{(2)}$ by rearranging the parts of $\lambda^{(1)}$ and
$\lambda^{(2)}$ in a weakly decreasing way. As an obvious generalization, for partitions $\lambda^{(1)},\lambda^{(2)},\dots ,\lambda^{(k-1)},\lambda^{(k)}$, we set
\begin{align}\label{eq: def of cons}
\dbs{t=1}{k}\lambda^{(t)} \seteq \lambda^{(1)}*\lambda^{(2)}*\cdots * \lambda^{(k-1)}*\lambda^{(k)}.
\end{align}

\begin{example}
For partitions $\lambda^{(1)}=(7,3,1)$, $\lambda^{(2)}=(8,6,6,3)$ and $\lambda^{(3)}=(7,5,4,1)$, we have
$$\dbs{t=1}{3}\lambda^{(t)}=(8,7,7,6,6,5,4,3,3,1,1).$$
\end{example}

The {\it Young diagram} $Y^\lambda$ associated to a partition $\lambda =(\lambda_1, \lambda_2, \dots , \lambda_k)$ is a finite collection of cells arranged in left-justified rows, with the $i$-th row length given by $\lambda_i$.

We also define a partial order $\subset$ on the set of all partitions, called the {\it inclusion order}, in the following way:
$$  \mu \subset \lambda \quad \text{ if and only if }  \quad  Y^\mu \subset Y^\lambda.$$

A \emph{skew partition}, denoted by $\lambda\skewpar \mu$, is a pair of two partitions $\lambda$ and $\mu$ satisfying $\mu\subset \lambda$.  For a skew partition $\lambda\skewpar\mu$, the {\it skew Young diagram} $Y^{\lambda\skewpar\mu}$ is the diagram obtained by removing cells corresponding to $Y^\mu$ from $Y^\lambda$. The notation $\lambda\skewpar\mu\vdash m$  means that the number of cells in $Y^{\lambda\skewpar\mu}$ is $m$.

We will identify a usual partition $\lambda$ with the skew partition $\lambda\skewpar\emptyset$. In this identification, every definition on the skew partitions in this section induces a definition on the usual partitions.

\begin{definition} \hfill
\begin{enumerate}
\item A \emph{tableau} $T$ is a filling of the cells in the skew Young diagram $Y^\lm$ with integers $1,2,\dots,m$ for some skew partition $\lm\vdash m$. In this case we say that the \emph{shape $\Sh(T)$} of the tableau $T$ is $\lm$.

\item A \emph{standard Young tableau} is a tableau in which the entries in each row and each column are increasing. We denote by $\mathcal S^\lm$ the set of standard Young tableaux of shape $\lm$.

\item A \emph{reverse standard Young tableau} is a tableau in which the entries in each row and each column are decreasing. We denote by $\RYT^\lm$ the set of reverse standard Young tableaux of shape $\lm$.
\end{enumerate}
\end{definition}

\begin{example} \label{ex: ST RT}
The following tableaux are a reverse standard Young tableau of shape $(4,3,1)$ and a  standard Young tableau of shape $(4,3,1)$:
$$T=\young(8643,721,5)  \in \RYT^{(4,3,1)},
\qquad T'=\young(1356,278,4) \in \mathcal S^{(4,3,1)}.$$
\end{example}

Note that there is an obvious bijection between $\RYT^{\lm}$ and $\mathcal S^{\lm}$ that replaces each integer $i$ with $m+1-i$, where $\lm\vdash m$. The two tableaux in Example~\ref{ex: ST RT} correspond to each other under this bijection. Thus we have $|\RYT^{\lm}|=|\mathcal S^{\lm}|$.
We will sometimes identify reverse standard Young tableaux and standard Young tableaux using this bijection.
We denote by $f^{\lambda}=|\RYT^{\lambda}|=|\mathcal S^{\lambda}|$. Recall that there is a well known formula for $f^{\lambda}$ called the hook-length formula.

{\it In this paper, we only consider reverse standard Young tableaux
except the last 3 sections. Hence, for simplicity, we call a reverse standard Young tableau  just a Young tableau.}

For later use, we define another notation related to a tableau.
\begin{definition} \label{def: T less}
For a Young tableau $T$ with $m$ cells, we denote
by $T_{> s}$ for $1 \le s \le m$ the tableau which is obtained by removing all cells filled with $t$ such that $t \le s$ and replacing $u>s$ with $u-s$ for all $u>s$.
\end{definition}

For $T$ in Example \ref{ex: ST RT},
$$T_{>1} = \young(7532,61,4). $$

\medskip

Let $\Ss_m^{(k)}$ denote the set of Young tableaux with $m$ cells and at most $k$ rows.
It is well known that the cardinality of $\Ss_m^{(k)}$ is equinumerous to the number of $(k+1,k,\ldots,1)$-avoiding involutions in the symmetric group $\mathfrak S_m$.

In the literature an explicit formula for $|\Ss_m^{(k)}|$ is known only for $k \le 5$ as follows.

\begin{theorem} \label{thm: Standard at most} \cite{GB89,Re91} We have
$$|\Ss_m^{(2)}|= \left( \begin{matrix} m \\ \lfloor \frac{m}{2} \rfloor \end{matrix} \right), \
|\Ss_m^{(3)}|= \displaystyle\sum_{i =0}^{\lfloor \frac{m}{2} \rfloor}\mathsf{C}_i \binom{m}{2i}, \
|\Ss_m^{(4)}|= \mathsf{C}_{\lceil \frac{m+1}{2} \rceil}\mathsf{C}_{\lfloor \frac{m+1}{2} \rfloor}, \
|\Ss_m^{(5)}|= 6 \displaystyle \sum_{i=0}^{\lfloor \frac{m}{2} \rfloor}\binom{m}{2i}
\dfrac{(2i+2)!\mathsf{C}_i}{(i+2)!(i+3)!},
$$
where $\mathsf{C}_m=\frac{1}{m+1}\binom{2m}{m}$ is the $m$-th Catalan number.
\end{theorem}

Note that each element in $\Ss_m^{(k)}$ can be expressed in terms of a sequence of strict partitions as follows:
$$\Ss_m^{(k)}=\left\{  \ula=(\lambda^{(1)},\dots,\lambda^{(\ell)}) \ | \ \ell \le k, \  \lambda^{(i)} \supset \lambda^{(i+1)} \ (1 \le i<\ell)  \text{ and }
\lambda^{(1)} * \cdots * \lambda^{(\ell)} =\uplambda(m) \right\}.$$
In Example \ref{ex: ST RT}, the tableau $T$ can be identified with $\left( (8,6,4,3) \supset (7,2,1) \supset (5) \right) \in \Ss_8^{(3)}$:
$$\young(8643,721,5) \ \longleftrightarrow \ \ula = \left( \ \ \yng(8,6,4,3) \supset \yng(7,2,1) \supset \yng(5)  \right) $$

\subsection{Tensor products of Young walls} As we have seen in Definition \ref{def: association},
we can construct a Young wall when we have a partition $\lambda$ and a fundamental weight $\ofw$ of level $1$. Since a (skew) Young tableau $T$
is identified with a sequence $\ula=(\lambda^{(1)}, \dots , \lambda^{(k)})$ of strict partitions, we can make a correspondence
between a (skew) Young tableau $T$ of shape $\mu\skewpar\lambda$ with $k$ rows
 and a $k$-fold tensor product of Young walls,
$$\yy^T_\uofw \text{ or } \yy^\ula_\uofw \seteq \YW^{\lambda^{(1)}}_{\ofw_{i_1}} \otimes \YW^{\lambda^{(2)}}_{\ofw_{i_2}} \otimes \cdots \otimes \YW^{\lambda^{(k)}}_{\ofw_{i_k}} \quad \text{ with } \ula = \ula_T,$$
for a fixed sequence $\uofw=(\ofw_{i_1},\ofw_{i_2},\dots,\ofw_{i_k})$ of fundamental weights of level $1$.

\begin{example} \label{ex: rigid begining}
For $\g=D^{(1)}_7$, let $\uofw= (\ofw_0,\ofw_0)$ and consider the Young wall $\boxed{\Lambda_{4}^{0,0}}$ in Example \ref{ex:Type HT}. Then we have the correspondence
$$T=\young(\cdot\cdot\cdot,321) \quad  \longleftrightarrow \quad \yy^{((0),\uplambda(3))}_{(\ofw_0,\ofw_0)}
= \boxed{\ofw_0} \otimes \YW^{\uplambda(3)}_{\ofw_0} = \boxed{\Lambda_{4}^{0,0}}. $$

For $\g=B^{(1)}_7$, consider $T=\young(\cdot432,51)$ of shape $(4,2)\skewpar (1)$ and $\uofw=(\ofw_0,\ofw_1)$. Then the corresponding Young wall  $\yy^{T}_\uofw$ is given by
$$  \resizebox{1.8cm}{1.1cm}
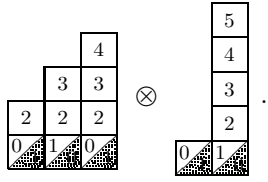
{\xy (0,0)*++{ \hwLfivefourB}\endxy}  \otimes  \resizebox{1.3cm}{1.5cm}{\xy (0,2)*++{ \hwLtwooneB}\endxy} . $$
\end{example}

\subsection{Some families of Young tableaux} In this subsection, we shall introduce special families of Young tableaux
and study the cardinalities of the families.

A {\em composition} $\lambda$ of $m$ is a sequence $(\lambda_1, \dots, \lambda_k)$ of nonnegative integers such that $\sum_{i=1}^k \lambda_i=m$.
\begin{definition} We say that a composition $\lambda$ of $m$ is
{\it almost even} when it satisfies one of the following conditions:
\begin{itemize}
\item If $m$ is odd, then it contains one odd part and the other parts are even.
\item If $m$ is even, then it contains two odd parts and the other parts are even.
\end{itemize}
We write $\lambda \Vdash_0 m$ to denote an almost even composition $\lambda$ of $m$.

We denote by $\Ae^{(k)}_m$ $(k \le m)$ the subset of $\Ss^{(k)}_{m}$ consisting of the tableaux $T$
such that $\Sh(T) \Vdash_0 m$ and call $T \in \Ae^{(k)}_m$ an {\it almost even tableau} of $m$ with at most $k$ rows.
\end{definition}

\begin{example}
For $m=5,6$ and $k=2$, we have
$$
\young(532,41) \in \Ae^{(2)}_5,  \ \ \young(5432,1) \in \Ae^{(2)}_5 \ \text{ and } \ \young(6532,41) \not \in \Ae^{(2)}_6,  \ \ \young(65432,1) \in \Ae^{(2)}_6.
$$
\end{example}

For $\epsilon \in \{ 0,1\}$ and $k\le m$, we denote by ${}^\epsilon\mathfrak{P}^{(k)}_m$  the subset of $\Ss^{(k)}_{m}$
consisting of the tableaux $T$ satisfying
\begin{align} \label{eq: conditions for pt}
\lambda \seteq \Sh(T) \vdash m \quad \text{ and } \quad \lambda_i \equiv \epsilon \ ({\rm mod} \ 2) \quad \text{ for all }  1 \le i \le k.
\end{align}
We say that $T \in {}^\epsilon\RPT^{(k)}_m$ is an {\it $\epsilon$-parity tableau} for $\epsilon \in \{ 0,1 \}$. We set $\RPT^{(k)}_m = {}^0\RPT^{(k)}_m \sqcup
{}^1\RPT^{(k)}_m$ and call it the set of {\it parity tableaux} of $m$ cells with at most $k$ rows.

\begin{example}
The following are examples of parity tableaux:
\begin{align*}
\young(532,4,1) \in {}^1\RPT^{(3)}_5, \quad \quad \young(6521,43) \in {}^0\RPT^{(2)}_6 .
\end{align*}
On the other hand, we have
\begin{align*}
\young(432,1) \not \in {}^0\RPT^{(3)}_4, \quad \quad  \young(532,41)\not \in \RPT^{(2)}_5.
\end{align*}
\end{example}

\begin{remark} \label{rmk-ae}
Note that $\Ae_{2m-1}^{(2)}=\Ss_{2m-1}^{(2)}$, and by Theorem  \ref{thm: Standard at most}, we have
$$|\Ae_{2m-1}^{(2)}|=|\Ss_{2m-1}^{(2)}|=\binom{2m-1}{m}.$$
Furthermore, one can observe that
\begin{itemize}
\item $\Ss_{2m}^{(2)} =\Ae_{2m}^{(2)} \sqcup {}^0\RPT_{2m}^{(2)}$ and $\Ae_{2m}^{(2)} = {}^1\RPT_{2m}^{(2)}$,
\item there exists a bijection  $\psi: \Ae_{2m}^{(2)} \to {}^0\RPT_{2m}^{(2)}$ such that $\psi(T)$ is the tableau which is obtained
by moving the cell filled with $1$ from its row in $T$ to the other row.
\end{itemize}
Since $|\Ss_{2m}^{(2)}|= \binom{2m}{m}$ by Theorem  \ref{thm: Standard at most}, we have
\begin{align} \label{eq: almost even at most 2}
|\Ae_{2m}^{(2)}|=|{}^0\RPT_{2m}^{(2)}|=|{}^1\RPT_{2m}^{(2)}|  =\dfrac{1}{2}\binom{2m}{m}=\binom{2m-1}{m}=|\Ss_{2m-1}^{(2)}|.
\end{align}
\end{remark}

\section{Lattice paths and triangular arrays}\label{sec:LP}

In this section, we find an interesting relationship among Young tableaux with at most $k=2$ or $3$ rows, triangular arrays related to lattice paths and composition
multiplicities of $m$-fold tensor products of irreducible $\mathfrak{sl}_2$-modules.

\subsection{Motzkin triangle}
\begin{definition} A {\it Motzkin path} is a path on the lattice
$\Z^2$ starting from $(0, 0)$, having three kinds steps called an \emph{up step} $U=(1, 1)$, a \emph{horizontal step} $H=(1, 0)$, a \emph{down step} $D=(1, -1)$, and not going below the $x$-axis.
\end{definition}

\begin{example} The following path is a Motzkin path from $(0,0)$ to $(10,1)$:
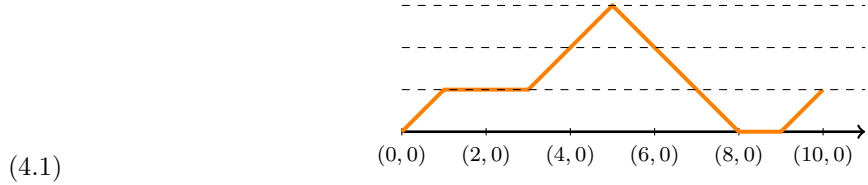
\begin{align}\label{MotzkinEx}
\begin{tikzpicture}[x=0.8cm,y=0.8cm, scale=0.7]
\foreach \x in {0,2,4,6,8,10}
\draw[shift={(\x-1,0)},color=black, ] (0pt,2pt) -- (0pt,-2pt) node[below] {\footnotesize $(\x,0)$};
\draw[line width=1 pt, ->] (-1,0)--(10,0);
\draw[orange, line width=1.5pt] (-1,0)--(0,1)--(2,1)--(3,2)--(4,3)--(6,1)--(7,0)--(8,0)--(9,1);
\foreach \y in {1,2,3}
    \draw[dashed](-1,\y)--(10,\y);
\end{tikzpicture}
\end{align}
We also express the above path as a sequence of steps by $UHHUUDDDHU$.
\end{example}

\begin{definition} A {\it generalized Motzkin number $\mathsf{M}_{(m,s)}$} for $m \ge s \ge 0$ is the number of all Motzkin paths
ending at the lattice point $(m,s)$. In particular, we write $\mathsf{M}_m =\mathsf{M}_{(m,0)}$ and call it the {\it $m$-th Motzkin number}.
\end{definition}

Interestingly, the Motzkin number $\mathsf{M}_m$ is also equal to the number
of all Young tableaux with $m$ cells and at most $3$ rows, see~\cite{Eu}.
That is, we have
\begin{align} \label{eq: Motzkin at most 3}
\mathsf{M}_m = |\Ss_m^{(3)}|=\sum_{i = 0}^{\lfloor \frac{m}{2} \rfloor}\mathsf{C}_i\binom{m}{2i}.
\end{align}
A recursive formula and a closed formula for $\mathsf{M}_{(m,s)}$ are known and easy to derive:
\begin{align}
\mathsf{M}_{(m,s)} & = \mathsf{M}_{(m-1,s)} + \mathsf{M}_{(m-1,s-1)}+ \mathsf{M}_{(m-1,s+1)} \label{eq: Motzkin triangle recursive} \\
&= \sum_{i=0}^{\lfloor s/2 \rfloor}\binom{m}{2i+m-s}\left(\binom{2i+m-s}{i} - \binom{2i+m-s}{i-1}\right).  \label{eq: Motzkin triangle closed}
\end{align}

Consider the following triangular array consisting of $\mathsf{M}_{(m,s)}$ and reflecting the recursive relation \eqref{eq: Motzkin triangle recursive}.
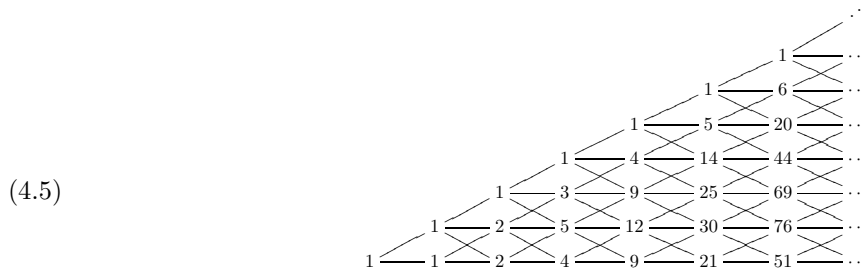
\begin{equation} \label{MT}
\raisebox{7em}{\scalebox{.7}{\xymatrix@R=0.5pc@C=2pc{&&&&&&&\iddots \\
&&&&&&1\ar@{-}[dr] \ar@{-}[r] \ar@{-}[ur]&\cdots \\
&&&&&1 \ar@{-}[dr] \ar@{-}[r] \ar@{-}[ur] &6 \ar@{-}[dr] \ar@{-}[r] \ar@{-}[ur]&\cdots \\
&&&&1\ar@{-}[dr] \ar@{-}[r] \ar@{-}[ur]&5\ar@{-}[dr] \ar@{-}[r] \ar@{-}[ur]&20\ar@{-}[dr] \ar@{-}[r] \ar@{-}[ur]&\cdots\\
&&&1\ar@{-}[dr] \ar@{-}[r] \ar@{-}[ur]&4\ar@{-}[dr] \ar@{-}[r] \ar@{-}[ur]&14\ar@{-}[dr] \ar@{-}[r] \ar@{-}[ur]&44\ar@{-}[dr] \ar@{-}[r] \ar@{-}[ur]&\cdots \\
&&1\ar@{-}[dr] \ar@{-}[r] \ar@{-}[ur]&3\ar@{-}[dr] \ar@{-}[r] \ar@{-}[ur]&9\ar@{-}[dr] \ar@{-}[r] \ar@{-}[ur]&25\ar@{-}[dr] \ar@{-}[r] \ar@{-}[ur]&69\ar@{-}[dr] \ar@{-}[r] \ar@{-}[ur]&\cdots \\
&1\ar@{-}[dr] \ar@{-}[r] \ar@{-}[ur]&2\ar@{-}[dr] \ar@{-}[r] \ar@{-}[ur]&5\ar@{-}[dr] \ar@{-}[r] \ar@{-}[ur]&12\ar@{-}[dr] \ar@{-}[r] \ar@{-}[ur]&30\ar@{-}[dr] \ar@{-}[r] \ar@{-}[ur]&76\ar@{-}[dr] \ar@{-}[r] \ar@{-}[ur]&\cdots \\
1 \ar@{-}[r] \ar@{-}[ur]&1  \ar@{-}[r] \ar@{-}[ur]&2 \ar@{-}[r] \ar@{-}[ur]&4 \ar@{-}[r] \ar@{-}[ur]&9  \ar@{-}[r] \ar@{-}[ur]&21  \ar@{-}[r] \ar@{-}[ur]&51  \ar@{-}[r] \ar@{-}[ur]&\cdots
}}}
\end{equation}
Here a solid line represents the contribution of a number to the number connected by the line in the next column. For example, we obtain $76$ as $25+30+21$ from the previous column.
We call this triangular array the {\it Motzkin triangle}.

\begin{remark} \label{rem: Motzkin}
For $m \in \Z_{\ge 0}$, let $V_m$ be the $(m+1)$-dimensional irreducible module over $\mathfrak{sl}_2$. In particular, the standard module $\mathbf{V}$ is (isomorphic to) $V_1$ and the adjoint module $\mathbb V$ is $V_2$. The Clebsch--Gordan formula yields
$$ V_m \otimes \mathbb{V} \simeq V_{m-2} \oplus V_m \oplus V_{m+2} \quad \text{ for } m \ge 2.$$
Using \eqref{eq: Motzkin triangle recursive}, one can show that  $\mathsf{M}_{(m,s)}$ is equal to  the multiplicity of $V_{2s+1}$ in $ \mathbf{V} \otimes \mathbb{V}^{\otimes m}$.
The same observation holds for $\mathsf{V} \seteq V_1 \oplus V_0$ (see \cite{BH}); that is, one can check that
$\mathsf{M}_{(m,s)}$ is equal to  the multiplicity of $V_{s}$ in $ \mathsf{V}^{\otimes m}$.
Thus we have
\begin{align} \label{eq: 3m}
3^m = \sum_{s=0}^m (s+1)\mathsf{M}_{(m,s)}.
\end{align}
\end{remark}

\subsection{Riordan triangle}
\begin{definition}
A {\it Riordan path} is a Motzkin path which has
without horizontal step on the $x$-axis.
\end{definition}

\begin{example} The following path is a Riordan path:
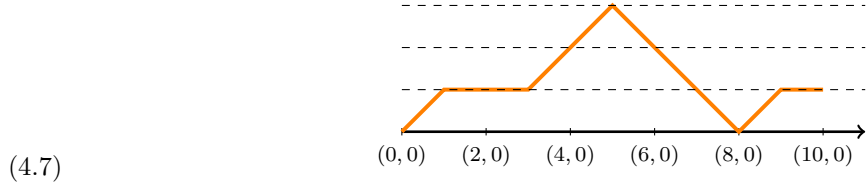
\begin{align}\label{RiordanEx}
\begin{tikzpicture}[x=0.8cm,y=0.8cm, scale=0.7]
\foreach \x in {0,2,4,6,8,10}
\draw[shift={(\x-1,0)},color=black, ] (0pt,2pt) -- (0pt,-2pt) node[below] {\footnotesize $(\x,0)$};
\draw[line width=1 pt, ->] (-1,0)--(10,0);
\draw[orange, line width=1.5pt] (-1,0)--(0,1)--(2,1)--(3,2)--(4,3)--(6,1)--(7,0)--(8,1)--(9,1);
\foreach \y in {1,2,3}
    \draw[dashed](-1,\y)--(10,\y);
\end{tikzpicture}
\end{align}
\end{example}
Note that the path in \eqref{MotzkinEx} is {\it not} a Riordan path.

\begin{definition} A {\it generalized Riordan number $\mathsf{R}_{(m,s)}$} for $m \ge s \ge 0$ is the number of all Riordan paths
ending at the lattice point $(m,s)$. In particular, we write $\mathsf{R}_m = \mathsf{R}_{(m,0)}$ and call it the {\it $m$-th Riordan number}.
\end{definition}

The Riordan number $\mathsf{R}_m$ has a closed formula: $\mathsf{R}_0=1$, $\mathsf{R}_1=0$ and
\begin{align} \label{eq: riordan closed}
\mathsf{R}_m= \dfrac{1}{m+1}\displaystyle\sum_{i=1}^{\lfloor m/2 \rfloor }\binom{m+1}{i}\binom{m-i-1}{i-1} \quad \text{ for $m \ge 2$}.
\end{align}
We see that $\mathsf{R}_{(m,s)}$ has a recursive formula
\begin{align}\label{eq: Riordan triangle recursive}
\mathsf{R}_{(m,s)} = \begin{cases}
\mathsf{R}_{(m-1,s)} + \mathsf{R}_{(m-1,s-1)}+ \mathsf{R}_{(m-1,s+1)} & \text{ if } s \ge 1, \\
\mathsf{R}_{(m-1,1)}   & \text{ if } s = 0.
\end{cases}
\end{align}

Consider the following triangular array consisting of $\mathsf{R}_{(m,s)}$ and reflecting the recursive formula \eqref{eq: Riordan triangle recursive}.
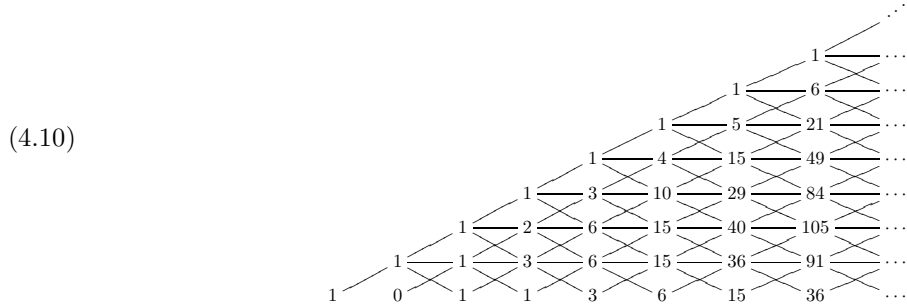
\begin{equation} \label{RT}
\raisebox{5em}{\scalebox{.7}{\xymatrix@R=0.5pc@C=2pc{
&&&&&&&&\iddots \\
&&&&&&&1\ar@{-}[dr] \ar@{-}[r] \ar@{-}[ur]&\cdots \\
&&&&&&1\ar@{-}[dr] \ar@{-}[r] \ar@{-}[ur]&6\ar@{-}[dr] \ar@{-}[r] \ar@{-}[ur]&\cdots \\
&&&&&1\ar@{-}[dr] \ar@{-}[r] \ar@{-}[ur]&5\ar@{-}[dr] \ar@{-}[r] \ar@{-}[ur]&21\ar@{-}[dr] \ar@{-}[r] \ar@{-}[ur]&\cdots \\
&&&&1\ar@{-}[dr] \ar@{-}[r] \ar@{-}[ur]&4\ar@{-}[dr] \ar@{-}[r] \ar@{-}[ur]&15\ar@{-}[dr] \ar@{-}[r] \ar@{-}[ur]&49\ar@{-}[dr] \ar@{-}[r] \ar@{-}[ur]&\cdots\\
&&&1\ar@{-}[dr] \ar@{-}[r] \ar@{-}[ur]&3\ar@{-}[dr] \ar@{-}[r] \ar@{-}[ur]&10\ar@{-}[dr] \ar@{-}[r] \ar@{-}[ur]&29\ar@{-}[dr] \ar@{-}[r] \ar@{-}[ur]&84\ar@{-}[dr] \ar@{-}[r] \ar@{-}[ur]&\cdots \\
&&1\ar@{-}[dr] \ar@{-}[r] \ar@{-}[ur]&2\ar@{-}[dr] \ar@{-}[r] \ar@{-}[ur]&6\ar@{-}[dr] \ar@{-}[r] \ar@{-}[ur]&15\ar@{-}[dr] \ar@{-}[r] \ar@{-}[ur]&40\ar@{-}[dr] \ar@{-}[r] \ar@{-}[ur]&105\ar@{-}[dr] \ar@{-}[r] \ar@{-}[ur]&\cdots \\
&1\ar@{-}[dr] \ar@{-}[r] \ar@{-}[ur]&1\ar@{-}[dr] \ar@{-}[r] \ar@{-}[ur]&3\ar@{-}[dr] \ar@{-}[r] \ar@{-}[ur]&6\ar@{-}[dr] \ar@{-}[r] \ar@{-}[ur]&15\ar@{-}[dr] \ar@{-}[r] \ar@{-}[ur]&36\ar@{-}[dr] \ar@{-}[r] \ar@{-}[ur]
&91\ar@{-}[dr] \ar@{-}[r] \ar@{-}[ur]&\cdots \\
1\ar@{-}[ur]&0\ar@{-}[ur]&1\ar@{-}[ur]&1\ar@{-}[ur]&3\ar@{-}[ur]&6\ar@{-}[ur]&15\ar@{-}[ur]&36\ar@{-}[ur]&\cdots
}}}
\end{equation}
We call this triangular array  the {\it Riordan triangle}.

\begin{remark}\label{rem: Riordan}
Let  $\mathbb{V}$ be the adjoint representation of $\mathfrak{sl}_2$ as before.
By the same argument as in Remark \ref{rem: Motzkin}, the number $\mathsf{R}_{(m,s)}$ is equal to the multiplicity of $V_{2s}$ in the decomposition of $\mathbb{V}^{\otimes m}$. Then we have the identity
$$3^m = \displaystyle\sum_{s=0}^m (2s+1)\mathsf{R}_{(m,s)}.$$
\end{remark}

Let $\NR_{(m,s)}=\mathsf{M}_{(m,s)}-\mathsf{R}_{(m,s)}$. In other words, $\NR_{(m,s)}$ is the number of Motzkin paths ending at $(m,s)$ which have at least one horizontal step on the $x$-axis.
\begin{lemma}\label{lem:NR}
For $m\ge s\ge 1$, we have
\[
\mathsf{R}_{(m,s)} = \NR_{(m,s-1)}.
\]
\end{lemma}
\begin{proof}
We prove this by constructing a bijection $\phi:A\to B$, where
$A$ is the set of Motzkin paths ending at $(m,s)$ with at least one horizontal step on the $x$-axis and $B$ is the set of Motzkin paths ending at $(m,s-1)$ with no horizontal step on the $x$-axis. Let $T=t_1t_2\dots t_m\in A$, where $t_1,t_2,\dots,t_m$ are the steps of $T$ in this order. Let $t_i$ be the first horizontal step on the $x$-axis. Then we define $\phi(T)=t_1\dots t_{i-1} (1,1) t_{i+1}\dots t_m$. It is easy to see that $\phi$ is a bijection from $A$ to $B$.
\end{proof}

There is a simple relation between Motzkin numbers and Riordan numbers.

\begin{lemma}\label{lem:MR}
For $m\ge0$, we have
\[
\mathsf{M}_m=\mathsf{R}_m+\mathsf{R}_{m+1}.
\]
\end{lemma}
\begin{proof}
By definition, we have $\mathsf{M}_m = \mathsf{R}_m +\NR_m$.
By Lemma~\ref{lem:NR}, we have
\[
\NR_m = \NR_{(m,0)} = \mathsf{R}_{(m,1)} = \mathsf{R}_{(m+1,0)} = \mathsf{R}_{m+1}. \qedhere
\]
\end{proof}

Note that $\sfR_{(m,s)}=\sfM_{(m,s)}=0$ if $m<s$.

\begin{proposition} \label{lem:rec_R}
For $m,s\ge1$, we have
\[
\sfR_{(m,s)} =  \sfM_{(m-1,s)} + \sfM_{(m-1,s-1)} -\sfR_{(m-1,s)},
\]
and
\[
\sfR_{(m,s)}  = \sum_{i=0}^{m-s} (-1)^i (\sfM_{(m-1-i,s)}+\sfM_{(m-1-i,s-1)}) .
\]
\end{proposition}
\begin{proof}
The left side of the first equation is
\[
\sfR_{(m,s)} = \sfR_{(m-1,s-1)}  +  \sfR_{(m-1,s)} + \sfR_{(m-1,s+1)}.
\]
The right side is
\[
\sfR_{(m-1,s)} + \NR_{(m-1,s)} + \sfR_{(m-1,s-1)}
+\NR_{(m-1,s-1)} -\sfR_{(m-1,s)}.
\]
By Lemma~\ref{lem:NR}, these two quantities are equal.

Using the first identity iteratively, we obtain the second identity.
\end{proof}

\begin{proposition}\label{prop:Dm3}
For $m \ge 1$, we have
$$\mathsf{R}_m=|\Ae_{m-1}^{(3)}|=|\RPT^{(3)}_m|.$$
\end{proposition}
\begin{proof}
One can see that $\Ss^{(3)}_m = \RPT^{(3)}_m \bigsqcup \Ae^{(3)}_m$.
Consider the map $\phi:\RPT^{(3)}_m \to
\Ae^{(3)}_{m-1}$ given by
$$T \longmapsto T_{> 1},$$ where $T_{>1}$ is defined in Definition \ref{def: T less}.
Then it is easy to check that the map $\phi$ is a bijection.
Thus we have
$|\RPT^{(3)}_m|=|\Ae^{(3)}_{m-1}|$. Now we use an induction on $m$. If $m=1$, then $|\RPT^{(3)}_1|=\mathsf{R}_1=0$.
Assume that $|\RPT^{(3)}_m|=\mathsf{R}_m$. Since $\mathsf M_m= |\Ss^{(3)}_m|$, we have
\begin{align*}
|\RPT^{(3)}_{m+1}|=|\Ae^{(3)}_{m}| & = \mathsf{M}_m-|\RPT^{(3)}_{m}| \\
& = \mathsf{M}_m-\mathsf{R}_m = \mathsf{R}_{m+1} \qquad \text{ by Lemma~\ref{lem:MR}}.  \qedhere
\end{align*}
\end{proof}

\begin{remark}
The set of parity tableaux $\RPT^{(3)}_{m}$ and the set of almost even tableaux $\Ae^{(3)}_{m-1}$
can be taken as tableaux models for the Riordan number $\mathsf{R}_m$, so much as the set $\Ss^{(3)}_{m}$ can be used to realize the Motzkin number $\mathsf{M}_m$.
\end{remark}

\subsection{Catalan triangle}
\begin{definition}
A {\it Dyck path} is a Motzkin path without horizontal steps.
\end{definition}

\begin{example} The following path is a Dyck path:
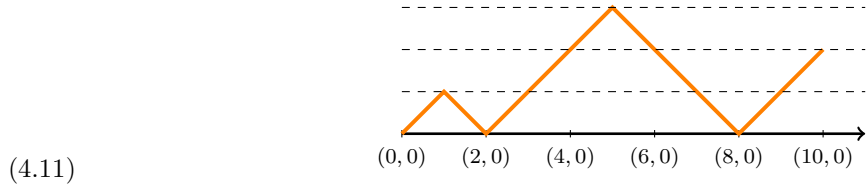
\begin{align}\label{DyckEx}
\begin{tikzpicture}[x=0.8cm,y=0.8cm, scale=0.7]
\foreach \x in {0,2,4,6,8,10}
\draw[shift={(\x-1,0)},color=black, ] (0pt,2pt) -- (0pt,-2pt) node[below] {\footnotesize $(\x,0)$};
\draw[line width=1 pt, ->] (-1,0)--(10,0);
\draw[orange, line width=1.5pt] (-1,0)--(0,1)--(1,0)--(2,1)--(3,2)--(4,3)--(6,1)--(7,0)--(9,2);
\foreach \y in {1,2,3}
    \draw[dashed](-1,\y)--(10,\y);
\end{tikzpicture}
\end{align}
\end{example}
Note that the path in \eqref{RiordanEx} is {\it not} a Dyck path.

\begin{definition} A {\it generalized Catalan number $\mathsf{C}_{(m,s)}$} for $m \ge s \ge 0$ is the number of all Dyck paths
ending at the lattice point $(m,s)$. In particular, we write $\mathsf{C}_m = \mathsf{C}_{(2m,0)}$ which is known as the {\it $m$-th Catalan number}.
\end{definition}

A recursive formula and a closed formula for $\mathsf{C}_{(m,s)}$ are also well-known:
\begin{equation}\label{eq: Catalan triangle recursive}
\begin{aligned}
& \mathsf{C}_{(m,s)} = \delta_{m \equiv_2 s} \  \dfrac{m! (s+1)}{ \frac{m+s+2}{2}!\frac{m-s}{2}! }, \ \ \mathsf{C}_{(m,s)} = \mathsf{C}_{(m-1,s+1)} + \mathsf{C}_{(m-1,s-1)},
\end{aligned}
\end{equation}
where we write $m \equiv_2 s$ for $m \equiv s \pmod 2$.

We have the following triangular array consisting of $\mathsf{C}_{(m,s)}$ and reflecting the recursive relation \eqref{eq: Catalan triangle recursive}.
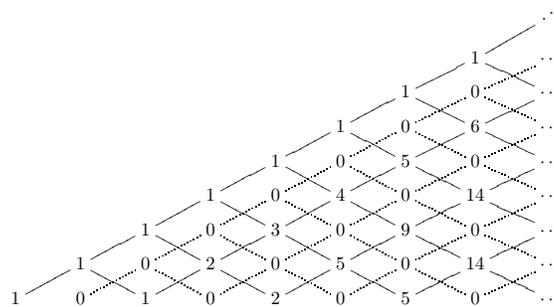
$$
\raisebox{5em}{\scalebox{.7}{\xymatrix@R=0.5pc@C=2pc{
&&&&&&&&\iddots \\
&&&&&&&1\ar@{-}[dr] \ar@{-}[ur]&\cdots \\
&&&&&&1\ar@{-}[dr] \ar@{-}[ur]&0\ar@{.}[dr] \ar@{.}[ur]&\cdots \\
&&&&&1\ar@{-}[dr] \ar@{-}[ur]&0\ar@{.}[dr] \ar@{.}[ur]&6\ar@{-}[dr] \ar@{-}[ur]&\cdots \\
&&&&1\ar@{-}[dr] \ar@{-}[ur]&0\ar@{.}[dr] \ar@{.}[ur]&5\ar@{-}[dr] \ar@{-}[ur]&0\ar@{.}[dr] \ar@{.}[ur]&\cdots\\
&&&1\ar@{-}[dr] \ar@{-}[ur]&0\ar@{.}[dr] \ar@{.}[ur]&4\ar@{-}[dr] \ar@{-}[ur]&0\ar@{.}[dr] \ar@{.}[ur]&14\ar@{-}[dr] \ar@{-}[ur]&\cdots \\
&&1\ar@{-}[dr] \ar@{-}[ur]&0\ar@{.}[dr] \ar@{.}[ur]&3\ar@{-}[dr] \ar@{-}[ur]&0\ar@{.}[dr] \ar@{.}[ur]&9\ar@{-}[dr] \ar@{-}[ur]&0\ar@{.}[dr] \ar@{.}[ur]&\cdots \\
&1\ar@{-}[dr] \ar@{-}[ur]&0\ar@{.}[dr] \ar@{.}[ur]&2\ar@{-}[dr] \ar@{-}[ur]&0\ar@{.}[dr] \ar@{.}[ur]&5\ar@{-}[dr] \ar@{-}[ur]&0\ar@{.}[dr] \ar@{.}[ur] &14\ar@{-}[dr] \ar@{-}[ur]&\cdots \\
1\ar@{-}[ur]&0\ar@{.}[ur]&1\ar@{-}[ur]&0\ar@{.}[ur]&2\ar@{-}[ur]&0\ar@{.}[ur]&5\ar@{-}[ur]&0\ar@{.}[ur]&\cdots
}}}
$$

\begin{remark} \label{rem: Catalan}
By the same argument as in Remark \ref{rem: Motzkin}, the number $\mathsf{C}_{(m,s)}$ is equal to the multiplicity of $V_{s}$ in the decomposition of $\mathbf{V}^{\otimes m}$.
\end{remark}

\begin{remark}\label{rem: Catalan model}
It is well-known that the number of standard tableaux of shape $\lambda=(m+s,m)$ coincides with the number $\mathsf{C}_{(2m+s,s)}$.
\end{remark}

\subsection{Pascal Triangle}

If we consider lattice paths from $(0,0)$ to $(m,s)$ for $m\ge s \ge 0$, having $U=(1,1)$ and $D=(1,-1)$, that may go below the $x$-axis, then the number $\mathsf{B}_{(m,s)}$  of such paths is given by
\[ \mathsf{B}_{(m,s)} = \delta_{m \equiv_2 s} \, \binom{m}{\frac{m-s}2}. \]
Clearly, we have $\mathsf{B}_{(m,s)} = \mathsf{B}_{(m-1,s+1)} + \mathsf{B}_{(m-1,s-1)}$ and the corresponding triangular array is the (half of the) Pascal triangle. The number $\mathsf{B}_{(m,s)}$ is also equal to the multiplicity of $V_{m+s}$  in the composition series of $V_m \otimes \mathbf{V}^{\otimes m}$ where $\mathbf{V}$  is the standard module over $\mathfrak{sl}_2$ as before.

We present the following triangular array consisting of $\mathsf{B}_{(m,s)}$ for reference.
\begin{equation}\label{eqn-Pas}
\raisebox{5em}{\scalebox{.7}{\xymatrix@R=0.5pc@C=2pc{
&&&&&&&&\iddots \\
&&&&&&&1\ar@{-}[dr] \ar@{-}[ur]&\cdots \\
&&&&&&1\ar@{-}[dr] \ar@{-}[ur]&0\ar@{.}[dr] \ar@{.}[ur]&\cdots \\
&&&&&1\ar@{-}[dr] \ar@{-}[ur]&0\ar@{.}[dr] \ar@{.}[ur]&7\ar@{-}[dr] \ar@{-}[ur]&\cdots \\
&&&&1\ar@{-}[dr] \ar@{-}[ur]&0\ar@{.}[dr] \ar@{.}[ur]&6\ar@{-}[dr] \ar@{-}[ur]&0\ar@{.}[dr] \ar@{.}[ur]&\cdots\\
&&&1\ar@{-}[dr] \ar@{-}[ur]&0\ar@{.}[dr] \ar@{.}[ur]&5\ar@{-}[dr] \ar@{-}[ur]&0\ar@{.}[dr] \ar@{.}[ur]&21\ar@{-}[dr] \ar@{-}[ur]&\cdots \\
&&1\ar@{-}[dr] \ar@{-}[ur]&0\ar@{.}[dr] \ar@{.}[ur]&4\ar@{-}[dr] \ar@{-}[ur]&0\ar@{.}[dr] \ar@{.}[ur]&15\ar@{-}[dr] \ar@{-}[ur]&0\ar@{.}[dr] \ar@{.}[ur]&\cdots \\
&1\ar@{-}[dr] \ar@{-}[ur]&0\ar@{.}[dr] \ar@{.}[ur]&3\ar@{-}[dr] \ar@{-}[ur]&0\ar@{.}[dr] \ar@{.}[ur]&10\ar@{-}[dr] \ar@{-}[ur]&0\ar@{.}[dr] \ar@{.}[ur] &35\ar@{-}[dr] \ar@{-}[ur]&\cdots \\
1\ar@{-}[ur]&0\ar@{.}[ur]&2\ar@{-}[ur]&0\ar@{.}[ur]&6\ar@{-}[ur]&0\ar@{.}[ur]&20\ar@{-}[ur]&0\ar@{.}[ur]&\cdots
}}}
\end{equation}

\section{Dominant maximal weights} \label{sec:Rep}
In this section, we investigate the set of dominant maximal weights of highest weight modules $V(\Lambda)$ over affine Kac--Moody algebras of classical types. We will see that most of the dominant maximal weights of levels $2$ and $3$ are essentially finite, and will classify them into the corresponding finite types.
Then, by Lemma \ref{Thm: depend on finite}, the multiplicities of distinct dominant maximal weights of the same finite type can be determined simultaneously even though they appear in highest weight modules over different affine Kac--Moody algebras.
In other words, the multiplicities of essentially finite dominant maximal weights depend only on their finite types.

Another goal of this section is to determine certain families of dominant maximal weights of all levels, which can be associated with pairs $(\uplambda(m),\uplambda(s))$ of staircase partitions
and are essentially finite of type $B_n$ or $D_n$. Again, applying Lemma \ref{Thm: depend on finite}, we see the following:
\begin{eqnarray} &&
\parbox{88ex}{
For two essentially finite dominant maximal weights of the same finite type,
which are associated with the same $(\uplambda(m),\uplambda(s))$, their multiplicities coincide with each other,
{\it even if their affine types are different.}
}\label{eq: coincidence of multiplicity}
\end{eqnarray}

Throughout this section, the (fundamental) weights $\ofw$ of level 1 will be written in boldface;
the weights $\Lambda$ of level $2$ in regular; the weights $\Uplambda$ of level $\ge 3$ in upright. As arguments and techniques are similar, some details are omitted for other types after we consider type $B_n^{(1)}$ thoroughly.

\subsection{Type $A_{n-1}^{(1)}$} \label{subsec-An1} This case was studied in \cite{JM,JM1,Ts,TW}. In this subsection, we briefly review their results and
show that the dominant maximal weights obtained in \cite{Ts,TW} are essentially finite. Hence we can reduce them as dominant weights for some
$L(\omega)$ over $A_{n-1}$.

\medskip

For $0 \le s < n$ and $1 \le \ell \le \left\lfloor \dfrac{n-s}{2} \right\rfloor$  and $1 \le u \le \left\lfloor \dfrac{s}{2} \right\rfloor$, we define
$\Lambda \seteq \ofw_0+\ofw_s$ and
\begin{align*}
& \lambda_{\ell,s}^n \seteq \sum_{k=n-\ell+1}^{n-1}(k-n+\ell)\alpha_k + \ell \sum_{i=0}^s \alpha_i + \sum_{j=s+1}^{\ell+s-1} (\ell-j+s)\alpha_j , \allowdisplaybreaks\\
& \mu_{u,s}^n \seteq \sum_{k=s-u+1}^{s-1}(k-s+u)\alpha_k + u \sum_{i=s}^{n-1} \alpha_i + \sum_{j=0}^{u-1} (u-j)\alpha_j.
\end{align*}

\begin{lemma} \cite[Theorem 1.4 (i)]{Ts}
For $V(\Lambda)$ over $A^{(1)}_{n-1}$,
$$ \mx^+(\Lambda|2) = \{ \Lambda \} \bigsqcup \left\{ \Lambda-\lambda_{\ell,s}^n \ | \ 1 \le \ell \le t  \seteq \left\lfloor \dfrac{n-s}{2} \right\rfloor \right \}
\bigsqcup \left\{ \Lambda-\mu_{u,s}^n \ | \ 1 \le u \le \left\lfloor \dfrac{s}{2} \right\rfloor \right\}.$$
\end{lemma}

The above lemma tells us that every element in $\mx^+(\Lambda|2)$ is essentially finite, since
\begin{align} \label{eq: evidence}
\ell+s < n-\ell+1 \quad \text{ and } \quad u < s-u+1.
\end{align}

Now we show that we obtain all the dominant weights of $L(\omega_t +\omega_{t+s})$ from $\mx^+(\Lambda|2)$. Since
$$J = [0,\ell+s-1] \, \bigsqcup \, [n-\ell+1,n-1] \seteq \Supp(\lambda_{\ell,s}^n) \subsetneq I \seteq [0,n-1]$$ and $\ell+s < n-\ell+1$ for all $\ell$,
we can choose $s+t$ as an extremal vertex (see \eqref{aff dyn}). Thus
$$ \Omega_1 \seteq  \{ \Lambda \} \bigsqcup \left\{ \Lambda-\lambda_{\ell,s}^n \ | \ 1 \le \ell \le t \right \}$$
can be considered as a subset of dominant maximal weights of $L(\omega_t+\omega_{t+s})$ over $A_{n-1}$
via the embedding
$$ [0,t+s-1] \sqcup [t+s+1,n] \hookrightarrow \{1,2,\ldots,n-1\} \text{ such that }  x \longmapsto a \equiv s+t-x \ ({\rm mod } \ n).$$
Hence $\Omega_1$ can be identified with
\begin{align} \label{eq: A omega 1}
\{ \omega_{t-r}+\omega_{t+s+r} \ | \ 0 \le r \le t\}
\end{align}
which is a subset of dominant weights of $L(\omega_t+\omega_{t+s})$. (Here we set $\omega_0 \seteq 0$.) By \cite[\S 13]{Hum}, $L(\omega_t+\omega_{t+s})$ has $(t+1)$-many dominant weights
and hence the set in \eqref{eq: A omega 1} indeed coincides with the set of dominant weights of $L(\omega_t+\omega_{t+s})$.

By a similar argument, the set
$$ \Omega_2 \seteq  \{ \Lambda \} \bigsqcup \left\{ \Lambda-\mu_{u,s}^n \ | \ 1 \le u \le t' \seteq \left\lfloor \dfrac{s}{2} \right\rfloor \right\} $$
can be identified with the dominant weights
\begin{align} \label{eq: A omega 2}
\{ \omega_{t'-r}+\omega_{n-s+t'-r} \ | \ 0 \le r \le t'  \}
\end{align}
of $L(\omega_{t'}+\omega_{n-s+t'})$ over $A_{n-1}$.

\subsection{Type $B_n^{(1)}$} Assume that $\g=B_n^{(1)}$.
If $\Lambda = \ofw_0 + \ofw_n$, one can check that there are only two maximal weights $\Lambda$ and $\ofw_1+\ofw_n -\updelta$, and their multiplicities are $1$ and $n$, respectively. When $\Lambda=\ofw_1+ \ofw_n$, the same is true with $\ofw_0$ replaced by $\ofw_1$.

Assume that $\Lambda$ is of level $2$, other than $\ofw_0 + \ofw_n$ and $\ofw_1 + \ofw_n$; that is,
\[\Lambda=(1+\delta_{i,0}+\delta_{i,n})\Lambda_i+\delta_{i,1}\ofw_0  = \begin{cases} 2\ofw_0 & \text{if }i=0, \\ \ofw_0+\ofw_1 & \text{if }i=1,  \\ 2 \ofw_n & \text{if }i =n , \\ \Lambda_i & \text{if }i \neq 0,1,n.  \end{cases} \]

Recall that \[\updelta=\alpha_0+\alpha_1+2(\alpha_2+\cdots +\alpha_{n}) \quad \text{ and } \quad c=h_0+h_1+2(h_2+\cdots+h_{n-1})+h_n,\] and we have
\begin{equation} \label{eq: Bn smax cond}
2\mathcal{C}_{{\rm af}} \cap (\overline{\Lambda}+\overline{\rl}) = \{ \lambda= \overline{\Lambda}+\sum_{i=1}^n m_i\alpha_i
 \ | \ \lambda(h_i)\ge 0 \ (1 \le i \le n), \ (\lambda|\theta) \le 2 \},
 \end{equation}
where $\theta=\alpha_1+2(\alpha_2+\cdots +\alpha_{n})$.

\begin{lemma} \label{lem: Bn1 HT smax}
Let $\Lambda =(\delta_{s,0}+\delta_{s,1})\ofw_0 + \Lambda_s$ $(0 \le s \le n-1)$. Then the following weights are in $\max^+(\Lambda|2)$, i.e., they are dominant maximal weights of $V(\Lambda):$
\begin{align}
& (1+\delta_{2u-1+s,n})\Lambda_{2u-1+s}-u\updelta  = \nonumber \\
& \hspace{10ex} \Lambda-\cont\left(\YW^{(n)*\uplambda(2u-2+s)}_{\ofw_0} \right)+\cont\left(\YW^{\uplambda(s-1)}_{\ofw_0} \right) \qquad \text{for } \quad 1+\delta_{s,0} \le u \le \lfloor (n-s+1)/2 \rfloor , \label{eq: Bn1 HT type smax ex} \\
& (1+\delta_{2u+s,n})  \Lambda_{2u+s}-u\updelta  = 
\Lambda - \cont\left(\YW^{\uplambda(2u-1+s)}_{\ofw_0}\right) + \cont\left(\YW^{\uplambda(s-1)}_{\ofw_0} \right)   \ \ \text{for} \ \ 0 \le u \le \lfloor (n-s)/2 \rfloor.
\label{eq: Bn1 HT type smax}
\end{align}
\end{lemma}

\begin{proof} The equalities in \eqref{eq: Bn1 HT type smax ex} and \eqref{eq: Bn1 HT type smax} can be checked by direct computations. In each equation in \eqref{eq: Bn1 HT type smax ex} and \eqref{eq: Bn1 HT type smax}, the RHS shows that the image of the weight under the orthogonal projection is in $\overline{\Lambda}+\overline{\rl}$, and the LHS shows that the image of the orthogonal projection belongs to $2\mathcal{C}_{{\rm af}}$. Thus the weights are in $\max^+(\Lambda|2)$ by Proposition \ref{prop: Kac bijection}.
\end{proof}

Let $\mathfrak g_n$ be the finite dimensional subalgebra of $\mathfrak g$, generated by $e_i, h_i, f_i$ for $i \in I_n :=I \setminus \{ n \}$, as in Section \ref{subsec-finite}. Then $\mathfrak g_n$ is of type $D_n$. For each dominant maximal weight $\mu=\Lambda - \sum_{i\in I} k_i \alpha_i$  in \eqref{eq: Bn1 HT type smax}, we have $k_n=0$ and so $\mu$ is essentially finite of type $D_n$.  Denote by $\omega$ the dominant integral weight  of $\mathfrak g_n$ corresponding to $\Lambda$ and consider the highest weight  module $L(\omega)$ of $\mathfrak g_n$ with highest weight $\omega$.

\begin{proposition} \label{prop-exhaust D2}
All the dominant weights of $L(\omega)$ over $\mathfrak g_n$ of type $D_n$ are obtained from
the weights in \eqref{eq: Bn1 HT type smax} through the correspondence $\Lambda - \sum_{i \in I_n} k_i \alpha_i  \mapsto  \omega - \sum_{i \in I_n} k_i \alpha_i$.
\end{proposition}

\begin{proof}
Since $$ n \not \in \Supp \left( \cont\left(\YW^{\uplambda(2u-1+s)}_{\ofw_0}\right)- \cont\left(\YW^{\uplambda(s-1)}_{\ofw_0} \right) \right),$$
we can take $n$ as an extremal vertex. Thus we can identify the weights in \eqref{eq: Bn1 HT type smax} with
$$  \left\{ \omega_{n-s-2k} \ | \ 1 \le k \le \lfloor (n-s)/2 \rfloor  \right\} \  \bigsqcup \ \{ (\delta_{s,0}+\delta_{s,1})\omega_{n}+\omega_{n-s} \},  $$
which is the subset of dominant weights of
$V((\delta_{s,0}+\delta_{s,1})\omega_{n}+\omega_{n-s})$ over $\mathfrak g_n$ via the embedding $$ I_n=[0,n-1] \to [1,n] \quad \text{such that} \quad i \longmapsto n-i.$$

By \cite[Lemma 2.6]{Ko87}, $V((\delta_{s,0}+\delta_{s,1})\omega_{n}+\omega_{n-s})$ has $(\lfloor (n-s)/2 \rfloor+1)$-many dominant weights
and hence the weights in \eqref{eq: Bn1 HT type smax} coincides with the set of dominant weight of $V((\delta_{s,0}+\delta_{s,1})\omega_{n}+\omega_{n-s})$.
\end{proof}

\begin{example} For $\g=B_9^{(1)}$ and $\Lambda=\Lambda_{3}$, the dominant maximal weight $\Lambda_7-2\updelta \in \max^+(\Lambda|2)$ can be written as follows:
\begin{align*}
\Lambda_7-2\updelta & = \Lambda-\left\{ 3\al_0+3\al_1 +\sum_{i=2}^{6} (7-i)\alpha_i \right\} + \left\{ \al_0+\al_1+\al_2 \right\} \allowdisplaybreaks\\
&= \Lambda-\cont\left( {\xy (0,0)*++{\hwsix}\endxy}\right)+\cont\left( {\xy (0,0)*++{\hwtwo}\endxy}\right)
\end{align*}
\end{example}

Define  $\YW^{\uplambda_\epsilon(n)}_{\ofw_n}$ $(\epsilon=0,1)$ to be the Young wall determined by the staircase partition $\uplambda (n)$ whose top of the first column is the half-thickness block with color $\epsilon$.

\begin{example}
The $\YW^{\uplambda_0(n)}_{\ofw_n}$ and $\YW^{\uplambda_1(n)}_{\ofw_n}$ for $B_3^{(1)}$ are given as follows:
$$
\YW^{\uplambda_0(3)}_{\ofw_3} = \resizebox{1.8cm}{1cm}{\xy (0,0)*++{\xy
(0,0)*{}="B1";(6,0)*{}="B2";
(0,6)*{}="T1";(6,6)*{}="T2";
"T1"; "B1" **\dir{-};"T2"; "B2" **\dir{-};"T2"; "T1" **\dir{-};"B2"; "B1" **\dir{-};
"B2"+(0,0.5)*{};"B2"+(-18,0.5)*{} **\dir{.}; "B2"+(0,1)*{};"B2"+(-18,1)*{} **\dir{.};"B2"+(0,1.5)*{};"B2"+(-18,1.5)*{} **\dir{.};
"B2"+(0,2)*{};"B2"+(-18,2)*{} **\dir{.};"B2"+(0,2.5)*{};"B2"+(-18,2.5)*{} **\dir{.};
"T1"+(-6,0); "T2"+(-6,0) **\dir{-};"T1"+(-6,0); "B1"+(-6,0) **\dir{-};"B2"+(-6,0); "B1"+(-6,0) **\dir{-};
"B2"+(0,3);"B2"+(-18,3)**\dir{-};
"T1"+(-12,0); "T2"+(-12,0) **\dir{-};"T1"+(-12,0); "B1"+(-12,0) **\dir{-};"B2"+(-12,0); "B1"+(-12,0) **\dir{-};
"T1"+(0,6); "B1"+(0,6) **\dir{-};"T2"+(0,6); "B2"+(0,6) **\dir{-};"T2"+(0,6); "T1"+(0,6) **\dir{-};
"T2"+(0,12); "B2"+(0,12) **\dir{-};"T2"+(0,12); "B1"+(0,12) **\dir{-};
"T1"+(-6,6); "T2"+(-6,6) **\dir{-};"T1"+(-6,6); "B1"+(-6,6) **\dir{-};
(3,1.5)*{3};(3,4.5)*{3};  (3,9)*{2}; (4.5,14)*{0};
(-3,1.5)*{3};(-3,4.5)*{3};  (-3,9)*{2}; 
(-9,1.5)*{3};(-9,4.5)*{3};
\endxy} \endxy}
\quad \text{and} \quad
\YW^{\uplambda_1(3)}_{\ofw_3}=\resizebox{1.8cm}{1cm}{\xy (0,0)*++{
\xy
(0,0)*{}="B1";(6,0)*{}="B2";
(0,6)*{}="T1";(6,6)*{}="T2";
"T1"; "B1" **\dir{-};"T2"; "B2" **\dir{-};"T2"; "T1" **\dir{-};"B2"; "B1" **\dir{-};
"B2"+(0,0.5)*{};"B2"+(-18,0.5)*{} **\dir{.}; "B2"+(0,1)*{};"B2"+(-18,1)*{} **\dir{.};"B2"+(0,1.5)*{};"B2"+(-18,1.5)*{} **\dir{.};
"B2"+(0,2)*{};"B2"+(-18,2)*{} **\dir{.};"B2"+(0,2.5)*{};"B2"+(-18,2.5)*{} **\dir{.};
"T1"+(-6,0); "T2"+(-6,0) **\dir{-};"T1"+(-6,0); "B1"+(-6,0) **\dir{-};"B2"+(-6,0); "B1"+(-6,0) **\dir{-};
"B2"+(0,3);"B2"+(-18,3)**\dir{-};
"T1"+(-12,0); "T2"+(-12,0) **\dir{-};"T1"+(-12,0); "B1"+(-12,0) **\dir{-};"B2"+(-12,0); "B1"+(-12,0) **\dir{-};
"T1"+(0,6); "B1"+(0,6) **\dir{-};"T2"+(0,6); "B2"+(0,6) **\dir{-};"T2"+(0,6); "T1"+(0,6) **\dir{-};
"T1"+(0,12); "B1"+(0,12) **\dir{-};"T1"+(0,12); "T2"+(0,12) **\dir{-};"B1"+(0,12); "T2"+(0,12) **\dir{-};
"T1"+(-6,6); "T2"+(-6,6) **\dir{-};"T1"+(-6,6); "B1"+(-6,6) **\dir{-};
(3,1.5)*{3};(3,4.5)*{3};  (3,9)*{2}; (1.5,16)*{1};
(-3,1.5)*{3};(-3,4.5)*{3};  (-3,9)*{2}; 
(-9,1.5)*{3};(-9,4.5)*{3};
\endxy} \endxy}.
$$
\end{example}

\begin{lemma} \label{lem: Bn1 HH smax}
Let $\Lambda=(1+\delta_{s,n})\Lambda_s+\delta_{s,1}\ofw_0$ $(1 \le s \le n)$. Then the following weights are in $\max^+(\Lambda|2)$:
\begin{align}
(1+\delta_{u,n})\Lambda_{u} & =\Lambda- \cont\left(\YW^{\uplambda(n-u)}_{\ofw_n} \right)
+ \cont\left(\YW^{\uplambda(n-s)}_{\ofw_n} \right) \ (2 \le u \le s),  \label{eq: Bn1 HH type smax 1}\\
\ofw_0+\ofw_1 & = \Lambda- \cont\left(\YW^{\uplambda(n-1)}_{\ofw_n} \right) + \cont\left(\YW^{\uplambda(n-s)}_{\ofw_n} \right), \label{eq: Bn1 HH type smax 2}\\
2\ofw_1 -\updelta & = \Lambda- \cont\left(\YW^{\uplambda_0(n)}_{\ofw_n} \right) + \cont\left(\YW^{\uplambda(n-s)}_{\ofw_n} \right), \label{eq: Bn1 HH type smax 3}\\
2\ofw_0 & = \Lambda- \cont\left(\YW^{\uplambda_1(n)}_{\ofw_n} \right) + \cont\left(\YW^{\uplambda(n-s)}_{\ofw_n} \right). \label{eq: Bn1 HH type smax 4} 
\end{align}
For $\Lambda=2\ofw_0$, we have
\begin{equation} \label{eq: Bn1 HT type smax 2}
2\ofw_1-2\updelta =2\ofw_0-2(\alpha_0+\sum_{i=2}^n \alpha_i) \in \mx^+(\Lambda|2).
\end{equation}
\end{lemma}

\begin{proof}
One can use the same argument as in Lemma \ref{lem: Bn1 HT smax}.
\end{proof}

Let $\mathfrak g_1$ (resp. $\g_0$)
be the finite dimensional subalgebra of $\mathfrak g$, generated by $e_i, h_i, f_i$ for $i \in I_1 :=I \setminus \{ 1 \}$ (resp. $i \in I_1 :=I \setminus \{ 0 \}$).
Then $\mathfrak g_1$ (resp. $\g_0$) is of type $B_n$.
One can see that each dominant maximal weight $\mu=\Lambda - \sum_{i\in I} k_i \alpha_i$  in Lemma \ref{lem: Bn1 HH smax}
is essentially finite of type $B_n$.

\begin{proposition} \label{prop-exhaust B2}
All the dominant weights of $L(\omega)$  over $B_n$  are obtained from
the weights in Lemma \ref{lem: Bn1 HH smax} through the correspondence $\Lambda - \sum_{i \in I_n} k_i \alpha_i  \mapsto  \omega - \sum_{i \in I_n} k_i \alpha_i$.
\end{proposition}

\begin{proof}
One can easily check that $0$ (resp. $1$) does not appears as an element of support for weights in
\eqref{eq: Bn1 HH type smax 1}, \eqref{eq: Bn1 HH type smax 2} and \eqref{eq: Bn1 HH type smax 4}
(resp. \eqref{eq: Bn1 HH type smax 1}, \eqref{eq: Bn1 HH type smax 2}, \eqref{eq: Bn1 HH type smax 3} and \eqref{eq: Bn1 HT type smax 2}). Hence we can take
$0$ (resp. $1$) as an extremal vertex.
Thus we can identify the weights in \eqref{eq: Bn1 HH type smax 1}, \eqref{eq: Bn1 HH type smax 2} and \eqref{eq: Bn1 HH type smax 4}
(resp. \eqref{eq: Bn1 HH type smax 1}, \eqref{eq: Bn1 HH type smax 2}, \eqref{eq: Bn1 HH type smax 3} and \eqref{eq: Bn1 HT type smax 2}) with
$$ \begin{cases} \left\{ \omega_{k} \ | \ 0 \le k \le s \right\}, & \text{ if } s \ne 0, \\
\left\{ 2\omega_{1},\omega_1,\omega_0 \right\} & \text{ if } s = 0, \end{cases}$$
which is the subset of dominant weights of
$L((\delta_{s,n})\omega_{n}+(1+\delta_{s,0})\omega_{s+\delta_{s,0}})$ over $B_n$ via the natural embedding
$$ I_0=[1,n] \to [1,n] \quad (\text{resp.} \ [0] \sqcup [2,n] \to [1,n]).$$

By \cite[Lemma 2.4]{Ko87}, $L((\delta_{s,n})\omega_{n}+(1+\delta_{s,0})\omega_{s+\delta_{s,0}})$ has $(s+1+\delta_{s,0})$-many dominant weights
and hence the weights in  \eqref{eq: Bn1 HH type smax 1}, \eqref{eq: Bn1 HH type smax 2} and \eqref{eq: Bn1 HH type smax 4}
(resp. \eqref{eq: Bn1 HH type smax 1}, \eqref{eq: Bn1 HH type smax 2}, \eqref{eq: Bn1 HH type smax 3} and \eqref{eq: Bn1 HT type smax 2})
coincides with the set of dominant weight of $L((\delta_{s,n})\omega_{n}+(1+\delta_{s,0})\omega_{s+\delta_{s,0}})$.
\end{proof}

Let ${\rm max}_{i}^+(\Lambda|2)$ be the set of the dominant maximal weights in Lemma \ref{lem: Bn1 HT smax} and
${\rm max}_{ii}^+(\Lambda|2)$ be the set of those in Lemma \ref{lem: Bn1 HH smax}.
Combining these two sets, we obtain the whole set of dominant maximal weights as stated in the following theorem.

\begin{theorem} \label{thm: B level 2}
Assume that $\g=B_n^{(1)}$ and $\Lambda=(\delta_{s,0}+\delta_{s,1})\ofw_0 + \delta_{s,n}\ofw_n+\Lambda_s$ $(0 \le s \le n)$  is of level $2$. Then we have the  union
$${\rm max}^+(\Lambda|2) = {\rm max}_{i}^+(\Lambda|2) \bigsqcup  {\rm max}_{ii}^+(\Lambda|2),$$
and the number of elements in ${\rm max}^+(\Lambda|2)$  is equal to $n+2$, since
$$|{\rm max}_{i}^+(\Lambda|2)|=n-s \quad  \text{ and } \quad |{\rm max}_{ii}^+(\Lambda|2)|=s+2.$$
\end{theorem}

Before we begin the proof of Theorem \ref{thm: B level 2}, we make some preparation. Recall that for a statement $P$, the number $\delta(P)$ is equal to 1 if $P$ is true and 0 if $P$ is false. Sometimes, we will write $\delta_P$  for $\delta(P)$.

Now we consider the conditions on  ${\rm max}^+(\Lambda|2)$ for $\Lambda=(\delta_{s,0}+\delta_{s,1})\ofw_0 + \delta_{s,n}\ofw_n+\Lambda_s$ $(0 \le s \le n)$.
For $\eta = \displaystyle \sum_{i=1}^n x_i \oal_i \in 2 \mathcal{C}_{{\rm af}} \cap (\overline{\Lambda}+\overline{\rl})$ such that $\eta \ne 0$, the condition \eqref{eq: Bn smax cond} tells us that
\begin{itemize}
\item[{\rm (1)}] $\eta(h_1)=2x_1-x_2 \ge -\delta_{1,s}$,
\item[{\rm (i)}] $\eta(h_i)=-x_{i-1}+2x_i-x_{i+1} \ge -\delta_{i,s}$ $(2 \le i \le n-1)$,
\item[{\rm (n)}] $\eta(h_n)= -2x_{n-1}+2x_n \ge -2\delta_{n,s}$,
\end{itemize}
and
$$(\eta|\theta)=x_2+(2-\delta_{s,1}-2\delta_{s,0}) \le 2.$$
Then by summing inequalities {\rm (2)}$\sim${\rm (n$-1$)} and $\frac{1}{2} \times {\rm (n)}$, we have
\begin{align} \label{eq: ineq}
-x_1+x_2 \ge -\delta(s > 1).
\end{align}

We have also that
\begin{enumerate}
\item[{\rm (a)}] for $s \le i \le n-1$,
$$ x_{i+1} \ge x_i \ \text{ and } \ x_i=x_{i+1} \ \text{ implies } \  x_i=x_{i+1}=x_{i+2}=\cdots = x_n;$$
\item[{\rm (b)}] for $1 \le i \le s-1$,
$$ -x_i+ x_{i+1} \ge -\delta(1 \le i< s);$$
\item[{\rm (c)}] for all $2 \le i \le n$,
$$ x_1+x_i \ge x_{i+1}- \delta(i\ge s) \times \delta(s \ge 1).$$
\end{enumerate}

With the inequality {\rm (1)}, the inequality \eqref{eq: ineq} implies that
$$  x_1 \ge -\delta(s \ge 1) \quad \text{ and } \quad  x_2 \ge -2\delta(s \ge 1).$$

\begin{proof}[Proof of Theorem \ref{thm: B level 2}]

(a) Assume that $\Lambda=2\ofw_0$. Then we have the inequalities
$$   0 \le x_1 \le x_2 \le 2 \quad \text{ and } \quad 2x_1 -x_2 \ge 0.$$
Then $(x_1,x_2)=(0,0)$, $(1,1)$, $(1,2)$ and $(2,2)$. Now one can prove that, for $\eta = \sum_{i=1}^n x_i \oal_i \in 2 \mathcal{C}_{{\rm af}} \cap \overline{\rl}$ such that $\eta \ne 0$, we have
$$
\eta = \begin{cases} \displaystyle\sum_{i=1}^u i \ \oal_i + u \sum_{t=u}^n \oal_t  & \text{ for some }  1 \le u \le n, \text{ or } \\
2\displaystyle\sum_{i=1}^n\oal_i=\overline{2\ofw_1-2\updelta}.
\end{cases}
$$
Here $\{ \sum_{i=1}^u i \ \oal_i + u \sum_{t=u}^n \oal_t \}$ contributes to \eqref{eq: Bn1 HT type smax ex} and \eqref{eq: Bn1 HT type smax}. \\

(b) Assume that $\Lambda=\ofw_0+\ofw_1$. Then we have the inequalities
$$  x_1 \ge -1, \ 1 \ge x_2 \ge -2, \ 2x_1 -x_2 \ge -1, \ -x_1 +2x_2-x_3 \ge 0 \text{ and } x_n \ge \cdots \ge x_2 \ge x_1.  $$
Then $(x_1,x_2)=(0,0)$, $(0,1)$, $(1,1)$ and $(-1,-1)$.
Now one can prove that, for $\eta = \overline{\ofw_1}+\sum_{i=1}^n x_i \oal_i \in 2 \mathcal{C}_{{\rm af}} \cap \overline{\rl}$ such that $\eta \ne \overline{\ofw_1}$, we have
$$
\eta = \begin{cases} \displaystyle\sum_{i=1}^u (i-1) \ \oal_i + (u-1) \sum_{t=u}^n \oal_t  & \text{ for }  2 \le u \le n, \\
 \displaystyle\sum_{i=1}^n\oal_i =\overline{2\ofw_1-\updelta}  \text{ or } -\displaystyle\sum_{i=1}^n\oal_i=\overline{2\ofw_0}.
\end{cases}
$$
Here $\{ \sum_{i=1}^u (i-1) \ \oal_i + (u-1) \sum_{t=u}^n \oal_t \}$ contributes to \eqref{eq: Bn1 HT type smax ex} and \eqref{eq: Bn1 HT type smax}. \\

(c) Assume that $\Lambda=\Lambda_s$ $(2\le s \le n-1)$ or $2\ofw_n$.
Then we have inequalities
\begin{align*}
& x_1 \ge -1, \quad 0 \ge x_2 \ge-2, \quad -x_1+x_2 \ge -1, \quad  x_{n} \ge x_{n-1} \ge \cdots \ge x_{s+1} \ge x_s,  \\
& -x_{i-1} + 2x_i-x_{i+1} \ge 0 \ \  \text{ for } i < s, \quad  2x_1-x_2 \ge 0, \\
& x_1 + x_i \ge x_{i+1} \text{ for } i \le s \quad \text{ and } \quad x_1 + x_i \ge x_{i+1}-1 \text{ for } i > s.
\end{align*}
Then $(x_1,x_2)=(0,0)$, $(1,0)$, $(-1,-2)$ and $(0,-1)$.

(1) Assume $x_1=0$. Then, for $2 \le i \le s-1$,
we have
\begin{align}\label{eq: ineq z}
 x_i \ge x_{i+1} \ge x_i-1.
\end{align}

(1-1) If there exists $1 \le u \le s-1$ such that $x_{i+1} = x_i-1$, take $t$ the smallest one; that is $x_{t+1}=-1$. Since
\begin{align}\label{eq: ineq p}
-x_{t}-2x_{t+1}-x_{t+2} \ge 0,
\end{align}
the inequality \eqref{eq: ineq z} implies $x_{t+2}=-2$. Repeating this process, we obtain $x_{h+1}=x_{h}-1$ for $t \le h \le s-1$.
Since
\begin{align*}
-x_{s-1}-2x_{s}-x_{s+1} \ge -1,
\end{align*}
we have $x_s=x_{s+1}$ and hence $x_s=x_{s+1}=\cdots=x_n$. Thus $\eta$ is of the following form
$$ \sum_{i=t}^s (i-t+1) \ \oal_i + (s-t+1) \sum_{t=s}^n \oal_t  \quad  \text{ for }  1 \le t \le s-1,$$
which contributes to \eqref{eq: Bn1 HH type smax 1} and \eqref{eq: Bn1 HH type smax 2}.

(1-2) Now we assume that $x_1=x_2=\cdots=x_s=0$. Then we have, for $u > s$,
$$ x_{u-1} \le x_u \le x_{u-1}+1.$$
Then, by applying the same method as in {\rm (a)}, we see that $\eta$ is of the following form:
$$ \sum_{i=s+1}^u (i-s) \ \oal_i + (u-s) \sum_{t=u}^n \oal_t  \quad  \text{ for }  s+1 \le u \le n$$
which contributes to \eqref{eq: Bn1 HT type smax ex} and \eqref{eq: Bn1 HT type smax}.

(2) Assume $(x_1,x_2)=(1,0)$. As in (1-1), we can conclude that
$$\eta = \oal_1 - \sum_{i=3}^s (i-2)\oal_i - (s-2)\sum_{j=s+1}^n \oal_j=\overline{2\ofw_1 -\updelta}.$$

(3) Assume $(x_1,x_2)=(-1,-2)$. As in (1-1), we can conclude that
\[
\eta = - \sum_{i=1}^s i \ \oal_i - s \sum_{j=s+1}^n \oal_j=\overline{2\ofw_0}.  \qedhere
\]
\end{proof}

\begin{definition} \label{def-sma-2}
Assume that $\eta \in {\rm max}^+(\Lambda|2)$ is of the form
\[ \eta = \Lambda- \cont\left(\YW^{\uplambda(m)}_{\ofw} \right) + \cont\left(\YW^{\uplambda(s)}_{\ofw} \right)
\quad \text{ or } \quad \Lambda- \cont\left(\YW^{(n)*\lambda(m-1)}_{\ofw} \right) + \cont\left(\YW^{\uplambda(s)}_{\ofw} \right)
, \] where $s\ge 0$ if $\Lambda
\neq 2 \ofw_0$ and $s=-1$ if $\Lambda = 2 \ofw_0$. (See Remark \ref{rem-oo} below.)
Then we define the {\em index} of the maximal weight $\eta$ to be $(m,s)$.
Similarly, if $\eta \in {\rm max}^+(\Lambda|2)$ is of the form
\[ \eta = \Lambda- \cont\left(\YW^{\uplambda_\epsilon(n)}_{\ofw} \right) + \cont\left(\YW^{\uplambda(s)}_{\ofw} \right), \quad \epsilon=0,1, \]
then define the {\em index} of the maximal weight $\eta$ to be $(n,s)$.
\end{definition}

\begin{remark} \label{rem-oo}
Though we have $\uplambda(0)=\uplambda(-1)=(0)$, we use $\uplambda(-1)$ when $\Lambda = 2\ofw_0$.
\end{remark}

Now we consider $\Uplambda$ of level $\ge 3$.
The following lemma is useful:

\begin{lemma} \label{lem-useful}
For any $\Lambda',\Lambda'' \in P^+$ with $\Lambda'(c)=k$ and $\Lambda''(c)=k'$, we have
$$\Lambda'' + {\rm max}^+(\Lambda'|k) \subset{\rm max}^+(\Lambda''+\Lambda'|k+k').$$
\end{lemma}

\begin{proof}
The assertion follows from Proposition \ref{prop: Kac bijection}.
\end{proof}

In the following lemma, we obtain maximal weights of level $3$ that do not come from those of level $2$.
\begin{lemma} \label{lem-level3B}
Let $\Uplambda =(1 + \delta_{s,0}+\delta_{s,1})\ofw_0 + \Lambda_s$ $(0 \le s \le n-1)$. Then the following weights are in $\max^+(\Lambda|3)$:
\begin{align} \label{eq: Bn1 HT type smax 3}
&\ofw_1+(1+\delta_{2u+s,n})\Lambda_{2u+s}-(u+1)\updelta  =     \Uplambda- \cont\left(\YW^{(n)*\uplambda(2u-1+s)}_{\ofw_1} \right)+(\alpha_1-\alpha_0)\\
 & \hskip 5 cm +\cont\left(\YW^{\uplambda(s-1)}_{\ofw_0} \right)  \text{ for }\ \delta_{s,0} + \delta_{s,1} \le u \le \lfloor (n-s)/2 \rfloor , \nonumber \allowdisplaybreaks\\
& \ofw_1+(1+\delta_{2u+1+s,n})  \Lambda_{2u+1+s}-(u+1)\updelta  = \Uplambda - \cont\left(\YW^{\uplambda(2u+s)}_{\ofw_1}\right)+ (\alpha_1-\alpha_0)  \label{eq: Bn1 HT type smax 4} \\
 &  \hskip 5 cm  +  \cont\left(\YW^{\uplambda(s-1)}_{\ofw_0} \right) \text{ for } \ \delta_{s,0} \le u \le \lfloor (n-1-s)/2 \rfloor, \nonumber \allowdisplaybreaks\\
\label{eq: Bn1 HH type smax 5}
& 3\ofw_1 -(2+\delta_{0,s}) \updelta =  \Uplambda-
\begin{cases} \left( 3\alpha_0+3 \displaystyle\sum_{i=2}^n \alpha_i \right) &\text{ if } s=0, \\
\left(2\alpha_0+2\displaystyle\sum_{i=2}^n \alpha_i \right) &\text{ if } s=1, \\[2ex]
\left( 2\alpha_0+ \displaystyle\sum_{i=2}^s(i+1)\alpha_i + (s+1) \sum_{j=s+1}^n \alpha_j \right) & \text{ if } 2\le s\le n-1,
\end{cases}  \allowdisplaybreaks\\
& \ofw_1+\Lambda_u -\updelta = \Uplambda- \left( \sum_{i=0}^u\alpha_i + \sum_{j=u+1}^s(j+1-u)\alpha_j+ (s+1-u)\sum_{t=s+1}^n \alpha_t  \right)  \quad (2 \le u \le s-1) .
\end{align}
\end{lemma}

\begin{proof}
The equalities  can be checked through direct computations. Then, as in the proof of Lemma \ref{lem: Bn1 HT smax}, we use Proposition \ref{prop: Kac bijection} to show that the weights are dominant maximal.
\end{proof}

We denote the set of weights in Lemma \ref{lem-level3B} by ${\rm max}^+_{iii}(\Uplambda|3)$.
By Lemma \ref{lem-useful}, we also have
\[
\ofw_0+{\rm max}^+(\Lambda|2) \subset {\rm max}^+(\ofw_0+\Lambda|3) \quad \text{ and } \quad \ofw_n+{\rm max}^+(\Lambda|2) \subset {\rm max}^+(\ofw_n+\Lambda|3),
\]
where $\Lambda$ is of level $2$.

\begin{theorem} \label{thm: B level 3}
We have
\begin{align*} {\rm max}^+(\ofw_0+\Lambda|3) &= (\ofw_0 + {\rm max}^+(\Lambda|2)) \ \bigsqcup \ {\rm max}^+_{iii}(\ofw_0+\Lambda|3)
\end{align*}
for $\Lambda = (\delta_{s,0}+\delta_{s,1})\ofw_0 + \Lambda_s$ $(0 \le s \le n-1)$, and
\begin{align*}
{\rm max}^+(\ofw_n+\Lambda|3) &= \ofw_n + {\rm max}^+(\Lambda|2)
\end{align*}
for $\Lambda=(1+\delta_{s,n})\Lambda_s+\delta_{s,1}\ofw_0$ $(1 \le s \le n)$.
In particular, the number of elements in ${\rm max}^+(\ofw_0+\Lambda|3)$ is equal to $2(n+1)$, and the number of elements in ${\rm max}^+(\ofw_n+\Lambda|3)$ is equal to $n+2$.
\end{theorem}

\begin{proof}
One can prove by applying a similar argument to that of the proof of Theorem \ref{thm: B level 2}.
\end{proof}

\begin{proposition} \label{prop-exhaust B3}
For $\Lambda \seteq (1+\delta_{s,n})\Lambda_s+\delta_{s,1}\ofw_0$ $(1 \le s \le n)$,
 the set  $\ofw_n+{\rm max}_{ii}^+(\Lambda|2)$ of  dominant maximal weights corresponds to
the set of dominant weights of $L((1+\delta_{s,n})\omega_{n}+\omega_{s})$ over $B_n$.
\end{proposition}

\begin{proof}
As in Proposition \ref{prop-exhaust B2}, one can show that the set $\ofw_n+{\rm max}_{ii}^+(\Lambda|2)$
corresponds to
$$  \left\{ \omega_n+\omega_{k} \ | \ 1 \le k \le s \right\} \bigsqcup \{ (1+\delta_{s,n})\omega_{n}+\omega_{s}\},  $$
which is a subset of dominant weights of
$L((1+\delta_{s,n})\omega_{n}+\omega_{s})$ over $B_n$.
By \cite[Lemma 2.4]{Ko87}, $L((1+\delta_{s,n})\omega_{n}+\omega_{s})$ has $(s+1)$-many dominant weights and hence our assertion follows.
\end{proof}

Define
\begin{align} \label{eq: bmega D}
\tilde{\omega}_s = \begin{cases}
\omega_s  & \text{ if } 1 \le s <n-1  \\
\omega_{n-1}+\omega_n & \text{ if } s=n-1, \\
2\omega_n & \text{ if } s=n.
\end{cases}
\end{align}

\begin{proposition} \label{prop-exhaust D3}
Let $\mathfrak a$ be  the set of dominant weights in \eqref{eq: Bn1 HH type smax 5} and $\mathfrak b$ those in \eqref{eq: Bn1 HT type smax}. Then the union of $\mathfrak a$ and $\ofw_0+\mathfrak b$ corresponds to
the set of dominant weights of $L(\omega)$ over $D_n$, where $\omega \seteq \omega_{n}+\tilde{\omega}_{n-s}$ for $ 0 \le s \le n-1$.
\end{proposition}

\begin{proof} Clearly, the sets $\mathfrak a$ and $\ofw_0+ \mathfrak b$ are disjoint.
As in Proposition \ref{prop-exhaust D2},  one can show that the union of $\mathfrak a$  and $\ofw_0+ \mathfrak b$ corresponds to
\begin{align}\label{eq: Riordan multi}
\begin{cases} \{ \tilde{\omega}_{s-i}+\omega_{n-\delta_i} \ | \ i =0,1, \ldots, s \} & \text{ if } s \le n-1,  \\
\{ \tilde{\omega}_{n-i}+\omega_{n-\delta_i} \ | \ i =0,2,3 \ldots, s \} & \text{ if } s =n, \end{cases}
\end{align}
which is a subset of dominant weights of $L(\omega)$.
Here $\tilde{\omega}_0$ is to be understood as  $0$ and $\delta_i =1$ if $i$ is an odd integer and $\delta_i =0$ otherwise.
By \cite[Lemma 2.6]{Ko87}, $L(\omega)$ over $D_{n}$
has $(n-s+\delta_{s \ne 0})$-many dominant weights and hence our assertion follows.
\end{proof}

\begin{definition} \label{def-sma-3}
Assume that $\eta \in {\rm max}^+(\ofw+\Lambda|3)$, and set $\Uplambda= \ofw+\Lambda$.

(1) If $\eta = \ofw + \mu$ with $\mu \in {\rm max}^+(\Lambda|2)$ of index $(m,s)$,
then we define the {\em index} of $\eta$ to be $(m,s)$.

(2) Assume that $\eta$ is of the form
\begin{align*} \eta & = \Uplambda- \cont\left(\YW^{\uplambda(m)}_{\ofw} \right)+(\alpha_1-\alpha_0)  +\cont\left(\YW^{\uplambda(s)}_{\ofw'} \right) \\ & \phantom{LLLLL} \text{or } \quad  \Uplambda - \cont\left(\YW^{(n)*\uplambda(m-1)}_{\ofw} \right)+(\alpha_1-\alpha_0)  +\cont\left(\YW^{\uplambda(s)}_{\ofw'} \right), \end{align*}
where $s\ge 0$ if $\Uplambda
\neq 3 \ofw_0$ and $s=-1$ if $\Uplambda = 3 \ofw_0$. (cf. Remark \ref{rem-oo}) Then we
define the {\em index} of the maximal weight $\eta$ to be $(m,s)$.
\end{definition}

We generalize Definition \ref{def-sma-3} to higher levels.

\begin{definition} \label{def-sma-h}
Assume that $\eta \in (k-1)\ofw+{\rm max}^+(\ofw+\Lambda|3)$ for $k \ge 1$, and write $\eta=(k-1) \ofw +\mu$ with $\mu \in {\rm max}^+(\ofw+\Lambda|3)$.  If $\mu$ is of index $(m,s)$, then we define the {\it index} of $\eta$ to be $(m,s)$.
\end{definition}

 Whenever the index is defined for a maximal weight $\eta \in {\rm max}^+(k\ofw+\Lambda|k+2)$, $k \ge 0$, the weight $\eta$ will be called a {\em staircase dominant maximal weight}. The set of staircase dominant maximal weights will be denoted by $\smax(k\ofw+\Lambda|k+2)$.

\medskip

We close this subsection with a conjecture on the number of the dominant maximal weights.
\begin{conjecture}
Assume that $\g= B_n^{(1)}$, and let $\ell \ge 2$.
\begin{enumerate}
\item The number of elements in ${\rm max}^+((\ell-2)\ofw_0 + \Lambda|\ell)$ is equal to
\[ \binom{n+ \lfloor \ell/2 \rfloor}{\lfloor \ell/2 \rfloor}+ \binom{n+ \lfloor (\ell-1)/2 \rfloor}{\lfloor (\ell-1)/2 \rfloor}. \]
\item The number of elements in ${\rm max}^+((\ell-2)\ofw_n+ \Lambda|\ell)$ is equal to
\[ \binom{n+ \lfloor \ell/2 \rfloor}{\lfloor \ell/2 \rfloor}+ \binom{n+ \lfloor \ell/2 \rfloor -1}{\lfloor \ell/2 \rfloor -1}. \]
\end{enumerate}
\end{conjecture}

\subsection{Type $C_n^{(1)}$} \label{subsec-Cn1}
Unlike other affine types, the set $\mx^+(\ofw_s|1)$ is not trivial for any fundamental weight $\ofw_s$ of type $C_n^{(1)}$, $0 \le s \le n$.

\medskip
For $0 \le s \le n$, we define
\begin{align}
 \zeta_{\ell,s}^n = &  \ell\alpha_0+2\ell \sum_{i=1}^{s} \alpha_i + \sum_{j=1}^{2\ell-1}(2\ell-j)\alpha_{s+j}  \quad (1 \le \ell \le \lfloor (n-s)/2 \rfloor), \label{eq: C zeta}\\
\xi_{u,s}^n = & \sum_{i=1}^{2u} i  \alpha_{s-2u+i} + 2u \sum_{j=1}^{n-s-1}\alpha_{s+j} + u\alpha_n  \quad (1 \le u \le \lfloor s/2 \rfloor).   \label{eq: C xi}
\end{align}

Using a similar argument to that of the proof of Theorem \ref{thm: B level 2}, one can prove the following theorem:

\begin{theorem} \label{thm: C level 1}
For $0 \le s \le n$, we have
$$ \mx^+(\ofw_s|1) = \{ \ofw_s \} \bigsqcup \left\{ \ofw_{s} - \zeta_{\ell,s}^n \ | \ 1 \le \ell \le \lfloor (n-s)/2 \rfloor \right\} \bigsqcup
\left\{ \ofw_{s} -\xi_{u,s}^n  \ | \  1 \le u \le \lfloor s/2 \rfloor  \right\}.$$
\end{theorem}

Now we show that every element in $\mx^+(\ofw_s|1)$ is essentially finite.  Since
$$\Supp(\zeta_{\ell,s}^p)=[0,2\ell-1+s] \subsetneq [0,n],$$
we can choose $n$ as an extremal vertex. Then the set
$$ \Omega_1 \seteq  \{ \ofw_s \} \bigsqcup \left\{ \ofw_s -\zeta_{\ell,s}^n \ | \ 1 \le \ell \le \lfloor (n-s)/2 \rfloor \right \}$$
can be considered as a subset of dominant maximal weights of $L(\omega_{n-s})$ over $C_n$
via the embedding
$$ [0,n-1] \hookrightarrow [1,n] \quad \text{ given by }  i \mapsto n-i.$$
Hence $\Omega_1$ can be identified with
\begin{align} \label{eq: C omega 1}
\{ \omega_{n-s-2k} \ | \ 0 \le k \le \lfloor (n-s)/2 \rfloor \}
\end{align}
which is a subset of dominant weights of $L(\omega_{n-s})$ (Here we set $\omega_0 \seteq 0$). By \cite[Lemma 2.5]{Ko87}, $L(\omega_{n-s})$  has $(\lfloor (n-s)/2 \rfloor+1)$-many dominant weights
and the set in \eqref{eq: C omega 1} coincides with the set of dominant weights of $L(\omega_{n-s})$ indeed.

In a similar way, the set
$$ \Omega_2 \seteq  \{ \ofw_s \} \bigsqcup \left\{ \ofw_{s} -\xi_{u,s}^n  \ | \  1 \le u \le \lfloor s/2 \rfloor \right\} $$
can be identified with the set of dominant weights
\begin{align} \label{eq: C omega 2}
\{ \omega_{s+1-2k} \ | \ 0 \le k \le \lfloor (s+1)/2 \rfloor  \}
\end{align}
of $L(\omega_{s+1})$ over $C_n$.

\subsection{Type $D_{n}^{(1)}$} Recall that the affine type $D_{n}^{(1)}$ has fundamental weights  $\ofw_0,\ofw_1,\ofw_{n-1},\ofw_n$ of level $1$. Since $(\ofw_0,\ofw_1)$ and $(\ofw_{n-1},\ofw_n)$ are symmetric, we only consider the case when
$$ \Lambda= (\delta_{s,0}+\delta_{s,1})\ofw_0+\Lambda_s   \quad (0 \le s \le n-2).$$

\begin{lemma} \label{lem: Dn(1) HT smax 2} \hfill

{\rm (1)} If $s$ is odd,  we have
\begin{equation}
\Lambda_0+\Lambda_1, \quad \Lambda_{2u+1}  \in \mx^+(\Lambda|2) \quad \text{ for }1 \le u \le \tfrac {s-1}2 ,
\end{equation}

 and if $s$ is even,
\begin{equation}
2\Lambda_0, \quad 2\Lambda_1-(1+\delta_{s,0}) \updelta, \quad \Lambda_{2u}  \in \mx^+(\Lambda|2) \quad \text{ for } 1\le u \le \tfrac s 2.
\end{equation}

{\rm (2)} For $1 \le u \le \lfloor(n-2-s)/2\rfloor$, the following weights are in $\max^+(\Lambda|2) \colon$
\begin{equation} \label{eq: Dn(1) HT smax 2}
\Lambda_{s+2u}-u\updelta  = \Lambda -\cont\left(\YW^{\uplambda(2u-1+s)}_{\ofw_0} \right) +\cont\left(\YW^{\uplambda(s-1)}_{\ofw_0} \right).
\end{equation}

{\rm (3)} Assume $n-s$ is an even integer. Then the following weights are in $\max^+(\Lambda|2) \colon$
\begin{equation} \label{eq: Dn(1) HT smax 2 even}
\begin{aligned}
& 2\ofw_{n}-\dfrac{n-s}{2}\, \updelta = \Lambda- \cont\left(\YW^{\uplambda_{n-1}(n-1)}_{\ofw_0} \right)+\cont\left(\YW^{\uplambda(s-1)}_{\ofw_0} \right),  \\
& 2\ofw_{n-1}-\dfrac{n-s}{2} \, \updelta = \Lambda- \cont\left(\YW^{\uplambda_{n}(n-1)}_{\ofw_0} \right)+\cont\left(\YW^{\uplambda(s-1)}_{\ofw_0} \right),
\end{aligned}
\end{equation}
where $\YW^{\uplambda_\epsilon(n-1)}_{\ofw_0}$ $(\epsilon=n-1,n)$ is the Young wall whose top of the first column is the half-thickness block with color $\epsilon$.

{\rm (4)} Assume $n-s$ is an odd integer. Then the following weight is in $\max^+(\Lambda|2) \colon$
\begin{equation} \label{eq: Dn(1) HT smax 2 odd}
\begin{aligned}
\ofw_{n-1}+\ofw_{n}-\dfrac{n-1-s}{2}\, \updelta =  \Lambda-\cont\left(\YW^{\uplambda(n-2)}_{\ofw_0} \right)  +\cont\left(\YW^{\uplambda(s-1)}_{\ofw_0} \right).
\end{aligned}
\end{equation}
\end{lemma}

\begin{proof}
The lemma can be prove using direct computation as in Lemma \ref{lem: Bn1 HT smax}, and we omit the details.
\end{proof}

\begin{remark}
We see that all the weights in Lemma \ref{lem: Dn(1) HT smax 2} (2)-(4) are essentially finite of type $D_n$.
\end{remark}

\begin{theorem} \label{thm: D level 2}
For $ \Lambda= (\delta_{s,0}+\delta_{s,1})\ofw_0+\Lambda_s$ $(0 \le s \le n-2)$ of level $2$, the set $\mx^+(\Lambda|2)$ is completely given by the maximal weights in Lemma \ref{lem: Dn(1) HT smax 2}. In particular,
we have
$$|\mx^+(\Lambda|2)|=\begin{cases} \tfrac {n+3}2 & \text{if $n$ is odd},  \\ \tfrac n 2 +3 & \text{if $n$ is even and $s$ is even}, \\ \tfrac n 2 & \text{otherwise}.  \end{cases}$$
\end{theorem}

\begin{proof}
One can prove the theorem by applying a similar strategy as in Theorem \ref{thm: B level 2}.
\end{proof}

We define the index of a maximal dominant weight in a similar way to Definition \ref{def-sma-2}.

\begin{definition}
Assume that $\eta \in {\rm max}^+(\Lambda|2)$ is of the form
\[ \eta = \Lambda- \cont\left(\YW^{\uplambda(m)}_{\ofw} \right) + \cont\left(\YW^{\uplambda(s)}_{\ofw} \right)
, \] where $s\ge 0$ if $\Lambda
\neq 2 \ofw_0$ and $s=-1$ if $\Lambda = 2 \ofw_0$. (See Remark \ref{rem-oo}.)
Then we define the {\em index} of the maximal weight $\eta$ to be $(m,s)$.
Similarly, assume that  $\eta \in {\rm max}^+(\Lambda|2)$ is of the form
\[ \eta = \Lambda- \cont\left(\YW^{\uplambda_\epsilon(n-1)}_{\ofw} \right) + \cont\left(\YW^{\uplambda(s)}_{\ofw} \right), \quad \epsilon=n-1,n, \] where $s\ge 0$ if $\Lambda
\neq 2 \ofw_0$ and $s=-1$ if $\Lambda = 2 \ofw_0$.
Then define the {\em index} of the maximal weight $\eta$ to be $(n-1,s)$.
\end{definition}

Now we consider highest weights of level $3$.

\begin{lemma}  \label{lem: Dn(1) HT smax 3} \hfill

{\rm (1)} The following weights are in $\max^+(\ofw_0+\Lambda|3) \colon$ For $0 \le u \le \lfloor(n-3-s)/2\rfloor$,
\begin{equation} \label{eq: Dn(1) HT smax 3}
\begin{aligned}
\ofw_1+\Lambda_{s+2u+1}-(u+1)\updelta & = \ofw_0+\Lambda -\cont\left(\YW^{\uplambda(2u+s)}_{\ofw_1} \right)+ (\alpha_1-\alpha_0)
+\cont\left(\YW^{\uplambda(s-1)}_{\ofw_0} \right).
\end{aligned}
\end{equation}

{\rm (2)} Assume $n-s$ is an even integer. Then the following weight is in $\max^+(\ofw_0+\Lambda|3) \colon$
\begin{equation} \label{eq: Dn(1) HT smax 3 even}
\begin{aligned}
\ofw_{1}+\ofw_{n-1}+\ofw_{n}-\dfrac{n-s}{2}\, \updelta & =  \ofw_0+\Lambda-\cont\left(\YW^{\uplambda(n-2)}_{\ofw_0} \right) +(\alpha_1-\alpha_0)+\cont\left(\YW^{\uplambda(s-1)}_{\ofw_0} \right).
\end{aligned}
\end{equation}

{\rm (3)} Assume $n-s$ is an odd integer. Then the following weights are in $\max^+(\ofw_0+\Lambda|3) \colon$ $t \in \{n-1,n\}$
\begin{equation} \label{eq: Dn(1) HT smax 3 odd}
\begin{aligned}
& \ofw_{1}+2\ofw_{t}-\dfrac{n-s+1}{2}\, \updelta = \ofw_0+\Lambda- \cont\left(\YW^{\uplambda_{t}(n-1)}_{\ofw_0} \right)+\delta_{s\equiv_2 0} (\alpha_1-\alpha_0)+\cont\left(\YW^{\uplambda(s-1)}_{\ofw_0} \right), 
\end{aligned}
\end{equation}
where we write $s \equiv_2 0$ for $s \equiv 0$ $($mod $2)$.
\end{lemma}

\begin{remark} \label{rem: Type B}
We see that all the weights in Lemma \ref{lem: Dn(1) HT smax 3} are essentially finite of type $D_n$.
\end{remark}

The following definition is an analogue of Definition \ref{def-sma-3}.

\begin{definition}
Assume that $\eta \in {\rm max}^+(\ofw_0+\Lambda|3)$, and set $\Uplambda= \ofw_0+\Lambda$.
\begin{enumerate}
\item If $\eta = \ofw_0 + \mu$ with $\mu \in {\rm max}^+(\Lambda|2)$ of index $(m,s)$,
then we define the {\em index} of $\eta$ to be $(m,s)$.
\item Assume that $\eta$ is of the form
\begin{align*} \eta & = \Uplambda- \cont\left(\YW^{\uplambda(m)}_{\ofw_0} \right)+(\alpha_1-\alpha_0)  +\cont\left(\YW^{\uplambda(s)}_{\ofw_0} \right) ,  \end{align*} where $s\ge 0$ if $\Uplambda
\neq 3 \ofw_0$ and $s=-1$ if $\Uplambda = 3 \ofw_0$. Then define the {\em index} of the maximal weight $\eta$ to be $(m,s)$.
\item Assume that  $\eta$ is of the form
\begin{align*} \eta & =   \Uplambda - \cont\left(\YW^{\uplambda_{\epsilon}(n-1)}_{\ofw_0} \right)+\delta_{s\equiv_2 0}(\alpha_1-\alpha_0)  +\cont\left(\YW^{\uplambda(s)}_{\ofw_0} \right), \quad \epsilon=n-1,n, \end{align*}  where $s\ge 0$ if $\Uplambda
\neq 3 \ofw_0$ and $s=-1$ if $\Uplambda = 3 \ofw_0$. We define the {\em index} of the maximal weight $\eta$ to be $(n-1,s)$.
\end{enumerate}
\end{definition}

Similarly, we consider higher levels to make the following definition.

\begin{definition}
Assume that $\eta \in (k-1)\ofw+{\rm max}^+(\ofw+\Lambda|3)$ for $k \ge 1$, and write $\eta=(k-1) \ofw +\mu$ with $\mu \in {\rm max}^+(\ofw+\Lambda|3)$.  If $\mu$ is of index $(m,s)$, then we define the {\it index} of $\eta$ to be $(m,s)$.
\end{definition}

Whenever the index is defined for a maximal weight $\eta \in {\rm max}^+(k\ofw+\Lambda|k+2)$, $k \ge 0$, the weight $\eta$ will be called a {\em staircase dominant maximal weight}. The set of staircase dominant maximal weights will be denoted by $\smax(k\ofw+\Lambda|k+2)$.

\subsection{Type $A_{2n-1}^{(2)}$} Recall that the affine type $A_{2n-1}^{(2)}$ has the fundamental weights $\ofw_0$ and $\ofw_1$
of level $1$. Let us take a level $2$ dominant integral weight $\Lambda$ of the form
$$ \Lambda= (\delta_{s,0}+\delta_{s,1})\ofw_0+\Lambda_s   \quad (0 \le s \le n).$$

\begin{lemma}   \label{lem: A2n-1(1) HT smax} \hfill

{\rm (1)} For $0 \le u \le \lfloor(n-s)/2\rfloor$, we have
\begin{equation} \label{eq: A2n-1(1) HT smax}
\begin{aligned}
(\delta_{s,0}+\delta_{s,1})\ofw_0+ \Lambda_{s+2u}-u\updelta & = \Lambda -\cont\left(\YW^{\uplambda(2u-1+s)}_{\ofw_0} \right)+\cont\left(\YW^{\uplambda(s-1)}_{\ofw_0} \right) \in \mx^+(\Lambda|2).
\end{aligned}
\end{equation}

{\rm (2)} For $1 \le u \le \left\lfloor \dfrac {s}{2} \right\rfloor$, we have
\begin{equation}\label{eq: A2n-1(1) C smax}
\begin{aligned}
&(1+\delta_{s-2u,0})\Lambda_{s-2u}+\delta_{s-2u,1}\ofw_1 \\
& \hspace{10ex} = \Lambda_s- \left( \sum_{i=s-2u+1}^{\max(s,n-1)} (i-s+2u)\alpha_i + 2u \sum_{j=s+1}^{n-1}\alpha_j+u\alpha_n\right) \in {\rm max}^+(\Lambda|2).
\end{aligned}
\end{equation}

{\rm (3)} If $s \ge 2$ is even, then we have
\begin{align} \label{eq: A2n-1(1) C smax-1}
2\ofw_1 -\updelta =\Lambda_s- \left( \sum_{i=2}^{\max(s,n-1)} i \ \alpha_i + \alpha_0 + s \sum_{j=s+1}^{n-1}\alpha_j+\dfrac{s}{2}\alpha_n\right) \in \mx^+(\Lambda|2).
\end{align}

{\rm (4)} When $s=0$, we have
\begin{align} \label{eq: A2n-1(1) C smax-2}
2\ofw_1 -2\updelta =\Lambda_s- \left( 2\sum_{i=2}^{n-1}  \alpha_i + 2\alpha_0+  \alpha_n \right) \in \mx^+(\Lambda|2).
\end{align}
\end{lemma}

\begin{remark} \label{rem: Type Dn}
We see that the weights in \eqref{eq: A2n-1(1) HT smax} are essentially finite of type $D_n$, and that those in
\eqref{eq: A2n-1(1) C smax}, \eqref{eq: A2n-1(1) C smax-1} and \eqref{eq: A2n-1(1) C smax-2} are essentially finite of type $C_n$.
\end{remark}

\begin{theorem}
For $\Lambda= (\delta_{s,0}+\delta_{s,1})\ofw_0+\Lambda_s$ $(0 \le s \le n)$ of level $2$,
the maximal weights in Lemma \ref{lem: A2n-1(1) HT smax} exhaust the whole set ${\rm max}^+(\Lambda|2)$.
Hence the number of elements in $\max^+(\Lambda|2)$ is $\lfloor n/2 \rfloor +2$ if $s$ is even, and $\lfloor (n-1)/2 \rfloor +1$ if $s$ is odd.
\end{theorem}

\begin{proof}
One can prove the theorem by applying a similar argument as in Theorem \ref{thm: B level 2}.
\end{proof}

\medskip

Now we consider highest weights of level $3$.

\begin{lemma}  \label{lem: A2n-1(1) HT smax 3}
The following weights are in $\max^+(\ofw_0+\Lambda|3) \colon$ For $0 \le u \le \lfloor(n-s)/2\rfloor$,
\begin{equation} \label{eq: A2n-1(1) HT smax 3}
\begin{aligned}
\ofw_1+\Lambda_{s+2u+1}-(u+1)\updelta & = \ofw_0+\Lambda -\cont\left(\YW^{\uplambda(2u+s)}_{\ofw_1} \right)+ (\alpha_1-\alpha_0)
+\cont\left(\YW^{\uplambda(s-1)}_{\ofw_0} \right).
\end{aligned}
\end{equation}
\end{lemma}

We define the index of the weights in \eqref{eq: A2n-1(1) HT smax} and \eqref{eq: A2n-1(1) HT smax 3}  as we did in Definition \ref{def-sma-2} and \ref{def-sma-3}, respectively,  and we extend it to higher levels as in Definition \ref{def-sma-h}. Similarly, whenever the index is defined for a maximal weight $\eta \in {\rm max}^+(k\ofw+\Lambda|k+2)$, $k \ge 0$, the weight $\eta$ will be called a {\em staircase dominant maximal weight}. The set of staircase dominant maximal weights will be denoted by $\smax(k\ofw+\Lambda|k+2)$.

\subsection{Type $A_{2n}^{(2)}$} Recall that the affine type $A_{2n}^{(2)}$ has the only  fundamental weight
$\ofw_0$ of level $1$.  Let us take level $2$ dominant integral weights $\Lambda$ as follows:
$$ \Lambda= \delta_{s,0}\ofw_0+\Lambda_s   \quad (0 \le s \le n).$$

\begin{lemma} \label{lem: A2n(2) HT smax 2} \hfill

{\rm (1)} For $0 \le u \le \lfloor(n-s)/2\rfloor$, we have
\begin{equation} \label{eq: A2n(2) HT smax 2}
\begin{aligned}
(1+\delta_{s+2u,n-1})\Lambda_{s+2u}-2u\updelta & = \Lambda -\cont\left(\YW^{\uplambda(2u-1+s)}_{\ofw_0} \right)  +\cont\left(\YW^{\uplambda(s)}_{\ofw_0} \right)  \in {\rm max}^+(\Lambda|2).
\end{aligned}
\end{equation}

{\rm (2)} For $1 \le u \le \left\lfloor \dfrac {s}{2} \right\rfloor$, we have
\begin{equation}\label{eq: A2n(2) C smax}
\begin{aligned}
& (1+\delta_{s-2u,0})\Lambda_{s-2u} = \Lambda_s- \left( \sum_{i=s-2u+1}^{\max(s,n-1)} (i-s+2u)\alpha_i + 2u  \sum_{j=s+1}^{n-1}\alpha_j+u\alpha_n\right) \in {\rm max}^+(\Lambda|2).
\end{aligned}
\end{equation}
\end{lemma}

\begin{remark} \label{rem: Type A2n2}
We see that the weights in \eqref{eq: A2n(2) HT smax 2} are essentially finite of type $B_n$, and that those in
\eqref{eq: A2n(2) C smax} are essentially finite of type $C_n$.
\end{remark}

\begin{theorem}
For $\Lambda= \delta_{s,0}\ofw_0+\Lambda_s$ $(0 \le s \le n)$ of level $2$,
the maximal weights in {\rm Lemma \ref{lem: A2n(2) HT smax 2}} exhaust the whole set ${\rm max}^+(\Lambda|2)$.
Hence the number of elements in $\max^+(\Lambda|2)$ is $(n+1)/2$ if $n$ is odd and $n/2+\delta_{s \equiv_2 0}$ if $n$ is even.
\end{theorem}

\begin{proof}
A similar argument as in Theorem \ref{thm: B level 2} can be used.
\end{proof}

As in the previous cases, one can determine dominant maximal weights of level $3$ highest weights. We leave it to interested readers.

\begin{definition}
Assume that $\eta \in {\rm max}^+(\Lambda|2)$ is of the form
\[ \eta = \Lambda- \cont\left(\YW^{\uplambda(m)}_{\ofw_0} \right) + \cont\left(\YW^{\uplambda(s)}_{\ofw_0} \right)
, \] where $s\ge 0$. Then we define the {\em index} of the maximal weight $\eta$ to be $(m,s)$.
\end{definition}
We extend the above definition to higher levels as before. Whenever the index is defined for a maximal weight $\eta \in {\rm max}^+(k\ofw+\Lambda|k+2)$, $k \ge 0$, the weight $\eta$ will be called a {\em staircase dominant maximal weight}. The set of staircase dominant maximal weights will be denoted by $\smax(k\ofw+\Lambda|k+2)$.

\subsection{Type $D_{n+1}^{(2)}$} Recall that the affine type $D_{n+1}^{(2)}$ has  the fundamental weights $\ofw_0,\ofw_n$ of level $1$.  Let us consider level $2$ dominant integral weights $\Lambda$:
$$ \Lambda= (\delta_{s,0}+\delta_{s,n})\ofw_0+\Lambda_s   \quad (0 \le s \le n-1).$$

\begin{lemma}  \label{lem: Dn+1(2) HT smax 2}
The following weights are in $\max^+(\Lambda|2) \colon$
\begin{align}
(1+\delta_{s+u,n})\Lambda_{s+u}-u\updelta & = \Lambda -\cont\left(\YW^{\uplambda(u-1+s)}_{\ofw_0} \right) +\cont\left(\YW^{\uplambda(s)}_{\ofw_0} \right)
\ ( 0 \le u \le n-s) \label{eq: Dn+1(2) HT smax 2 0}, \\
(1+\delta_{u,0})\Lambda_{u} & =\Lambda- \cont\left(\YW^{\uplambda(n-u)}_{\ofw_n} \right)
+ \cont\left(\YW^{\uplambda(n-s)}_{\ofw_n} \right) \ (1 \le u \le s).  \label{eq: Dn+1(2) HT smax 2 n}
\end{align}
\end{lemma}

\begin{remark}
We see that the weights in \eqref{eq: Dn+1(2) HT smax 2 0} and \eqref{eq: Dn+1(2) HT smax 2 n} are essentially finite of type $B_n$.
\end{remark}

\begin{theorem}
For $\Lambda= (\delta_{s,0}+\delta_{s,n})\ofw_0+\Lambda_s$ $(0 \le s \le n-1)$ of level $2$, the maximal weights in Lemma \ref{lem: Dn+1(2) HT smax 2} exhaust the whole set ${\rm max}^+(\Lambda|2)$.
The number of elements in $\max^+(\Lambda|2)$ is $n+1$.
\end{theorem}

Dominant maximal weights of level $3$ highest weights can be determined as in the previous types, and we leave it to interested readers.

\begin{definition}
Assume that $\eta \in {\rm max}^+(\Lambda|2)$ is of the form
\[ \eta = \Lambda- \cont\left(\YW^{\uplambda(m)}_{\ofw} \right) + \cont\left(\YW^{\uplambda(s)}_{\ofw} \right)
, \quad \ofw=\ofw_0, \ofw_n, \] where $s\ge 0$. Then we define the {\em index} of the maximal weight $\eta$ to be $(m,s)$.
\end{definition}
We extend the above definition to higher levels as before. The set of staircase dominant maximal weights is defined in a similar way as in the previous subsections.

\subsection{Classification of staircase dominant maximal weights} \label{subsec-class}

As we have observed in the previous subsections, the staircase maximal weights in ${\rm smax}^+(\Uplambda)$ are essentially finite of type $B_n$ or $D_n$.  Hence we classify
the staircase dominant maximal weights into two classes according to their finite types, and make the following definition.

\begin{definition}
Define $\hmaxp(\Uplambda|k)$ (resp. $\tmaxp(\Uplambda|k)$) to be the set of staircase dominant maximal weights of $\Uplambda$ of level $k \ge 2$,  that are essentially finite of type $B_n$ (resp. $D_n$).
\end{definition}

\begin{remark}[Indices for $\hmaxp(\Uplambda|k)$ and $\tmaxp(\Uplambda|k)$] \label{rem: indices} \hfill
\begin{enumerate}
\item[{\rm (1)}] For $k \ge 2$, the indices for $\hmaxp(\Uplambda|k)$ are given as
follows (see Lemma \ref{lem: Bn1 HH smax}):
\begin{align*}
\{ (m,s) \ | \ m \ge s \ge 0 \}.
\end{align*}
\item[{\rm (2)}] For $k \ge 2$, the indices for $\tmaxp(\Uplambda|k)$ are given as
follows (see Lemma \ref{lem: Bn1 HT smax}, \ref{lem-level3B} and \eqref{eq: Riordan multi}):
\begin{align} \label{eq: indices dmsk}
\begin{cases}
\{ (m,s-1) \ | \ s \ge 0, \  m \ge s-1 \text{ and } m \not\equiv_2 s \} \setminus \{ (0,-1) \} & \text{if } k =2, \\
\{ (m,s-1) \ | \ s \ge 0 \text{ and }  m \ge s-1  \} \setminus \{ (0,-1) \} & \text{if } k \ge 3.
\end{cases}
\end{align}
\end{enumerate}
\end{remark}

We make a table to show which affine types are related to each type of staircase dominant maximal weights.
\begin{table}[ht]
\renewcommand{\arraystretch}{1.4}
\centering
\small
\begin{tabular}[c]{|c|c|c|} \hline
\quad Staircase Type \quad & Affine Types  \\ \hline \quad
$\hmaxp(\Uplambda|k)$ \quad & \quad $B^{(1)}_n$, $A_{2n}^{(2)}$, $D_{n+1}^{(2)}$ \quad \\ \hline
\quad
$\tmaxp(\Uplambda|k)$ \quad & \quad $B^{(1)}_n$, $D_n^{(1)}$, $A_{2n-1}^{(2)}$ \quad \\ \hline
\end{tabular}
\end{table}

\section{ Weight multiplicities and  (spin) rigid Young tableaux}
\label{sec:WM}
In this section, we will introduce the notion of {\it $($spin$)$ rigid Young tableaux}, and show that the set of these tableaux is
equinumerous to the set of crystal basis elements in $\mathbf{B}(\Uplambda)_\eta$
for staircase dominant maximal weights $\eta \in {\rm smax}^+(\Uplambda|k)$, $k\ge 2$.
As noted in \eqref{eq: coincidence of multiplicity}, it suffices to consider their finite types.
Hence in this section we only consider affine type $B_n^{(1)}$ and the sets  ${\rm smax}_\Ss^+(\Uplambda|k)$ and ${\rm smax}_\sS^+(\Uplambda|k)$.

Considering the crystal rules for Young walls, one can prove the following lemma.

\begin{lemma}\label{lem:highest yw}
For strict partitions $\lambda^{(1)},\dots,\lambda^{(k)}$  with $\max \{ \lambda^{(1)}_1,\dots,\lambda^{(k)}_1 \} \le n$, the Young wall
$\yy^{(\lambda^{(1)},\dots,\lambda^{(k)})}_{(\ofw^{(1)},\dots,\ofw^{(k)})}$ corresponds to a highest weight if and only if
the following condition holds: $\lambda^{(i)}=\uplambda(s_i)$ for $i=1,2,\dots,k$ for some nonnegative integers $s_1,\dots,s_k$ with $s_1=0$.
\end{lemma}



\begin{definition}
For strict partitions $\lambda^{(1)}$ and $\lambda^{(2)}$, $\ofw$ and $\ofw'$ of the same type, we define
$s_{\ofw,\ofw'}(\lambda^{(1)},\lambda^{(2)})$ to be the smallest
 nonnegative integer $s$ satisfying
\begin{equation}
\label{eq: s condition}
\YW^{\lambda^{(1)}}_\ofw \supset (\YW^{\lambda^{(2)}}_{\ofw'}) _{\ge s+1}.
\end{equation}
\end{definition}

The following lemma implies that the quantity $s_{\ofw,\ofw'}(\lambda^{(1)},\lambda^{(2)})$ is invariant under application of $\tilde{e}_i$'s.

\begin{proposition}\label{prop:invariant}
For strict partitions $\lambda^{(1)},\lambda^{(2)}$ with $\max \{ \lambda^{(1)}_1,\lambda^{(2)}_1 \} \le n$,
suppose that
\[
\tilde{e}_i(\yy^{(\lambda^{(1)},\lambda^{(2)})}_{(\ofw,\ofw')}) = \yy^{(\lambda',\lambda'')}_{(\ofw,\ofw')}.
\]
Then $s_{\ofw,\ofw'} (\lambda^{(1)},\lambda^{(2)})=s _{\ofw,\ofw'} (\lambda',\lambda'')$.
\end{proposition}

\begin{proof}
Let $s=s_{\ofw,\ofw'}(\lambda^{(1)},\lambda^{(2)})$ and $s'=s_{\ofw,\ofw'}(\lambda',\lambda'')$.
Let $\epsilon=0$ if $\Lambda^{(1)}$ and $\Lambda^{(2)}$ are of type $\Ss$, and $\epsilon=1$ if they are of type $\sS$.
  The assumption implies that we have either
  \begin{enumerate}
  \item $\lambda'=\lambda^{(1)}$ and $\| \lambda^{(2)}\skewpar \lambda''\|=1$ or
  \item $\lambda''=\lambda^{(2)}$ and $\| \lambda^{(1)}\skewpar \lambda'\|=1$.
  \end{enumerate}

Since the second case can be proved similarly, we will only consider the first case. Since $\lambda'=\lambda^{(1)}\supset \lambda^{(2)}_{\ge s+1} \supset \lambda''_{\ge s+1}$, if $s\le \epsilon$, then it is the smallest possible and we have $s'=s$. Now assume that $s\ge 1+\epsilon$.  Let $j$ be the unique integer such that $\lambda^{(2)}_j = \lambda''_j+1$.  In order to show  $s=s'$, it suffices to show $\lambda' \not\supset \lambda''_{\ge s-\epsilon}$. For a contradiction, suppose that $\lambda' \supset \lambda''_{\ge s-\epsilon}$. Then we have $\lambda^{(1)}=\lambda' \supset \lambda''_{\ge s-\epsilon}$ and $\lambda^{(1)}\not\supset \lambda^{(2)}_{\ge s-\epsilon}$. Since $\lambda^{(2)}$ and $\lambda''$ differ by only one part, we obtain that $\lambda^{(1)}$ must have a part equal to $t-1$, where $t:=\lambda^{(2)}_j = \lambda''_j+1$.  Moreover, by considering the Young diagrams of $\lambda^{(1)}$, $\lambda^{(2)}_{\ge s-\epsilon}$, and $\lambda''_{\ge s-\epsilon}$, one can see that the position of the part $t-1$ in $\lambda^{(1)}$ is equal to the position of the part $t$  in $\lambda^{(2)}_{\ge s-\epsilon}$. Therefore, we have $j\ge s-\epsilon$ and
\[
\lambda^{(1)}_{j-s+\epsilon+1}=(\lambda^{(2)}_{\ge s-\epsilon})_{j-s+\epsilon+1}-1 = \lambda^{(2)}_{j}-1 = t-1.
\]
If $j=s-\epsilon$, then $\lambda^{(1)}_1=t-1$ and ${\rm sig}_i(\YW^{\lambda^{(1)}}_\ofw)=(+)$. If $j\ge s-\epsilon+1$, then by the assumption $\lambda^{(1)}=\lambda' \supset \lambda''_{\ge s-\epsilon}$, we have
\[
\lambda^{(1)}_{j-s+\epsilon} \ge (\lambda''_{\ge s-\epsilon})_{j-s+\epsilon} = \lambda''_{j-1}=\lambda^{(2)}_{j-1}>\lambda^{(2)}_{j}=t.
\]
Thus we also have ${\rm sig}_i(\YW^{\lambda^{(1)}}_\ofw)=(+)$. This means that
\[
\tilde{e}_i(\yy^{(\lambda^{(1)},\lambda^{(2)})}_{(\ofw,\ofw')}) =
\tilde{e}_i(\YW^{\lambda^{(1)}}_\ofw) \otimes \YW^{\lambda^{(2)}}_{\ofw'} = 0,
\]
which is a contradiction. Therefore, we must have $\lambda' \not\supset \lambda''_{\ge s-\epsilon}$, which implies $s=s'$.
\end{proof}

\subsection{Case  $\hmaxp(\Uplambda|k)$} In this subsection, we assume that $\eta$ is an element of $\hmaxp(\Uplambda|k)$
and that $\boxed{\ofw}$ is of type $\Ss$.

Let $k \in \Z_{\ge 1}$ and $s \in \Z_{\ge 0}$. A skew Young tableau $T$ of shape $\mu \skewpar (s^{k-1})$ with $m$ cells for a partition $\mu$ of length $k$ is naturally identified with a sequence of strict partitions $$(\lambda^{(1)},\lambda^{(2)},\ldots,\lambda^{(k-1)},\lambda^{(k)})$$ such that $\lambda^{(1)}*\lambda^{(2)}*\cdots * \lambda^{(k-1)}*\lambda^{(k)}=\uplambda(m)$, $\lambda^{(i)} \supset \lambda^{(i+1)}$ for $1 \le i \le k-2$ and $\lambda^{(k-1)} \supset \lambda^{(k)}_{\ge s+1}$.
For example, take $k=3$ and $s=1$ and we identify  the following
skew Young  tableau with the corresponding sequence of partitions
$$\young(\cdot754,\cdot31,62) \longleftrightarrow \left( (7,5,4),(3,1),(6,2) \right). $$
From now on, we will freely use this identification of skew tableaux and sequences of strict partitions.

\begin{definition} \label{def: rigid}
For $k \in \Z_{\ge 1}$ and $s,m \in \Z_{\ge 0}$, let $T=(\lambda^{(1)},\lambda^{(2)},\ldots,\lambda^{(k-1)},\lambda^{(k)})$ be a skew Young tableau of shape $\mu \skewpar (s^{k-1})$ with $m$ cells for a partition $\mu$ of length $k$. 
Then $T$ is called a {\it rigid Young tableau of index $(m,s)$
with $k$ rows} if $s=0$ or $s\ge1$ and
\begin{align} \label{eq: Bsm condition}
\lambda^{(k-1)} \not\supset \lambda^{(k)}_{\ge s}.
\end{align}
We denote by ${}_s\Ss_m^{(k)}$  the set of all rigid Young tableaux of index $(m,s)$ with $k$ rows. In particular, we have
${}_0\Ss_m^{(k)}=\Ss^{(k)}_m$.
\end{definition}
Note that if $T=(\lambda^{(1)},\lambda^{(2)},\ldots,\lambda^{(k-1)},\lambda^{(k)})$ is a rigid tableau of index $(m,s)$, then $\ell(\lambda^{(k)}) \ge s$.
The condition \eqref{eq: Bsm condition} says that a shift of the last row to the right by $1$ makes the tableau violate the column-strictness.

\begin{example} \hfill
\begin{enumerate}
\item $T=\left((432),(51)\right) \in {}_{1}\Ss^{(2)}_5$ since
$$ \young(\cdot432,51) \text{ is a skew Young tableau but } \young(432,51) \text{ is not a Young tableau}.$$
On the other hand, $\left((532),(41)\right) \not \in {}_{1}\Ss^{(2)}_5$ since
$$ \young(\cdot532,41) \text{ is a skew Young tableau and } \young(532,41) \text{ is also a Young tableau}.$$
\item $\young(\cdot\cdot\cdot\ot\te87,\cdot\cdot\cdot\oo91,65432) \in {}_{3}\Ss^{(3)}_{12}$ since $\young(\cdot\cdot\ot\te87,\cdot\cdot\oo91,65432)$ is not a skew Young tableau.
\vspace{0.1cm}
\item We also have $T=\left( (0),(0),(2,1) \right) \longleftrightarrow \young(\cdot\cdot,\cdot\cdot,21) \in {}_{2}\Ss^{(3)}_2$.
\end{enumerate}
\end{example}

\begin{proposition} \label{prop: hh k=2 step 2}
  For strict partitions $\lambda^{(1)}$ and $\lambda^{(2)}$ with $\max \{ \lambda^{(1)}_1,\lambda^{(2)}_1 \} \le n$, the Young wall
$\yy^{(\lambda^{(1)},\lambda^{(2)})}_{(\ofw,\ofw)}$ is connected to $\boxed{\ofw} \otimes \YW^{\uplambda(s)}_{\ofw}$
where $s =s_{\ofw,\ofw}(\lambda^{(1)},\lambda^{(2)})$.
Conversely, for strict partitions $\lambda^{(1)}$ and  $\lambda^{(2)}$ with $\max \{ \lambda^{(1)}_1,\lambda^{(2)}_1 \} \le n$,
if the Young wall
$\yy^{(\lambda^{(1)},\lambda^{(2)})}_{(\ofw,\ofw)}$ is connected to $\boxed{\ofw} \otimes \YW^{\uplambda(s)}_{\ofw}$, then
$s =s_{\ofw,\ofw}(\lambda^{(1)},\lambda^{(2)})$.
\end{proposition}

\begin{proof}
If we apply $\tilde{e}_i$'s to $\yy^{(\lambda^{(1)},\lambda^{(2)})}_{(\ofw,\ofw)}$ until no longer possible, we obtain a Young wall corresponding to a highest weight vector. By Lemma~\ref{lem:highest yw}, the resulting Young wall is of the form $\boxed{\ofw}\otimes \YW^{\uplambda(r)}_\ofw$ for some $r\ge0$. By Proposition~\ref{prop:invariant}, we have
\[
s =s_{\ofw,\ofw}(\lambda^{(1)},\lambda^{(2)}) = s_{\ofw,\ofw}(\emptyset, \uplambda(r)) = r.
\]
The converse is obtained by using the fact that
$\boxed{\ofw} \otimes \YW^{\uplambda(s)}_{\ofw}$ and $\boxed{\ofw} \otimes \YW^{\uplambda(s')}_{\ofw}$ are not connected for $s\ne s'$.
\end{proof}


As in Introduction, define
\begin{align} \label{eq: bmega B}
\tilde{\omega}_s \seteq \begin{cases} 2 \omega_n & \text{ if } s=n ,\\ \omega_s & \text{ otherwise.} \end{cases}
\end{align}
Let $L({\omega})$ be the highest weight module with highest weight ${\omega}$ over the finite dimensional Lie algebra of type $B_n$.

We have the following result:

\begin{proposition} \label{thm: level 2 hh rigid}
For $\eta \in \hmaxp(\Lambda|2)$ of index $(m,s)$, we have
$$ \dim(V(\Lambda)_\eta) = |{}_s\Ss_m^{(2)}| =  \dim \big( L(\tilde{\omega}_{n-s})_{\tilde{\omega}_{n-m}} \big).$$
\end{proposition}

\begin{proof}
Note that $$\cont(\yy^{T}_{(\ofw,\ofw)}) =\cont(\YW^{\uplambda(m)}_{\ofw}) \quad \text{ for any } T \in {}_s\Ss_m^{(2)}.$$ By Proposition \ref{prop: hh k=2 step 2},
the set $\{ \yy^{T}_{(\ofw,\ofw)} \ | \  T \in {}_s\Ss_m^{(2)} \}$ forms the crystal basis for
$V(\Lambda)_\eta$, which implies our assertion. The last equality follows from Proposition \ref{prop-exhaust B2} and \eqref{ena-sam}.
\end{proof}

Now, we obtain the main theorem of this subsection:

\begin{theorem} \label{thm: level k hh rigid}
Assume that $k\ge 2$ and $0 \le s \le m$. Then, for $\eta \in \hmaxp(\Uplambda|k)$  of index $(m,s)$, we have
$$ \dim V(\Uplambda)_\eta  = |{}_s\Ss_m^{(k)}|= \dim  L((k-2)\omega_n+\tilde{\omega}_{n-s})_{(k-2)\omega_n+\tilde{\omega}_{n-m}} .$$
\end{theorem}

\begin{proof} Since $s \le m \le n$, a Young wall $\yy \in \mathbf{B}(\Uplambda)_\eta$ connected to $\boxed{\Uplambda}\seteq \boxed{(k-1)\ofw} \otimes \YW_\ofw^{\uplambda(s)}$
cannot contain a removable $\updelta$. Thus, for each $\yy \in \mathcal{Z}(\ofw)^{\otimes k}$ connected to $\boxed{\Uplambda}$, there exists a
sequence of strict partitions $\ula=(\lambda^{(1)},\lambda^{(2)},\ldots,\lambda^{(k-1)},\lambda^{(k)})$ satisfying $\lambda^{(1)}*\lambda^{(2)}*\cdots * \lambda^{(k-1)}*\lambda^{(k)}=\uplambda(m)$, and hence $\yy=\yy^\ula_{\uofw}.$

Let $t$ be the smallest integer such that $t<k$ and $\lambda^{(t)} \not \supset \lambda^{(t+1)}$.
If there is no such integer, we let $t=k$. If $t<k$, we also define $u$ to be the smallest nonnegative integer satisfying
\[
\lambda^{(t)} \supset \lambda^{(t+1)}_{\ge u+1}.
\]

By applying the arguments in Proposition~\ref{prop: hh k=2 step 2} to the higher levels,
one can see that if $t < k$,
$$ \text{$\yy$ is connected to $\boxed{  (k-1) \ofw} \otimes \YW_\ofw^{\uplambda(s)} \iff t=k-1 \text{ and } u=s  \iff \ula \in {}_s\Ss_m^{(k)}$},$$
and if $t=k$,
\[
\text{$\yy$ is connected to $\boxed{k\ofw} \iff t=k \iff \ula \in {}_0\Ss_m^{(k)}$}.   \qedhere
\]
\end{proof}

As a special case, when $s=0$,  the numbers $|\Ss_m^{(k)}|$ for $m \le n$ are the multiplicities of maximal weights of $V(k\ofw)$. Explicit formulas for the numbers $|\Ss_m^{(k)}|$ are given in Theorem  \ref{thm: Standard at most} for $1\le k \le 5$. We will obtain a closed formula for $|\Ss_m^{(6)}|$
in Corollary \ref{Cor: Sm6}.
In \cite{TW}, Tsuchioka and Watanabe studied the case $\Uplambda= k\ofw_0$ for types $A^{(2)}_{2n}$ and $D^{(2)}_{n+1}$.

\subsection{Case  $\tmaxp(\Uplambda|k)$} In this subsection, we will deal with $\eta$ in $\tmaxp(\Uplambda|k)$.
Throughout this section, we assume that $\ofw$ is of type $\sS$.

\begin{proposition} \label{prop: Talbelaux ht in smax} \hfill
\begin{enumerate}
\item
For  strict partitions $\lambda^{(1)},\lambda^{(2)}$ such that
$$\text{$\max\{\lambda^{(1)}_1,\lambda^{(2)}_1\} \le n$,
$\lambda^{(1)}  \supset \lambda^{(2)}_{\ge 2s}$ and, $s=1$ or $\lambda^{(1)} \not \supset \lambda^{(2)}_{\ge 2s-2}$ for  some $s \ge 2$},$$
the Young wall \ $\YW^{\lambda^{(1)}}_{\ofw_0} \otimes \YW^{\lambda^{(2)}}_{\ofw_1}$ is connected to
$\boxed{\Lambda_{2s-1}} \seteq  \boxed{\ofw_0} \otimes \YW^{\uplambda(2s-2)}_{\ofw_1}$.

\item For strict partitions $\lambda^{(1)},\lambda^{(2)}$ such that $$\text{$\max\{\lambda^{(1)}_1,\lambda^{(2)}_1\} \le n$,
$\lambda^{(1)}  \supset  \lambda^{(2)}_{\ge 2s+1}$ and $\lambda^{(1)} \not \supset \lambda^{(2)}_{\ge 2s-1}$ for  some $s \ge 1$},$$
the Young wall \ $\YW^{\lambda^{(1)}}_{\ofw_0} \otimes \YW^{\lambda^{(2)}}_{\ofw_0}$ is connected to
$\boxed{\Lambda_{2s}} \seteq   \boxed{\ofw_0} \otimes \YW^{\uplambda(2s-1)}_{\ofw_0}$.
\end{enumerate}
\end{proposition}
\begin{proof}
By Remark \ref{rmk: pattern up to one column}, the patterns appearing in $\YW^{\lambda^{(1)}}_{\ofw_0}$ and $(\YW^{\lambda^{(2)}}_{\ofw_1})_{\ge 2s}$ coincide with each other.
By applying $\tilde{e}_i$'s until no longer possible, we obtain a Young wall corresponding to its highest weight vector. By Proposition~\ref{prop:invariant}, its highest weight vector
is of the form $\boxed{\ofw_0} \otimes \YW^{\uplambda(2t)}_{\ofw_1}$ for some $2t\ge0$. 
By Lemma~\ref{lem:highest yw},
\[
2s-2 =s_{\ofw_0,\ofw_1}(\lambda^{(1)},\lambda^{(2)}) = s_{\ofw_0,\ofw_1}(\emptyset, \uplambda(2t)) = 2t.
\]
This proves the first statement.

The second statement follows similarly with the consideration on patterns.
\end{proof}

Recall that each $\eta \in \tmaxp(\Lambda|2)$ is of index $(2m-1+s,s-1)$ (see \eqref{eq: Bn1 HT type smax}).

\begin{theorem} \label{thm: level 2 ht}
For $\eta \in \tmaxp(\Lambda|2)$ of index $(2m-1+s,s-1)$,
set $\epsilon=0$ if $s$ is even and $\epsilon=1$ otherwise. Then
\begin{eqnarray} &&
\parbox{90ex}{
$\yy \in \mathbf{B}((\delta_{s,0}+\delta_{s,1})\ofw_0+\Lambda_s)_\eta$ $(1 \le s <n)$ if and only if $\yy=\YW^{\lambda^{(1)}}_{\ofw_0} \otimes \YW^{\lambda^{(2)}}_{\ofw_{\epsilon}}$ satisfies
\begin{enumerate}
\item[{\rm (a)}] $\lambda^{(1)} * \lambda^{(2)} = \uplambda(2m-1+s)$,
\item[{\rm (b)}] $\begin{cases}
\lambda^{(1)} \supset \lambda^{(2)}_{\ge s+1} \text{ and }\lambda^{(1)} \not \supset \lambda^{(2)}_{\ge s-1} & \text{ if } s \ge 2, \\
\lambda^{(1)} \supset \lambda^{(2)}_{\ge 2} & \text{ if } s =1,\\
\lambda^{(1)} \supset \lambda^{(2)} & \text{ if } s =0.
\end{cases}$
\end{enumerate}}\label{eq: iff level2 ss}
\end{eqnarray}
\end{theorem}

\begin{proof}
The ``if" part follows from Proposition~\ref{prop: Talbelaux ht in smax}. Now it suffices to prove the ``only if" part.
Since $\eta$ corresponds to $(\uplambda(2m-1+s),\uplambda(s-1))$ for $2m-1+s \le n$, $\yy$ should be of the form
$\YW^{\lambda^{(1)}}_{\ofw_0} \otimes \YW^{\lambda^{(2)}}_{\ofw_{\epsilon}}$ for some pair of strict partitions $(\lambda^{(1)},\lambda^{(2)})$.
Note that any pair of strict partitions $(\lambda^{(1)},\lambda^{(2)})$ has the largest $t$ satisfying one of three conditions in (b) of \eqref{eq: iff level2 ss}.
One can also check that $\max\{\lambda^{(1)}_1,\lambda^{(2)}_1 \} \le n$. Then the ``only if" part follows from the form of weight $\eta$ and Proposition~\ref{prop: Talbelaux ht in smax} again; that is,
$s=t$ and $\lambda^{(1)} * \lambda^{(2)} = \uplambda(2m-1+s)$ by \eqref{eq: Bn1 HT type smax}.
\end{proof}

Let $k \in \Z_{\ge 1}$ and $s \in \Z_{\ge 0}$. Recall that a skew Young tableaux $T$ of shape $\mu \skewpar (s^{k-1})$ with $m$ cells for a partition $\mu$ of length $k$ is identified with a sequence of strict partitions $$(\lambda^{(1)},\lambda^{(2)},\ldots,\lambda^{(k-1)},\lambda^{(k)})$$ such that $\lambda^{(1)}*\lambda^{(2)}*\cdots * \lambda^{(k-1)}*\lambda^{(k)}=\uplambda(m)$, $\lambda^{(i)} \supset \lambda^{(i+1)}$ for $1 \le i \le k-2$ and $\lambda^{(k-1)} \supset \lambda^{(k)}_{\ge s+1}$.

Now we define a family of tableaux which will play an important role for type $\mathfrak D$ constructions.

\begin{definition} \label{def: spin rigid}
For $s,m \in \Z_{\ge 0}$ with  $m \ge s-1$, let $T$ be a skew Young tableau of shape $\mu \skewpar (s^{k-1})$  with $m$ cells for a partition $\mu$ of length $k$, which is identified with the sequence of strict partitions
$$\ula=(\lambda^{(1)},\lambda^{(2)},\ldots,\lambda^{(k-1)},\lambda^{(k)}) \quad \text{ with } \lambda_i \seteq \ell(\lambda^{(i)}), \ i=1, \dots, k.$$
Then $T$ is called a {\it spin rigid Young tableau of index $(m,s)$
with $k$ rows} if
it satisfies the following conditions:
\begin{enumerate}
\item[{\rm (a)}] $(\lambda_1,\lambda_2, \dots, \lambda_{k-1} , \lambda_k+s) \Vdash_0 m+s$,
\item[{\rm (b)}] if $s \ge 2$, then  $\lambda^{(k-1)} \not\supset \lambda^{(k)}_{\ge s-1}$.
\end{enumerate}
We denote by ${}_s\sS_m^{(k)}$ the set of all spin rigid Young tableaux of index $(m,s)$ with $k$ rows. In particular,
${}_0\sS_m^{(k)}=\Ae^{(k)}_m$ and hence ${}_0\sS_{2m-1}^{(2)}= \Ss_{2m-1}^{(2)}$. (See Remark \ref{rmk-ae}.)
\end{definition}

Note that the condition (b) implies $\ell(\lambda^{(k)}) \ge \max\{0,s-1\}$.  The condition (b) says that a shift of the last row to the right by $2$ makes the tableau violate the column-strictness.  The condition (a) naturally arises when we connect a spin rigid tableau with a staircase dominant maximal weight through a tensor product of Young walls. See Lemma \ref{lem: cont ht} below.

We will color the columns of a spin rigid Young tableau in white and gray as follows to indicate the corresponding columns of Young walls starting from $0$-blocks and $1$-blocks.

\medskip
\emph{The first column of spin rigid Young tableaux $T \in {}_{2s}\sS_m^{(k)}$ is colored in white
while the first column of spin rigid Young tableaux $T \in {}_{2s+1}\sS_m^{(k)}$ is colored in gray.}
\medskip

\begin{example} \hfill

(1) We have
$$ T =
\begin{ytableau}
 \cdot& *(gray!40) \cdot&  4& *(gray!40) 2&  1  \\
 \cdot& *(gray!40) \cdot\\
 3
\end{ytableau} \in {}_2\sS_4^{(3)}, \quad \text{ since } \quad
\begin{ytableau}
 \cdot& *(gray!40) \cdot&  4& *(gray!40) 2&  1  \\
 \cdot& *(gray!40) \cdot\\
\cdot& *(gray!40) \cdot & 3
\end{ytableau} \quad
\text{ is not a skew tableau. }
$$
Here $T$ corresponds to $\ula=((4,2,1),(0),(3))$.

The set ${}_2\sS_4^{(3)}$ consists of the following $15$ spin rigid Young tableaux:
\begin{align*}
\ytableausetup{smalltableaux}
&\begin{ytableau}
 \cdot& *(gray!40) \cdot&  3& *(gray!40) 2&  1  \\
 \cdot& *(gray!40) \cdot\\
 4
\end{ytableau}, \
\begin{ytableau}
 \cdot& *(gray!40) \cdot&  4& *(gray!40) 2&  1  \\
 \cdot& *(gray!40) \cdot\\
 3
\end{ytableau}, \
\begin{ytableau}
 \cdot& *(gray!40) \cdot&  4& *(gray!40) 3&  1  \\
 \cdot& *(gray!40) \cdot\\
 2
\end{ytableau}, \
\begin{ytableau}
 \cdot& *(gray!40) \cdot&  4& *(gray!40) 3&  2  \\
 \cdot& *(gray!40) \cdot\\
 1
\end{ytableau}, \
\begin{ytableau}
 \cdot& *(gray!40) \cdot&  3& *(gray!40) 1\\
 \cdot& *(gray!40) \cdot&  2\\
 4
\end{ytableau}, \
\begin{ytableau}
 \cdot& *(gray!40) \cdot&  4& *(gray!40) 2\\
 \cdot& *(gray!40) \cdot&  1\\
 3
\end{ytableau}, \
\begin{ytableau}
 \cdot& *(gray!40) \cdot&  4& *(gray!40) 3\\
 \cdot& *(gray!40) \cdot&  1\\
 2
\end{ytableau},  \allowdisplaybreaks\\
&
\begin{ytableau}
 \cdot& *(gray!40) \cdot&  3& *(gray!40) 2\\
 \cdot& *(gray!40) \cdot&  1\\
 4
\end{ytableau}, \
\begin{ytableau}
 \cdot& *(gray!40) \cdot&  4& *(gray!40) 1\\
 \cdot& *(gray!40) \cdot&  2\\
 3
\end{ytableau}, \
\begin{ytableau}
 \cdot& *(gray!40) \cdot&  4\\
 \cdot& *(gray!40) \cdot&  3\\
 2 & *(gray!40) 1
\end{ytableau}, \
\begin{ytableau}
 \cdot& *(gray!40) \cdot&  4\\
 \cdot& *(gray!40) \cdot&  2\\
 3 & *(gray!40) 1
\end{ytableau}, \
\begin{ytableau}
 \cdot& *(gray!40) \cdot&  3\\
 \cdot& *(gray!40) \cdot&  2\\
 4 & *(gray!40) 1
\end{ytableau}, \
\begin{ytableau}
 \cdot& *(gray!40) \cdot&  2\\
 \cdot& *(gray!40) \cdot&  1\\
 4 & *(gray!40) 3
\end{ytableau}, \
\begin{ytableau}
 \cdot& *(gray!40) \cdot&  4\\
 \cdot& *(gray!40) \cdot&  1\\
 3 & *(gray!40) 2
\end{ytableau}, \
\begin{ytableau}
 \cdot& *(gray!40) \cdot&  3\\
 \cdot& *(gray!40) \cdot&  1\\
 4 & *(gray!40) 2
\end{ytableau}.
\end{align*}

(2) The set ${}_3\sS_4^{(3)}$ consists of the following $10$ spin rigid Young tableaux:
\begin{align*}
\ytableausetup{smalltableaux}
& \begin{ytableau}
*(gray!40) \cdot&  \cdot& *(gray!40) \cdot&  2& *(gray!40) 1 \\
*(gray!40) \cdot&  \cdot& *(gray!40) \cdot  \\
*(gray!40) 4&  3
\end{ytableau}, \
\begin{ytableau}
*(gray!40) \cdot&  \cdot& *(gray!40) \cdot&  3& *(gray!40) 2 \\
*(gray!40) \cdot&  \cdot& *(gray!40) \cdot  \\
*(gray!40) 4&  1
\end{ytableau}, \
\begin{ytableau}
*(gray!40) \cdot&  \cdot& *(gray!40) \cdot&  3& *(gray!40) 1 \\
*(gray!40) \cdot&  \cdot& *(gray!40) \cdot  \\
*(gray!40) 4&  2
\end{ytableau}, \
\begin{ytableau}
*(gray!40) \cdot&  \cdot& *(gray!40) \cdot&  4& *(gray!40) 1 \\
*(gray!40) \cdot&  \cdot& *(gray!40) \cdot  \\
*(gray!40) 3&  2
\end{ytableau}, \
\begin{ytableau}
*(gray!40) \cdot&  \cdot& *(gray!40) \cdot&  4& *(gray!40) 3 \\
*(gray!40) \cdot&  \cdot& *(gray!40) \cdot  \\
*(gray!40) 2&  1
\end{ytableau} \
 \begin{ytableau}
*(gray!40) \cdot&  \cdot& *(gray!40) \cdot&  4& *(gray!40) 2 \\
*(gray!40) \cdot&  \cdot& *(gray!40) \cdot  \\
*(gray!40) 3&  1
\end{ytableau}, \
\begin{ytableau}
*(gray!40) \cdot&  \cdot& *(gray!40) \cdot&  1 \\
*(gray!40) \cdot&  \cdot& *(gray!40) \cdot  \\
*(gray!40) 4&  3 & *(gray!40) 2
\end{ytableau}, \
\begin{ytableau}
*(gray!40) \cdot&  \cdot& *(gray!40) \cdot&  2 \\
*(gray!40) \cdot&  \cdot& *(gray!40) \cdot  \\
*(gray!40) 4&  3 & *(gray!40) 1
\end{ytableau}, \
\begin{ytableau}
*(gray!40) \cdot&  \cdot& *(gray!40) \cdot&  3 \\
*(gray!40) \cdot&  \cdot& *(gray!40) \cdot  \\
*(gray!40) 4&  2 & *(gray!40) 1
\end{ytableau}, \
\begin{ytableau}
*(gray!40) \cdot&  \cdot& *(gray!40) \cdot&  4 \\
*(gray!40) \cdot&  \cdot& *(gray!40) \cdot  \\
*(gray!40) 3&  2 & *(gray!40) 1
\end{ytableau}. \
\end{align*}
\end{example}

When $\Uplambda= (k-2+\delta_{s,1})\ofw_0+\Lambda_{2s-1}$, the crystal $\mathbf{B}(\Uplambda)$ is embedded into $\mathcal{Z}(\ofw_0)^{\otimes k-1} \otimes \mathcal{Z}(\ofw_1)$, and when $\Uplambda = (k-2)\ofw_0+(1+\delta_{s,0})\Lambda_{2s}$, the crystal
$\mathbf{B}(\Uplambda)$ is embedded into $\mathcal{Z}(\ofw_0)^{\otimes k}$. Hence we use  gray  color to distinguish the columns of Young walls starting with $1$-blocks 
with those starting with $0$-blocks.  For example, we have
\begin{align}
& \resizebox{1.6cm}{1.2cm}{\xy (0,0)*++{ \hwLfivefourB}\endxy}  \otimes  \resizebox{1.2cm}{1.6cm}{\xy (0,3)*++{ \hwLtwooneB}\endxy}
\longleftrightarrow \
\YW^{(4,3,2)}_{\ofw_0} \otimes \YW^{(5,1)}_{\ofw_1}
\longleftrightarrow \
 \begin{ytableau}
*(gray!40) \cdot &  4& *(gray!40) 3&  2 \\
*(gray!40) 5&  1
\end{ytableau} \ \in {}_1\sS_5^{(2)}, \allowdisplaybreaks \label{eq: spin st ex1}\\
& \resizebox{1.6cm}{1.2cm}{\xy (0,0)*++{ \hwLfivefourB}\endxy}  \otimes  \resizebox{1.2cm}{1.6cm}{\xy (0,3)*++{ \hwLtwooneBp}\endxy}
\longleftrightarrow \
\YW^{(4,3,2)}_{\ofw_0} \otimes \YW^{(5,1)}_{\ofw_0} \longleftrightarrow \
 \begin{ytableau}
 \cdot& *(gray!40) \cdot&  4 & *(gray!40) 3&  2\\
 5& *(gray!40) 1
\end{ytableau} \
 \in {}_2\sS_5^{(2)} . \label{eq: spin st ex2}
\end{align}
Note that the cells filled with white (resp. gray) color represent the columns starting with $0$-blocks (resp. $1$-blocks).
In \eqref{eq: spin st ex2}, we use
$$\young(\cdot\cdot432,51) \quad \text{ instead of } \quad \young(\cdot432,51)$$
so that each column of the tableaux may have the same color.

Let
$$\uofw=\begin{cases} (\ofw_0,\ldots,\ofw_0,\ofw_0) & \text{ if $s$ is even},\\ (\ofw_0,\ldots,\ofw_0,\ofw_1) & \text{ if $s$ is odd}. \end{cases}$$
The following lemma follows from the definitions of ${}_s\sS_m^{(k)}$ and $\tmaxp(\Uplambda|k)$:

\begin{lemma} \label{lem: cont ht}
Let $s,m \in \Z_{\ge 0}$ with  $n \ge m \ge s-1$, and $\Uplambda = (k-2+\delta_{s,0}+\delta_{s,1})\ofw_0+\Lambda_{s}$, $k \ge 2$. Then, for $T \in {}_s\sS_m^{(k)}$, we have
\[ \cont(\yy_{\uofw}^T) =
\begin{cases}
\cont(\YW_{\ofw_1}^{\uplambda(m)})-(\alpha_1-\alpha_0) & \text{ if $s \equiv_2 m$}, \\
\cont(\YW_{\ofw_0}^{\uplambda(m)})  & \text{ otherwise},\\
\end{cases} \]
and the tableau $T$ is associated with $\eta \in \tmaxp(\Uplambda|k)$ of index $(m,s-1)$ such that
\begin{align*} \cont(\yy_{\uofw}^T)- \cont(\YW_{\ofw_0}^{\uplambda(s-1)}) & = \Uplambda-\eta.
\end{align*}
\end{lemma}

Recall the set of indices for $\tmaxp(\Uplambda|k)$ in \eqref{eq: indices dmsk}.  The following is the main theorem of this subsection:

\begin{theorem} \label{thm: level k ht spin rigid}
Assume that $k\ge 2$. Then, for $\eta \in \tmaxp(\Uplambda|k)$ of index $(m,s-1)$, we have
$$ \dim V(\Uplambda)_\eta = |{}_s\sS_m^{(k)}|=\dim L \big( (k-2)\omega_n+\tilde{\omega}_{n-s} \big)_\mu,$$
where the definition of $\tilde{\omega}_s$ is given in \eqref{eq: bmega D}
and the weights $\mu$ are given by
\[  \mu = \begin{cases}
(k-2)\omega_n + \tilde{\omega}_{n-m-1} & \text{if } k=2, \text{ or } k \ge 3 \text{ and } m \not \equiv_2 s, \\
(k-3)\omega_n +\omega_{n-1}+ \tilde{\omega}_{n-m-1} & \text{if } k \ge 3 \text{ and } m \equiv_2 s.
\end{cases} \]
\end{theorem}

\begin{proof}
Let $\eta, \eta' \in \tmaxp(\Uplambda|k)$ of index $(m,s-1)$.
If $\eta$ is associated with $(\uplambda(m), \uplambda(s-1))$ and $\eta'$ with $((n)*\uplambda(m-1), \uplambda(s-1))$, one can see that $\dim V(\Uplambda)_\eta = \dim V(\Uplambda)_{\eta'}$ by replacing
the role of $(n)*\uplambda(m-1)$ with that of $\uplambda(m)$ to construct a one-to-one correspondence between the corresponding sets of tensor products of Young walls. Thus we only need to consider $\eta$ associated with $(\uplambda(m), \uplambda(s-1))$.

Set
$$\boxed{\Uplambda}\seteq \begin{cases}
\ \ \boxed{(k-1)\ofw_0} \otimes \YW_{\ofw_0}^{\uplambda(s-1)} & \text{ if $s$ is even}, \\[1ex]
\ \ \boxed{(k-1)\ofw_0} \otimes \YW_{\ofw_1}^{\uplambda(s-1)} & \text{ if $s$ is odd}.
\end{cases}$$

Since  $m \le n$, a Young wall $\yy \in \mathbf{B}(\Uplambda)_\eta $ connected to $\boxed{\Uplambda}$
cannot contain a removable $\updelta$.
Hence Lemma \ref{lem: cont ht} tells us that $\yy \in \mathbf{B}(\Uplambda)_\eta $ corresponds to a
sequence of strict partitions $\ula=(\lambda^{(1)},\lambda^{(2)},\ldots,\lambda^{(k-1)},\lambda^{(k)})$ satisfying
the condition {\rm (a)} in Definition \ref{def: spin rigid}:
$$ \yy = \yy^\ula_{\uofw} \quad \text{ where } \uofw=\begin{cases} (\ofw_0,\ldots,\ofw_0,\ofw_0) & \text{ if $s$ is even},\\ (\ofw_0,\ldots,\ofw_0,\ofw_1) & \text{ if $s$ is odd}. \end{cases} $$

Note that if $\ell(\lambda^{(k)}) < \max\{0,s-1\}$, then
$ \yy$ cannot be connected to $\boxed{\Uplambda}$.
Now the condition  {\rm (b)} in Definition \ref{def: spin rigid}
follows to represent the columns of Young walls starting with $1$-blocks  from Proposition \ref{prop: Talbelaux ht in smax} and Theorem \ref{thm: level 2 ht}.
\end{proof}

We record the special case $s=0$ as a corollary for reference to be used later.

\begin{corollary} \label{cor: ht AE s=0 enumeration}
The numbers $|\Ae_m^{(k)}|$ of almost even tableaux of $m$ with at most $k$ rows
are the multiplicities of dominant maximal weights for $V(k\ofw)$ and hence the multiplicities of dominant weights for $V(k\omega_n)$.
\end{corollary}

For the rest of this subsection, we investigate  relationship between ${}_0\sS^{(k)}_{m}$ and ${}_1\sS^{(k)}_{m-1}$, which will be used in Section \ref{sec-level3}.  Set $\Uplambda =(k-1) \ofw_0+ \ofw_1$ for $k \ge 3$.
The crystal $\mathbf{B}(\Uplambda)$ can also be realized by the subcrystal of $\mathcal{Z}(\ofw_1) \otimes \mathcal{Z}(\ofw_0)^{\otimes k-1}$ (as opposed to $\mathcal{Z}(\ofw_0)^{\otimes k-1} \otimes \mathcal{Z}(\ofw_1)$)
connected to $\boxed{\ofw_1} \otimes \boxed{(k-1)\ofw_0}$. By applying the argument in this subsection, one can prove that
the crystal basis of $V(\Uplambda)_\eta$ for $\eta \in \tmaxp((k-1) \ofw_0+ \ofw_1| k)$ is realized by
\begin{equation} \label{Rmk: setnimus}
{}_0\sS^{(k)}_{m} \setminus \ytableausetup{smalltableaux} \begin{ytableau}  m \end{ytableau} \seteq \{ T \setminus \ytableausetup{smalltableaux} \begin{ytableau}  m \end{ytableau} \ |
 \ T \in {}_0\sS^{(k)}_{m} \} ,
\end{equation}
where $\eta \in {\rm smax}^+(\Uplambda|k)$ is of index $(m-1,0)$ and $T \setminus \begin{ytableau}  m \end{ytableau} $ is the tableau obtained by removing
the cell $\ytableausetup{smalltableaux} \begin{ytableau}  m \end{ytableau}$ located in the position $(1,1)$. For example, when $m=6$ and $k=3$,
$$
\resizebox{1.2cm}{1.2cm}{\xy (0,0)*++{ \hwLfivefourBs}\endxy}  \otimes  \resizebox{1.3cm}{1.6cm}{\xy (0,3)*++{ \hwLtwooneBt}\endxy}  \otimes  \resizebox{0.8cm}{0.5cm}{\xy (0,0)*++{ \hst}\endxy}
\ \longleftrightarrow \
T \setminus \begin{ytableau}  6 \end{ytableau} =
\ytableausetup{nosmalltableaux}
\begin{ytableau}
 \cdot& *(gray!40) 4&  3 \\
 5& *(gray!40) 2\\
 1
\end{ytableau} \quad
\text{ where }
T = \ytableausetup{nosmalltableaux}
\begin{ytableau}
 6& *(gray!40) 4&  3 \\
 5& *(gray!40) 2 \\
 1
\end{ytableau}\in {}_0\sS^{(3)}_{6}.
$$
One the other hand, by Theorem \ref{thm: level k ht spin rigid}, the crystal basis of $V(\Uplambda)_\eta$  is also realized by ${}_1\sS^{(k)}_{m-1}$ consisting of spin rigid Young tableaux.

Hence we can conclude that
$$
|{}_0\sS^{(k)}_{m}|=|{}_0\sS^{(k)}_{m} \setminus \ytableausetup{smalltableaux} \begin{ytableau}  m \end{ytableau} | = |{}_1\sS^{(k)}_{m-1} |,$$
which will explain the correspondence with the equation $\mathsf{R}_{(m,0)} = \mathsf{R}_{(m-1,1)}$ in \eqref{eq: Riordan triangle recursive}
(see Section \ref{sec-level3} below).

\begin{example}
The set ${}_0\sS^{(3)}_{4} \setminus \ytableausetup{smalltableaux} \begin{ytableau}  4 \end{ytableau}$  is given as follows:
\begin{align*}
\ytableausetup{smalltableaux}
&\begin{ytableau}
 \cdot& *(gray!40) 3\\
 2 \\
 1
\end{ytableau}, \
\begin{ytableau}
 \cdot& *(gray!40) 2\\
 3 \\
 1
\end{ytableau}, \
\begin{ytableau}
 \cdot& *(gray!40) 1\\
 3 \\
 2
\end{ytableau}, \
\begin{ytableau}
 \cdot& *(gray!40) 2 & 1\\
 3
\end{ytableau}, \
\begin{ytableau}
 \cdot& *(gray!40) 3 & 1\\
 2
\end{ytableau}, \
\begin{ytableau}
 \cdot& *(gray!40) 3 & 2\\
 1
\end{ytableau}.
\end{align*}
On the other hand, the set ${}_1\sS^{(3)}_{3}$ is given as follows:
\begin{align*}
\ytableausetup{smalltableaux}
& \begin{ytableau}
*(gray!40) \cdot&  3& *(gray!40) 1  \\
*(gray!40) \cdot&  2
\end{ytableau}, \
\begin{ytableau}
*(gray!40) \cdot&  3& *(gray!40) 2 \\
*(gray!40) \cdot&  1
\end{ytableau}, \
\begin{ytableau}
*(gray!40) \cdot&  3& *(gray!40) 2  &  1
\end{ytableau}, \
\begin{ytableau}
*(gray!40) \cdot&  3\\
*(gray!40) \cdot&  2 \\
*(gray!40) 1
\end{ytableau}, \
\begin{ytableau}
*(gray!40) \cdot&  3\\
*(gray!40) \cdot&  1 \\
*(gray!40) 2
\end{ytableau}, \
\begin{ytableau}
*(gray!40) \cdot&  2\\
*(gray!40) \cdot&  1 \\
*(gray!40) 3
\end{ytableau}.
\end{align*}
\end{example}

We make a summary of the observation made above as a corollary:
\begin{corollary}  \label{cor: ht kLambda0+Lambda_1}
Set $\Uplambda =(k-1) \ofw_0+ \ofw_1$ for $k \ge 2$.
Then the number of the almost even tableaux of  $m \ge 1$ with at most $k$ rows appears as the multiplicity of a maximal weight $\eta \in {\rm smax}^+(\Uplambda|k)$ of index $(m-1,0)$. That is, we have
$$| _0\Ae_m^{(k)}| =|{}_1\sS^{(k)}_{m-1} |=\dim (V(\Uplambda)_\eta). $$
\end{corollary}

\begin{remark}
Explicit formulas for the numbers $|\Ae_m^{(k)}|$ for $1 \le k \le 5$ will be given in Theorem  \ref{thm:card_S}.
Thus we know explicitly  the multiplicities of $\eta \in {\rm smax}^+(\Uplambda|k)$ of indices $(m,-1)$ and $(m-1,0)$  for $1 \le k \le 5$.
\end{remark}

\section{Level $2$ weight multiplicities: Catalan and Pascal triangles}
\label{sec:level 2}
In this section, we prove that all the multiplicities
of the (staircase) dominant maximal  weights of level $2$ are generalized Catalan numbers or binomial coefficients. As will be indicated in Section \ref{sec-class}, the results can be obtained through classical constructions. We will provide a different proof, which  utilizes a new insertion scheme for (spin) rigid Young tableaux and makes  the Catalan and Pascal triangles compatible with the insertion scheme. This insertion scheme will naturally generalize in the next section to the case of level $3$ weights, where classical constructions do not easily generalize.

\subsection{Classical realizations} \label{sec-class}

Now we restate and give an alternative proof for \cite[Theorem 1.4 (ii)]{Ts}, which was on the affine type $A^{(1)}_{n-1}$:

\begin{theorem} $($cf. \cite[Theorem 1.4 (ii)]{Ts}$)$  For finite type $A_{n-1}$, we have
$$\dim L(\omega_{t}+\omega_{t+s})_{\omega_{t-k}+\omega_{t+s+k}}=\mathsf{C}_{(s+2k,s)} \quad \text{ for } 0 \le k \le t,$$ where $\mathsf{C}_{(m,s)}$ are generalized Catalan numbers.
\end{theorem}

\begin{proof}
By Kashiwara--Nakashima realization (\cite{KN94}) of the crystal basis for $\mathbf{B}(\omega_{t}+\omega_{t+s})$ via semi-standard tableaux filled with $1,2,\ldots,n$, the dimension
$\dim L(\omega_{t}+\omega_{t+s})_{\omega_{t-k}+\omega_{t+s+k}}$ is the same as the number of semi-standard tableaux $T$ (the convention for
semi-standard tableaux in \cite{KN94} is different from ours) satisfying the following conditions:
\begin{itemize}
\item $\Sh(T)=(2^t,1^s)$,
\item for every $1\le i\le t-k$, the two cells in the $i$-th row are filled with $i$,
\item the remaining $2k+s$ cells are filled with the distinct numbers
  $ t-k+1, t-k+2,\ldots, t+k+s$.
\end{itemize}
Hence Remark \ref{rem: Catalan model} implies our assertion.
\end{proof}

In Section \ref{subsec-An1}, we showed that every dominant maximal weight of a highest weight $\Lambda$ of level $2$ is essentially finite of type $A_{n-1}$. Thus we obtain the following corollary:

\begin{corollary} For finite type $A_{n-1}$, assume that $\eta \in \mx^+(\Lambda|2)$. Then the multiplicity of $\eta$ is a generalized Catalan number.
\end{corollary}

Generalized Catalan numbers also appear for type $C_n$ as one can see in the following theorem.
\begin{theorem} \label{thm: multiplicity C level 1}
For finite type $C_n$, $1 \le s \le n$ and $0 \le i \le \lfloor \frac s 2 \rfloor$, we have
$$\dim L(\omega_s)_{\omega_{s-2i}} = \mathsf{C}_{(n-s+2i,n-s)}.$$
\end{theorem}

\begin{proof}
This is a consequence of the exterior power realization of the
fundamental representation (see \cite[Theorem 17.5]{FH}) since
\[
\mathsf{C}_{(n-s+2i,n-s)} = \binom{n-(s-2i)}{i}-\binom{n-(s-2i)}{i-1}.  \qedhere
\]
\end{proof}

In Section \ref{subsec-Cn1}, we showed that every dominant maximal weight of a highest weight $\ofw_s$ of level $1$ over type $C_n^{(1)}$ is essentially finite of type $C_n$. For types $A_{2n-1}^{(2)}$ and $A_{2n}^{(2)}$, we determined dominant maximal weights which are essentially finite of type $C_n$. See Remarks \ref{rem: Type Dn} and \ref{rem: Type A2n2}. Thus we obtain the following corollary:

\begin{corollary} Assume that $\eta$ is a dominant maximal weight which is essentially finite of type $C_n$ for a highest weight $\Lambda$ of level $1$ over type $C_n^{(1)}$ or of level $2$ over type $A_{2n-1}^{(2)}$ or $A_{2n}^{(2)}$. Then the multiplicity of $\eta$ is a generalized Catalan number.
\end{corollary}

The following theorem shows that binomial coefficients appear as weight multiplicities for finite types $B_n$ and $D_n$.

\begin{theorem} \label{thm: multiplicity B,D level 2}
For $1 \le s \le n$, we have
$$ \begin{cases}
\dim L(\tilde{\omega}_s)_{\tilde{\omega}_{k}} = \matr{n-k}{\lfloor \frac{s-k}{2} \rfloor} & \text{ if $L(\tilde{\omega}_s)$ is over $B_n$,} \vspace*{0.2cm} \\
\dim L(\tilde{\omega}_s)_{\tilde{\omega}_{k}} = \matr{n-k -\delta_{n,s}}{\frac{s-k}2}  & \text{ if $L(\tilde{\omega}_s)$ is over $D_n$ and $s \equiv_2 k$.}
\end{cases}$$
\end{theorem}

\begin{proof}
By the exterior power realization of the
fundamental representation in \cite[Theorem 19.2, Theorem 19.14]{FH}, one can prove this assertion.
\end{proof}

 We remark
here that it seems difficult in general to prove the above results using the Kashiwara--Nakashima realization for finite types $B_n$ and $D_n$.

Though we can use Theorem \ref{thm: multiplicity B,D level 2} to describe the multiplicities of maximal weights in $\hmaxp(\Lambda|2)$ and $\tmaxp(\Lambda|2)$, we will develop a new method in the next subsections for the reason mentioned at the beginning of this section.

\subsection{Insertion of a box} \label{sec-insertion}
\begin{definition}
Let $\ula=(\lambda^{(1)},\ldots,\lambda^{(k)})$ be a sequence of strict partitions  with $\dbs{j=1}{k}\lambda^{(j)}=\uplambda(m-1)$. For $1 \le  u  \le k$,
we define the {\em insertion of $(m)$ into the $u$-th partition} by
$$ \ula \pls{u} (m)= ({\lambda'}^{(1)},\ldots,{\lambda'}^{(k)})  $$
where
$$
\begin{cases}
{\lambda'}^{(j)}=\lambda^{(j)} &\text{ if } j \ne u, \\
{\lambda'}^{(u)}=(m)*\lambda^{(u)} &\text{ if } j = u.
\end{cases}
$$
Then $\ula \pls{u} (m)=({\lambda'}^{(1)},\ldots,{\lambda'}^{(k)}) $ is a new sequence of strict partitions with $\dbs{j=1}{k}{\lambda'}^{(j)}=\uplambda(m)$.
\end{definition}

The operation $\pls{u} (m)$ is to be understood as an {\it insertion} of the box $\,\young(m)\, $ into the $u$-th row of a skew-tableaux. For example, we have
\begin{align*}
\young(7542,631) \pls{1} (8) = \young({{\color{red}8}}7542,631)
\end{align*}

\subsection{Case  $\hmaxp(\Lambda|2)$} \label{sec-hhh}
We start with a simple observation.
For $T=(\lambda^{(1)},\ldots,\lambda^{(k)}) \in {}_s\Ss^{(k)}_m$, the number $m$ can  only appear as the first part of the first partition or as the first part of the last partition. That is, we have
\begin{equation} \label{observation}
m = \begin{cases}
\lambda^{(1)}_1 \text{ or } \lambda^{(k)}_1 & \text{ if } s \ge 1, \\
\lambda^{(1)}_1 & \text{ if } s=0.
\end{cases}
\end{equation}

\begin{example} \label{ex: pascal hh 1} \hfill

(1) ${}_1\Ss^{(2)}_5$ consists of the following $10$ rigid Young tableaux:
\begin{align*}
\ytableausetup{smalltableaux}
& \begin{ytableau}
\cdot &  4 &  3&  2 & 1 \\
5
\end{ytableau}, \
 \begin{ytableau}
\cdot &  4 &  2 & 1 \\
5&  3
\end{ytableau}, \
 \begin{ytableau}
\cdot &  4 &  3 & 1 \\
5&  2
\end{ytableau}, \
 \begin{ytableau}
\cdot &  4 &  3 & 2 \\
5&  1
\end{ytableau}, \
 \begin{ytableau}
\cdot &  5 &  2 & 1 \\
4&  3
\end{ytableau}, \
 \begin{ytableau}
\cdot &  4 &  3  \\
5&  2 & 1
\end{ytableau}, \
 \begin{ytableau}
\cdot &  5 &  3  \\
4&  2 & 1
\end{ytableau}, \
 \begin{ytableau}
\cdot &  4 &  2  \\
5&  3 & 1
\end{ytableau}, \
 \begin{ytableau}
\cdot &  5 &  4  \\
3&  2 & 1
\end{ytableau}, \
 \begin{ytableau}
\cdot &  5 &  2  \\
4&  3 & 1
\end{ytableau}.
\end{align*}

(2) ${}_3\Ss^{(2)}_5$ consists of the following $5$ rigid Young tableaux:
\begin{align*}
\ytableausetup{smalltableaux}
& \begin{ytableau}
\cdot &  \cdot &  \cdot&  2 & 1 \\
5 & 4 & 3
\end{ytableau}, \
\begin{ytableau}
\cdot &  \cdot &  \cdot&  5 \\
4 & 3 & 2 & 1
\end{ytableau}, \
\begin{ytableau}
\cdot &  \cdot &  \cdot&  4 \\
5 & 3 & 2 & 1
\end{ytableau}, \
\begin{ytableau}
\cdot &  \cdot &  \cdot&  3 \\
5 & 4 & 2 & 1
\end{ytableau}, \
\begin{ytableau}
\cdot &  \cdot &  \cdot&  2 \\
5 & 4 & 3 & 1
\end{ytableau}, \
\end{align*}
\end{example}

\begin{lemma} \label{lem: insertion hh 2}
For $T=(\lambda,\mu) \in {}_s\Ss^{(2)}_{m-1}$, we have
$$T \pls{1} (m) \in {}_{s-1}\Ss^{(2)}_{m} \quad \text{ and } \quad T \pls{2} (m) \in {}_{s+1}\Ss^{(2)}_{m}.$$
\end{lemma}

\begin{proof}
Recall that $(\lambda,\mu) \in {}_s\Ss^{(2)}_{m-1}$ for $s\ge 1$ implies
$$ {\rm (i)} \ \lambda_i < \mu_{s+i-1} \text{ and } \lambda_i > \mu_{s+i} \text{ for some } 1 \le i \le \ell(\lambda) \quad \text{ or } \quad {\rm (ii)} \ \ell(\mu)-s = \ell(\lambda).$$
Since
$ ((m)*\lambda)_1 =m$, $((m)*\lambda)_i =\lambda_{i+1} $ and $\ell((m)*\lambda)=\ell(\lambda)+1$,
we can conclude that $$T \pls{1} (m) \in {}_{s-1}\Ss^{(2)}_{m}.$$
Similarly, the facts that $ ((m)*\mu)_1 =m$, $((m)*\mu)_i=\mu_{i+1}$ and $\ell((m)*\mu)=\ell(\mu)+1$ implies
\[
T \pls{2} (m) \in {}_{s+1}\Ss^{(2)}_{m}.  \qedhere
\]
\end{proof}

\begin{remark}
For $m \in \mathbb Z_{\ge 1}$, the sets ${}_m\Ss^{(2)}_m$ and ${}_m\Ss^{(2)}_{m+1}$ are described as follows:
\begin{align}
{}_m\Ss^{(2)}_m = \{ ((0),\uplambda(m)) \} \quad \text{ and } \quad {}_m\Ss^{(2)}_{m+1} = \{ ((m+1),\uplambda(m)) \}.
\end{align}
Hence  $|{}_m\Ss^{(2)}_m|=|{}_m\Ss^{(2)}_{m+1}|=1$.
\end{remark}

Let $L({\omega})$ be the highest weight module with highest weight ${\omega}$ over the finite dimensional Lie algebra of type $B_n$. Recall the definition of $\tilde{\omega}_s$ in \eqref{eq: bmega B}.
\begin{theorem} \label{them: pascal hh}
Let $\eta \in \hmaxp(\Lambda|2)$ of index $(m,s)$.
For every $s \le m$,
$$|{}_{s}\Ss^{(2)}_{m}| = \binom{m}{\left\lfloor \frac{m-s}{2} \right\rfloor}=\dim V(\Lambda)_\eta = \dim L(\tilde{\omega}_{n-s})_{\tilde{\omega}_{n-m}}.$$
\end{theorem}

\begin{proof}
By \eqref{observation}, for each $T=(\lambda,\mu) \in {}_{s}\Ss^{(2)}_{m}$  with $s \ge 1$, we have
$$ \lambda_1 =m \quad \text{ or } \quad \mu_1=m.$$
Thus
$$T = T_1 \pls{1} (m) \quad \text{ or } \quad T = T_2 \pls{2} (m)$$
for some $T_1 \in {}_{s+1}\Ss^{(2)}_{m-1}$ or $T_2 \in {}_{s-1}\Ss^{(2)}_{m-1}$  respectively.
Particularly, $T \in \Ss^{(2)}_m$ is of the form $T' \pls{1} (m)$ for some $T' \in \Ss^{(2)}_{m-1} \sqcup  {}_{1}\Ss^{(1)}_{m-1}$.
Since the sets $\big( {}_{s+1}\Ss^{(2)}_{m-1} \big) \pls{1} (m)$ and $\big({}_{s-1}\Ss^{(2)}_{m-1}\big) \pls{2} (m)$ are distinct,
our assertion follows from
$$|\Ss_m^{(2)}|=|{}_0\Ss_m^{(2)}| =\binom{m}{ \lfloor \frac{m}{2} \rfloor}, \qquad |{}_m\Ss_m^{(2)}|=\binom{m}{ \lfloor \frac{m-m}{2} \rfloor}=1=\binom{m}{ \lfloor \frac{m+1-m}{2} \rfloor}= |{}_m\Ss_{m+1}^{(2)}|$$
and
$$\left(  {}_{s-1}\Ss^{(2)}_{m-1}  \pls{2} (m) \right)  \bigsqcup \left( {}_{s+1}\Ss^{(2)}_{m-1}  \pls{1} (m) \right) = {}_{s}\Ss^{(2)}_{m}   $$
corresponding to $\binom{n}{k}=\binom{n-1}{k}+\binom{n-1}{k-1}$.
\end{proof}

The following lattice diagram  illustrates the above theorem and realizes the Pascal triangle:
\begin{align} \label{eq: pascal triange}
\raisebox{5em}{\scalebox{.9}{\xymatrix@R=0.1pc@C=1.8pc{ &&&& \cdots \\
&&& {\scriptsize \young(\cdot\cdot\cdot,321)}\ar@{=>}[ur] \ar[dr] &&\\
&& {\scriptsize \young(\cdot\cdot,21)} \ar[dr] \ar@{=>}[ur] \ar[dr] &&\cdots\\
& {\scriptsize \young(\cdot,1)} \ar@{=>}[ur] \ar[dr] && {\scriptsize \young(\cdot2,31)} \ {\scriptsize \young(\cdot21,3)} \ {\scriptsize \young(\cdot3,21)} \ar@{=>}[ur] \ar[dr]&&\\
\emptyset \ar@{=>}[ur] \ar[dr] && {\scriptsize \young(2,1)} \ {\scriptsize \young(21)} \ar@{=>}[ur] \ar[dr]  &&\cdots\\
& {\scriptsize \young(1)}\ar@{=>}[dr] \ar[ur] &&  {\scriptsize \young(32,1)} \ {\scriptsize \young(321)} \  {\scriptsize \young(31,2)}  \ar@{=>}[dr] \ar[ur] &&\\
&& {\scriptsize \young(\cdot1,2)}\ar@{=>}[dr] \ar[ur] \ar@{=>}[dr] \ar[ur] &&\cdots\\
&&& {\scriptsize \young(\cdot\cdot1,32)} \ar@{=>}[dr] \ar[ur] &&\\
&&&&\cdots
}}}
\end{align}
Here $\Rightarrow$ denotes insertion $\pls{2}$ into the second row (or
partition) and $\to$ denotes insertion $\pls{1}$ into the first one. By taking the cardinality of the tableaux at each position, we obtain the Pascal triangle.

\begin{example}
In Example \ref{ex: pascal hh 1}, we can see that
$$|{}_1\Ss^{(2)}_5| = \binom{5}{\left\lfloor \frac{5-1}{2} \right\rfloor}=\binom{5}{2}=10 \quad \text{ and } \quad |{}_3\Ss^{(2)}_5| = \binom{5}{\left\lfloor \frac{5-3}{2} \right\rfloor}=\binom{5}{1}=5.$$
Furthermore, we get $|{}_2\Ss^{(2)}_6|=10+5=\binom{6}{\left\lfloor \frac{6-2}{2} \right\rfloor}$ from the insertion scheme:
\begin{align*}
\ytableausetup{smalltableaux}
& \begin{ytableau}
\cdot & \cdot & 4 &  3&  2 & 1 \\
\six & 5
\end{ytableau}, \
\begin{ytableau}
\cdot & \cdot & 4 & 2 & 1 \\
\six & 5 & 3
\end{ytableau}, \
\begin{ytableau}
\cdot & \cdot & 4 & 3 & 1 \\
\six & 5 & 2
\end{ytableau}, \
\begin{ytableau}
\cdot & \cdot & 4 & 3 & 2 \\
\six & 5 & 1
\end{ytableau}, \
\begin{ytableau}
\cdot & \cdot & 5 & 2 & 1 \\
\six & 4 & 3
\end{ytableau}, \
\begin{ytableau}
\cdot & \cdot & 4 & 3 \\
\six & 5 & 2 & 1
\end{ytableau}, \
\begin{ytableau}
\cdot & \cdot & 5 & 3 \\
\six & 4 & 2 & 1
\end{ytableau}, \
\begin{ytableau}
\cdot & \cdot & 4 & 2 \\
\six & 5 & 3 & 1
\end{ytableau}, \
\begin{ytableau}
\cdot & \cdot & 5 & 4 \\
\six & 3 & 2 & 1
\end{ytableau}, \
\begin{ytableau}
\cdot & \cdot & 5 & 2 \\
\six & 4 & 3 & 1
\end{ytableau}, \\
& \hspace{25ex}
\begin{ytableau}
\cdot & \cdot & \six & 2 & 1 \\
5 & 4 & 3
\end{ytableau}, \
\begin{ytableau}
\cdot & \cdot & \six & 5 \\
4 & 3 & 2 & 1
\end{ytableau}, \
\begin{ytableau}
\cdot & \cdot & \six & 4 \\
5 & 3 & 2 & 1
\end{ytableau}, \
\begin{ytableau}
\cdot & \cdot & \six & 3 \\
5& 4 & 2 & 1
\end{ytableau}, \
\begin{ytableau}
\cdot & \cdot & \six & 2 \\
5 & 4 & 3 & 1
\end{ytableau}.
\end{align*}
%
\end{example}

\begin{corollary} For $m \ge s \ge 0$, set $$a=\lfloor (m-s)/2 \rfloor \quad \text{ and } \quad b=m-a.$$
We have a bijective map between $${}_{s}\Ss^{(2)}_{m} \quad \text{ and } \quad \mathfrak L(a,b),$$
where $\mathfrak L(a,b)$ denotes the set of paths in the Pascal triangle \eqref{eqn-Pas}
starting from $(0,0)$ to $(m,b-a)$ using the vectors $(1,1)$ and $(1,-1)$.
\end{corollary}

\begin{proof}
For $T \in {}_{s}\Ss_{m}^{(2)}$, we first assume that $s \equiv_2 m$. Then
we record the vector $v_m$ as
\begin{itemize}
\item $(1,1)$ if $\begin{cases} \text{$T = T' \pls{2} (m)$ for some $T'\in {}_{s-1}\Ss_{m-1}^{(2)}$ with $s \ge 1$, or} \\
\text{$T = T' \pls{1} (m)$ for some $T'\in \Ss_{m-1}^{(2)}$}, \end{cases}$
\item $(1,-1)$ if $T = T' \pls{1} (m)$ for some $T'\in {}_{s+1}\Ss_{m-1}^{(2)}$.
\end{itemize}
Now we assume that $s-1 \equiv_2 m$.
Then
we record the vector $v_m$ as
\begin{itemize}
\item $(1,-1)$ if $\begin{cases} \text{$T = T' \pls{2} (m)$ for some $T'\in {}_{s-1}\Ss_{m-1}^{(2)}$ with $s \ge 1$, or} \\
\text{$T = T' \pls{1} (m)$ for some $T'\in \Ss_{m-1}^{(2)}$}, \end{cases}$
\item $(1,1)$ if $T = T' \pls{1} (m)$ for some $T'\in {}_{s+1}\Ss_{m-1}^{(2)}$.
\end{itemize}
Then, by induction on $m$, we obtain the sequence of vectors $(v_1, v_2, \dots , v_m)$ corresponding to a path in the Pascal triangle.
\end{proof}

\begin{example} For
$$T  = \young(\cdot\cdot6532,8741) \in {}_{2}\Ss_{8}^{(2)},$$
we have $a=3$ and $b=5$.
Then the tableau $T$ corresponds to the following lattice path:
\begin{align*}
\begin{tikzpicture}[x=0.8cm,y=0.8cm, scale=0.7]
\foreach \x in {0,2,4,6,8}
\draw[shift={(\x-1,0)},color=black, ] (0pt,2pt) -- (0pt,-2pt) node[below] {\footnotesize $(\x,0)$};
\draw[line width=1 pt, ->] (-1,0)--(8,0);
\draw[orange, line width=1.5pt] (-1,0)--(0,1)--(3,-2)--(7,2);
\foreach \y in {-2,-1,1,2}
    \draw[dashed](-1,\y)--(8,\y);
\end{tikzpicture}
\end{align*}
\end{example}

\subsection{Case  $\tmaxp(\Lambda|2)$}
By Lemma \ref{Thm: depend on finite}, we may assume that $\g=B^{(1)}_{n}$
and
$$
\Lambda = (\delta_{s,0}+ \delta_{s,1})\ofw_0+\Lambda_s \quad (0 \le s \le n-1)
$$
throughout this subsection.

As in \eqref{observation}, the same property holds for $T=(\lambda^{(1)},\ldots,\lambda^{(k)}) \in {}_s\sS^{(k)}_m$ to have
$$
m = \begin{cases}
\lambda^{(1)}_1 \text{ or } \lambda^{(k)}_1 & \text{ if } s \ge 1, \\
\lambda^{(1)}_1 & \text{ if } s=0.
\end{cases}
$$

\begin{example} \label{ex: pascal ht 1} \hfill

(1) The set ${}_1\sS^{(2)}_4$ consists of the following $10$ spin rigid Young tableaux:
\begin{align*}
\ytableausetup{smalltableaux}
& \begin{ytableau}
*(gray!40) \cdot&  4& *(gray!40) 3&  2 &*(gray!40) 1
\end{ytableau}, \
\begin{ytableau}
*(gray!40) \cdot&  3& *(gray!40) 2&  1  \\
*(gray!40) 4
\end{ytableau}, \
\begin{ytableau}
*(gray!40) \cdot&  4& *(gray!40) 2&  1  \\
*(gray!40) 3
\end{ytableau}, \
\begin{ytableau}
*(gray!40) \cdot&  4& *(gray!40) 3&  1  \\
*(gray!40) 2
\end{ytableau}, \
\begin{ytableau}
*(gray!40) \cdot&  4& *(gray!40) 3&  2  \\
*(gray!40) 1
\end{ytableau}, \
\begin{ytableau}
*(gray!40) \cdot&  4& *(gray!40) 3  \\
*(gray!40) 2 &  1
\end{ytableau}, \
\begin{ytableau}
*(gray!40) \cdot&  4& *(gray!40) 2  \\
*(gray!40) 3 &  1
\end{ytableau}, \
\begin{ytableau}
*(gray!40) \cdot&  3& *(gray!40) 2  \\
*(gray!40) 4 &  1
\end{ytableau}, \
\begin{ytableau}
*(gray!40) \cdot&  4& *(gray!40) 1  \\
*(gray!40) 3 &  2
\end{ytableau}, \
\begin{ytableau}
*(gray!40) \cdot&  3& *(gray!40) 1  \\
*(gray!40) 4 &  2
\end{ytableau}. \
\end{align*}

(2) The set ${}_3\sS^{(2)}_4$ consists of the following $5$ spin rigid Young tableaux:
\begin{align*}
\ytableausetup{smalltableaux}
& \begin{ytableau}
*(gray!40) \cdot&  \cdot& *(gray!40) \cdot&  2 &*(gray!40) 1 \\
*(gray!40) 4 &  3
\end{ytableau}, \
\begin{ytableau}
*(gray!40) \cdot&  \cdot& *(gray!40) \cdot&  1\\
*(gray!40) 4 &  3& *(gray!40) 2
\end{ytableau}, \
\begin{ytableau}
*(gray!40) \cdot&  \cdot& *(gray!40) \cdot&  2\\
*(gray!40) 4 &  3& *(gray!40) 1
\end{ytableau}, \
\begin{ytableau}
*(gray!40) \cdot&  \cdot& *(gray!40) \cdot&  3\\
*(gray!40) 4 &  2& *(gray!40) 1
\end{ytableau}, \
\begin{ytableau}
*(gray!40) \cdot&  \cdot& *(gray!40) \cdot&  4\\
*(gray!40) 3 &  2& *(gray!40) 1
\end{ytableau}. \
\end{align*}
\end{example}

\begin{lemma}
For any $(\lambda,\mu) \in {}_s\sS^{(2)}_{m-1}$, we have
$$(\lambda,\mu) \pls{1} (m) \in {}_{s-1}\sS^{(2)}_{m} \quad \text{ and } \quad (\lambda,\mu) \pls{2} (m) \in {}_{s+1}\sS^{(2)}_{m}.$$
\end{lemma}

\begin{proof}
Recall Definition \ref{def: spin rigid}. In particular, since $k=2$, we have $m \not\equiv_2 s$. Then one can use a similar argument to that of the proof of Lemma \ref{lem: insertion hh 2}.
\end{proof}

Let $L({\omega})$ be the highest weight module with highest weight ${\omega}$ over the finite dimensional Lie algebra of type $D_n$. Recall the definition of $\tilde{\omega}_s$ in \eqref{eq: bmega D}.

\begin{theorem} \label{them: pascal ht}
Let $\eta \in \tmaxp(\Lambda|2)$ of index $(2u-1+s,s-1)$. For $s \ge 0$ and $u \ge 0$,
$$|{}_{s}\sS^{(2)}_{2u-1+s}| = \binom{2u+s -\delta_{s,0}}{u} =\dim V(\Lambda)_\eta = \dim L(\tilde{\omega}_{n-s})_{\tilde{\omega}_{n-s-2u}} .$$
\end{theorem}

\begin{proof}
With Corollary \ref{cor: ht kLambda0+Lambda_1} and the fact that
$$|{}_{s}\sS^{(2)}_{s-1}| = | \left\{ \big((0),\uplambda(s-1)) \right\}| =1,$$
one can apply a similar argument to that of the proof of Theorem \ref{them: pascal hh}.
\end{proof}

\begin{example}
From Example \ref{ex: pascal ht 1}, we see that
$$|{}_1\sS^{(2)}_4| = \binom{4+1}{2}=\binom{5}{2}=10 \quad \text{ and } \quad |{}_3\sS^{(2)}_4| = \binom{2+3}{1}=\binom{5}{1}=5.$$
Furthermore, we get $|{}_2\sS^{(2)}_5|=10+5=\binom{4+2}{2}$  from the insertion scheme:
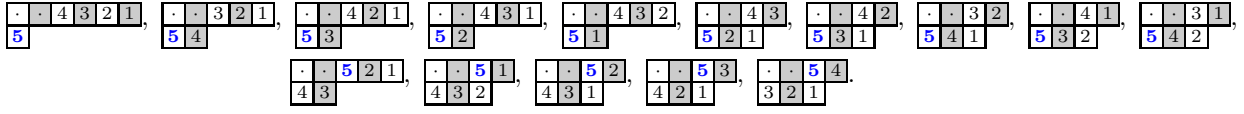
\begin{align*}
\ytableausetup{smalltableaux}
& \begin{ytableau}
\cdot & *(gray!40) \cdot&  4& *(gray!40) 3&  2 &*(gray!40) 1\\
 \fv
\end{ytableau}, \
\begin{ytableau}
\cdot & *(gray!40) \cdot&  3& *(gray!40) 2&  1  \\
 \fv & *(gray!40) 4
\end{ytableau}, \
\begin{ytableau}
\cdot & *(gray!40) \cdot&  4& *(gray!40) 2&  1  \\
 \fv & *(gray!40) 3
\end{ytableau}, \
\begin{ytableau}
\cdot & *(gray!40) \cdot&  4& *(gray!40) 3&  1  \\
 \fv & *(gray!40) 2
\end{ytableau}, \
\begin{ytableau}
\cdot & *(gray!40) \cdot&  4& *(gray!40) 3&  2  \\
 \fv & *(gray!40) 1
\end{ytableau},  \
\begin{ytableau}
\cdot & *(gray!40) \cdot&  4& *(gray!40) 3  \\
 \fv & *(gray!40) 2 &  1
\end{ytableau}, \
\begin{ytableau}
\cdot & *(gray!40) \cdot&  4& *(gray!40) 2  \\
 \fv & *(gray!40) 3 &  1
\end{ytableau}, \
\begin{ytableau}
\cdot & *(gray!40) \cdot&  3& *(gray!40) 2  \\
 \fv & *(gray!40) 4 &  1
\end{ytableau}, \
\begin{ytableau}
\cdot & *(gray!40) \cdot&  4& *(gray!40) 1  \\
 \fv & *(gray!40) 3 &  2
\end{ytableau}, \
\begin{ytableau}
\cdot & *(gray!40) \cdot&  3& *(gray!40) 1  \\
 \fv & *(gray!40) 4 &  2
\end{ytableau}, \\
& \hspace{25ex}
\begin{ytableau}
 \cdot& *(gray!40) \cdot&  \fv& *(gray!40) 2 & 1 \\
 4 & *(gray!40) 3
\end{ytableau}, \
\begin{ytableau}
 \cdot& *(gray!40) \cdot&  \fv& *(gray!40) 1\\
 4 & *(gray!40) 3&  2
\end{ytableau}, \
\begin{ytableau}
 \cdot& *(gray!40) \cdot&  \fv& *(gray!40) 2\\
 4 & *(gray!40) 3&  1
\end{ytableau}, \
\begin{ytableau}
 \cdot& *(gray!40) \cdot&  \fv& *(gray!40) 3\\
 4 & *(gray!40) 2&  1
\end{ytableau}, \
\begin{ytableau}
 \cdot& *(gray!40) \cdot&  \fv& *(gray!40) 4\\
 3 & *(gray!40) 2&  1
\end{ytableau}. \
\end{align*}
\end{example}

\section{Level $3$ weight multiplicities: Motzkin and Riordan triangles} \label{sec-level3}

As a special case $k=3$ in Theorems~\ref{thm: level k hh rigid} and \ref{thm: level k ht spin rigid}, the multiplicity of
$\eta \in \hmaxp(\Uplambda|3)$ of index $(m,s)$ is equal to the number of rigid Young tableaux
\[
 \dim(V(\Uplambda)_\eta)= |{}_s\Ss_m^{(3)}| = \dim\left(L(\omega_n+\tilde{\omega}_{n-s})_{\omega_n+\tilde{\omega}_{n-m}}\right),
\]
and the multiplicity
of $\eta \in \tmaxp(\Uplambda|3)$ of index $(m,s-1)$ is equal to the number of spin rigid Young tableaux
\[
\dim(V(\Uplambda)_\eta) = | {}_s\sS_m^{(k)}| = \dim\left(L (\omega_n+\tilde{\omega}_{n-s})_{\mu} \right),
\]
where $\mu =\omega_n + \tilde{\omega}_{n-m-1}$ if $m \not \equiv_2 s$ and $\mu =\omega_{n-1} + \tilde{\omega}_{n-m-1}$ if $m \equiv_2 s$.

In this section, we will prove that these multiplicities are equal to the generalized Motzkin numbers and the generalized Riordan numbers respectively.

\begin{theorem}
  \label{thm: motzkin}
For $m\ge s \ge 0$, we have
\[
|{}_s\Ss_m^{(3)}|=\mathsf{M}_{(m,s)}.
\]
\end{theorem}

\begin{theorem}
  \label{thm: riordan}
 For $m \ge s \ge 0$, we have
\[
  |{}_s \sS^{(3)}_m| = \mathsf{R}_{(m+1,s)}.
\]
\end{theorem}

\begin{remark} \hfill
\begin{enumerate}
\item[{\rm (1)}] Note that $ |{}_0 \sS^{(3)}_0|=0=\mathsf{R}_{(1,0)}$. For $m \ge 1$, we have proved in
 Corollary~\ref{cor: ht kLambda0+Lambda_1} that
$$ |{}_0 \sS^{(3)}_m| =  |{}_1\sS^{(3)}_{m-1}|.$$
Hence $$|{}_1 \sS^{(3)}_m| = \mathsf{R}_{(m+1,1)} = \mathsf{R}_{(m+2,0)} = |{}_0 \sS^{(3)}_{m+1}|.$$
Thus, for Theorem \ref{thm: riordan}, it is enough to prove when $s \ge 1$.
\item[{\rm (2)}] Note that $\dim L(3\omega_n)_{3\omega_n}=1=\mathsf{R}_{(0,0)}$. In \eqref{eq: Riordan multi}, we saw that $\tilde{\omega}_{n-1}+\omega_{n-1}$ is not a dominant weight of $L(3\omega_n)$.
Then Theorem \ref{thm: riordan} can be restated as $$  \mathsf{R}_{(m,s)}= \dim V(\omega_n+\tilde{\omega}_{n-s})_{\tilde{\omega}_{n-m}+\omega_{n-\delta(m \not\equiv_2 s)}} \quad \text{ for any } m \ge s \ge 0,$$
which explains the relationship with Riordan triangle better.
\end{enumerate}
\end{remark}

In  Section~\ref{subsec:RS}, we  show Theorems~\ref{thm: motzkin} and \ref{thm: riordan} using the Robinson--Schensted algorithm. In  Section~\ref{subsec:IS} we prove Theorem~\ref{thm: motzkin} using a generalization of the insertion scheme in Section~\ref{sec:level 2}.

\subsection{Proof by the RS algorithm} \label{subsec:RS}
Up until now, in this paper, we have used reverse standard Young tableaux. However, in this subsection we will consider standard Young tableaux (or SYTs
for short), which are more suitable for the usual Robinson--Schensted algorithm.

Recall that a composition $\lambda=(\lambda_1,\dots,\lambda_k)$ is called \emph{almost-even} if the number of odd parts is exactly 1 or 2. Note that for an almost-even composition $\lambda$ of $m$, the number of odd parts is $1$ if $m$ is odd, and $2$ if $m$ is even. An \emph{almost-even partition} is a partition that is almost-even when considered as a composition.

Let $\lambda=(\lambda_1,\dots,\lambda_k)$ be a partition. We say that $\lambda$  is a \emph{parity partition} if $\lambda_i \equiv_2 \lambda_j$ for all $1\le i,j\le k$.

\begin{definition} \hfill
\begin{enumerate}
\item Let $\SYT_m^{(k)}$ be the set of SYTs of shape
  $\lambda\vdash m$ for some partition $\lambda=(\lambda_1,\dots,\lambda_k)$.
\item Let ${}_s\SYT^{(k)}_{m}$ be the set of SYTs of shape
  $\lambda\skewpar (s^{k-1})\vdash m$ for some partition $\lambda
  =(\lambda_1,\dots,\lambda_k)$ of size $(m+s(k-1))$.
\item Let ${}_s\PT^{(k)}_{m}$ be the set of SYTs of shape
  $\lambda\skewpar (s^{k-1})\vdash m$ for some parity partition
  $\lambda =(\lambda_1,\dots,\lambda_k)$ of size $(m+s(k-1))$.
\item Let ${}_s\AAE^{(k)}_{m}$ be the set of SYTs of shape
  $\lambda\skewpar (s^{k-1})\vdash m$ for some partition
  $\lambda =(\lambda_1,\dots,\lambda_k)$ of size $(m+s(k-1))$ such that
$(\lambda_1-s,\dots,\lambda_{k-1}-s,\lambda_k+s)$ is almost-even.
\end{enumerate}
\end{definition}

Using the obvious bijection between the SYTs and the reverse standard Young tableaux, we obtain the following lemma.

\begin{lemma}\label{lem:BDSAE}
We have
\begin{align}
&  |{}_s \Ss^{(k)}_{m}|=|{}_s\SYT^{(k)}_{m}|-|{}_{s-1}\SYT^{(k)}_{m}|,   \label{eq:BSS}\\
&   |{}_s\sS^{(k)}_{m}|  =|{}_s\AAE^{(k)}_{m}|-|{}_{s-2}\AAE^{(k)}_{m}|,  \label{eq:DAA}
\end{align}
where we define ${}_{t}\SYT^{(k)}_{m}={}_{t}\AAE^{(k)}_{m}=\emptyset$ if $t<0$.
\end{lemma}

In order to prove Theorems~\ref{thm: motzkin} and \ref{thm: riordan}, we will find formulas for $|{}_s\SYT^{(3)}_{m}|$ and $|{}_s\AAE^{(3)}_{m}|$.
We need the following lemma which can be taken as an equivalent definition of ${}_s\AAE^{(3)}_{m}$. Notice that this lemma is not true
for ${}_s\AAE^{(k)}_{m}$ in general.

\begin{lemma}
The set ${}_s\AAE^{(3)}_{m}$ consists of the SYTs of shape
  $\lambda\skewpar (s,s)\vdash m$ for some almost-even partition
  $\lambda =(\lambda_1,\lambda_2,\lambda_3)$ of size $m+2s$.
\end{lemma}
\begin{proof}
It is sufficient to show that $(\lambda_1,\lambda_2,\lambda_3)$ is almost-even if and only if $(\lambda_1-s,\lambda_2-s,\lambda_3+s)$ is almost-even. This is trivial if $s$ is even. Suppose that $s$ is odd. Let
$t$ be the number of odd parts in $(\lambda_1,\lambda_2,\lambda_3)$.
Then the number of odd parts in $(\lambda_1-s,\lambda_2-s,\lambda_3+s)$ is $3-t$. Since $t\in\{1,2\}$ if and only if $3-t\in \{1,2\}$, we have that
$(\lambda_1,\lambda_2,\lambda_3)$ is almost-even if and only if
$(\lambda_1-s,\lambda_2-s,\lambda_3+s)$ is almost-even.
\end{proof}

Our main tool is the Robinson--Schensted algorithm. Let us first fix some notations.
  A \emph{permutation} of $\{1,2,\dots,n\}$ is a bijection $\pi:\{1,2,\dots,n\}\to \{1,2,\dots,n\}$. We denote by $\Sym_n$ the set of permutations of $\{1,2,\dots,n\}$.  As usual, we will also write a permutation $\pi\in \Sym_n$ as a word $\pi=\pi_1\pi_2\dots \pi_n$, where $\pi_i = \pi(i)$.

\begin{definition}
  An \emph{involution} is a permutation $\pi\in\Sym_n$ such that $\pi^2$ is the identity permutation $12\dots n$. We denote by $\Inv_n$ the set of involutions in $\Sym_n$. Let $\pi\in\Inv_n$. Then for every $1\le i\le n$, we have either $\pi(i) = i$ or $\pi(i)=j$ and $\pi(j)=i$ for some $j\ne i$. If $\pi(i)=i$, we call $i$ a \emph{fixed point} of $\pi$. If $\pi(i)=j$ for $i\ne j$, we say that $i$ and $j$ are \emph{connected} in $\pi$.  If there are no four integers $a<b<c<d$ such that $a$ and $d$ are connected and $b$ and $c$ are connected in $\pi$, we say that $\pi$ is \emph{non-nesting}. We denote by $\NI_n$ the set of non-nesting involutions in $\Inv_n$.
\end{definition}

\begin{definition}
  For a permutation $\pi\in\Sym_n$ and an integer
  $0\le k\le n$, we denote by $\pi_{\le k}$ the permutation in $\Sym_k$
  obtained from $\pi$ by removing every integer greater than $k$. Similarly, for  a SYT $T$ with $n$ cells and an integer $0\le
  k\le n$, we denote by $T_{\le k}$ the SYT with
  $k$ cells obtained from $T$ by removing every cell with entry
  greater than $k$.
\end{definition}

For a permutation $\pi\in \Sym_n$, let $P(\pi)$ and $Q(\pi)$ be the
insertion tableau and the recording tableau respectively via the
Robinson--Schensted algorithm. The following properties of the
Robinson--Schensted algorithm are well known, see \cite{Sagan}.

\begin{itemize}
\item The map $\pi\mapsto (P(\pi),Q(\pi))$ is a bijection from $\Sym_n$
  to the set of pairs $(P,Q)$ of SYTs of the same
  shape with $n$ cells.
\item For $\pi\in \Sym_n$, we have $P(\pi^{-1}) = Q(\pi)$. Therefore,
  the map $\pi\mapsto P(\pi)$ gives a bijection from $\Inv_n$ to the set of SYTs with $n$  cells.
\item For $\pi\in \Sym_n$ and $1\le k\le n$, we have $P(\pi_{\le
    k})=P(\pi)_{\le k}$.
\item For $\pi=\pi_1\dots\pi_n\in \Sym_n$, the number of rows of $P(\pi)$ is equal to the length of a longest decreasing subsequence of $\pi_1\dots\pi_n$.
\end{itemize}

These properties implies the following proposition.

\begin{proposition}
The map $\pi\mapsto P(\pi)$ is a bijection from $\NI_n$ to $\SYT_n^{(3)}$.
\end{proposition}

The following lemma is the main lemma in this subsection.

\begin{lemma}\label{lem:RSinv}
Let ${}_s\NI_m$ be the set of elements $\pi\in\NI_{2s+m}$ satisfying the following condition: there exists an integer $0\le t\le s$ such that
\begin{itemize}
\item $2i-1$ and $2i$ are connected in $\pi$ for all $1\le i\le t$,
\item $2j-1$ is connected to an integer greater than $2s$ and $2j$ is a fixed point for all $t+1\le j\le s$.
\end{itemize}

Let ${}_s \overline{\SYT}_m^{(3)}$ be the set of elements $T\in \SYT_{2s+m}^{(3)}$ satisfying the following condition: $T_{\le 2s}$ is the SYT of shape $(s,s)$ such that the $i$th column consists of $2i-1$ and $2i$ for all $1\le i\le s$.

Then the map $\pi\mapsto P(\pi)$ is a bijection from
${}_s\NI_m$ to ${}_s \overline{\SYT}_m^{(3)}$.
\end{lemma}
\begin{proof}
Let $\pi\in \Inv_{2s+m}$ and $T=P(\pi)\in \SYT_{2s+m}$. It is sufficient to show that $\pi\in {}_s\NI_m$ if and only if $T\in {}_s \overline{\SYT}_m^{(3)}$.

  Suppose that $\pi\in {}_s\NI_m$. Then we have
\[
T_{\le 2s} = P(\pi)_{\le 2s} = P(\pi_{\le 2s}).
\]
Since $\pi\in \Inv_{2s+m}$, we obtain that
\[
\pi_{\le 2s} = 2,1, 4,3, \dots, 2t-1, 2t, 2t+2,2t+4,\dots,2s,
2t+1,2t+3,\dots,2s-1.
\]
Then $T_{\le 2s}=P(\pi)_{\le 2s} =P(\pi_{\le 2s})$ is the desired SYT of shape $(s,s)$ and we obtain  $T\in {}_s \overline{\SYT}_m^{(3)}$.

Now suppose that $T\in {}_s \overline{\SYT}_m^{(3)}$. Let $t$ be the
largest integer such that $2i-1$ and $2i$ are connected in $\pi$ for
all $1\le i\le t$. If there is no such integer, we set $t=0$. If $t\ge s$, we are done.  Assume that $t<s$. By
the definition of $t$, we have that $2t+1$ is connected to some
integer $j>2t+2$ in $\pi$. We claim that $2t+2$ is a fixed point. For a
contradiction, suppose that $2t+2$ is connected to some integer
$r>2t+2$ in $\pi$. If $r<j$, then the four integers $2t+1<2t+2<r<j$ violate the condition for a non-nesting involution, which is a contradiction. If $r>j$, then
\[
\pi_{\le 2t+2} = 2,1,4,3,\dots,2t,2t-1, 2t+1,2t+2.
\]
The insertion tableau of this permutation is not equal to
$T_{\le 2t+2}$, which is a contradiction to
\[
P(\pi_{\le 2t+2})  = P(\pi)_{\le 2t+2} = T_{\le 2t+2}.
\]
Therefore, $2t+2$ must be a fixed point of $\pi$. Moreover, $2t+1$ is connected to an integer greater than $2s$. To see this suppose that
$2t+1$ is connected to an integer $j\le 2s$. Then $\pi_{\le 2s}$ has a decreasing sequence $j, 2t+2, 2t+1$ of length $3$. Then the insertion tableau of
$\pi_{\le 2s}$ would have at least $3$ rows and it cannot be
$T_{\le 2s}$.  Therefore, $2t+1$ must be connected to an integer
greater than $2s$. By the same argument, we can show that $2i-1$ is connected to an integer greater than $2s$ and
$2i$ is a fixed point for all $t\le i\le s$. This finishes the proof.
\end{proof}

Now we recall a well-known bijection between the non-nesting involutions and the Motzkin paths. For $\pi\in\NI_n$, let $\phi(\pi)$ be the
Motzkin path $L$ constructed as follows. If $i$ is a fixed point of $\pi$,
 the $i$th step of $L$ is a horizontal step. If $i$ and $j$ are connected in $\pi$ for $i<j$,  the $i$th step of $L$ is an up
step and the $j$th step of $L$ is a down step. It is easy to see that
$\phi$ is a bijection from $\NI_n$ to the set
of Motzkin paths of length $n$.

\begin{proposition}\label{prop:sT3}
We have
\[
|{}_s\SYT^{(3)}_{m}| = \sum_{t=0}^s \mathsf{M}_{(m,t)}.
\]
\end{proposition}
\begin{proof}
First, observe that there is a natural bijection from ${}_s\SYT^{(3)}_{m}$ to the set ${}_s\overline{\SYT}^{(3)}_{m}$ in Lemma~\ref{lem:RSinv}. Such a bijection can be constructed as follows. For $T\in {}_s\SYT^{(3)}_{m}$, let $T'$ be the SYT obtained from $T$ by increasing every entry in $T$ by $2s$ and filling the two empty cells in the $i$th column with $2i-1$ and $2i$ for all $1\le i\le s$. Thus, by Lemma~\ref{lem:RSinv}, we have
\[
|{}_s\SYT^{(3)}_{m}|=|{}_s\overline{\SYT}^{(3)}_{m}| = |{}_s\mathcal{NI}_{m}|.
\]

Now consider $\pi\in {}_s\NI_{m}$ and the corresponding Motzkin path $\phi(\pi)$ from $(0,0)$ to $(2s+m,0)$.
By definition of ${}_s\NI_m$ in Lemma~\ref{lem:RSinv}, there is an integer $0\le t\le s$ such that the first $2s$ steps of
  $\phi(\pi)$ are $(UD)^t(UH)^{s-t}$. Therefore if we take $L$ to be
  the path consisting of the first $m$ steps of the reverse path of
  $\phi(\pi)$, then $L$ is a Motzkin path from $(0,0)$ to $(m,t)$. It
  is easy to see that the map $\pi\mapsto L$ is a bijection from
  ${}_s\NI_{m}$ to the set of all Motzkin paths from $(0,0)$ to $(m,t)$ for some $0\le t \le m$.  Thus we have
\[
 |{}_s\mathcal{NI}_{m}| = \sum_{t=0}^s \mathsf{M}_{(m,t)},
\]
which completes the proof.
\end{proof}

Now we have all the ingredients to prove Theorem~\ref{thm: motzkin}.
\begin{proof}[Proof of Theorem~\ref{thm: motzkin}]
By \eqref{eq:BSS} and Proposition~\ref{prop:sT3}, we have
\[
|{}_s \Ss^{(3)}_{m}|=|{}_s\SYT^{(3)}_{m}|-|{}_{s-1}\SYT^{(3)}_{m}|=\mathsf{M}_{(m,s)}.  \qedhere
\]
\end{proof}

In order to prove Theorem~\ref{thm: riordan} we need two lemmas.

\begin{lemma} \label{lem:TPP} For  integers $m\ge0$ and $s\ge0$, we have
  \begin{align*}
  |{}_s\SYT^{(3)}_{m}| =|{}_s\PT^{(3)}_{m}|+|{}_s\AAE^{(3)}_{m}|, \ \
  |{}_s\AAE^{(3)}_{m}| =|{}_s\PT^{(3)}_{m+1}| \ \text{ and } \
  |{}_s\SYT^{(3)}_{m}| =|{}_s\PT^{(3)}_{m}|+|{}_s\PT^{(3)}_{m+1}|.
\end{align*}
\end{lemma}
\begin{proof}
For the first identity, consider a tableau $T\in {}_s\SYT^{(3)}_{m}$. Then the shape of $T$ is $\lambda=(\lambda_1,\lambda_2,\lambda_3)$ with $\lambda\skewpar(s,s)\vdash m$. It is easy to see that $\lambda$ is either a parity partition or an almost-even partition. Thus we obtain the first identity.

For the second identity, consider a tableau $T\in {}_s\AAE^{(3)}_{m}$. Then the shape of $T$ is an almost-even partition $\lambda=(\lambda_1,\lambda_2,\lambda_3)$ with $\lambda\skewpar(s,s)\vdash m$. If $m$ is even, then only one of $\lambda_1,\lambda_2,\lambda_3$ is even, and if $m$ is odd, only one of them is even. Thus, in any case, one of $\lambda_1,\lambda_2,\lambda_3$ has a different parity than the others. Suppose that $\lambda_i$ is the one with the different parity.  Let $T'$ be the tableau obtained from $T$ by increasing every entry by $1$ and add a new cell at the end of $\lambda_i$. Then $T'\in {}_s\PT^{(3)}_{m+1}$. The map $T\mapsto T'$ gives a bijection from ${}_s\AAE^{(3)}_{m}$ to ${}_s\PT^{(3)}_{m+1}$. Thus we obtain the second identity.

The third identity follows from the first two identities.
\end{proof}

\begin{lemma} \label{lem:PPR}
For integers $m\ge0$ and $s\ge1$, we have
\[
|{}_s\PT^{(3)}_{m}|-|{}_{s-2}\PT^{(3)}_{m}|=\mathsf{R}_{(m,s)}.
\]
\end{lemma}

\begin{proof}
We will prove this by induction on $m$ when $s\ge1$ is fixed.
If $m=0$, then both sides are zero. Now suppose that the statement
\begin{equation}
  \label{eq:indhyp}
  |{}_s\PT^{(3)}_{m}|-|{}_{s-2}\PT^{(3)}_{m}|=\mathsf{R}_{(m,s)}
\end{equation}
is true for $m\ge0$.
By Lemma~\ref{lem:TPP}, we have
\[
  |{}_s\SYT_{m}^{(3)}|-|{}_{s-2}\SYT_{m}^{(3)}|
  =|{}_s\PT_{m}^{(3)}|+|{}_s\PT_{m+1}^{(3)}|
  - |{}_{s-2}\PT_{m}^{(3)}|-|{}_{s-2}\PT_{m+1}^{(3)}|.
\]
By Proposition~\ref{prop:sT3} and Proposition~\ref{lem:rec_R}, we have
\[
  |{}_s\SYT_{m}^{(3)}|-|{}_{s-2}\SYT_{m}^{(3)}|
 = \mathsf{M}_{(m,s)}+\mathsf{M}_{(m,s-1)}
 = \mathsf{R}_{(m,s)}+\mathsf{R}_{(m+1,s)}.
\]
Thus,
\begin{equation}
  \label{eq:PPPPRR}
  (|{}_s\PT_{m}^{(3)}| - |{}_{s-2}\PT_{m}^{(3)}|)
  +(|{}_s\PT_{m+1}^{(3)}| -|{}_{s-2}\PT_{m+1}^{(3)}|)
  = \mathsf{R}_{(m,s)}+\mathsf{R}_{(m+1,s)}.
\end{equation}
By \eqref{eq:indhyp} and \eqref{eq:PPPPRR}, we obtain that
\[
  |{}_s\PT^{(3)}_{m+1}|-|{}_{s-2}\PT^{(3)}_{m+1}|=\mathsf{R}_{(m+1,s)}.
\]
Thus, by induction, the statement is true for all $m\ge0$.
\end{proof}

Now we give a proof of Theorem~\ref{thm: riordan}.

\begin{proof}[Proof of Theorem~\ref{thm: riordan}]
By \eqref{eq:DAA}, Lemmas~\ref{lem:TPP} and \ref{lem:PPR}, we have
\[
|{}_s\sS^{(3)}_{m}|  =|{}_s\AAE^{(3)}_{m}|-|{}_{s-2}\AAE^{(3)}_{m}|
=  |{}_s\PT^{(3)}_{m+1}|-|{}_{s-2}\PT^{(3)}_{m+1}|=\mathsf{R}_{(m+1,s)}.
\]
Thus, we have $|{}_s \sS^{(3)}_m| = \mathsf{R}_{(m+1,s)}$. Our assertion for $s=0$ follows from Corollary \ref{cor: ht kLambda0+Lambda_1}.
\end{proof}


\subsection{Proof by insertion scheme}
\label{subsec:IS}

In this subsection, we will prove that all the multiplicities of $\eta \in \hmaxp(\Uplambda|3)$ are generalized Motzkin numbers
$\mathsf{M}_{(m,s)}$ using an insertion scheme which generalizes the one in Section \ref{sec:level 2}.
Namely,
we will introduce a new kind of jeu du taquin which realizes the recursive formula \eqref{eq: Motzkin triangle recursive}:
$$\mathsf{M}_{(m,s)} = \mathsf{M}_{(m-1,s)} + \mathsf{M}_{(m-1,s-1)}+ \mathsf{M}_{(m-1,s+1)}.$$
As its corollary, we have a bijective map between $\{ {}_{s}\Ss_{m}^{(3)} \ | \  0 \le s \le m \}$ and the set of all Motzkin paths.

Note that, for $T=(\lambda,\mu,\nu) \in {}_{s}\Ss_{m}^{(3)}$,
$$ \begin{cases}
\lambda_1 =m \text{ or } \nu_1 =m & \text{ if } s>0, \\
\lambda_1 =m & \text{ if } s=0.
\end{cases}  $$

\begin{lemma} For $T=(\lambda,\mu,\nu) \in {}_{s}\Ss_{m-1}^{(3)}$, we have
$$T \pls{1} (m) \in {}_{s}\Ss_{m}^{(3)} \quad\text{ and }\quad T \pls{3} (m) \in {}_{s+1}\Ss_{m}^{(3)}.$$
\end{lemma}

\begin{proof}
In the definition of ${}_{s}\Ss_{m-1}^{(3)}$ (Definition \ref{def: rigid}), the conditions are relevant only with $\mu$ and $\nu$. Hence
$T \pls{1} (m) \in {}_{s}\Ss_{m}^{(3)}$, since $(m)*\lambda \supset \mu$ and  nothing happens to $\mu$ and $\nu$. The second assertion follows from the second assertion
of Lemma \ref{lem: insertion hh 2}.
\end{proof}

\begin{example} \label{ex: jeu de taquin preview}
The set ${}_{0}\Ss_{3}^{(3)}$ consists of four tableaux
$$ \young(321), \ \young(32,1), \ \young(31,2), \ \young(3,2,1),$$
and the set ${}_{1}\Ss_{3}^{(3)}$ has five elements
$$ \young(\cdot2,\cdot1,3), \ \young(\cdot3,\cdot1,2), \ \young(\cdot21,\cdot,3), \ \young(\cdot31,\cdot,2), \ \young(\cdot32,\cdot,1).$$
Using the operations $\pls{3} (4)$ and $\pls{1} (4)$, we get the elements in ${}_{1}\Ss_{4}^{(3)}$ from ${}_{0}\Ss_{3}^{(3)}$ and ${}_{1}\Ss_{3}^{(3)}$ as follows:
\begin{align*}
& \young(\cdot321,\cdot,\bfo), \ \young(\cdot32,\cdot1,\bfo), \ \young(\cdot31,\cdot2,\bfo), \ \young(\cdot3,\cdot2,\bfo1), \
\young(\cdot\bfo2,\cdot1,3), \ \young(\cdot\bfo3,\cdot1,2), \ \young(\cdot\bfo21,\cdot,3), \ \young(\cdot\bfo31,\cdot,2), \ \young(\cdot\bfo32,\cdot,1).
\end{align*}
\end{example}

\begin{remark} \label{rem: observation 1,3}
One can observe that an element $(\lambda,\mu,\nu) \in {}_{s}\Ss_{m}^{(3)}$ obtained from ${}_{s}\Ss_{m-1}^{(3)}$ in the above way can be distinguished
from others by the following characterization:
$$ \lambda_1=m  \quad \text{ and } \quad (\lambda_{\ge 2},\mu,\nu) \in {}_{s}\Ss_{m-1}^{(3)}.$$
Similarly, an element $(\lambda,\mu,\nu) \in {}_{s}\Ss_{m}^{(3)}$ obtained from ${}_{s-1}\Ss_{m-1}^{(3)}$ can be distinguished
from others by the following characterization:
$$ \nu_1=m  \quad \text{ and } \quad (\lambda,\mu,\nu_{\ge 2}) \in {}_{s-1}\Ss_{m-1}^{(3)}.$$

But there are elements in ${}_{s}\Ss_{m}^{(3)}$ which cannot be obtained from ${}_{s}\Ss_{m-1}^{(3)}$ or ${}_{s-1}\Ss_{m-1}^{(3)}$. For example,
there are elements in ${}_{1}\Ss_{4}^{(3)}$ which do not appear in Example \ref{ex: jeu de taquin preview}:
$$ \young(\cdot4,\cdot3,21), \ \young(\cdot4,\cdot2,31), \ \young(\cdot41,\cdot2,3).$$
\end{remark}

\begin{lemma} \label{lem: *3}
Let $T=(\lambda,\mu,\nu) \in {}_{s}\Ss_{m}^{(3)}$ with $m \ge 1$. If $\nu_1=m$, then
$$ s \ge 1 \quad  \text{ and } \quad  T = T' \pls{3} (m)  \quad \text{ for some } T' \in {}_{s-1}\Ss_{m-1}^{(3)}.$$
\end{lemma}

\begin{proof}
This assertion follows from the definition of ${}_{s}\Ss_{m}^{(k)}$ directly.
\end{proof}

Now we will construct an algorithm to get elements $(\lambda,\mu,\nu)$ of ${}_{s}\Ss_{m}^{(3)}$ from ${}_{s+1}\Ss_{m-1}^{(3)}$. By Remark \ref{rem: observation 1,3}
and Lemma \ref{lem: *3}, such an element in ${}_{s}\Ss_{m}^{(3)}$  should satisfy the following conditions:
\begin{align} \label{eq: *2}
\lambda_1=m  \quad \text{ and } \quad (\lambda_{\ge 2},\mu,\nu) \not\in {}_{s}\Ss_{m-1}^{(3)} \ (\text{ or equivalently } \ \lambda_{\ge 2} \not \supset \mu ).
\end{align}
In tableaux notation, the construction of $T=(\lambda,\mu,\nu) \in {}_{s}\Ss_{m}^{(3)}$ from $T'=(\lambda',\mu',\nu') \in {}_{s+1}\Ss_{m-1}^{(3)}$ can be understood
as filling the top-rightmost empty cell with $m$ and performing jeu de taquin to fill the empty cell right below.  For example, for given
\begin{align} \label{eq: setting T'}
 T' =\ytableausetup{boxsize=1.2em}
 \begin{ytableau}
  \cdot& \cdot & *(blue!60) {\color{yellow}\cdot} & 12 & 10 & 8& 7 \\
  \cdot& \cdot & *(blue!60) {\color{yellow}\cdot} & 11 & 9 & 1\\
  6&5&4&3&2
\end{ytableau} \in  {}_3\Ss^{(3)}_{12},
\end{align}
we put $13$ in the top blue cell
\begin{align} \label{eq: setting T}
\ytableausetup{nosmalltableaux}
 \begin{ytableau}
  \cdot& \cdot & *(blue!60) \tit & 12 & 10 & 8& 7 \\
  \cdot& \cdot & *(blue!60) {\color{yellow}\cdot} & 11 & 9 & 1\\
  6&5&4&3&2
\end{ytableau} .
\end{align}

Now we explain the jeu de taquin to fill the remaining blue cell.
\begin{algorithm} [Rigid jeu de taquin] \label{Alg: JdT1}
Assume that $T'$ is given, and fill the top-rightmost empty cell with $m$ as described above.   Take the reference point to be the empty cell in the second row.
\begin{enumerate}
\item Perform $\swarrow^{1}$ on the north-east cell in the first row and $\gets_1$ on the other cells in the first row. If the resulting tableau is standard, terminate the process; otherwise (recover the original tableau and) go to (2).
\item Perform $\upp{3}$ on the south cell in the third row and $\gets_3$ on the other cells in the third row. If the resulting tableau is standard, terminate the process; otherwise (recover the original tableau and) go to (3).
\item Perform $\gets_2$ on the east cell to switch the position of the empty cell and go to (1).
\end{enumerate}
Denote the resulting tableau by $T$. We call this process the {\it rigid jeu de taquin $($of level $3)$}.
\end{algorithm}

By applying the operation (1) of Algorithm \ref{Alg: JdT1} to \eqref{eq: setting T}, we have
\begin{align*}
\ytableausetup{nosmalltableaux}
 \begin{ytableau}
  \cdot& \cdot & *(blue!60) \tit & *(red!60)10 & *(red!60)8& *(red!60)7 \\
  \cdot& \cdot & *(yellow!50) 12 & 11 & 9 & 1\\
  6&5&4&3&2
\end{ytableau} \ .
\end{align*}
The cell $\young(\ot)$ moves from the first row to second row $\swarrow^{1}$ and the cells
$\young(\te87)$
located on the right hand side of $\young(\ot)$ are shifted by $1$ to the left $\gets_1$. Thus we shall denote the operation (1) by $\swarrow^{1}\, \gets_1$. Clearly,  the resulting tableau is not standard.

We apply the operation (2) in Algorithm \ref{Alg: JdT1} to \eqref{eq: setting T} to obtain
$$
\ytableausetup{nosmalltableaux}
 \begin{ytableau}
  \cdot& \cdot & *(blue!60) \tit & 12 & 10 & 8& 7 \\
  \cdot& \cdot & *(yellow!50) 4 &  11 & 9 & 1\\
  6&5&*(red!60)3&*(red!60)2
\end{ytableau} \ .
$$
The cell $\young(4)$ moves from the third row to second row $\upp{3}$ and the cells
$\young(32)$
located on the right hand side of $\young(4)$ are shifted by $1$ to the left $\gets_3$. Thus we shall denote the operation (2) by
 $\upp{3}\, \gets_3$. The resulting tableau is not standard either.

Now perform the operation (3) in Algorithm \ref{Alg: JdT1} to \eqref{eq: setting T} and obtain
\begin{equation} \label{eqn-1311}
\ytableausetup{nosmalltableaux}
 \begin{ytableau}
  \cdot& \cdot & *(blue!60) \tit & 12 & 10 & 8& 7 \\
  \cdot& \cdot & *(blue!60) {\color{yellow}\cdot} & 11 & 9 & 1\\
  6&5&4&3&2
\end{ytableau}
\quad
\leadsto
\quad
\ytableausetup{nosmalltableaux}
 \begin{ytableau}
  \cdot& \cdot & *(blue!60) \tit & 12 & 10 & 8& 7 \\
  \cdot& \cdot & *(yellow!50)11 & *(blue!60) {\color{yellow}\cdot} & 9 & 1\\
  6&5&4&3&2
\end{ytableau} \ .
\end{equation}

One can easily see that neither of the operations (1) and (2) performed on the new tableau in \eqref{eqn-1311} produces a standard tableau. Thus we perform the operation (3) to obtain
\begin{equation} \label{eqn-129}
\ytableausetup{nosmalltableaux}
 \begin{ytableau}
  \cdot& \cdot & *(blue!60) \tit & 12 & 10 & 8& 7 \\
  \cdot& \cdot & 11  & *(yellow!50)9 & *(blue!60) {\color{yellow}\cdot} & 1\\
  6&5&4&3&2
\end{ytableau} \ .
\end{equation}

Now we perform the operation (1) on the tableau \eqref{eqn-129} and obtain
$$
\ytableausetup{nosmalltableaux}
 \begin{ytableau}
  \cdot& \cdot & *(blue!60) \tit & 12 & 10 & 7 \\
  \cdot& \cdot & 11  & 9 & *(blue!60) {\color{yellow}8} & 1\\
  6&5&4&3&2
\end{ytableau}
\in {}_2\Ss^{(3)}_{13},
$$
which is standard. In this way, we have obtained a tableau $T \in {}_2\Ss^{(3)}_{13}$ from $T' \in {}_3\Ss^{(3)}_{12}$.

Clearly, the process terminates in finite steps, and one can check that the resulting tableau $T$ in {\rm Algorithm \ref{Alg: JdT1} satisfies the conditions in \eqref{eq: *2}
and is contained in ${}_{s}\Ss_{m}^{(3)}$.
Furthermore, we can construct the reverse of the rigid jeu de taquin easily.

\begin{algorithm}[Reverse rigid jeu de taquin] \label{Alg: RJDT}
Assume that  $T=(\lambda,\mu,\nu) \in {}_{s}\Ss_{m}^{(3)}$ satisfies \eqref{eq: *2}. Remove $m$ from its cell. Take the reference point to be the leftmost non-empty cell, say $\mathfrak c$, in the second row.
\begin{enumerate}
\item Perform $\rightarrow_3$ on the cells in the third row from the rightmost cell all the way to the south cell of $\mathfrak c$, and $\dw{2}$ on the cell $\mathfrak c$, and $\rightarrow_2$ on the cells, if any, which were at the left-side of $\mathfrak c$.  If the resulting tableau is standard, terminate the process; otherwise (recover the original tableau and) go to (2).
\item Perform $\rightarrow_1$ on the cells in the first row from the rightmost cell all the way  to the northeast cell of $\mathfrak c$,  and $\nearrow_2$ on the cell $\mathfrak c$, and $\rightarrow_2$ on the cells, if any, which were at the left-side of $\mathfrak c$.  If the resulting tableau is standard, terminate the process; otherwise (recover the original tableau and) go to (3).
\item Take the east cell to be new $\mathfrak c$ for the next round, and make it the reference point, and go to (1).
\end{enumerate}
Denote the resulting tableau by $T'$. We call this process the {\it reverse rigid jeu de taquin $($of level $3)$}.
\end{algorithm}

One can check that the resulting tableau $T'$ in {\rm Algorithm \ref{Alg: RJDT}} is contained in ${}_{s+1}\Ss_{m-1}^{(3)}$. It is also easy to see that Algorithm \ref{Alg: RJDT} is an inverse process of Algorithm \ref{Alg: JdT1}.

\begin{example}
 For a given
\begin{align} \label{eq: setting T' again}
 T =\ytableausetup{nosmalltableaux}
 \begin{ytableau}
  \cdot & 13 & 12 & 10 & 7 &5 \\
  \cdot & 11 & 9 & 8 & 1\\
  6&4&3&2
\end{ytableau} \in  {}_1\Ss^{(3)}_{13},
\end{align}
one can check that it satisfies the conditions in \eqref{eq: *2}.
Now we delete $13$.
$$
\begin{ytableau}
  \cdot & *(yellow!50) \cdot & 12 & 10 & 7 &5 \\
  \cdot & *(blue!60){\color{yellow}11} & 9 & 8 & 1\\
  6&4&3&2
\end{ytableau}
$$
Since $\nu_1 =6 < 11=\mu_1$, (1) in Algorithm \ref{Alg: RJDT} fails, and since
 $\mu_1 =11 < 12=\lambda_1$, (2) fails. Hence we apply {\rm (3)}  to change the reference point (in blue color):
$$
\begin{ytableau}
  \cdot & \cdot & 12 & 10 & 7 &5 \\
  \cdot & 11 & *(blue!60){\color{yellow}9} & 8 & 1\\
  6&4&3&2
\end{ytableau}
$$
As (1) and (2) fail again, we apply (3) to obtain
$$
\begin{ytableau}
  \cdot & \cdot & 12 & 10 & 7 &5 \\
  \cdot & 11 & 9 & *(blue!60){\color{yellow}8} & 1\\
  6&4&3&2
\end{ytableau}
$$
Now (2) works to produce a standard tableau:

$$
T' = \begin{ytableau}
  \cdot & \cdot & 12 & 10 &  *(red!60)8 &  *(red!60)7 &  *(red!60)5 \\
  \cdot & \cdot & *(yellow!50)11 & *(yellow!50)9 & 1\\
  6&4&3&2
\end{ytableau} \in {}_2\Ss^{(3)}_{12}.
$$
To check that it is an inverse process, we add $13$ again and see:
$$
\begin{ytableau}
  \cdot & *(blue!60){\color{yellow}13} & 12 & 10 & 8 &  7 &5 \\
  \cdot & \cdot & 11 & 9 & 1\\
  6&4&3&2
\end{ytableau} \leadsto
\begin{ytableau}
  \cdot & 13 & 12 &  10 &  8 &  7 &5 \\
  \cdot & *(yellow!50) 11 & *(yellow!50)9 &  *(blue!60) \cdot & 1\\
  6&4&3&  2
\end{ytableau}
\leadsto
\begin{ytableau}
  \cdot & 13 & 12 &  10  &  *(yellow!50) 7 & *(yellow!50)5 \\
  \cdot & 11 & 9 & *(red!60)  8 &   1\\
  6&4&3& 2
\end{ytableau} =T.
$$
\end{example}

\begin{theorem} \label{thm-bi-mo}
The rigid-type jeu de taquin gives a bijection between $${}_{s+1}\Ss_{m-1}^{(3)} \quad \text{ and } \quad
{}_{s}\Ss_{m}^{(3)} \setminus \left( {}_{s}\Ss_{m-1}^{(3)}\pls{1} (m) \bigsqcup  {}_{s-1}\Ss_{m-1}^{(3)}\pls{3} (m) \right).$$
\end{theorem}

\begin{proof}
Our assertion follows from Algorithm \ref{Alg: JdT1} and Algorithm \ref{Alg: RJDT} which are inverses to each other.
\end{proof}

Now we give another proof of Theorem~\ref{thm: motzkin}.

\begin{proof}[Proof of Theorem~\ref{thm: motzkin}]

From Theorem \ref{thm-bi-mo}, we have
$$|{}_{s+1}\Ss_{m-1}^{(3)}|+|{}_{s}\Ss_{m-1}^{(3)}|+|{}_{s-1}\Ss_{m-1}^{(3)}| = |{}_{s}\Ss_{m}^{(3)}|,$$ which is the same as  \eqref{eq: Motzkin triangle recursive}. Since we have
$|{}_{m}\Ss_{m}^{(3)}|=1,$ we are done.
\end{proof}

\begin{corollary}
We have a bijective map between ${}_{s}\Ss_{m}^{(3)}$ and $\mathfrak M_{(m,s)}$
where $\mathfrak M_{(m,s)}$ is the set of Motzkin paths ending at $(m,s)$
\end{corollary}

\begin{proof} Assume that we have $T \in {}_{s}\Ss_{m}^{(3)}$.
For each step of the reverse rigid jeu de taquin (removing the cell $\young(m)$), we record the vector $v_m$ as
\begin{itemize}
\item $(1,0)$ if $T = T' \pls{1} (m)$ for some $T'\in {}_{s}\Ss_{m-1}^{(3)}$,
\item $(1,1)$ if $T = T' \pls{3} (m)$ for some $T'\in {}_{s-1}\Ss_{m-1}^{(3)}$,
\item $(1,-1)$ if $T$ can be obtained from $T'\in {}_{s+1}\Ss_{m-1}^{(3)}$.
\end{itemize}
Then, by induction on $m$, we obtain the sequence of vectors corresponding to a Motzkin path.
\end{proof}

\begin{example}  \label{exa-exa}
For
$$
\ytableausetup{smalltableaux}
T=(\lambda,\mu,\nu) = \begin{ytableau}
  \cdot& \cdot &  \cdot & 12 & 10 & 8& 7 \\
  \cdot& \cdot &  \cdot & 11 & 9 & 1\\
  6&5&4&3&2
\end{ytableau} \in  {}_3\Ss^{(3)}_{12}
$$
we see $\nu_1 \ne 12$ and
$$
\ytableausetup{smalltableaux}
\begin{ytableau}
  \cdot& \cdot &  \cdot & 10 & 8& 7 \\
  \cdot& \cdot &   \cdot & 11 & 9 & 1\\
  6&5&4&3&2
\end{ytableau} \not \in  {}_3\Ss^{(3)}_{11}.
$$
Hence $v_{12}=(1,-1)$ and $T$ can be obtained from
$$
\ytableausetup{smalltableaux}
T'=\begin{ytableau}
  \cdot& \cdot& \cdot &  \cdot & 11 & 10 &  8& 7 \\
  \cdot& \cdot& \cdot &  \cdot & 9 & 1\\
  6&5&4&3&2
\end{ytableau} \in  {}_4\Ss^{(3)}_{11}.
$$
Now we have
\begin{align*}
 \underset{(1,0)}{\gets}
& \ytableausetup{smalltableaux}
\begin{ytableau}
  \cdot& \cdot& \cdot &  \cdot & 10 &  8& 7 \\
  \cdot& \cdot& \cdot &  \cdot & 9 & 1\\
  6&5&4&3&2
\end{ytableau}
\underset{(1,-1)}{\gets}
\begin{ytableau}
  \cdot& \cdot& \cdot &  \cdot &  \cdot & 9 &  8& 7 \\
  \cdot& \cdot& \cdot &  \cdot &  \cdot  & 1\\
  6&5&4&3&2
\end{ytableau}
\underset{(1,0)}{\gets}
\begin{ytableau}
  \cdot& \cdot& \cdot &  \cdot &  \cdot &  8& 7 \\
  \cdot& \cdot& \cdot &  \cdot &  \cdot  & 1\\
  6&5&4&3&2
\end{ytableau}
\underset{(1,0)}{\gets}
\begin{ytableau}
  \cdot& \cdot& \cdot &  \cdot &  \cdot &  7 \\
  \cdot& \cdot& \cdot &  \cdot &  \cdot  & 1\\
  6&5&4&3&2
\end{ytableau}
 \underset{(1,-1)}{\gets}
\begin{ytableau}
  \cdot& \cdot& \cdot &  \cdot &  \cdot &  \cdot \\
  \cdot& \cdot& \cdot &  \cdot &  \cdot & \cdot \\
  6&5&4&3&2& 1
\end{ytableau}
 \underset{(1,1)^6}{\gets}
 \emptyset.
\end{align*}
Thus $T$ corresponds to the Motzkin path given below:
\begin{align*}
\begin{tikzpicture}[x=0.8cm,y=0.8cm, scale=0.7]
\foreach \x in {0,2,4,6,8,10,12}
\draw[shift={(\x-1,0)},color=black, ] (0pt,2pt) -- (0pt,-2pt) node[below] {\footnotesize $(\x,0)$};
\draw[line width=1 pt, ->] (-1,0)--(12,0);
\draw[orange, line width=1.5pt] (-1,0)--(5,6)--(6,5)--(8,5)--(9,4)--(10,4)--(11,3);
\foreach \y in {1,2,3,4,5,6}
    \draw[dashed](-1,\y)--(12,\y);
\end{tikzpicture}
\end{align*}
\end{example}

\begin{remark}
In \cite{Eu}, Eu constructed a bijection between ${}_{0}\Ss_{m}^{(3)}$ and $\mathfrak M_{(m,0)}$.
His bijection gives paths different from those obtained by our bijection.
\end{remark}

\section{Some level $k$ weight multiplicities when $k\to\infty$: Bessel triangle} \label{sec:Bessel}

In this section we will compute level $k$ weight multiplicities
$|_{s}\Ss^{(k)}_{m}|$ and $|_{s}\sS^{(k)}_{m}|$ when $k$ is as large as $m$ (or $m/2$).
Recall that we have $_0\Ss^{(k)}_{m}= \Ss^{(k)}_{m}$ and $_0\sS^{(k)}_{m} =\sS^{(k)}_{m}$. Let $\RYT_m$ be the set of reverse SYTs with $m$ cells and $\mathcal S_m$ be the set of SYT with $m$ cells.

First, observe that if $k\ge m$, the set $\Ss^{(k)}_{m}$ is the same as the set $\RYT_m$.  Since $|\SYT_m|$ is equal to the number of involutions in  $\Sym_m$, we have
\begin{equation}
  \label{eq:Binfty}
\mathsf{B}_m^{(\infty)} \seteq \lim_{k \to \infty}|\Ss^{(k)}_m|=|\RYT_m|=|\SYT_m|=\sum_{s=0}^{\lfloor m/2 \rfloor}
\binom{m}{2s}(2s-1)!!,
\end{equation}
where $(2s-1)!!=1\cdot 3 \cdots (2s-1)$.
Similarly, if $k\ge m$, the set $\sS^{(k)}_m$ becomes the set of Young tableaux with $m$ cells that have exactly one or two rows of odd length depending on the parity of $m$. Using a well known property of the Robinson--Schensted algorithm we can deduce that $\lim_{k \to \infty} |\sS^{(k)}_m|$ is the number of involutions in $\Inv_m$ with one or two fixed points.

In Section~\ref{subsec:Dmk} we find formulas for $|\sS^{(k)}_{2m}|$ when $k\ge m-1$ and for $|\sS^{(k)}_{2m-1}|$ when $k\ge m-2$. Our formulas (Theorems~\ref{thm:D2m-1k} and \ref{thm:D2mk}) imply that
\begin{equation}
  \label{eq:Dinfty}
\mathsf{D}_m^{(\infty)} \seteq \lim_{k \to \infty}|\Ae_{m}^{(k)}| =
\begin{cases}
\quad\quad m!! & \text{ if $m$ is odd}, \\
 \dfrac{m}{2} \times (m-1)!! & \text{ if $m$ is even}.
\end{cases}
\end{equation}
In Section~\ref{subsec:sBmk} we find a formula for
$|{}_s\Ss^{(k)}_m|$ when $k\ge m-s$ and compute the limit of
$|{}_s\Ss^{(k)}_m|$ as $k\to\infty$.
In Section~\ref{subsec:sDmk} we find a formula for
$|{}_s\sS^{(k)}_m|$ when $k\ge m-s+1$ and compute the limit of
$|{}_s\sS^{(k)}_m|$ as $k\to\infty$.

\subsection{The limit of $|\Ae^{(k)}_{m}|$ when $k \to \infty$}
\label{subsec:Dmk}

The following lemma is well-known (\cite[Exercise 3.12]{Sagan}). Here we identify a reverse standard Young tableau with a standard Young tableau using the obvious bijection.

\begin{lemma} \label{lem:schensted}
The Robinson--Schensted algorithm gives a bijection between the set of Young tableaux
of $n$ cells with $k$ columns of odd length and the set of involutions
of $\{1,2,\dots,n\}$ with $k$ fixed points.
\end{lemma}

Let $I(m,k)$ denote the number of involutions of $\{1,2,\dots,m\}$ with
$k$ fixed points. It is easy to see that
$$
I(2m,0) = I(2m-1,1) =(2m-1)!!, \qquad
I(2m,2) = m \cdot I(2m,0) = m (2m-1)!!.
$$

\begin{theorem}\label{thm:D2m-1k}
For an odd integer $2m-1$ and any $k \ge m$,
$$|\Ae^{(k)}_{2m-1}| = (2m-1)!!.$$
\end{theorem}
\begin{proof}
Since $k \ge m$, any Young tableau of $2m-1$ cells has at most $m-1$ (nonzero) rows of  even length. Thus $\Ae^{(k)}_{2m-1}$ is the set of Young tableaux of $2m-1$ cells with exactly one row of odd length and there is no restriction on the number of rows. By taking the
conjugate, this number is also equal to the number of Young tableaux of $2m-1$ cells with exactly one column of odd length. By
Lemma~\ref{lem:schensted}, this is equal to the number of involutions
of $\{1,2,\dots,2m-1\}$ with one fixed point. Thus we get
$|\Ae^{(k)}_{2m-1}|  = I(2m-1,1)=(2m-1)!!$.
\end{proof}

\begin{theorem}\label{thm:D2mk}
For an even integer $2m$ and any $k \ge m+1$,
$$|\Ae^{(k)}_{2m}| = m (2m-1)!!=\dfrac{(2m)!}{(m-1)!2^{m}}.$$
\end{theorem}
\begin{proof}
This can be shown by the same argument as in the proof of the previous theorem.
\end{proof}

\begin{corollary} \label{cor: reduction} For each $m$,
$$|\Ae^{(m-1)}_{2m-1}|=(2m-1)!!-\mathsf{C}_m.$$
\end{corollary}

\begin{proof}
Note that
$$\Ae^{(m)}_{2m-1} \setminus \Ae^{(m-1)}_{2m-1} = \RYT^{\lambda}$$
where $\lambda=(2,2,\ldots,2,1)\vdash 2m-1$. Since $|\RYT^{\lambda} |=f^\lambda = \mathsf{C}_m$, our assertion follows.
\end{proof}

By applying the same strategy as in Corollary \ref{cor: reduction}, we have the following corollary:

\begin{corollary} For each $m$, we have
\begin{enumerate}
\item $|\Ae^{(m)}_{2m}|=m (2m-1)!!-3\dfrac{2m!}{(m-1)!(m+2)!}.$
 \item $|\Ae^{(m-2)}_{2m-1}|=(2m-1)!!-\mathsf{C}_m-\dfrac{(2m-1)!}{m!(m-3)!}-\matr{2m-1}{m+1}.$
\item $|\Ae^{(m-1)}_{2m}|=m (2m-1)!!-3\dfrac{2m!}{(m-1)!(m+2)!}-\dfrac{4}{m+2} \times \dfrac{(2m-1)!}{m!(m-2)!}.$
\end{enumerate}
\end{corollary}

Since $\Ss^{(k)}_m$ and $\Ae_{m}^{(k)}$ can be understood as special cases of
$_{s}\Ss^{(k)}_{m}$ and $_{s}\sS^{(k)}_{m}$ respectively, in the next two subsections we will investigate
$$\lim_{k \to \infty}|_{s}\Ss^{(k)}_{m}| \quad \text{ and } \quad \lim_{k \to \infty}|_{s}\sS^{(k)}_{m}|.$$

\subsection{The limit of $|_{s}\Ss^{(k)}_{m}|$  when $k \to \infty$}
\label{subsec:sBmk}

\begin{proposition}\label{prop:9.6}
Let $k\ge m-s+2$. Then
$$|_{s}\Ss^{(k)}_{m}| = \binom{m}{s} \times \mathsf{B}^{(\infty)}_{m-s},$$
where $\mathsf{B}^{(\infty)}_{m}$ is the number defined in \eqref{eq:Binfty}.
\end{proposition}
\begin{proof}
Let $T\in {}_{s}\Ss^{(k)}_{m}$. Since the $k$th row of $T$ has at least $s$ cells, the first $k-1$ rows can have at most $m-s$ cells. Since $m-s\le k-2$, the $(k-1)$st row must be empty. Thus the $k$th row of $T$ has exactly $s$ cells. Such a tableau can be constructed by selecting $s$ integers from $\{1,2,\dots,m\}$ for the $k$th row and filling the remaining $m-s$ integers in a Young diagram so that the entries are increasing in each row and column. The number of ways to do this is equal to $\binom{m}{s} \times \mathsf{B}^{(\infty)}_{m-s}$.
\end{proof}

\begin{remark}
By similar arguments, one can show the following identities:
\[
|_{s}\Ss^{(m-s+1)}_{m}| = \binom{m}{s} \times \mathsf{B}^{(\infty)}_{m-s}
-\binom{m-1}{s-1} \ \ \text{ and } \ \
|_{s}\Ss^{(m-s)}_{m}| = \binom{m}{s} \times \mathsf{B}^{(\infty)}_{m-s}
- \binom{m-1}{s-1}(m-s-1).
\]
\end{remark}

\begin{corollary}\label{cor:B_infty}
  For positive integers $s \le m$,
$${}_s\mathsf{B}^{(\infty)}_m \seteq \lim_{k \to \infty}|_{s}\Ss^{(k)}_{m}| = \binom{m}{s} \times \mathsf{B}^{(\infty)}_{m-s}.$$
\end{corollary}

The triangular array consisting of $\{ {}_s \mathsf{B}^{(\infty)}_m  \}$ is given as follows:
\begin{align} \label{eq: n,s involution}
\begin{array}{cccccccccccc}
&&&&&\iddots&\iddots&\iddots&\iddots&\vdots\\
&&&&1&5&30&140&700&\cdots \\
&&&1&4&20&80&350&1456&\cdots \\
&&1&3&12&40&150&546&2128&\cdots \\
&1&2&6&16&50&156&532&1856&\cdots \\
1&1&2&4&10&26&76&232&764&\cdots
\end{array}
\end{align} where the bottom row is the number of involutions in $\Sym_m$.

\subsection{The limit of $|_{s}\sS^{(k)}_{m}|$  when $k \to \infty$} \label{subsec:sDmk}

\begin{theorem} \label{thm: k infty ht not mod 2}
Assume that we have a pair of positive integers $2 \le s \le m$ satisfying $s \not\equiv_2 m$. Then, for $k\ge m-s+3$, we have
\[
|_{s}\sS^{(k)}_{m}| = \binom{m+1}{s}  \times (m-s)!!.
\]
Therefore, we have a closed formula for the limit as follows:
\begin{align} \label{eq: Bessel}
{}_s\mathsf{D}^{(\infty)}_m \seteq \displaystyle\lim_{k \to \infty}|_{s}\sS^{(k)}_{m}|= \binom{m+1}{s}  \times (m-s)!!.
\end{align}
\end{theorem}
\begin{proof}
Let $ T \in {}_{s}\sS^{(k)}_{m}$. By the same arguments as in the proof of Proposition~\ref{prop:9.6}, the $k$th row of $T$ has $s-1$ or $s$ cells.
Now we consider the two cases separately.

(1) The $k$th row of $T$ has $s$ cells. Let $T'$ be the tableau obtained from the first $k-1$ rows of $T$ by relabeling the integers with $1,2,\dots,m-s$ with respect to their relative order. Then $T'$ is
an almost even tableau of the odd number $m-s$. The number of such tableaux $T'$ is $\mathsf{D}^{(\infty)}_{m-s} = (m-s)!!$. Since we can select the entries in the $k$th row of $T$ freely, there are $\binom ms$ ways to do this. Thus, the number of tableaux $T$ in this case is
$\binom ms (m-s)!!$.

(2) The $k$th row of $T$ has $s-1$ cells. Let $T'$ be the tableau obtained from the first $k-1$ rows of $T$ by relabeling the integers with $1,2,\dots,m-s+1$ with respect to their relative order. Then
all the rows of $T'$ have even length. By the same arguments as in the proof of Theorem~\ref{thm:D2m-1k},  the number of such tableaux $T'$ is equal to $I(m-s+1,0)=(m-s)!!$, the number of fixed-point free involutions.
Similarly to the first case, the number of tableaux $T$ in this case is
$\binom{m}{s-1}(m-s)!!$.

By the above two cases, we have
\[
|_{s}\sS^{(k)}_{m}| =
\binom{m}{s}(m-s)!! + \binom{m}{s-1}(m-s)!!
=\binom{m+1}{s}(m-s)!!.  \qedhere
\]
\end{proof}

\begin{theorem} \label{thm: k infty ht mod 2}
Assume that a given pair of positive integers $2 \le s \le m$ satisfies $s \equiv_2 m$. Then for a $k\ge m-s+3$, we have
\[
|_{s}\sS^{(k)}_{m}| = \binom{m}{s} \times \mathsf{D}^{(\infty)}_{m-s}+\binom{m}{s-1}\times \mathsf{D}^{(\infty)}_{m-s+1},
\]
where $\mathsf{D}^{(\infty)}_{m}$ is given in \eqref{eq:Dinfty}.
Therefore, we have a closed formula for the limit as follows:
\begin{align} \label{eq: spin-Bessel}
 {}_s\mathsf{D}^{(\infty)}_m\seteq \lim_{k\to\infty}
|_{s}\sS^{(k)}_{m}| = \binom{m}{s} \times \mathsf{D}^{(\infty)}_{m-s}+\binom{m}{s-1}\times \mathsf{D}^{(\infty)}_{m-s+1}.
\end{align}
\end{theorem}

\begin{proof}
The proof is almost identical to the proof of Theorem \ref{thm: k infty ht not mod 2}.
\end{proof}

The closed formula \eqref{eq: Bessel} is known to compute the triangular array consisting of coefficients of Bessel polynomials (\cite[A001497]{OEIS}):
\begin{equation} \label{BT} \begin{array}{cccccccccccc}
& &&&&&&&\iddots&\iddots&\cdots \\
& &&&&&&1&36&990&\cdots \\ 
& &&&&&1&28&630&13860&\cdots \\ 
& &&&&1&21&378&6930&135135&\cdots\\ 
& &&&1&15&210&3150&51975&945945&\cdots \\ 
& &&1&10&105&1260&17325&270270&4729725&\cdots \\ 
& &1&6&45&420&4725&62370&945945&16216200&\cdots \\ 
& 1&3&15&105&945&10395&135135&2027025&34459425&\cdots \\ 
1&1&3&15&105&945&10395&135135&2027025&34459425&\cdots
\end{array}
\end{equation}
where the lowest two rows are $\mathsf{D}^{(\infty)}_{2m-1} = (2m-1)!!$.
We call this triangular array {\em Bessel triangle}.

\section{Standard Young tableaux with a fixed number of rows of odd length}
\label{sec:YT odd}

In this section we consider SYTs with a fixed number of rows of odd length.
We denote by $\SYT_m$ the set of SYTs with $m$ cells.
Recall that $\SYT_m^{(k)}$ is the set of SYTs with $m$ cells and at most $k$ rows, and that there is an obvious bijection from $\SYT_m^{(k)}$ to $\mathfrak B_m^{(k)}$. The main objects in this section are the sets $\SYT_m^{(k)}$ and their subsets $\SYT_m^{(k,t)}$ defined  below.

\begin{definition}
For $0\le t\le k$, we denote by $\SYT_m^{(k,t)}$ the set of SYTs with $m$ cells, at most $k$ rows and exactly $t$ rows of odd length.
\end{definition}

Observe that by the obvious bijection between SYTs and reverse standard Young tableaux, we have
\begin{equation}
  \label{eq:SD}
|\sfS_{2m}^{(k,2)}| = |\Ae_{2m}^{(k)}| \quad \text{ and } \quad |\sfS_{2m-1}^{(k,1)}| = |\Ae_{2m-1}^{(k)}|.
\end{equation}
Thus, $|\sfS_{m}^{(k,t)}|$ can be thought of as a generalization of $|\Ae_{m}^{(k)}|$. In this section, we study the cardinalities of
$\sfS_{m}^{(k)}$ and $\sfS_{m}^{(k,t)}$.

In Section~\ref{subsec:Smkt}, we express $|\SYT_m^{(k)}|$ in terms of $|\SYT_i^{(k,0)}|$ and $|\SYT_i^{(k,k)}|$ (Proposition~\ref{prop:kk}). Using this relation and some known results, we find an explicit formula
for $\sfS_{m}^{(k,t)}$ for every $0\le t\le k\le 5$ (Theorem~\ref{thm:card_S}). In Section~\ref{subsec:trace}, we express $|\SYT_m^{(k)}|$ as an integral over the orthogonal group $\OG(k)$ with respect to the normalized Haar measure (Theorem~\ref{thm:1-Tr}).  In Section~\ref{sec:evaluation-integrals}, we evaluate  this integral to find an explicit formula for $|\sfS^{(k)}_m|= |\mathfrak B_m^{(k)}|$ (Theorem~\ref{thm:Selberg}).

\subsection{The cardinality of $\sfS_{m}^{(k,t)}$ for $0\le t\le k\le 5$}
\label{subsec:Smkt}

In this subsection we give an explicit formula for $\sfS_{m}^{(k,t)}$
for every $0\le t\le k\le 5$.  Note that $\sfS_{m}^{(k,t)}=\emptyset$ if $m \not\equiv_2 t$.
Since it is trivial for $k=0,1$, we consider $k\ge2$. Recall that
\[
\mathsf{R}_m= \dfrac{1}{m+1}\displaystyle\sum_{i=1}^{ \lfloor m/2 \rfloor }\binom{m+1}{i}\binom{m-i-1}{i-1}.
\]

\begin{theorem}\label{thm:card_S}
We have a formula for $|\sfS_{m}^{(k,t)}|$ for $0\le t\le k\le 5$ as follows:
\begin{eqnarray*}
&\text{ For $k=2$, } &
|\sfS^{(2,0)}_{2m}|  =|\sfS^{(2,1)}_{2m-1}| =|\sfS^{(2,2)}_{2m}|= \binom{2m-1}{m}. \allowdisplaybreaks\\
&\text{ For $k=3$, } &|\sfS^{(3,0)}_{2m}| = |\sfS^{(3,1)}_{2m-1}| = |\Ae_{2m-1}^{(3)}|=\mathsf{R}_{2m}, \allowdisplaybreaks\\
&& |\sfS^{(3,2)}_{2m}| = |\sfS^{(3,3)}_{2m-1}| = |\Ae_{2m}^{(3)}|=\mathsf{R}_{2m+1}. \allowdisplaybreaks\\
&\text{ For $k=4$, } &
|\sfS^{(4,0)}_{2m}|  =|\sfS^{(4,1)}_{2m-1}| = |\Ae_{2m-1}^{(4)}|= \binom{\mathsf{C}_m+1}2, \allowdisplaybreaks\\
&& |\sfS^{(4,2)}_{2m}| = |\Ae_{2m}^{(4)}|=\mathsf{C}_m\mathsf{C}_{m+1} - \mathsf{C}_{m}^2=\dfrac{3(2m)!^2}{(m-1)!m!(m+1)!(m+2)!}, \allowdisplaybreaks\\
&& |\sfS^{(4,3)}_{2m-1}| = |\sfS^{(4,4)}_{2m}| = \binom{\mathsf{C}_m}2. \allowdisplaybreaks\\
&\text{ For $k=5$, } &
|\sfS^{(5,0)}_{2m}|  = |\sfS^{(5,1)}_{2m-1}| = |\Ae_{2m-1}^{(5)}|=\sum_{i=0}^{m} \binom{2m}{2i} \mathsf{C}_i\mathsf{C}_{i+1} - \sum_{i=0}^{m-1} \binom{2m}{2i+1}\mathsf{C}_{i+1}^2, \qquad \qquad \qquad \qquad  \allowdisplaybreaks\\
&& |\sfS^{(5,2)}_{2m}| = |\Ae_{2m}^{(5)}|=\sum_{i=0}^{m} \frac{2i}{i+3}\binom{2m}{2i} \mathsf{C}_i\mathsf{C}_{i+1}- \sum_{i=0}^{m-1} \frac{2i}{i+3}\binom{2m}{2i+1} \mathsf{C}_{i+1}^2, \allowdisplaybreaks\\
&& |\sfS^{(5,3)}_{2m-1}| =\sum_{i=0}^{m-1} \frac{2i}{i+3}\binom{2m-1}{2i} \mathsf{C}_i\mathsf{C}_{i+1} - \sum_{i=0}^{m-1} \frac{2i}{i+3}\binom{2m-1}{2i+1} \mathsf{C}_{i+1}^2, \allowdisplaybreaks\\
&& |\sfS^{(5,4)}_{2m}| = |\sfS^{(5,5)}_{2m-1}| = \sum_{i=0}^{m-1} \binom{2m-1}{2i}\mathsf{C}_i \mathsf{C}_{i+1}- \sum_{i=0}^{m-1}\binom{2m-1}{2i+1} \mathsf{C}_{i+1}^2.
\end{eqnarray*}
\end{theorem}

Before proving this theorem we first find some relations between
the numbers $|\sfS^{(k,t)}_m|$ and $|\sfS^{(k)}_m|$.

\begin{lemma}\label{lem:kk1}
We have
  \[
|\sfS^{(k,k)}_m| = |\sfS^{(k,k-1)}_{m-1}| \quad \text{ and } \quad
|\sfS^{(k,0)}_m| = |\sfS^{(k,1)}_{m-1}|.
\]
\end{lemma}
\begin{proof}
The map deleting the cell with $m$ gives a bijection from $\sfS^{(k,k)}_m$ to $\sfS^{(k,k-1)}_{m-1}$. The same map also gives a bijection from $\sfS^{(k,0)}_m$ to $\sfS^{(k,1)}_{m-1}$.
\end{proof}

The next lemma is the key lemma in this section. The proof is based on the Robinson--Schensted algorithm and a sign-reversing involution. Recall that an SYT is a filling of a Young diagram $\lambda\vdash m$ with integers $1,2,\dots,m$. We need to extend this definition to a \emph{partial SYT} which is a filling of a Young diagram with distinct integers such that the entries are increasing in each row and each column.

\begin{lemma}\label{lem:kk2}
For integers $k\ge1$ and $m\ge0$, we have
\[
|\sfS^{(k,0)}_m|-|\sfS^{(k,k)}_m|
= \sum_{i=0}^m (-1)^{m-i} \binom mi |\sfS^{(k-1)}_i|.
\]
\end{lemma}
\begin{proof}
Let $X$ be the set of pairs $(T,A)$ of a partial SYT $T$ and a
subset $A$ of $\{1,2,\dots,m\}$ such that $T$ has at most $k-1$ rows
and the set of entries of $T$ is
$\{1,2,\dots,m\}\setminus A$. Then we have
\[
\sum_{i=0}^m (-1)^{m-i} \binom mi |\sfS^{(k-1)}_i|
=\sum_{(T,A)\in X} (-1)^{|A|}.
\]

We define $Y$ to be the set of pairs $(P,H)$ of an SYT $P$ and a
sequence $H=(t_1,t_2,\dots,t_k)$ such that
\begin{itemize}
\item $P$ has at most $k$ rows, and
\item if $\lambda=(\lambda_1,\dots,\lambda_k)$ is the shape of $P$
  (some $\lambda_i$ can be zero),
  then
$0\le t_i\le \lambda_i-\lambda_{i+1}$ for all $1\le i\le k-1$ and
$t_k=\lambda_k$.
\end{itemize}
Note that if $\mu=(\mu_1,\dots,\mu_k)$ is defined by $\mu_i=\lambda_i-t_i$ for $1\le i\le k$, then the second condition above means that $\mu\subset\lambda$ and $\lambda\skewpar\mu$ is a skew partition whose Young diagram contains at most one cell in each column. Such a skew partition is called a \emph{horizontal strip}. By identifying the sequence $H$ and the skew partition $\lambda\skewpar\mu$, one can consider $H$ as a horizontal strip of $P$ which contains all cells in row $k$ of $P$.

We claim that there is a bijection from $X$ to $Y$ such that if $(T,A)\in X$ corresponds to $(P,H)\in Y$, then $|A| = t_1+t_2+\cdots + t_k$.  For $(T,A)\in X$, let $P$ be the SYT obtained from $T$ by inserting the elements of $A$ in increasing order via the Robinson--Schensted algorithm and $H=(t_1,\dots,t_k)$ be the sequence of integers such that $t_i$ is the number of newly added cells in row $i$. In other words, if $\Sh(P)=\lambda=(\lambda_1,\dots,\lambda_k)$ and $\Sh(T)=\mu=(\mu_1,\dots,\mu_k)$, then $t_i=\lambda_i-\mu_i$. It is well known that if $i<j$ and $i$ is inserted to a partial SYT $T$ and $j$ is inserted to the resulting tableau via the Robinson--Schensted algorithm, then the newly added cell after inserting $j$ is strictly to the right of the newly added cell after inserting $i$. This property implies that $\lambda\skewpar\mu$ is a horizontal strip and the cells in it have been added from left to right. Therefore, we can recover $(T,A)$ from $(P,H)$ using the inverse map of the Robinson--Schensted algorithm and this proves the claim.

By the above claim, we have
\[
\sum_{(T,A)\in X} (-1)^{|A|} = \sum_{(P,H)\in Y} (-1)^{t_1+\dots+t_k}.
\]

Now we define a map $\phi$ on $Y$ as follows. Suppose that $(P,H)\in Y$ and the shape of $P$ is $\lambda=(\lambda_1,\dots,\lambda_k)$ and $H=(t_1,\dots,t_k)$.  Find the smallest $i\le k-1$ such that $t_i$ is an odd integer or $t_i$ is an even integer less than $\lambda_i-\lambda_{i-1}$. In this case we define $\phi(P,H)=(P,H')$, where $H'=(t_1',\dots,t_k')$ is obtained from $H$ by replacing $t_i$ by $t_i-1$ if $t_i$ is odd and by $t_i+1$ if $t_i$ is even.  If there is no such integer $i$, then we define $\phi(P,H)=(P,H)$. It is easy to see that $\phi$ is an involution on $Y$ such that if $\phi(P,H)=(P,H')$ and $H\ne H'$, then $(-1)^{t_1+\dots+t_k} = -(-1)^{t_1'+\dots+t_k'}$. Moreover, if $\phi(P,H)=(P,H)$, then $t_i=\lambda_i-\lambda_{i+1}$ is even for all $1\le i\le k-1$. This can happen only if $P\in \sfS^{(k,0)}_m$ or $P\in\sfS^{(k,k)}_m$.  If $\phi(P,H)=(P,H)$ for $P\in \sfS^{(k,0)}_m$, then $(-1)^{t_1+\dots+t_k}=(-1)^{t_k}=1$.  If $\phi(P,H)=(P,H)$ for $P\in \sfS^{(k,k)}_m$, then $(-1)^{t_1+\dots+t_k}=(-1)^{t_k}=-1$.  Therefore, $\phi$ is a sign-reversing involution and we have \[ \sum_{(P,H)\in Y} (-1)^{t_1+\dots+t_k} = |\sfS^{(k,0)}_m|-|\sfS^{(k,k)}_m|, \] which finishes the proof.
\end{proof}

Applying the principle of inclusion and exclusion to Lemma~\ref{lem:kk2}, we obtain the following proposition.

\begin{proposition}\label{prop:kk}
For integers $k\ge1$ and $m\ge0$, we have
\[
|\sfS^{(k-1)}_m| = \sum_{i=0}^m \binom mi
\left(|\sfS^{(k,0)}_i|-|\sfS^{(k,k)}_i|\right) .
\]
\end{proposition}

Now we prove Theorem~\ref{thm:card_S}.

\begin{proof}[Proof of Theorem~\ref{thm:card_S}]

  We have already proved the formulas for $k=2$ in \eqref{eq: almost even at most 2} and for $k=3$ in Proposition~\ref{prop:Dm3}.  Now we consider the cardinality of $\sfS^{(k,t)}_{m}$ for $k=4$.

Recall that we have a formula for $|\sfS^{(4)}_{m}|=|\Ss^{(4)}_m|$ in Theorem~\ref{thm: Standard at most}:
\begin{align*}
|\sfS^{(4)}_{2m}|=\mathsf{C}_m\mathsf{C}_{m+1} \quad \text{ and } \quad |\sfS^{(4)}_{2m-1}| =\mathsf{C}_m\mathsf{C}_{m}.
\end{align*}

Since $2m$ is even,
\begin{equation}
  \label{eq:1}
|\sfS^{(4,0)}_{2m}| +|\sfS^{(4,2)}_{2m}| +|\sfS^{(4,4)}_{2m}|
=|\sfS^{(4)}_{2m}| =\mathsf{C}_m \mathsf{C}_{m+1}.
\end{equation}
By Lemma~\ref{lem:kk1}, we have
\begin{equation}
  \label{eq:2}
|\sfS^{(4,0)}_{2m}| +|\sfS^{(4,4)}_{2m}| = |\sfS^{(4,1)}_{2m-1}| + |\sfS^{(4,3)}_{2m-1}|
=|\sfS^{(4)}_{2m-1}| =\mathsf{C}_m^2.
\end{equation}
By Lemma~\ref{lem:kk2}, we have
\begin{equation}
  \label{eq:3}
|\sfS^{(4,0)}_{2m}| - |\sfS^{(4,4)}_{2m}|
 = \sum_{i=0}^{2m} (-1)^i\binom{2m}{i} |\sfS^{(3)}_{i}| = \mathsf{C}_m.
\end{equation}
In \eqref{eq:3}, we used the fact that $|\sfS^{(3)}_{i}|=\mathsf{M}_i$ and
\[
 \sum_{i=0}^{2m} (-1)^i\binom{2m}{i} \mathsf{M}_i = \mathsf{C}_m,
\]
which can be obtained from the following identity using
inclusion-exclusion:
\[
\mathsf{M}_m = \sum_{i=0}^{\lfloor m/2 \rfloor} \binom{m}{2i} \mathsf{C}_i.
\]

By \eqref{eq:1}, \eqref{eq:2} and \eqref{eq:3}, we obtain the formulas
for $|\sfS^{(4,0)}_{2m}|$, $|\sfS^{(4,2)}_{2m}|$ and
$|\sfS^{(4,4)}_{2m}|$. By Lemma~\ref{lem:kk1}, we obtain the formulas
for $|\sfS^{(4,1)}_{2m-1}|$ and $|\sfS^{(4,3)}_{2m-1}|$.

Now we consider the cardinality of $\sfS^{(k,t)}_{m}$ for $k=5$.
First, we have
\begin{equation}
  \label{eq:11}
|\sfS^{(5,0)}_{2m}| +|\sfS^{(5,2)}_{2m}| +|\sfS^{(5,4)}_{2m}|
=|\sfS^{(5)}_{2m}| \quad \text{ and } \quad
|\sfS^{(5,1)}_{2m-1}| +|\sfS^{(5,3)}_{2m-1}| +|\sfS^{(5,5)}_{2m-1}|
=|\sfS^{(5)}_{2m-1}|.
\end{equation}
By Lemma~\ref{lem:kk1}, we have
\begin{equation}
  \label{eq:13}
|\sfS^{(5,0)}_{2m}| = |\sfS^{(5,1)}_{2m-1}| \quad \text{ and } \quad
|\sfS^{(5,5)}_{2m+1}| = |\sfS^{(5,4)}_{2m}|.
\end{equation}
By Lemma~\ref{lem:kk2}, we have
\begin{equation}
\begin{aligned}
& |\sfS^{(5,0)}_{2m}| - |\sfS^{(5,5)}_{2m}| = |\sfS^{(5,0)}_{2m}|
 = \sum_{i=0}^{2m} (-1)^i\binom{2m}{i} |\sfS^{(4)}_{i}|, \\
&|\sfS^{(5,0)}_{2m-1}| - |\sfS^{(5,5)}_{2m-1}| = -|\sfS^{(5,5)}_{2m-1}|
 = \sum_{i=0}^{2m-1} (-1)^i\binom{2m-1}{i} |\sfS^{(4)}_{i}|.
 \end{aligned}
\end{equation}
By solving the above equations, we obtain the desired formulas.
\end{proof}

\subsection{Traces of orthogonal matrices}
\label{subsec:trace}

There is an interesting integral representation of the number $|\sfS_{2m}^{(k,0)}|$ as follows, see Example 2 on page 423 in \cite{Mac}: \begin{equation}
  \label{eq:mac}
\int_{\OG(k)} \trace(X)^m d\mu(X) = |\sfS_{m}^{(k,0)}|.
\end{equation}
Here, the integral is taken with respect to the normalized Haar
measure $\mu$ on the orthogonal group $\OG(k)$ consisting of $k \times k$
orthogonal matrices. Note that if $m$ is odd, we have
$|\sfS_{m}^{(k,0)}|=0$.
Thus, by \eqref{eq:SD} and Lemma~\ref{lem:kk1}, we have
\begin{equation}
  \label{eq:mac2}
|\Ae^{(k)}_{2m-1}| = |\sfS_{2m}^{(k,0)}|=\int_{\OG(k)} \trace(X)^{2m} d\mu(X).
\end{equation}
In this subsection we show that $|\sfS_{m}^{(k,k)}|$ and
$|\sfS^{(k)}_m|$ also have similar integral representations.

For a symmetric function $f(x_1,\dots,x_k)$ with $k$ variables and $X\in \OG(k)$, we define $f(X)$ by $f(X)=f(e^{i\theta_1},\dots,e^{i\theta_k})$, where $e^{i\theta_1},\dots,e^{i\theta_k}$ are the eigenvalues of $X$.  Note that $\trace(X ^m)=p_m(X)$, where $p_m(x_1,\dots,x_k)=x_1^m+\cdots+x_k^m$ is the $m$-th power sum symmetric function.

We need the following known result, see \cite[pp.420--421]{Mac}:
\begin{equation}
  \label{eq:4}
\int_{\OG(k)} s_\lambda(X)d\mu(X) = \left\{
  \begin{array}{ll}
    1 & \mbox{if every part of $\lambda$ is even,}\\
0 & \mbox{otherwise,}
  \end{array}\right .
\end{equation}
where $s_\lambda$ is the Schur function.

\begin{proposition}\label{prop:Smkk}
We have
\[
|\sfS_{m}^{(k,k)}|=\int_{\OG(k)} \det(X) {\rm Tr}(X)^{m} d\mu(X).
\]
\end{proposition}
\begin{proof}
Note that
\[
\trace(X)^{m} = p_1(X)^{m} = \sum_{\lambda\vdash m, \,\ell(\lambda)\le k} f^\lambda s_\lambda(X),
\]
where $ f^\lambda $ is the number of standard Young tableaux of shape $\lambda$.
Since
\[
x_1\dots x_ks_{\lambda}(x_1,\dots,x_k) = s_{\lambda+(1^k)}(x_1,\dots,x_k)
\]
for $\lambda$ with at most $k$ rows, we have $\det(X) s_\lambda(X)=s_{\lambda+(1^{k})}(X)$. Thus,
\[
\int_{\OG(k)} \det(X) \trace(X)^{m} d\mu(X)
= \sum_{\lambda\vdash m, \, \ell(\lambda)\le k}  f^\lambda
\int_{\OG(k)} s_{\lambda+(1^{k})}(X) d\mu(X).
\]
By \eqref{eq:4}, this is equal to $|\sfS_{m}^{(k,k)}|$.
\end{proof}

Now we give an integral expression for the number SYTs with $m$ cells and  at most $k$ rows.

\begin{theorem}\label{thm:1-Tr}
For integers $k,m\ge0$, we have
\[
|\mathfrak B_m^{(k)}|=|\sfS^{(k)}_m| = \int_{\OG(k+1)} (1-\det(X))(1+{\rm Tr}(X))^{m} d\mu(X).
\]
\end{theorem}
\begin{proof}
By Proposition~\ref{prop:kk},
\[
|\sfS^{(k)}_m| = \sum_{i=0}^{m} \binom{m}{i} \left( |\sfS^{(k+1,0)}_{i}|
- |\sfS^{(k+1,k+1)}_{i}| \right).
\]
By \eqref{eq:mac} and Proposition~\ref{prop:Smkk}, we have
\begin{multline*}
|\sfS^{(k)}_m| = \sum_{i=0}^{m} \binom{m}{i} \left( \int_{\OG(k+1)} {\rm Tr}(X)^{i} d\mu(X)
- \int_{\OG(k+1)} \det(X){\rm Tr}(X)^{i} d\mu(X) \right)\\
= \int_{\OG(k+1)} (1-\det(X))
\left(\sum_{i=0}^{m} \binom{m}{i} {\rm Tr}(X)^{i}\right) d\mu(X).
\end{multline*}
We then obtain the desired identity using the binomial theorem.
\end{proof}

\subsection{Evaluation of integrals}
\label{sec:evaluation-integrals}

In this subsection we obtain an explicit formula for the number of SYTs with $m$ cells and at most $k$ rows by evaluating the integral in Theorem~\ref{thm:1-Tr}.
For the reader's convenience we recall a well-known fact on the normalized Haar measure on the orthogonal
group $\OG(k)$ due to Weyl \cite{Weyl}, see also \cite[Remarks 3 on p. 57]{DM}.

For any orthogonal matrix $A\in\OG(n)$,  the eigenvalues of $A$ lie on the unit circle.
Let $P_n(e^{i\theta_1},e^{i\theta_2},\dots,e^{i\theta_n})$ be the
probability that a random matrix $A\in\OG(n)$ has the given
eigenvalues $e^{i\theta_1},e^{i\theta_2},\dots,e^{i\theta_n}$ for
$\theta_1,\dots,\theta_n\in [0,2\pi)$. Here,
we assume that $A$ is selected randomly with respect to the normalized Haar measure.
Then this probability is given as follows.

\begin{proposition}\label{prop:Haar}
 For $k\ge1$, $\epsilon\in\{1,-1\}$ and $\theta_1,\dots,\theta_k\in [0,\pi]$ we have
\begin{align*}
  P_{2k}(e^{\pm i\theta_1},e^{\pm i\theta_2},\dots,e^{\pm i\theta_k})
&= \frac{2^{k^2-2k+1}}{\pi^k k!} \prod_{1\le r<s\le k}(\cos\theta_r-\cos\theta_s)^2,\\
  P_{2k+2}(\pm1, e^{\pm i\theta_1},e^{\pm i\theta_2},\dots,e^{\pm i\theta_{k}})
&= \frac{2^{k^2-1}}{\pi^k k!} \prod_{t=1}^{k}(1-\cos^2\theta_t)
\prod_{1\le r<s\le k}(\cos\theta_r-\cos\theta_s)^2,\\
P_{2k+1}(\epsilon, e^{\pm i\theta_1},e^{\pm i\theta_2},\dots,e^{\pm i\theta_{k}})
&= \frac{2^{k^2-k-1}}{\pi^k k!} \prod_{t=1}^{k}(1-\epsilon\cos\theta_t)
\prod_{1\le r<s\le k}(\cos\theta_r-\cos\theta_s)^2.
\end{align*}
\end{proposition}

We denote by $\OG_+(k)$ (resp.~$\OG_-(k)$) the set of matrices
$A\in\OG(k)$ with $\det(A)=1$ (resp.~$\det(A)=-1$).

Now we give an explicit formula for $|\sfS^{(k)}_m|$.

\begin{theorem}\label{thm:Selberg}
For $k\ge1$ and $m\ge0$, we have
\[
|\sfS^{(2k)}_m|
=\sum_{t_1+\dots+t_k=m}\binom{m}{t_1,\dots,t_k}
\det\left(
\binom{t_i+2k-i-j}{\lfloor \frac{t_i+2k-i-j}2 \rfloor}
\right)_{i,j=1}^k,
\]
\[
|\sfS^{(2k+1)}_m|
=\sum_{t_0+t_1+\dots+t_k=m}\binom{m}{t_0,t_1,\dots,t_k}
\det\left(
C\left(\frac{t_i+2k-i-j}2\right)
\right)_{i,j=1}^k,
\]
where $C(x)=\frac{1}{x+1}\binom{2x}x$ if $x$ is an integer and $C(x )=0$ otherwise.
\end{theorem}

\begin{proof}
By Theorem~\ref{thm:1-Tr} and Proposition~\ref{prop:Haar}, we have
  \begin{align*}
|\sfS^{(2k)}_m|&=2\int_{\OG_-(2k+1)} (1+\mathrm{Tr}(X))^m d\mu(X)\\
&=\frac{2^{k^2-2k}}{\pi^k k!} \int_{[0,\pi]^k}(2\cos\theta_1+\dots+2\cos\theta_k)^m
\prod_{1\le r<s\le k}(\cos\theta_r-\cos\theta_s)^2
\prod_{i=1}^{k}(1+\cos\theta_i)  d\theta_i\\
&=\frac{2^{k^2-2k+m}}{\pi^k k!}
\sum_{t_1+\dots+t_k=m}\binom{m}{t_1,\dots,t_k}
\int_{[0,\pi]^k} \det(x_i^{t_i+k-j})_{i,j=1}^k\det(x_i^{k-j})_{i,j=1}^k
\prod_{i=1}^{k}(1+\cos\theta_i)  d\theta_i\\
&=\frac{2^{k^2-2k+m}}{\pi^k k!}
\sum_{t_1+\dots+t_k=m}\binom{m}{t_1,\dots,t_k}
\sum_{\sigma,\tau\in \Sym_n} \mathrm{sgn}(\sigma) \mathrm{sgn}(\tau)
\int_{[0,\pi]^k} x_i^{t_i+2k-\sigma(i)-\tau(i)}
\prod_{i=1}^{k}(1+\cos\theta_i)  d\theta_i,
  \end{align*}
where $x_i=\cos\theta_i$. When $\sigma\in \Sym_n$ is fixed, since $t_i$'s are symmetric, we can replace $t_i$ by $t_{\sigma(i)}$. We can also replace $\tau$ by $\tau\sigma$. Then the resulting summand is independent of $\sigma$. Thus, we obtain
\[
|\sfS^{(2k)}_m|=\frac{2^{k^2-2k+m}}{\pi^k}
\sum_{t_1+\dots+t_k=m}\binom{m}{t_1,\dots,t_k}
\sum_{\tau\in \Sym_n} \mathrm{sgn}(\tau)
\prod_{i=1}^{k} \int_0^\pi x^{t_i+2k-i-\tau(i)}
(1+\cos\theta)  d\theta.
\]
By expressing the second sum as a determinant and using the fact
\[
\int_0^{\pi} \cos^n \theta(1+\cos \theta) d\theta = \frac{\pi}{2^n} \cb{n},
\]
we obtain the desired formula. The second formula can be proved similarly.
\end{proof}

In the literature there is an explicit formula for $|\sfS^{(k)}_m|$ for $k\le 5$.
As a corollary of Theorem~\ref{thm:Selberg}, we obtain a double-sum formula for $|\sfS^{(6)}_m|$.

\begin{corollary} \label{Cor: Sm6}
Letting $\gamma_n =  \cb{n}$, we have
\[
|\sfS^{(6)}_m| = \sum_{i+j+k=m} \binom{m}{i,j,k}
\det
\begin{pmatrix}
\gamma_{i+4}&\gamma_{i+3}&\gamma_{i+2}\\
\gamma_{j+3}&\gamma_{j+2}&\gamma_{j+1}\\
\gamma_{k+2}&\gamma_{k+1}&\gamma_{k}
\end{pmatrix}.
\]
\end{corollary}

There is another way to compute $|\sfS^{(k)}_m|$ using symmetric functions due to Gessel~\cite[Section~6]{Gessel}. It would be interesting to find a connection between his result and Theorem~\ref{thm:Selberg}.  Eu et al.~\cite{Eu2} found a bijection between $\sfS^{(k)}_m$
and the set of certain colored Motzkin paths.

We also note that the integrals in the proof of Theorem~\ref{thm:Selberg} are Selberg-type integrals, see \cite{FW}.  There is a combinatorial interpretation for Selberg-type integrals, see \cite[Exercise 1.10 (b)]{EC2}. Recently, a connection between SYTs and the Selberg integral was found in \cite{KO}.  There is also a combinatorial interpretation for a $q$-analog of the Selberg integral, see \cite{KS}.  It would be interesting to study the combinatorial aspects of the formulas in Theorem~\ref{thm:Selberg} and their $q$-analogs.

\end{document}